\RequirePackage{silence}
\immediate\write18{makeindex \jobname.nlo -s nomencl.ist -o \jobname.nls}
\documentclass[a4paper, reqno, oneside]{book}

\usepackage[UKenglish]{babel}
\usepackage{amsmath, amssymb, amsthm}
\usepackage{scrextend}
\usepackage[hidelinks]{hyperref}
\usepackage{xpatch}
\usepackage{color}
\usepackage{tikz-cd}
\usepackage{enumitem}
\usepackage{theoremref}
\usepackage{imakeidx}
\usepackage[refpage]{nomencl}
\usepackage[toc]{appendix}
\usepackage{mdframed}
\usepackage{ifthen}
\usepackage{xspace}

\makeindex[title=Index of terms]
\makenomenclature
\renewcommand{\nomname}{Index of symbols}

\newcommand{\term}[1]{\emph{#1}\index{#1}}

\usepackage[inline,final]{showlabels}
\showlabels{thlabel}

\makeindex[name=notes,title=Notes throughout the document (red text),columns=1]

\setlist[description]{font=\normalfont\scshape}

\xpatchcmd{\proof}{\itshape}{\normalfont\bfseries}{}{}
\newtheoremstyle{repeat}{}{}{\itshape}{}{\bfseries}{.}{.5em}{#3, repeated}

\newtheorem{theorem}{Theorem}[section]
\newtheorem{proposition}[theorem]{Proposition}
\newtheorem{lemma}[theorem]{Lemma}
\newtheorem{corollary}[theorem]{Corollary}

\newtheorem{claim}{Claim}[theorem]

\theoremstyle{definition}
\newtheorem{definition}[theorem]{Definition}
\newtheorem{remark}[theorem]{Remark}
\newtheorem{convention}[theorem]{Convention}
\newtheorem{example}[theorem]{Example}

\theoremstyle{repeat}
\newtheorem*{repeated-theorem}{Repeat}

\newcommand{\E}{\mathcal{E}}
\newcommand{\F}{\mathcal{F}}

\renewcommand{\L}{\mathcal{L}}

\newcommand{\MM}{\mathfrak{M}}

\newcommand{\U}{\mathcal{U}}

\newcommand{\N}{\mathbb{N}}
\newcommand{\Z}{\mathbb{Z}}
\newcommand{\Q}{\mathbb{Q}}
\newcommand{\R}{\mathbb{R}}

\DeclareMathOperator{\tp}{tp}

\DeclareMathOperator{\Aut}{Aut}
\renewcommand{\S}{\operatorname{S}}
\DeclareMathOperator{\Th}{Th}
\DeclareMathOperator{\dom}{dom}
\renewcommand{\d}{\operatorname{d}}

\DeclareMathOperator{\Diag}{\operatorname{Diag}}
\DeclareMathOperator{\Obs}{Obs}
\DeclareMathOperator{\Mor}{Mor}

\DeclareMathOperator{\cf}{cf}
\DeclareMathOperator{\cod}{cod}
\DeclareMathOperator{\linspan}{span}

\makeatletter
\newcommand{\dotminus}{\mathbin{\text{\@dotminus}}}

\newcommand{\@dotminus}{%
  \ooalign{\hidewidth\raise1ex\hbox{.}\hidewidth\cr$\m@th-$\cr}%
}
\makeatother

\renewcommand{\phi}{\varphi}
\newcommand{\equivls}{\equiv^\textup{Ls}}
\newcommand{\op}{{\textup{op}}}
\newcommand{\pc}{{\textup{pc}}}
\renewcommand{\u}{{\textup{u}}}
\newcommand{\qf}{{\textup{qf}}}
\newcommand{\fo}{{\textup{fo}}}
\newcommand{\eq}{{\textup{eq}}}
\newcommand{\heq}{{\textup{heq}}}

\newcommand{\pos}{{\textup{pos}}}

\newcommand{\OP}{$\mathsf{OP}$\xspace}
\newcommand{\TP}[1][]{\ifthenelse{\equal{#1}{}}{$\mathsf{TP}$}{$\mathsf{TP_{#1}}$}\xspace}
\newcommand{\SOP}[1][]{\ifthenelse{\equal{#1}{}}{$\mathsf{SOP}$}{$\mathsf{SOP_{#1}}$}\xspace}
\newcommand{\IP}{$\mathsf{IP}$\xspace}

\newcommand{\NTP}[1][]{\ifthenelse{\equal{#1}{}}{$\mathsf{NTP}$}{$\mathsf{NTP_{#1}}$}\xspace}
\newcommand{\NSOP}[1][]{\ifthenelse{\equal{#1}{}}{$\mathsf{NSOP}$}{$\mathsf{NSOP_{#1}}$}\xspace}

\def\Ind#1#2{#1\setbox0=\hbox{$#1x$}\kern\wd0\hbox to 0pt{\hss$#1\mid$\hss}
\lower.9\ht0\hbox to 0pt{\hss$#1\smile$\hss}\kern\wd0}
\def\ind{\mathop{\mathpalette\Ind{}}}
\def\Notind#1#2{#1\setbox0=\hbox{$#1x$}\kern\wd0\hbox to 0pt{\mathchardef
\nn="3236\hss$#1\nn$\kern1.4\wd0\hss}\hbox to 0pt{\hss$#1\mid$\hss}\lower.9\ht0
\hbox to 0pt{\hss$#1\smile$\hss}\kern\wd0}
\def\nind{\mathop{\mathpalette\Notind{}}}

\title{Positive Logic:\\An Introduction for Model Theorists}
\author{Mark Kamsma}
\date{\today}
 
\begin{document}

\maketitle

\chapter*{Abstract}
\thispagestyle{empty}
\phantomsection\addcontentsline{toc}{chapter}{Abstract}

Positive logic is a generalisation of full first-order logic that does not have negation built in. Still, many model-theoretic ideas, tools and techniques work perfectly fine in positive logic. Importantly, there is a compactness theorem. With some care, many classical results hold in the generality of positive logic without giving up any strength.

In these self-contained notes we give an introduction to model theory in positive logic. We give a complete treatment of the basics of positive model theory and then we move on to deeper model-theoretic concepts. First, we discuss countable categoricity, where we work towards a theorem that characterises countably categorical positive theories. After that, we briefly discuss how the convenient formalism of monster models goes through in positive logic as usual. This is helpful in the remainder of the notes, where we discuss simple and stable theories. The main aim in those chapters is to develop dividing independence and prove Kim-Pillay style theorems. For a smoother treatment we assume thickness, which is the relatively mild assumption that being an indiscernible sequence is type-definable. We finish by discussing two big applications of positive logic: hyperimaginaries and continuous logic. For the former we define an $(-)^\heq$ construction, analogous to the $(-)^\eq$ construction for imaginaries in full first-order logic. Where the $(-)^\heq$ construction is problematic in full first-order logic, it does stay within the framework in positive logic and it preserves many nice properties. For the latter we explain how continuous logic can be studied as a special case of positive logic, making it so that all abstract model-theoretic results in positive logic apply to continuous theories.

In the appendix we provide a quick guide to the material covered in these notes, including very brief proof sketches.

\setcounter{tocdepth}{1} 
\tableofcontents

\newpage
\thispagestyle{empty}

\begin{center}
\vspace*{7cm}
\Large{\textit{In memoriam\\Jaap}}
\end{center}

\newpage
\chapter*{Acknowledgements}
\thispagestyle{empty}
\addcontentsline{toc}{chapter}{Acknowledgements}

I would like to thank Jan Dobrowolski, Jonathan Kirby, Rosario Mennuni and Alberto Miguel-G\'{o}mez for their feedback on these notes.

While writing these notes, the author was supported by Marie Sk\l{}odowska-Curie grant number 101130801.
\begin{flushright}
    \includegraphics[height=1.5cm]{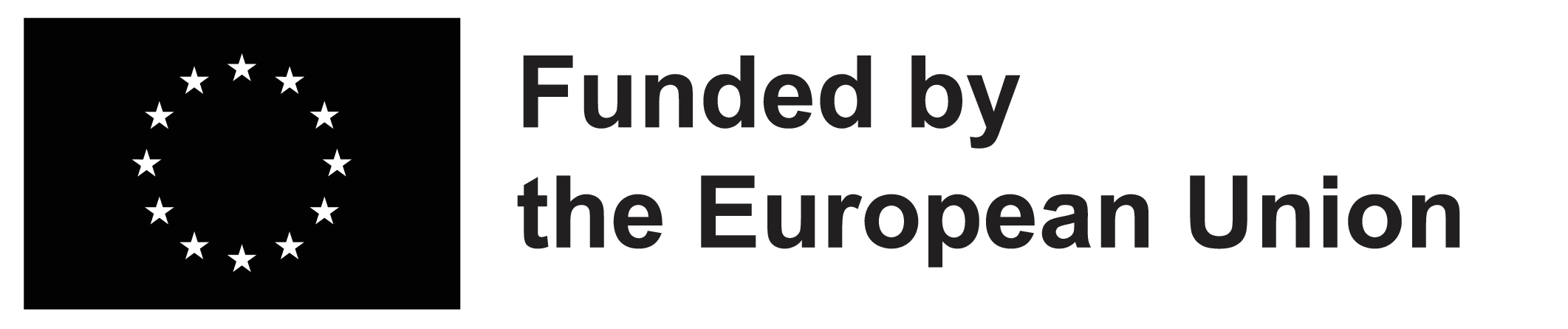}
\end{flushright}

\section*{Feedback}
Any feedback is more than welcome! Please send it to \href{mailto:mark@markkamsma.nl}{mark@markkamsma.nl}.

\chapter{Introduction}
\label{ch:introduction}
Positive logic is a generalisation of full first-order logic that does not have negation built in. That is, the allowed connectives are $\top$, $\bot$, $\wedge$ and $\vee$, and only existential quantification $\exists$ is allowed. This is a generalisation, because any desired amount of negation can be added back in through a process called Morleyisation.

Many model-theoretic ideas, tools and techniques still work perfectly fine in positive logic. Importantly, there is a compactness theorem. With some care, many classical results hold in the generality of positive logic without giving up any strength.

A notable example of the use of positive logic is that hyperimaginaries can be added to the monster model, similar to how we can add imaginaries in full first-order logic without any issue. Another important example is that positive logic also subsumes continuous logic, in the sense that any monster model of a continuous theory can be viewed as the monster model of a positive theory, allowing abstract results in positive logic to be applied in continuous logic. Both these (classes of) examples are discussed in these notes.

Only basic knowledge of model theory is assumed. Some remarks are aimed at a more advanced audience that is familiar with the full first-order version of whatever is discussed. The later chapters (simplicity and stability, Chapters \ref{ch:simple-theories} and \ref{ch:stable-theories}) technically require no prerequisite knowledge, but might be lacking in motivation if the reader is not familiar with simplicity and stability in full first-order logic.

\section{The purpose of these notes}
\label{sec:why-these-notes}
The main purpose of these notes is to provide an introduction to positive logic, and to present an overview of some of the deep model theory that can be done, and has been done, in positive logic. This is not a survey of all the work in positive logic to date. We chose to treat a positive version of what is usually referred to as the Ryll-Nardzewski theorem (a characterisation of countably categorical theories) and a positive treatment of simplicity and stability.

The main results that we present are not original. The purpose is not to present new results, but to present a self contained introduction to and overview of positive model theory. Even though these results are already present in literature, they are spread out over multiple papers, sometimes using different terminology and notation. By consolidating this existing body of work, we can give a simpler and smoother presentation. With less technical statements the hope is that this makes positive model theory accessible to a larger audience.
\section{Overview}
\label{sec:overview}
We give a brief summary of each chapter.
\begin{enumerate}
\setcounter{enumi}{1}
\item In this chapter we treat the basics of positive logic. We start by discussing the syntax and fundamental notions such as homomorphisms, immersions, positively closed models (p.c.\ models) and, importantly, compactness. We then discuss familiar constructions from full first-order logic, such as types, amalgamation and the downward L\"owenheim-Skolem theorem. We finish the chapter with a hierarchy of tameness properties that positive theories can enjoy, namely: being Boolean, Hausdorff, semi-Hausdorff or thick.
\item The goal of this chapter is to prove a positive version of what is usually referred to as the Ryll-Nardzewski theorem, a characterisation of countably categorical theories. We prove an omitting types theorem along the way and introduce positively saturated, atomic and prime models.
\item As is common in model theory, we will wish to work in a monster model. In this chapter we discuss how such monster models can be constructed, similarly to how it is done in full first-order logic. Being able to now work in a monster model, we give a positive version of some common model-theoretic tools, such as the construction of indiscernible sequences, and a treatment of Lascar strong types.
\item The goal of this chapter is to prove a Kim-Pillay style theorem: a characterisation of simple theories in terms of the existence of an independence relation, that must then be given by non-dividing and is thus unique. Throughout we will assume the positive theory we work with is thick, as this allows us to give a much simpler and smoother treatment of simplicity that is very close to the usual treatment in full first-order logic.
\item Continuing the previous chapter, we link stability to simplicity. We approach stability through independence relations, with the goal of this chapter being a Kim-Pillay style theorem for stable theories.
\item In this final chapter we give two classes of examples of positive theories. First we discuss how one can add hyperimaginaries as elements to the monster model without leaving the framework of positive logic. So we get an $(-)^\heq$ construction, similar to the $(-)^\eq$ construction from full first-order logic. After that we discuss how continuous logic can be studied through positive logic by giving an explicit description of how to turn a continuous monster model into a monster model of a positive theory. This allows us to apply all the abstract theory that has been developed for positive logic (e.g., simplicity and stability) to continuous logic. We give a brief example of how this translation can be used to get a Kim-Pillay style theorem for simple theories in continuous logic.
\end{enumerate}
Finally, in appendix \ref{ch:lazy-model-theoritician} we give an as-brief-as-possible summary of the necessities for positive logic. This can be viewed as a cheat sheet and a quick start guide for readers that wish to use positive logic.
\section{Bibliographic remarks}
\label{sec:bibliographic-remarks-introduction}
To make these notes self contained, all proofs and details are worked out. We thus refer as little as possible to other sources in the text. As mentioned before, the main results are not original, so at the end of each chapter there is a section with bibliographic remarks. In these sections we briefly discuss original sources and further reading.

\chapter{Basics}
\label{ch:basics}
We cover the basics of positive logic, assuming a basic background in full first-order logic. The main objects of study in positive logic are positively closed models (p.c.\ models, see \thref{pc-model}). Such models work well with the positive fragment of full first-order logic, but they do not work well with negations. Importantly, we still have a compactness theorem for positive formulas (\thref{compactness}).

We discuss types (Section \ref{sec:types-and-type-spaces}) and properties of the category of (p.c.) models (Section \ref{sec:properties-of-category-of-models}), which is all very similar to full first-order logic after making some natural adjustments.

An important difference between positive logic and full first-order logic is that behaviour that is always type-definable in full first-order logic is generally no longer type-definable in positive logic. In practice it is often the case that enough of such behaviour is still type-definable in a positive theory. It is thus useful to have a hierarchy of how `nice' a positive theory is in this regard, yielding the notions of Boolean, Hausdorff, semi-Hausdorff and thick theory, which are discussed in Section \ref{sec:boolean-hausdorff-semi-hausdorff-thick}.

\section{Formulas, homomorphisms, immersions and p.c.\ models}
\label{sec:very-basics}
We assume the reader is familiar with the following definitions, which are exactly the same as in the classical treatment for full first-order logic:
\begin{itemize}
\item \term{signature} or \term{language}, consisting of \emph{constant symbols}\index{constant symbol}, \emph{function symbols}\index{function symbol} and \emph{relation symbols}\index{relation symbol}, possibly multi-sorted;
\item a \term{structure} in a given language (we allow empty structures, which are essentially propositional structures);
\item the recursive definition of a formula in full first-order logic, and satisfaction of such a formula in a structure.
\end{itemize}
We also adapt some conventions that are standard in model theory:
\begin{itemize}
\item in general treatments we often leave the signature implicit and just assume to work in some fixed signature $\L$ (that is often left out of the notation);
\item the theories we consider are assumed to be consistent;
\item we will not distinguish tuples from single elements, so the notation $a \in M$ means that $a$ is some tuple in $M$;
\item usually lowercase letters $a, b, c, \ldots$ will denote tuples of elements in a structure while $x, y, z, \ldots$ will denote (tuples of) variables, uppercase letters $A, B, C, \ldots$ will denote sets where we use $M$ and $N$ for models;
\item unions are written in juxtaposition, so $AB$ just means $A \cup B$.
\end{itemize}
\begin{definition}
\thlabel{positive-formulas-and-theory}
A \term{positive formula} is one that is built from atomic formulas using the connectives $\top, \bot, \wedge, \vee$ and the existential quantifier $\exists$. Our signature will always include the symbol for equality $=$, but not necessarily the symbol for inequality.

An \term{h-inductive sentence} is one of the form $\forall x(\phi(x) \to \psi(x))$, where $\phi(x)$ and $\psi(x)$ are positive formulas. An \term{h-universal sentence} is an h-inductive sentence of the form $\forall x(\phi(x) \to \bot)$.

A \term{positive theory} is a set of h-inductive sentences.
\end{definition}
We stress that inequality is generally not a positive formula, which is relevant in certain applications of positive logic. For example, when treating hyperimagarinies (Section \ref{sec:hyperimaginaries}) or continuous logic (Section \ref{sec:continuous-logic}).

Semantically h-inductive sentences are precisely the sentences expressing that one positively definable set is included in another. The `h' in the name refers to ``homomorphism'' and comes from their interaction with homomorphisms when taking directed unions (see also the discussion before \thref{directed-unions}).

The h-universal sentences are precisely the negations of positive sentences, up to equivalence. So they can always be written in the form $\neg \exists x \phi(x)$, where $\phi(x)$ is positive quantifier-free. Such a formula is equivalent to $\forall x \neg \phi(x)$, hence the name.

In categorical logic, and mainly in topos theory, positive logic is also studied under the name of \term{coherent logic}. Note that our h-inductive sentences are then exactly \emph{coherent sequents}.
\begin{convention}[Stay positive!]
\thlabel{stay-positive}
As these notes are about positive logic, we will drop the ``positive'' from terms from now on. That is, we will just say ``formula'' and ``theory'' instead of ``positive formula'' and ``positive theory'' respectively. If we ever need to refer to full first-order logic, we will explicitly say so.
\end{convention}
\begin{remark}
\thlabel{normal-forms-of-formulas}
One easily verifies by induction that every formula $\phi(x)$ is equivalent to a formula of the form $\exists y \psi(x, y)$, where $\psi(x, y)$ is positive quantifier-free. Another useful normal form is that $\phi(x)$ is always equivalent to a formula of the form $\phi_1(x) \vee \ldots \vee \phi_n(x)$, such that for each $1 \leq i \leq n$ the formula $\phi_i(x)$ is of the form $\exists  y \psi_i(x, y)$, where $\psi_i(x, y)$ is a conjunction of atomic formulas.
\end{remark}
\begin{definition}
\thlabel{homomorphism}
A \term{homomorphism} is a function $f: M \to N$ between structures that preserves relation symbols and commutes with function symbols and constant symbols. That is:
\begin{enumerate}
\item $f(c_M) = c_N$ for every constant symbol $c$;
\item $f(g_M(a)) = g_N(f(a))$ for every function symbol $g$ and every $a \in M$;
\item for any $a \in M$ and any relation symbol $R$ we have that $M \models R(a)$ implies $N \models R(a)$.
\end{enumerate}
In this situation we also call $N$ a \term{continuation} of $M$.
\end{definition}
\begin{proposition}
\thlabel{homomorphism-preserves-truth}
A function $f: M \to N$ is a homomorphism iff for every formula $\phi(x)$ and every $a \in M$ we have
\[
M \models \phi(a) \implies N \models \phi(f(a)).
\]
\end{proposition}
\begin{proof}
The right to left direction is trivial, and the other direction follows by induction on the complexity of the formula.
\end{proof}
\begin{remark}
\thlabel{homomorphisms-and-sentences}
\thref{homomorphism-preserves-truth} says in particular that for any sentence $\phi$ we have $M \models \phi$ implies $N \models \phi$. This holds even if $M$ is empty (the quantification ``every $a \in M$'' in \thref{homomorphism-preserves-truth} is ignored for sentences).

The contrapositive of this statement says precisely that for any h-universal sentence $\chi$ we have that $N \models \chi$ implies $M \models \chi$.
\end{remark}
\begin{definition}
\thlabel{diagram}
Let $M$ be an $\L$-structure. We write $\L(M)$\nomenclature[LM]{$\L(M)$}{Language $\L$ with additional constants for $M$} for the language that is $\L$ together with a new constant symbol for each element of $M$. We view $M$ as an $\L(M)$-structure by interpreting each constant symbol as the corresponding element.

The \term{positive diagram} of an $\L$-structure $M$ is the set $\Diag(M)$\nomenclature[DiagM]{$\Diag(M)$}{Positive diagram of $M$} of all positive quantifier-free $\L(M)$-sentences that are true in $M$.
\end{definition}
\begin{remark}
\thlabel{model-of-diagram-is-homomorphism}
We will often implicitly use the fact that a model $N$ of $\Diag(M)$ is essentially the same thing as a homomorphism $f: M \to N$.
\end{remark}
\begin{definition}
\thlabel{immersion}
A homomorphism $f: M \to N$ is called an \term{immersion} if for every formula $\phi(x)$ and every $a \in M$ we have
\[
M \models \phi(a) \Longleftrightarrow N \models \phi(f(a)).
\]
\end{definition}
Just like in \thref{homomorphisms-and-sentences} we note that nothing in the definition of an immersion excludes the case where $M$ is empty. It just means that we have $M \models \phi$ if and only if $N \models \phi$ for every sentence $\phi$.
\begin{definition}
\thlabel{pc-model}
We call a model $M$ of a theory $T$ a \term{positively closed model}, or \term{p.c.\ model}, if the following equivalent conditions hold:
\begin{enumerate}[label=(\roman*)]
\item every homomorphism $f: M \to N$ with $N \models T$ is an immersion;
\item for every $a \in M$ and $\phi(x)$, if there is a homomorphism $f: M \to N$ with $N \models T$ and $N \models \phi(f(a))$ then already $M \models \phi(a)$;
\item for every $a \in M$ and $\phi(x)$ such that $M \not \models \phi(a)$ there is $\psi(x)$ such that $T \models \neg \exists x(\phi(x) \wedge \psi(x))$ and $M \models \psi(a)$.
\end{enumerate}
\end{definition}
\begin{lemma}
\thlabel{equivalence-pc-characterisations}
The conditions in \thref{pc-model} are indeed equivalent.
\end{lemma}
\begin{proof}
\underline{(i) $\Rightarrow$ (ii)} By definition of being an immersion.

\underline{(ii) $\Rightarrow$ (iii)} Let $a \in M$ and $\phi(x)$ be such that $M \not \models \phi(a)$. The theory $T \cup \Diag(M) \cup \{\phi(a)\}$ is inconsistent. This is because a model $N$ of this theory would be a model of $T$, admit a homomorphism $f: M \to N$ and be such that $N \models \phi(f(a))$ from which $M \models \phi(a)$ would follow by our assumption. Hence there is some $\chi(a, b) \in \Diag(M)$ such that $T$ is inconsistent with $\{\phi(a), \chi(a, b)\}$. As $a$ and $b$ do not appear in $T$ this means that $T \models \neg \exists x(\phi(x) \wedge \exists y \chi(x, y))$. Taking $\psi(x)$ to be $\exists y \chi(x, y)$ then completes the proof.

\underline{(iii) $\Rightarrow$ (i)} Let $f: M \to N$ be a homomorphism with $N \models T$. Let $a \in M$ and $\phi(x)$ be such that $N \models \phi(f(a))$. Suppose for a contradiction that $M \not \models \phi(a)$. Then by assumption there is $\psi(x)$ with $M \models \psi(a)$ and $T \models \neg \exists x(\phi(x) \wedge \psi(x))$. As $f$ is a homomorphism we must have $N \models \psi(f(a))$, but then $N \models \phi(f(a)) \wedge \psi(f(a))$ while being a model of $T$. So we arrive at a contradiction and conclude $M \models \phi(a)$, as required.
\end{proof}
\begin{remark}
\thlabel{pc-model-vs-ec-model}
Some authors also use the name ``existentially closed model'' or ``e.c.\ model'' for what we call a p.c.\ model. This terminology stems from the approach where we consider embeddings between models instead of homomorphisms (see also \thref{common-morleyisations}(i)). In that case, a model $M$ is an e.c.\ model precisely when every quantifier-free formula (potentially using negations) that has a solution in some bigger model $N \supseteq M$ already has a solution in $M$. So this is really about the \emph{existence} of solutions.

Being positively closed is similar, but more is going on: p.c.\ models make as much true as possible. For example, the only p.c.\ models of the empty theory (in the empty signature) are singletons. It is instructive to see why this happens, as it emphasises two points. Namely that being p.c.\ is not just about finding `new' solutions to equations, but also that as many things as possible have to be true about existing elements. The second point is that equality is one of these things that can be true about elements, meaning in particular that when moving from an arbitrary model to a p.c.\ model some elements might have to be identified.

So let $M$ be a p.c.\ model of the empty theory. Then $N = M \cup \{*\}$ is another model and $M \subseteq N$ is a homomorphism of models. As $N \models \exists x(x=x)$ we must have $M \models \exists x(x=x)$, so $M$ is inhabited. To see that $M$ must be a singleton, let $a,b \in M$ and consider the homomorphism of models $f: M \to N'$, where $N' = \{*\}$. Applying \thref{pc-model}(ii) to the formula $x = y$ we find $N' \models f(a) = f(b)$, and so $M \models a = b$, as required.

We also note that some authors use the name ``positively existentially closed model'' or ``pec model'' for what we call a p.c.\ model.
\end{remark}
\begin{definition}
\thlabel{negation}
Let $T$ be a theory and $\phi(x)$ be a formula. A formula $\psi(x)$ such that $T \models \neg \exists x(\phi(x) \wedge \psi(x))$ is called an \emph{obstruction of $\phi(x)$}\index{obstruction of a formula}.
\end{definition}
Using the above terminology we can rephrase \thref{pc-model}(iii) as follows: for every $a$ and $\phi(x)$ such that $M \not \models \phi(a)$ there is an obstruction $\psi(x)$ of $\phi(x)$ such that $M \models \psi(a)$.
\begin{remark}
\thlabel{negation-instead-of-obstruction}
In some literature the term ``a negation'' is used instead of ``an obstruction''. This can be confused with \emph{the} negation of a formula, hence the choice to use the current terminology. The term ``a denial'' has also appeared.
\end{remark}
\term{Geometric logic} allows infinite disjunctions. That is, a geometric formula is one that is built from atomic formulas using $\top$, $\bot$, $\wedge$, $\exists$ and infinite disjunctions, such that it has only finitely many free variables.
\begin{corollary}
\thlabel{geometric-axiomatisation}
The class of p.c.\ models of a theory $T$ can always be axiomatised using geometric logic. More explicitly, writing $\Obs(\phi)$ for the set of obstructions of $\phi$ (with respect to $T$) we have that the following geometric theory axiomatises the class of p.c.\ models of $T$:
\[
T \cup \left\{ \forall x \left(\phi(x) \vee \bigvee \Obs(\phi) \right) : \phi(x) \text{ is a formula} \right\}.
\]
\end{corollary}
\begin{proof}
Being a model of the specified geometric theory is clearly equivalent to \thref{pc-model}(iii).
\end{proof}
\begin{definition}
\thlabel{universal-theory-kaiser-hull}
Let $T$ be a theory. The \term{h-universal theory} of $T$ is defined as\nomenclature[Tu]{$T^\u$}{The h-universal theory of $T$}
\[
T^\u = \{ \chi \text{ an h-universal sentence} : T \models \chi \}.
\]
The \term{Kaiser hull} of $T$ is defined as\nomenclature[Tec]{$T^\pc$}{Kaiser hull of $T$}
\[
T^\pc = \{ \chi \text{ an h-inductive sentence} : M \models \chi \text{ for every p.c.\ model } M \text{ of } T \}.
\]
\end{definition}
It turns out that $T^\u$ and $T^\pc$ have the same p.c.\ models as $T$, and are respectively the minimal and maximal such theories in a precise sense. This is shown later in \thref{universal-kaiser-hull-ec-models}.
\begin{lemma}
\thlabel{models-of-universal-consequences}
Let $T$ be a theory in full first-order logic. The models of $T^\u$ are precisely those $M_0$ such that admit a homomorphism $f: M_0 \to M$ into some model $M$ of $T$.
\end{lemma}
\begin{proof}
One direction is clear: if $f: M_0 \to M$ is a homomorphism then $M_0 \models T^\u$ because $M \models T^\u$. We prove the other direction, so let $M_0 \models T^\u$. We will show that $T \cup \Diag(M_0)$ is consistent. If it would be inconsistent there would be $\phi(a) \in \Diag(M_0)$, where $a \in M_0$, such that $T \models \neg \phi(a)$. As $a$ does not appear in $T$ this just means that $T \models \neg \exists x \phi(x)$. Hence $\neg \exists x \phi(x) \in T^\u$, contradicting $M_0 \models T^\u$.
\end{proof}
\begin{definition}
\thlabel{directed-system}
Recall that a \term{directed poset} is a poset $I$ such that for any $i_1, \ldots, i_n \in I$ there is an upper bound $j \in I$. A \term{directed system} of $\L$-structures is a functor from a directed poset $I$ into the category of $\L$-structures and homomorphisms. More precisely, we have $\L$-structures $(M_i)_{i \in I}$ together with a homomorphism $f_{ij}: M_i \to M_j$ for every $i \leq j$ such that:
\begin{enumerate}[label=(\roman*)]
\item $f_{ii}$ is the identity function,
\item $f_{jk}f_{ij} = f_{ik}$ for all $i \leq j \leq k$.
\end{enumerate}
The \emph{union of a directed system} $(M_i)_{i \in I}$ or \term{directed union} \emph{of} $(M_i)_{i \in I}$ is defined as follows. For the underlying set we take $M = \coprod_{i \in I} M_i / {\sim}$ where $\sim$ is the equivalence relation defined as follows: for $a \in M_i$ and $b \in M_j$ we have $a \sim b$ if and only if there is $k \geq i, j$ such that $f_{ik}(a) = f_{ik}(b)$. For $a \in M_i$ we write $[a]$ for the equivalence class of $a$ (so $[a] \in M$). Then we make $M$ into an $\L$-structure as follows:
\begin{itemize}
\item for every constant symbol $c$ we set $c_M = [c_{M_i}]$ for any $i \in I$;
\item for every relation symbol $R(x_1, \ldots, x_n)$ we set $M \models R([a_1], \ldots, [a_n])$ if there is some $i \in I$ and $a_1', \ldots, a_n' \in M_i$ representing $[a_1], \ldots, [a_n]$ with $M_i \models R(a_1', \ldots, a_n')$;
\item for every $n$-ary function symbol $f$ we do the following, given $[a_1], \ldots, [a_n] \in M$ we let $i$ be such that there are $a_1', \ldots, a_n' \in M_i$ representing $[a_1], \ldots, [a_n]$ and we set $f^M([a_1], \ldots, [a_n]) = [f^{M_i}(a_1', \ldots, a_n')]$.
\end{itemize}
In the special case where $I$ is a linear order we call $(M_i)_{i \in I}$ a \term{chain}. If $I$ is an ordinal and $(M_i)_{i \in I}$ is such that for every limit ordinal $\ell \in I$ the structure $M_\ell$ is the directed union of $(M_i)_{i < \ell}$ we call $(M_i)_{i \in I}$ a \term{continuous chain}.
\end{definition}
\begin{remark}
\thlabel{directed-system-remarks}
Some remarks about directed systems and unions of them, using the notation from \thref{directed-system}.
\begin{enumerate}[label=(\roman*)]
\item The structure on $M$ is well-defined because the maps in the system are homomorphisms and because the system is directed.
\item If all the homomorphisms in the directed system $(M_i)_{i \in I}$ are inclusions then the underlying set of $M$ is just the set-theoretic union of $(M_i)_{i \in I}$. We often think of homomorphisms as inclusions, so by abuse of notation we will often disregard the equivalence relation. That is, given $a \in M_i$ we will also write $a \in M$ instead of $[a] \in M$. Conversely, given $a \in M$ (which now denotes $[a]$) we will write $a \in M_i$ if there is $a' \in M_i$ such that $a' \sim a$.
\item Following the previous point, we will write $\bigcup_{i \in I} M_i$ for the union of the system $(M_i)_{i \in I}$, viewed as an $\L$-structure.
\item The obvious inclusions (technically just maps) from a directed system $(M_i)_{i \in I}$ into its union $M = \bigcup_{i \in I} M_i$ are homomorphisms. This follows immediately from how we defined the structure on $M$.
\end{enumerate}
\end{remark}
\begin{proposition}
\thlabel{formulas-in-directed-union}
Let $(M_i)_{i \in I}$ be a directed system and let $M = \bigcup_{i \in I} M_i$ be its union. Then for any formula $\phi(x)$ and any $a \in M$ we have $M \models \phi(a)$ if and only if there is some $i \in I$ such that $a \in M_i$ and $M_i \models \phi(a)$.
\end{proposition}
\begin{proof}
This is straightforward induction on the complexity of $\phi$. For atomic formulas the statement holds by definition. Then each induction step is straightforward using the fact that the system is directed and that all the maps in the system are homomorphisms.
\end{proof}
The following theorem explains the name ``h-inductive'' sentence. The `h' stands for ``homomorphism'' and directed unions are sometimes also called \emph{inductive limits}\index{inductive limit}.
\begin{theorem}
\thlabel{directed-unions}
Let $T$ be a theory in full first-order logic. Then the following are equivalent:
\begin{enumerate}[label=(\roman*)]
\item $T$ can be axiomatised using h-inductive sentences, that is: there is an equivalent theory $T'$ that contains only h-inductive sentences;
\item the union of a directed system of models of $T$ is again a model of $T$;
\item the union of a chain of models of $T$ is again a model of $T$.
\end{enumerate}
\end{theorem}
\begin{proof}
\underline{(i) $\Rightarrow$ (ii)} Let $\forall x(\phi(x) \to \psi(x))$ be an h-inductive sentence in $T'$ and let $(M_i)_{i \in I}$ be a directed system of models of $T$. It is enough to show that $\forall x(\phi(x) \to \psi(x))$ holds in $M = \bigcup_{i \in I} M_i$. So let $a \in M$ be such that $M \models \phi(a)$. By \thref{formulas-in-directed-union} there is $i \in I$ such that $a \in M_i$ and $M_i \models \phi(a)$. As $M_i$ is a model of $T$, and hence of $T'$, we have $M_i \models \psi(a)$ from which $M \models \psi(a)$ follows.

\underline{(ii) $\Rightarrow$ (iii)} Trivial.

\underline{(iii) $\Rightarrow$ (i)} Let $T'$ be the set of h-inductive consequences of $T$. Let $M_0 \models T'$, we will show that $M_0 \models T$. Write $\Pi$ for the set of all h-universal $\L(M_0)$-sentences that are true in $M_0$. We claim that $T \cup \Diag(M_0) \cup \Pi$ is consistent. If not then there are $\phi(a) \in \Diag(M_0)$ and $\psi(a) \in \Pi$ such that $T \models \neg(\phi(a) \wedge \psi(a))$. Note that $\psi(x)$ is of the form $\neg \chi(x)$ for some positive formula $\chi(x)$. So because $a$ does not appear in $T$ we have $T \models \forall x(\phi(x) \to \chi(x))$. This last sentence is h-inductive, hence $\forall x(\phi(x) \to \chi(x)) \in T'$. Then $M_0 \models \forall x(\phi(x) \to \chi(x))$, so, as $M_0 \models \phi(a)$, we have $M_0 \models \chi(a)$. However, this contradicts $M_0 \models \psi(a)$. So we conclude that $T \cup \Diag(M_0) \cup \Pi$ is indeed consistent.

We then find a model $N_0$ of $T \cup \Diag(M_0) \cup \Pi$. That is, $N_0$ is a model of $T$ and there is a homomorphism $f_0: M_0 \to N_0$. Now take $\Delta$ to be the elementary diagram of $M_0$ (i.e.\ all $\L(M_0)$-sentences in full first-order logic that are true in $M_0$). Then by construction $\Delta^\u = \Pi$, so by \thref{models-of-universal-consequences} and the fact that $N_0 \models \Pi$ there is a model $M_1 \models \Delta$ with a homomorphism $g_0: N_0 \to M_1$. In particular $g_0 f_0: M_0 \to M_1$ is an elementary embedding because $M_1 \models \Delta$ and thus $M_1 \models T'$.

Repeating this construction, we find an infinite chain of homomorphisms
\[
M_0 \xrightarrow{f_0} N_0 \xrightarrow{g_0} M_1 \xrightarrow{f_1} N_1 \xrightarrow{g_1} M_2 \xrightarrow{f_2} \ldots
\]
such that $M_i \models T'$, $N_i \models T$ and $g_i f_i$ is an elementary embedding for all $i < \omega$. Set $U = \bigcup_{i < \omega} M_i = \bigcup_{i < \omega} N_i$. By assumption $U$ is then a model of $T$. It is also the union of a chain of elementary embeddings, so the inclusion $M_0 \to U$ is an elementary embedding. We conclude that $M_0 \models T$, as required.
\end{proof}
\begin{proposition}
\thlabel{pc-models-closed-under-directed-unions}
Let $(M_i)_{i \in I}$ be a directed system of p.c.\ models of some theory $T$. Then $M = \bigcup_{i \in I} M_i$ is again a p.c.\ model of $T$.
\end{proposition}
\begin{proof}
By \thref{directed-unions} $M$ is a model of $T$, so it remains to verify that it is p.c. We check \thref{pc-model}(ii). Let $a \in M$, $\phi(x)$ and $f: M \to N$ be such that $N \models \phi(f(a))$ and $N \models T$. By directedness there is $i \in I$ such that $a \in M_i$. Write $g: M_i \to N$ for the composition $M_i \to M \xrightarrow{f} N$. So we have $N \models \phi(g(a))$ and hence $M_i \models \phi(a)$ because $M_i$ is p.c. We conclude that indeed $M \models \phi(a)$.
\end{proof}
\begin{definition}
\thlabel{size-of-theory}
For a theory $T$ we let $|T|$\nomenclature[T]{$|T|$}{Size of the theory $T$} be the cardinality of the set of formulas, up to logical equivalence. We call $T$ a \term{countable theory} if $|T|$ is countable.
\end{definition}
In some approaches $|T|$ is defined as the maximum between $\aleph_0$ and the cardinality of the set of formulas in the language of $T$. However, all we really care about is how many formulas there are up to equivalence modulo $T$. For example, a signature could have many constant symbols, but if the theory declares them to be all equal then there is effectively only one constant symbol. At the same time, we typically have infinitely many distinct formulas, even up to equivalence (e.g., $x_1 = x_1 \wedge \ldots \wedge x_n = x_n$ for all $n < \omega$).
\begin{theorem}[Continue to p.c.\ model]
\thlabel{continue-to-pc-model}
Every model $M$ of a theory $T$ can be continued to a p.c.\ model of $T$. That is, there is some p.c.\ model $N$ of $T$ with a homomorphism $f: M \to N$.
\end{theorem}
\begin{proof}
We will construct a chain $(M_n)_{n < \omega}$ of models of $T$, with $M_0 = M$, such that the following holds: for any $a \in M_n$ and $\phi(x)$, if there is a homomorphism $f: M_{n+1} \to N$ for some $N \models T$ with $N \models \phi(f(a))$ then $M_{n+1} \models \phi(a)$.

Assume that $M_n$ has been constructed. Let $\Phi$ be the set of all formulas with parameters in $M_n$. Enumerate $\Phi$ as $(\phi_i(a_i))_{i < |T| + |M_n|}$. We construct a chain $(M_n^i)_{i < |T| + |M_n|}$ of models of $T$ as follows. We set $M_n^0 = M_n$ and at limit stages we take the union. For the successor step we assume to have constructed $M_n^i$. If there is a homomorphism $f: M_n^i \to N$ with $N \models T$ and $N \models \phi_i(f(a_i))$ we set $M_n^{i+1} = N$ and take $f$ to be the next link in the chain. If no such homomorphism exists we set $M_n^{i+1} = M_n^i$. Having constructed this chain, we set $M_{n+1} = \bigcup_{i < |T| + |M_n|} M_n^i$.

We verify the induction hypothesis for $M_{n+1}$. Let $a \in M_n$ and $\phi(x)$ be any formula. Let $i < |T| + |M_n|$ be such that $\phi = \phi_i$ and $a = a_i$. Assume there is a homomorphism $f: M_{n+1} \to N$ for some $N \models T$ with $N \models \phi(f(a))$. Composing $f$ with the obvious inclusion $M_n^i \to M_{n+1}$ we obtain a homomorphism $g: M_n^i \to N$ with $N \models \phi(g(a))$. By how we constructed $M_n^{i+1}$ this means that $M_n^{i+1} \models \phi(a)$ and thus $M_{n+1} \models \phi(a)$, as required.

Having constructed $(M_n)_{n < \omega}$ we set $N = \bigcup_{n < \omega} M_n$. To check that $N$ is p.c.\ we let $a \in N$ and $\phi(x)$ be any formula. Assume there is some homomorphism $f: N \to N'$ with $N' \models T$ and $N' \models \phi(f(a))$. Let $n < \omega$ be such that $a \in M_n$. Composing $f$ with the obvious inclusion $M_{n+1} \to N$ we find a homomorphism $g: M_{n+1} \to N'$. Then $N' \models \phi(g(a))$, so by the induction hypothesis we have $M_{n+1} \models \phi(a)$ and thus $N \models \phi(a)$, as required.
\end{proof}
\begin{lemma}
\thlabel{t-pc-u-is-t-u}
For any theory $T$ we have $(T^\pc)^\u = T^\u$.
\end{lemma}
\begin{proof}
As $T^\pc \models T$ we have $T^\pc \models T^\u$, so $T^\u \subseteq (T^\pc)^\u$. Now let $M$ be any model of $T$ and continue $M$ to a p.c.\ model $N$. Then $N \models T^\pc$ and hence $N \models (T^\pc)^\u$. By definition $(T^\pc)^\u$ contains only h-universal sentences, so $M \models (T^\pc)^\u$ because $N$ is a continuation of $M$. As $M$ was an arbitrary model of $T$ we have $T \models (T^\pc)^\u$. We conclude that $(T^\pc)^\u \subseteq T^\u$, and the result follows.
\end{proof}
\begin{theorem}
\thlabel{universal-kaiser-hull-ec-models}
Let $T$ and $T'$ be theories in the same language. Then $T$ and $T'$ have the same p.c.\ models if and only if $T^\pc \models T'$ and $T' \models T^\u$. In particular, $T^\u$ and $T^\pc$ have the same p.c.\ models as $T$ and are respectively the minimal and maximal such theories.
\end{theorem}
\begin{proof}
We first prove the left to right direction. As $T$ and $T'$ have the same p.c.\ models, every p.c.\ model of $T$ is a model of $T'$. So by the definition of $T^\pc$ we have $T' \subseteq T^\pc$ and thus $T^\pc \models T'$. To prove $T' \models T^\u$ we let $M$ be any model of $T'$ and $\chi \in T^\u$. We can continue $M$ to a p.c.\ model $N$ of $T'$. So $N$ is also a p.c.\ model of $T$ and in particular a model of $T^\u$, thus $N \models \chi$. As $\chi$ is h-universal and $N$ is a continuation of $M$ we must have $M \models \chi$. Since $M$ was an arbitrary model of $T'$ we get $T' \models \chi$ and we conclude $T' \models T^\u$.

Now we prove the right to left direction. Let $M$ be a p.c.\ model of $T$. Then $M \models T^\pc$ and hence $M \models T'$ because $T^\pc \models T'$. To see that $M$ is p.c.\ we verify property \thref{pc-model}(iii). So let $a \in M$ and $\phi(x)$ be such that $M \not \models \phi(a)$. Then, because $M$ is p.c.\ for $T$, there is $\psi(x)$ with $M \models \psi(a)$ and $T \models \neg \exists x(\phi(x) \wedge \psi(x))$. We thus have $\neg \exists x(\phi(x) \wedge \psi(x)) \in T^\u$, as this is an h-universal sentence. Since $T' \models T^\u$ we thus have $T' \models \neg \exists x(\phi(x) \wedge \psi(x))$, and we conclude that $M$ is also p.c.\ for $T'$.

Now let $M$ be a p.c.\ model of $T'$. We have that $(T^\pc)^\u \models (T')^\u$ and $(T')^\u \models T^\u$. So by \thref{t-pc-u-is-t-u} we have $(T^\pc)^\u = (T')^\u = T^\u$. We first show that any homomorphism $f: M \to N$ with $N \models T$ must be an immersion. Assume for a contradiction that there is $a \in M$ and $\phi(x)$ such that $N \models \phi(f(a))$ while $M \not \models \phi(a)$. As $M$ is a p.c.\ model of $T'$ there must be $\psi(x)$ such that $M \models \psi(a)$ and $T' \models \neg \exists x(\phi(x) \wedge \psi(x))$. However, that means that $N \models \phi(f(a)) \wedge \psi(f(a))$, contradicting $\neg \exists x(\phi(x) \wedge \psi(x)) \in (T')^\u = T^\u$. We are left to show that $M$ is indeed a model of $T$. Using that $M \models T^\u$ we find a homomorphism $f: M \to N$ where $N \models T$, by \thref{models-of-universal-consequences}. By the above discussion $f$ is in fact an immersion. Let $\forall x(\phi(x) \to \psi(x)) \in T$ be an h-inductive sentence and let $a \in M$ be such that $M \models \phi(a)$. Then $N \models \phi(f(a))$ and so $N \models \psi(f(a))$. Using the fact that $f$ is an immersion we have $M \models \psi(a)$. So $M \models \forall x(\phi(x) \to \psi(x))$ and we conclude that $M$ is indeed a model of $T$, which completes the proof.
\end{proof}

When doing model theory in positive logic we are interested in the p.c.\ models. When we restrict our attention to p.c.\ models we do still have compactness, but only for positive formulas.
\begin{theorem}[Compactness for positive formulas]
\thlabel{compactness}
Let $T$ be a theory and let $\Sigma(x)$ be a set of positive formulas. Suppose that for every finite $\Sigma_0(x) \subseteq \Sigma(x)$ there is $M \models T$ with $a \in M$ such that $M \models \Sigma_0(a)$. Then there is a p.c.\ model $N$ of $T$ with $a \in N$ such that $N \models \Sigma(a)$.
\end{theorem}
\begin{proof}
By the compactness theorem for full first-order logic we find a model $M'$ of $T$ and $a' \in M'$ such that $M' \models \Sigma(a')$. Continue $M'$ to a p.c.\ model $N$ of $T$. Then because $\Sigma(x)$ only contains positive formulas we have $N \models \Sigma(f(a'))$. So we set $a = f(a')$, which concludes the proof.
\end{proof}
To illustrate that we can generally not get more compactness, we consider the following two examples.
\begin{example}
\thlabel{compactness-failure}
Consider the theory $T$ with a symbol for inequality and $\omega$ many disjoint unary predicates $P_n(x)$. Then p.c.\ models of $T$ are precisely those which consist of $\omega$-many disjoint infinite sets, one for each predicate. If we had full compactness then the set
\[
\Sigma(x) = \{ \neg P_n(x) : n < \omega \}
\]
would have a realisation in some p.c.\ model, which is impossible.
\end{example}
\begin{example}
\thlabel{bounded-set-compactness}
It could happen that there is a definable set that is infinite and bounded. This does not contradict compactness: it just means that inequality is not positively definable on that set. Such situations might arise when adding hyperimaginaries as parameters, which can be done in positive logic (see \thref{hyperimaginary-does-not-preserve-boolean}), but we give a simpler example here.

The signature consists of $\omega$ many constant symbols $\{c_i\}_{i < \omega}$. The theory $T$ then asserts that all of these constant symbols are distinct, i.e.\ $c_i \neq c_j$ for all $i \neq j$. There is precisely one p.c.\ model of $T$ (up to isomorphism), which consists of just the interpretations of the constant symbols. So the trivial definable set $x = x$ is bounded (i.e., it is countable in every p.c.\ model), but infinite. Again, with full compactness we would run into trouble because
\[
\Sigma((x_i)_{i < \omega_1}) = \{ x_i \neq x_j : i < j < \omega_1 \}
\]
would then yield a realisation with uncountably many elements.
\end{example}
\begin{definition}
\thlabel{positive-quantifier-elimination}
We say that a theory $T$ has \term{positive quantifier elimination} if every formula is equivalent to a positive quantifier-free formula modulo $T$.
\end{definition}
\begin{proposition}
\thlabel{regular-formulas-eliminate-quantifiers-implies-positive-quantifier-elimination}
A theory $T$ has positive quantifier elimination if for every formula $\exists y \phi(x, y)$, where $y$ is a single variable and $\phi(x, y)$ is a conjunction of atomic formulas, there is a quantifier-free formula $\psi(x)$ such that $T \models \forall x(\exists y \phi(x, y) \leftrightarrow \psi(x))$.
\end{proposition}
\begin{proof}
We claim that for every formula of the form $\exists y_1 \ldots y_k \phi(x, y_1, \ldots, y_n)$, with $\phi(x, y_1, \ldots, y_n)$ a conjunction of atomic formulas and $y_i$ a single variable for all $1 \leq i \leq n$, is equivalent to a quantifier-free formula modulo $T$. Every formula is equivalent to a disjunction of such formulas, so it is enough to prove the claim.

The proof of the claim comes down to repeatedly applying the assumption. We work out the details. We will show by induction on $1 \leq i \leq n$ that there are quantifier-free formulas $\psi_i(x, y_1, \ldots, y_{n-i})$ such that $\psi_i(x, y_1, \ldots, y_{n-i})$ is equivalent to $\exists y_{n-i+1} \ldots y_n \phi(x, y_1, \ldots, y_n)$ modulo $T$. Then $\psi_n(x)$ is equivalent to $\exists y_1 \ldots y_n \phi(x, y_1, \ldots, y_n)$ modulo $T$, as required.

For the base case we simply take $\psi_0(x, y_1, \ldots, y_n)$ to be $\phi(x, y_1, \ldots, y_n)$. Now assume we have constructed $\psi_i(x, y_1, \ldots, y_{n-i})$. We may assume that $\psi_i(x, y_1, \ldots, y_{n-i})$ is of the form
\[
\alpha_1(x, y_1, \ldots, y_{n-i}) \vee \ldots \vee \alpha_k(x, y_1, \ldots, y_{n-i}),
\]
where $\alpha_j(x, y_1, \ldots, y_{n-i})$ is a conjunction of atomic formulas for all $1 \leq j \leq k$. So $\exists y_{n-i} \psi_i(x, y_1, \ldots, y_{n-i})$ is equivalent to
\[
\exists y_{n-i} \alpha_1(x, y_1, \ldots, y_{n-i}) \vee \ldots \vee \exists y_{n-i} \alpha_k(x, y_1, \ldots, y_{n-i}).
\]
By our assumption each of the disjuncts is equivalent to a quantifier-free formula modulo $T$, so $\exists y_{n-i} \psi_i(x, y_1, \ldots, y_{n-i})$ is equivalent to a quantifier-free formula modulo $T$, which will be our $\psi_{i+1}(x, y_1, \ldots, y_{n-i-1})$. The equivalence (modulo $T$) to $\exists y_{n-i} \ldots y_n \phi(x, y_1, \ldots, y_n)$ follows immediately from the construction.
\end{proof}

\section{Types and type spaces}
\label{sec:types-and-type-spaces}
\begin{definition}
\thlabel{type}
Let $M$ be a p.c.\ model of some theory $T$ and let $a \in M$. The \term{type} of $a$ in $M$ is given by:\nomenclature[tpaM]{$\tp(a; M)$}{Type of $a$ in $M$}
\[
\tp(a; M) = \{ \phi(x) : M \models \phi(a) \}.
\]
If the elements of $a$ are indexed by $I$, we call this an \emph{$I$-type}\index{type!I-type@$I$-type}, in particular if $a$ is finite with $|a| = n$ (in a single-sorted language) we call this an \emph{$n$-type}\index{type!n-type@$n$-type}.
\end{definition}
The notation $\tp(a; M)$ makes sense for any structure $M$, not just p.c.\ models, as we see in the proposition below. We wish to reserve the terminology ``type'' for the situation in \thref{type}, see also \thref{partial-type}.
\begin{proposition}
\thlabel{pc-model-iff-types-are-maximal}
Let $M$ be a model of some theory $T$. Then $M$ is a p.c.\ model if and only if for every $a \in M$ we have that $\tp(a; M)$ is a maximal consistent set of formulas (modulo $T$).
\end{proposition}
\begin{proof}
We first prove the left to right direction. Let $a \in M$ and set $p(x) = \tp(a; M)$ and let $\phi(x)$ be any formula that is not in $p(x)$. Then $M \not \models \phi(a)$, so because $M$ is p.c.\ there must be an obstruction $\psi(x)$ of $\phi(x)$ such that $M \models \psi(a)$. We thus have $\psi(x) \in p(x)$ which means that $p(x) \cup \{\phi(x)\}$ is inconsistent with $T$. We conclude that $p(x)$ is indeed maximal consistent, modulo $T$.

For the converse we let $a \in M$ and $\phi(x)$ be arbitrary. Let $f: M \to N$ be a homomorphism with $N \models T$ and $N \models \phi(f(a))$. Write $p(x) = \tp(a; M)$ and $q(x) = \tp(f(a); N)$. As $f$ is a homomorphism we have $p(x) \subseteq q(x)$, and because $q(x)$ is realised in a model of $T$ it is consistent modulo $T$. By maximality of $p(x)$ we must thus have $p(x) = q(x)$. So $\phi(x) \in q(x) = p(x)$ and we see that $M \models \phi(a)$, as required.
\end{proof}
By compactness we have that types are precisely the maximally consistent sets of formulas. Some authors use ``type'' also for any consistent set of formulas (or for those of the form $\tp(a; M)$, where $M$ is not necessarily p.c.). For us a ``type'' will always be a maximal consistent set of formulas, while we make the following definition for the other case.
\begin{definition}
\thlabel{partial-type}
Let $T$ be a theory. A \term{partial type}\index{type!partial} is any consistent set of formulas (modulo $T$). So an alternative definition for \emph{type} is: a maximal partial type.
\end{definition}
\begin{definition}
\thlabel{type-space}
Let $T$ be a theory. For an index set $I$ we define the \emph{type space of $I$-types of $T$}\index{type space of $T$} $\S_I(T)$\nomenclature[STI]{$\S_I(T)$}{Type space of $I$-types of $T$} as follows. The underlying set is the set of $I$-types of $T$. For a set of formulas $\Sigma(x)$ we write\nomenclature[Sigmax]{$[\Sigma(x)]$}{Closed set in $\S(T)$ given by the set of formulas $\Sigma(x)$}
\[
[\Sigma(x)] = \{ p(x) \in \S_I(T) : \Sigma(x) \subseteq p(x) \},
\]
and for formulas we simplify the notation $[\{\phi(x)\}]$ to $[\phi(x)]$\nomenclature[phix]{$[\phi(x)]$}{Closed set in $\S(T)$ given by the formula $\phi(x)$}. We topologise $\S_I(T)$ by taking the sets of the form $[\Sigma(x)]$ as closed sets. As we often do not care so much about the index set $I$ we may drop it from the notation and just write $\S(T)$\nomenclature[ST]{$\S(T)$}{Type space of $T$}.
\end{definition}
The topology in the above definition is well-defined. It is easy to see that the proposed closed sets are closed under arbitrary intersections and we have $[\bot] = \emptyset$ and $[\top] = \S(T)$. To see that things are closed under finite unions we note that
\[
[\Sigma(x)] \cup [\Sigma'(x)] = [\{\phi(x) \vee \phi'(x) : \phi(x) \in \Sigma(x) \text{ and } \phi'(x) \in \Sigma'(x)\}].
\]
\begin{proposition}
\thlabel{type-space-properties}
The space $\S(T)$ is a compact $T_1$ space. The latter means that given any two points $p, q \in \S(T)$ there is a closed set $A \subseteq \S(T)$ with $p \in A$ and $q \not \in A$.
\end{proposition}
\begin{proof}
Compactness of $\S(T)$ follows from the compactness theorem, using the following formulation of topological compactness: given a family $\F$ of closed sets with the finite intersection property (i.e.\ $\bigcap \F_0 \neq \emptyset$ for any finite $\F_0 \subseteq \F$) we have that $\bigcap \F \neq \emptyset$. The $T_1$ property follows from the maximality of the types: if $p(x)$ and $q(x)$ are types then there is some $\phi(x) \in p(x)$ with $\phi(x) \not \in q(x)$, so $A = [\phi(x)]$ is the required closed set.
\end{proof}
\section{Properties of the category of (p.c.) models}
\label{sec:properties-of-category-of-models}
In this section we consider various basic properties for the category of (p.c.) models of a fixed theory, such as the amalgamation property (\thref{amalgamation-bases}), downward L\"owenheim-Skolem (\thref{lowenheim-skolem}) and the joint continuation property (\thref{jcp}). We have already seen that these categories have directed unions (i.e., directed colimits, see \thref{directed-unions,pc-models-closed-under-directed-unions}).
\begin{lemma}[Amalgamation lemma]
\thlabel{amalgamation-lemma}
Let $M$ and $M'$ be two $\L$-structures and let $a \in M$ and $b \in M'$ be (possibly infinite) tuples of matching length. If $\tp(b; M') \subseteq \tp(a; M)$ then there is an $\L$-structure $N$ with an elementary embedding $f: M \to N$ and a homomorphism $g: M' \to N$ such that $f(a) = g(b)$.
\end{lemma}
\begin{proof}
Introduce a tuple of new constants $c$ matching $a$ and $b$ to form $\L'$. We extend $M$ and $M'$ to $\L'$-structures by interpreting $c$ as $a$ and $b$ respectively. Let $\Delta$ be the $\L'$-elementary diagram of $M$. That is, all $\L'(M)$-sentences in  full first-order logic that are true in $M$. Let $\Sigma$ be the positive diagram $\Diag(M')$ of $M'$, considered as an $\L'$-structure. It suffices to show that $\Delta \cup \Sigma$ is consistent.

Let $\phi(c, d) \in \Sigma$ where $d$ is some tuple of elements from $M'$, distinct from $c$. Then $\exists y \phi(x, y) \in \tp(b; M')$ and thus $M \models \exists y \phi(a, y)$ because $\tp(b; M') \subseteq \tp(a; M)$. So there is $d' \in M$ with $M \models \phi(a, d')$ and we thus see that $M$ is a model of $\Delta \cup \{\phi(c, d)\}$. We conclude that $\Delta \cup \Sigma$ is finitely consistent and hence consistent.
\end{proof}
\begin{corollary}
\thlabel{amalgamation-bases}
Let $M \xleftarrow{f} M_0 \xrightarrow{g} M'$ be a span of $\L$-structures, where $f$ is a homomorphism and $g$ is an immersion. Then there exists an amalgamation $M \xrightarrow{f'} N \xleftarrow{g'} M'$, that is $f'f = g'g$, where $f'$ is an elementary embedding and $g'$ is a homomorphism.
\end{corollary}
\begin{proof}
By the assumptions on $f$ and $g$ we have $\tp(g(M_0); M') = \tp(M_0; M_0) \subseteq \tp(f(M_0); M)$, so we can apply \thref{amalgamation-lemma} with $a = f(M_0)$ and $b = g(M_0)$.
\end{proof}
\begin{proposition}
\thlabel{immersions-coherent-and-pc-iff-immersed-in-pc}
Some facts:
\begin{enumerate}[label=(\roman*)]
\item if $f: M_1 \to M_2$ and $g: M_2 \to M_3$ are homomorphisms such that $gf$ is an immersion then $f$ is an immersion;
\item if $f: M \to N$ is an immersion and $N$ is a p.c.\ model of $T$ then $M$ is a p.c.\ model of $T$.
\end{enumerate}
\end{proposition}
\begin{proof}
\underline{(i)} Let $a \in M_1$ and $\phi(x)$ be some formula such that $M_2 \models \phi(f(a))$. Then we have $M_3 \models \phi(gf(a))$ and hence $M_1 \models \phi(a)$, as required.

\underline{(ii)} We first prove that $M$ is a model of $T$. Let $\forall x(\phi(x) \to \psi(x)) \in T$ and let $a \in M$ such that $M \models \phi(a)$. Then $N \models \phi(f(a))$ and thus $N \models \psi(f(a))$ because $N$ is a model of $T$. As $f$ is an immersion we have $M \models \psi(a)$ and we conclude that $M$ is a model of $T$.

To show that $M$ is a p.c.\ model we let $g: M \to N'$ be any homomorphism. By \thref{amalgamation-bases} there are $N \xrightarrow{f'} U \xleftarrow{g'} N'$ where $f'$ is a homomorphism and $g'$ is an elementary embedding with $f'f = g'g$. As $N$ is p.c.\ we have in fact that $f'$ is an immersion and so $f'f$ is an immersion. By (i) we conclude that $g$ is an immersion, as required.
\end{proof}
\begin{theorem}[Downward L\"owenheim-Skolem for p.c.\ models]
\thlabel{lowenheim-skolem}
Let $M$ be a p.c.\ model of $T$ and let $A \subseteq M$. Then there is a p.c.\ model $M_0 \subseteq M$ with $A \subseteq M_0$ and $|M_0| \leq |A| + |T|$, such that the inclusion is an elementary embedding.
\end{theorem}
\begin{proof}
By the usual L\"owenheim-Skolem theorem we find $M_0 \subseteq M$ with $A \subseteq M_0$ and $|M_0| \leq |A| + |T|$, such that the inclusion is an elementary embedding. In particular the inclusion is an immersion, so $M_0$ is a p.c.\ model by \thref{immersions-coherent-and-pc-iff-immersed-in-pc}.
\end{proof}
In full first-order logic there are many equivalent definitions of a complete theory. The important ones are that all models satisfy the same set of full first-order sentences and that any two models can be jointly elementarily embedded into a third model (which allows us to build monster models, see Section \ref{sec:monster-models}). The latter is easily generalised to positive logic, and gives rise to the definition below. The positive equivalent of the former---i.e., all p.c.\ models satisfy the same set of h-inductive sentences---is equivalent condition (iv) below.
\begin{definition}
\thlabel{jcp}
A theory $T$ is said to have the \term{joint continuation property}, or \term{JCP} for short, if the following equivalent conditions hold.
\begin{enumerate}[label=(\roman*)]
\item For any two models $M$ and $M'$ of $T$ there is a model $N$ of $T$ with homomorphisms $M \to N \leftarrow M'$.
\item For any two p.c.\ models $M$ and $M'$ of $T$ there is a model $N$ of $T$ with homomorphisms $M \to N \leftarrow M'$.
\item For any two h-universal sentences $\phi$ and $\psi$ we have that $T \models \phi \vee \psi$ implies $T \models \phi$ or $T \models \psi$.
\item For every p.c.\ model $M$ of $T$, $T^\pc$ is precisely the set of all h-inductive sentences that are true in $M$.
\item For some p.c.\ model $M$ of $T$, $T^\pc$ is precisely the set of all h-inductive sentences that are true in $M$.
\item For every p.c.\ model $M$ of $T$, $T^\u$ is precisely the set of all h-universal sentences that are true in $M$.
\item For some p.c.\ model $M$ of $T$, $T^\u$ is precisely the set of all h-universal sentences that are true in $M$.
\end{enumerate}
\end{definition}
\begin{lemma}
\thlabel{equivalence-jcp-characterisations}
The conditions in \thref{jcp} are indeed equivalent.
\end{lemma}
\begin{proof}
We prove the following implications.
\[
\begin{tikzcd}
\text{(iv)} \arrow[d, Rightarrow] & \text{(ii)} \arrow[d, Rightarrow, bend left]                                  & \text{(vi)} \arrow[d, Rightarrow]   \\
\text{(v)} \arrow[rd, Rightarrow] & \text{(i)} \arrow[u, Rightarrow, bend left] \arrow[lu, Rightarrow] \arrow[ru, Rightarrow] & \text{(vii)} \arrow[ld, Rightarrow] \\
                                  & \text{(iii)} \arrow[u, Rightarrow]                                                        &                                    
\end{tikzcd}
\]
\underline{(i) $\Leftrightarrow$ (ii)} The right to left direction is trivial. The other follows after continuing each of $M$ and $M'$ into a p.c.\ model.

\underline{(i) $\Rightarrow$ (iv)} Suppose for a contradiction that there is some h-inductive sentence $\forall x(\phi(x) \to \psi(x))$ that is true in some p.c.\ model $M$ but is not in $T^\pc$. By definition of $T^\pc$ this means that there must be a p.c.\ model $M'$ such that $M' \not \models \forall x(\phi(x) \to \psi(x))$. So there is $a \in M'$ with $M' \models \phi(a)$ and $M' \not \models \psi(a)$. As $M'$ is p.c.\ there is an obstruction $\psi'(x)$ of $\psi(x)$ such that $M' \models \psi'(a)$. That is, we have $M' \models \exists x(\phi(x) \wedge \psi'(x))$. Using (i) we find a model $N$ of $T$ with homomorphisms $M \to N \leftarrow M'$. Each of these homomorphisms is an immersion, because $M$ and $M'$ are p.c. We thus have $N \models \exists x(\phi(x) \wedge \psi'(x))$ and so $M \models \exists x(\phi(x) \wedge \psi'(x))$. However, the latter contradicts $M \models \forall x(\phi(x) \to \psi(x))$.

\underline{(iv) $\Rightarrow$ (v)} Trivial.

\underline{(v) $\Rightarrow$ (iii)} Immediate, using the fact that $(T^\pc)^\u = T^\u$ (\thref{t-pc-u-is-t-u}).

\underline{(i) $\Rightarrow$ (vi)} Suppose for a contradiction that there is some h-universal sentence $\neg \phi$ that is true in some p.c.\ model $M$ but is not in $T^\u$. Then there is a model $M'$ such that $M' \models \phi$. Using (i) we find a model $N$ of $T$ with homomorphisms $M \to N \leftarrow M'$. We thus have $N \models \phi$, but then $M \models \phi$, because $M$ is a p.c.\ model so $M \to N$ is an immersion, a contradiction.

\underline{(vi) $\Rightarrow$ (vii) $\Rightarrow$ (iii)} Trivial.

\underline{(iii) $\Rightarrow$ (i)} Let $M$ and $M'$ be models of $T$. We prove that $\Diag(M) \cup \Diag(M')$ is consistent. If not, then there would be $\phi(a) \in \Diag(M)$ and $\psi(b) \in \Diag(M')$ such that $T \models \neg (\phi(a) \vee \psi(b))$. We may assume $a$ and $b$ to be disjoint, so we get $T \models \neg \exists x \phi(x) \vee \neg \exists y \psi(y)$. So by (iii) we must have either $T \models \neg \exists x \phi(x)$ contradicting $\phi(a) \in \Diag(M)$, or $T \models \neg \exists y \psi(y)$ contradicting $\psi(b) \in \Diag(M')$.
\end{proof}
\begin{definition}
\thlabel{maximal-pc-model}
We call a p.c.\ model $M$ of a theory $T$ a \term{maximal p.c.\ model} if any model $N$ of $T$ admits a homomorphism $N \to M$ into $M$.
\end{definition}
An example of a theory with a maximal model is the empty theory considered in \thref{pc-model-vs-ec-model}, where we established that the singletons are the p.c.\ models. So any singleton is a maximal model of that theory. It can also happen that the maximal p.c.\ model is infinite, see \thref{thick-not-semi-hausdorff}. This is in contrast to full first-order logic, where having a maximal model means that all models are finite.

\section{Boolean, (semi-)Hausdorff and thick}
\label{sec:boolean-hausdorff-semi-hausdorff-thick}
Even though positive logic does not have negation built in, we can add back as much as we desire. This is done through a process called Morleyisation, as described below.
\begin{definition}
\thlabel{morleyisation}
A \term{positive fragment} $\Delta$ of a language $\L$ is a set of formulas in full first-order logic that contains all atomic formulas and is closed under sub-formulas, change of variables, conjunction and disjunction. Given such a positive fragment $\Delta$ we define the ($\Delta$-)\term{Morleyisation} $\Mor(\Delta)$\nomenclature[MorDelta]{$\Mor(\Delta)$}{$\Delta$-Morleyisation} to be the following positive theory. We extend the language to include a relation symbol $R_\phi(x)$ for every $\phi(x) \in \Delta$. Then we inductively add h-inductive sentences to $\Mor(\Delta)$ so that it expresses that $R_\phi(x)$ and $\phi(x)$ are equivalent for every $\phi(x) \in \Delta$ (see below for details).
\end{definition}
\begin{remark}
\thlabel{common-morleyisations}
Fix some language $\L$ and let $\Delta$ be a positive fragment. Given any $\L$-theory $T$ whose axioms are of the form $\forall x(\phi(x) \to \psi(x))$ with $\phi, \psi \in \Delta$ we can, and will, naturally view $T$ as a positive theory by considering $\Mor(\Delta) \cup T'$. Here $T'$ is obtained from $T$ by replacing each $\forall x(\phi(x) \to \psi(x)) \in T$ by $\forall x(R_\phi(x) \to R_\psi(x))$.

There are two particularly interesting cases of the above situation. Write $\Delta_\qf$\nomenclature[Deltaqf]{$\Delta_\qf$}{Set of quantifier-free full first-order formulas} for the set of all quantifier-free full first-order formulas, and $\Delta_\fo$\nomenclature[Deltafo]{$\Delta_\fo$}{Set of all full first-order formulas} for the set of all full first-order formulas. We consider Morleyisation in each case.
\begin{enumerate}[label=(\roman*)]
\item Working with a $\Delta_\qf$-Morleyised theory $T$ is equivalent to saying that the negation of every relation symbol (including equality) is positive definable. In particular, all homomorphisms between models of such a theory $T$ are embeddings. We call such a theory a \term{Pillay theory}. If the theory furthermore has the property that every span of homomorphisms between models $M \xleftarrow{f} M_0 \xrightarrow{g} M'$ can be amalgamated (i.e., there are homomorphisms of models $M \xrightarrow{f'} N \xleftarrow{g'} M'$ such that $f'f = g'g$) then we call it a \term{Robinson theory}.

The p.c.\ models of a Pillay theory are what are classically called existentially closed models or e.c.\ models (see also \thref{pc-model-vs-ec-model}). Often one would prove that such a theory is model complete, or at least has a model companion, so that it can be studied as a full first-order theory. Using positive logic we no longer need to worry about such things. We can just study the theory as a positive theory, even when the theory is not companionable.
\item We can study full first-order logic as a special case of positive logic by working with $\Delta_\fo$-Morleyised theories. In this case every formula in full first-order logic will be equivalent to a positive formula (in fact, to an atomic one). So the notions of homomorphism, immersion and elementary embedding all coincide. This also means that every model is a p.c.\ model. See also \thref{boolean-theory}. Having JCP is now equivalent to the theory being complete (see also the discussion before \thref{jcp}).
\end{enumerate}
\end{remark}
\begin{convention}
\thlabel{full-first-order-logic-as-special-case}
Whenever we say we that we consider a theory $T$ in full first-order logic as a positive theory, we mean its Morleyised version as described in \thref{common-morleyisations}(ii). In this light we will view positive logic as a proper generalisation of full first-order logic.
\end{convention}
\begin{lemma}
\thlabel{morleyisation-does-exist}
The positive theory $\Mor(\Delta)$ described in \thref{morleyisation} does indeed exist.
\end{lemma}
\begin{proof}
We add h-inductive sentences to $\Mor(\Delta)$ based on the complexity of a formula $\phi(x) \in \Delta$. We may assume that $\phi(x)$ is built using the connectives $\vee, \wedge, \neg$ and $\exists$, as any other full first-order connectives can be treated as abbreviations for these connectives. So we split into the following cases, based on the outermost connective in $\phi(x)$.

\underline{Atomic $\phi(x)$.} We can simply add the sentences $\forall x(\phi(x) \to R_\phi(x))$ and $\forall x(R_\phi(x) \to \phi(x))$, which are h-inductive because $\phi(x)$ is atomic.

\underline{Connectives $\vee$ and $\wedge$.} Write $\bigcirc$ for the relevant connective (i.e. either $\vee$ or $\wedge$) so that $\phi(x)$ is $\phi_1(x) \bigcirc \phi_2(x)$. We then add the sentences $\forall x(R_\phi(x) \to R_{\phi_1}(x) \bigcirc R_{\phi_2}(x))$ and $\forall x(R_{\phi_1}(x) \bigcirc R_{\phi_2}(x) \to R_\phi(x))$.

\underline{Connective $\neg$.} So $\phi(x)$ is of the form $\neg \psi(x)$. We then add the sentences $\forall x(R_\phi(x) \vee R_\psi(x))$ and $\forall x \neg(R_\phi(x) \wedge R_\psi(x))$.

\underline{Connective $\exists$.} So $\phi(x)$ is of the form $\exists y \psi(x, y)$. We add the sentences $\forall x(R_\phi(x) \to \exists y R_\psi(x, y))$ and $\forall x (\exists y R_\psi(x, y) \to R_\phi(x))$.

One now easily proves by induction on the complexity of the formula $\phi(x)$ that $\Mor(\Delta) \models \forall x(\phi(x) \leftrightarrow R_\phi(x))$.
\end{proof}
\begin{definition}
\thlabel{boolean-theory}
We call a theory $T$ \emph{Boolean}\index{Boolean theory} if the following equivalent conditions hold.
\begin{enumerate}[label=(\roman*)]
\item Every model of $T$ is a p.c.\ model.
\item Every homomorphism between models of $T$ is an immersion.
\item For every positive formula $\phi(x)$ there is a positive formula $\psi(x)$ such that $T \models \forall x(\neg \phi(x) \leftrightarrow \psi(x))$.
\item For every full first-order formula $\phi(x)$ there is a positive formula $\psi(x)$ such that $T \models \forall x(\phi(x) \leftrightarrow \psi(x))$.
\item Every homomorphism between models of $T$ is an elementary embedding.
\end{enumerate}
\end{definition}
Clearly any Morleyised full first-order theory is Boolean, so following \thref{full-first-order-logic-as-special-case} we use these terms interchangeably.

The term Boolean in \thref{boolean-theory} refers to the fact that for such theories the distributive lattice of positively definable sets is in fact a Boolean algebra. Some sources use the term \term{positively model complete}, but Boolean seems more descriptive.
\begin{lemma}
\thlabel{homomorphism-preserved-formulas-are-positive}
Let $T$ be a theory. Suppose that $\phi(x)$ is a full first-order formula such that for every homomorphism $f: M \to N$ of models of $T$, and any $a \in M$ we have $M \models \phi(a)$ implies $N \models \phi(f(a))$. Then $\phi(x)$ is equivalent to a positive formula modulo $T$.
\end{lemma}
\begin{proof}
We first prove the following claim. Let $M$ be any model of $T$ and $a \in M$ such that $M \models \phi(a)$. Then there is positive $\psi(x)$ such that $T \models \forall x(\psi(x) \to \phi(x))$ and $M \models \psi(a)$.

To prove the claim we consider the set of formulas $T \cup \Diag(M) \cup \{\neg \phi(a)\}$. This cannot be consistent, as that would give us a homomorphism $f: M  \to N$ with $N \models T$ and $N \not \models \phi(f(a))$. There is thus some $\chi(a, b) \in \Diag(M)$ such that $T \models \chi(a, b) \to \phi(a)$. As $a$ and $b$ do not appear in $T$ this means that $T \models \forall x(\exists y \chi(x, y) \to \phi(x))$, and taking $\psi(x)$ to be $\exists y \chi(x, y)$ proves the claim.

Let $\Psi(x)$ be the set of all formulas $\neg \psi(x)$ such that $\psi(x)$ is positive and implies $\phi(x)$, modulo $T$. We will show that $T \cup \Psi(x) \cup \{\phi(x)\}$ is inconsistent. If it were consistent then we find a model $M$ of $T$ and $a \in M$ such that $M \models \Psi(a)$ and $M \models \phi(a)$. Using the claim there must be a positive $\psi(x)$ that implies $\phi(x)$ modulo $T$ such that $M \models \psi(a)$. However, by definition $\neg \psi(x) \in \Psi(x)$, so this contradicts $M \models \Psi(a)$.

Let $\{\neg \psi_1(x), \ldots, \neg \psi_n(x)\} \subseteq \Psi(x)$ such that $T \cup \{\neg \psi_1(x), \ldots, \neg \psi_n(x), \phi(x)\}$ is inconsistent. Then $T \models \forall x(\phi(x) \to \psi_1(x) \vee \ldots \vee \psi_n(x))$, and we conclude that $\phi(x)$ is equivalent to the positive formula $\psi_1(x) \vee \ldots \vee \psi_n(x)$ modulo $T$.
\end{proof}
\begin{lemma}
\thlabel{equivalence-boolean-characterisations}
The conditions in \thref{boolean-theory} are indeed equivalent.
\end{lemma}
\begin{proof}
We prove (i) $\Rightarrow$ (ii) $\Rightarrow$ (iii) $\Rightarrow$ (iv) $\Rightarrow$ (v) $\Rightarrow$ (i). The first and last implication are trivial, so we prove the remaining three.

\underline{(ii) $\Rightarrow$ (iii)} Let $\phi(x)$ be a positive formula. As every homomorphism between models of $T$ is an immersion, truth of $\neg \phi(x)$ is also preserved upwards by such homomorphisms. By \thref{homomorphism-preserved-formulas-are-positive} we conclude that $\neg \phi(x)$ is equivalent to a positive formula, modulo $T$.

\underline{(iii) $\Rightarrow$ (iv)} This easily follows by induction on the complexity of the full first-order formula $\phi(x)$, where we use (iii) for the induction step with $\neg$. We treat the remaining connectives such as $\to$ and $\forall$ as abbreviations using the connectives for positive formulas together with $\neg$.

\underline{(iv) $\Rightarrow$ (v)} To verify that a homomorphism $f: M \to N$ is an elementary embedding it suffices to check that for every full first-order formula $\phi(x)$ and every $a \in M$ we have that $M \models \phi(a)$ implies $N \models \phi(f(a))$. This immediately follows from (iv) because $\phi(x)$ will be equivalent to a positive formula, modulo $T$, and hence its truth is preserved upwards by homomorphisms.
\end{proof}
The following definitions, except for being Boolean, are taken from \cite{ben-yaacov_thickness_2003}. These assumptions are very useful for developing (neo)stability theory for positive logic, while the weaker ones---thickness, and even being semi-Hausdorff---are relatively mild. Before we define them, we first need to recall the notion of an indiscernible sequence.
\begin{definition}
\thlabel{indiscernible-sequence}
Fix a theory $T$. An \term{indiscernible sequence} is an infinite sequence $(a_i)_{i \in I}$ in some p.c.\ model $M$ such that for any $i_1 < \ldots < i_n$ and $j_1 < \ldots < j_n$ in $I$ we have
\[
\tp(a_{i_1} \ldots a_{i_n}; M) = \tp(a_{j_1} \ldots a_{j_n}; M).
\]
\end{definition}
So a sequence is indiscernible precisely when any two subsequences of the same order-type have the same type.
\begin{definition}
\thlabel{hausdorff-thick}
Let $T$ be a theory. We call $T$:
\begin{itemize}
\item \emph{Boolean} if every formula in full first-order logic is equivalent to some positive formula, modulo $T$ (or any of the equivalent statements from \thref{boolean-theory});
\item \emph{Hausdorff}\index{Hausdorff theory} if for any two distinct types $p(x)$ and $q(x)$ there are $\phi(x) \not \in p(x)$ and $\psi(x) \not \in q(x)$ such that $T^\pc \models \forall x(\phi(x) \vee \psi(x))$;
\item \emph{semi-Hausdorff}\index{semi-Hausdorff theory} if equality of types is type-definable, so there is a partial type $\Omega(x, y)$ such that for any $a, b$ in any p.c.\ model $M$ we have $\tp(a; M) = \tp(b; M)$ if and only if $M \models \Omega(a, b)$;
\item \emph{thick}\index{thick theory} if being an indiscernible sequence is type-definable, so there is a partial type $\Theta((x_i)_{i < \omega})$ such that for any sequence $(a_i)_{i < \omega}$ in any p.c.\ model $M$ we have that $(a_i)_{i < \omega}$ is indiscernible if and only if $M \models \Theta((a_i)_{i < \omega})$.
\end{itemize}
\end{definition}
The reason for the name Hausdorff is that this corresponds to the type spaces being Hausdorff, see \thref{hausdorff-equivalences}.

We mentioned Boolean theories in \thref{hausdorff-thick} again because they fit very well in the hierarchy mentioned there, as is apparent from the following proposition.
\begin{proposition}
\thlabel{hausdorff-implies-thick}
Boolean implies Hausdorff implies semi-Hausdorff implies thick.
\end{proposition}
\begin{proof}
\underline{Boolean implies Hausdorff.} Let $p(x)$ and $q(x)$ be distinct types. Pick any $\phi(x) \in q(x)$ such that $\phi(x) \not \in p(x)$. Because the theory is Boolean there must be $\psi(x)$ that is equivalent to $\neg \phi(x)$, modulo the theory. So we have $\psi(x) \not \in q(x)$ while also $T \models \forall x(\phi(x) \vee \psi(x))$, so in particular $T^\pc \models \forall x(\phi(x) \vee \psi(x))$, as required.

\underline{Hausdorff implies semi-Hausdorff.} Define
\begin{align*}
\Omega(x, y) = \{\phi(x, y) :\; &\text{for all } a, b \text{ in any p.c.\ model } M \text{ with } \tp(a) = \tp(b)\\
&\text{we have } M \models \phi(a, b)\}.
\end{align*}
Let $a, b$ be arbitrary in some arbitrary p.c.\ model $M$. By construction we have that $\tp(a) = \tp(b)$ implies $M \models \Omega(a, b)$. For the other direction we prove the contrapositive. So suppose that $\tp(a) \neq \tp(b)$. Because the theory is Hausdorff there are $\phi(x) \not \in \tp(a)$ and $\psi(x) \not \in \tp(b)$ such that $T^\pc \models \forall x(\phi(x) \vee \psi(x))$. The latter means that by definition of $\Omega(x, y)$ we then have $(\phi(x) \wedge \phi(y)) \vee (\psi(x) \wedge \psi(y)) \in \Omega(x, y)$. The former means that $M \not \models (\phi(a) \wedge \phi(b)) \vee (\psi(a) \wedge \psi(b))$, hence $M \not \models \Omega(a, b)$, as required.

\underline{Semi-Hausdorff implies thick.} Define the partial type $\Theta((x_i)_{i < \omega})$ as:
\[
\bigcup \{ \Omega(x_{i_1}, \ldots, x_{i_n}; x_{j_1}, \ldots, x_{j_n}) : n < \omega, i_1 < \ldots < i_n < \omega, j_1 < \ldots < j_n < \omega \}.
\]
Here $\Omega(x_{i_1}, \ldots, x_{i_n}; x_{j_1}, \ldots, x_{j_n})$ is the partial type expressing that $x_{i_1}, \ldots, x_{i_n}$ and $x_{j_1}, \ldots, x_{j_n}$ have the same type, which exists by the semi-Hausdorff assumption. So $\Theta((x_i)_{i < \omega})$ expresses that any two finite subsequences of $(x_i)_{i < \omega}$ of the same length have the same type, and hence it expresses indiscernibility.
\end{proof}
The following characterisations of Hausdorff theories (\thref{separate-types-equivalences}) are useful in practice.
\begin{definition}
\thlabel{aph}
We say that a theory $T$ has the \term{h-amalgamation property} or \term{APh} if any span $M \xleftarrow{f} M_0 \xrightarrow{g} M'$ of homomorphisms between models of $T$ can be amalgamated (i.e., there are homomorphisms $M \xrightarrow{f'} N \xleftarrow{g'} M'$ with $N \models T$ such that $f'f = g'g$).
\end{definition}
\begin{definition}
\thlabel{separable-types}
We call types $p(x)$ and $q(x)$ \emph{separable}\index{separable types} for a theory $T$ if there are $\phi(x) \not \in p(x)$ and $\psi(x) \not \in q(x)$ such that $T \models \forall x(\phi(x) \vee \psi(x))$. We say that a theory $T$ \term{separates types} if any two distinct types are separable for $T$.
\end{definition}
So a theory $T$ is Hausdorff iff $T^\pc$ separates types.
\begin{lemma}
\thlabel{maximal-extensions-are-realised}
Let $M$ be a model of some theory $T$ and let $a \in M$. Write $\pi(x) = \tp(a; M)$ and let $p(x) \in \S(T)$ be a maximal type with $\pi(x) \subseteq p(x)$. Then there is a homomorphism $f: M \to N$ with $N \models T$ and $N \models p(f(a))$.
\end{lemma}
\begin{proof}
It is enough to show that $\Diag(M) \cup p(a)$ is consistent. So let $\phi(a, b) \in \Diag(M)$ where $b$ is a tuple of parameters from $M$ disjoint from $a$. Then $\exists y \phi(a, y) \in \pi(a) \subseteq p(a)$. As $p(a)$ is consistent, we conclude by compactness that $\Diag(M) \cup p(a)$ is consistent.
\end{proof}
\begin{proposition}
\thlabel{separate-types-equivalences}
The following are equivalent for a theory $T$:
\begin{enumerate}[label=(\roman*)]
\item $T$ separates types;
\item for any $M \models T$ and $a \in M$ there is some type $p(x)$ such that for any homomorphism $f: M \to N$ where $N$ is a p.c.\ model of $T$ we have that $N \models p(f(a))$;
\item $T$ has APh.
\end{enumerate}
\end{proposition}
\begin{proof}
We prove (i) $\Leftrightarrow$ (ii) $\Leftrightarrow$ (iii).

\underline{(i) $\Rightarrow$ (ii)} Suppose for a contradiction that there are homomorphisms $f: M \to N$ and $g: M \to N'$, where $N$ and $N'$ are p.c.\ models, such that $\tp(f(a); N) \neq \tp(g(a); N')$. As $T$ separates types we find $\phi(x) \not \in \tp(f(a); N)$ and $\psi(x) \not \in \tp(g(a); N')$ such that $T \models \forall x(\phi(x) \vee \psi(x))$. Then we must either have $M \models \phi(a)$ or $M \models \psi(a)$. In the first case we get $N \models \phi(f(a))$, contradicting $\phi(x) \not \in \tp(f(a); N)$, and in the second case we get $N' \models \psi(g(a))$, contradicting $\psi(x) \not \in \tp(g(a); N')$.

\underline{(ii) $\Rightarrow$ (i)} Let $p(x)$ and $q(x)$ be two distinct types and consider the set of formulas
\[
\Sigma(x) = \{ \neg \phi(x) : \phi(x) \not \in p(x) \text{ or } \phi(x) \not \in q(x) \}.
\]
We claim that $\Sigma(x)$ is inconsistent with $T$. If it were consistent then there would be a model $M$ of $T$ and some $a \in M$ with $M \models \Sigma(a)$. By construction $\tp(a; M) \subseteq p(x) \cap q(x)$. So by \thref{maximal-extensions-are-realised} there are homomorphisms $f: M \to N$ and $g: M \to N$ into models of $T$, which we may assume to be p.c.\ models by \thref{continue-to-pc-model}, such that $\tp(f(a); N) = p(x)$ and $\tp(g(a); N') = q(x)$. This contradicts (ii), and so $\Sigma(x)$ must be inconsistent with $T$.

There are thus $\neg \phi_1(x), \ldots, \neg \phi_n(x) \in \Sigma(x)$ such that $T \models \forall x(\phi_1(x) \vee \ldots \vee \phi_n(x))$. Let $I = \{i : \phi_i(x) \not \in p(x)\}$ and $J = \{j : \phi_j(x) \not \in q(x)\}$. By definition of $\Sigma(x)$ we have that $I \cup J = \{1, \ldots, n\}$. Let $\psi_I(x) = \bigvee_{i \in I} \phi_i(x)$ and $\psi_J(x) = \bigvee_{i \in J} \phi_i(x)$. Then $T \models \forall x(\psi_I(x) \vee \psi_J(x))$ while $\psi_I(x) \not \in p(x)$ and $\psi_J(x) \not \in q(x)$, so $T$ separates $p(x)$ and $q(x)$.

\underline{(ii) $\Rightarrow$ (iii)} Let $M \xleftarrow{f} M_0 \xrightarrow{g} M'$ be a span of homomorphisms. We may assume $M$ and $M'$ to be p.c.\ models. Let $a$ be a tuple that enumerates $M_0$. By assumption $\tp(f(a); M) = \tp(g(a); M')$. So we find the required amalgamation by \thref{amalgamation-lemma}.

\underline{(iii) $\Rightarrow$ (ii)} Let $M \models T$ and $a \in M$. Continue $M$ into some p.c.\ model $N$ by a homomorphism $f: M \to N$. We claim that $p(x) = \tp(f(a); N)$ is as described in (ii). Let $g: M \to N'$ be any other homomorphism with $N'$ a p.c.\ model. By APh we find $N \xrightarrow{f'} N^* \xleftarrow{g'} N'$ such that $f'f = g'g$. We thus see that $\tp(g(a); N') = \tp(g'g(a); N^*) = \tp(f'f(a); N^*) = \tp(f(a); N) = p(x)$, where we used that $f'$ and $g'$ are immersions because $N$ and $N'$ are p.c.\ models.
\end{proof}
\begin{proposition}
\thlabel{hausdorff-equivalences}
The following are equivalent for a theory $T$:
\begin{enumerate}[label=(\roman*)]
\item $T$ is Hausdorff;
\item $\S_I(T)$ is a Hausdorff space for all index sets $I$;
\item $T^\pc$ separates types;
\item $T^\pc$ has APh.
\end{enumerate}
\end{proposition}
\begin{proof}
The equivalence (i) $\Leftrightarrow$ (iii) is just a reformulation and (iii) $\Leftrightarrow$ (iv) is \thref{separate-types-equivalences}. It remains to show (i) $\Leftrightarrow$ (ii).

Being a Hausdorff space can be formulated as follows: for any distinct points $p$ and $q$ there are closed sets $A$ and $B$ such that $p \not \in A$ and $q \not \in B$ while $A \cup B$ is the entire space. With this in mind (i) $\Rightarrow$ (ii) follows easily: for distinct types $p(x)$ and $q(x)$ we $\phi(x) \not \in p(x)$ and $\psi(x) \not \in q(x)$ such that $T^\pc \models \forall x(\phi(x) \vee \psi(x))$ and consider the closed sets $[\phi(x)]$ and $[\psi(x)]$.

We prove (ii) $\Rightarrow$ (i). Let $p(x)$ and $q(x)$ be distinct types and let $[\pi(x)]$ and $[\rho(x)]$ be such that $p(x) \not \in [\pi(x)]$ and $q(x) \not \in [\rho(x)]$ while $[\pi(x)] \vee [\rho(x)]$ is the entire space. Here we used that any closed set in $\S_I(T)$ is of the form $[\sigma(x)]$ for some set of formulas $\sigma(x)$. As $p(x) \not \in [\pi(x)]$ by definition means that $\pi(x) \not \subseteq p(x)$, we find $\phi(x) \in \pi(x)$ such that $\phi(x) \not \in p(x)$. Similarly we find $\psi(x) \in \rho(x)$ such that $\psi(x) \not \in q(x)$. We have $[\pi(x)] \subseteq [\phi(x)]$ and $[\rho(x)] \subseteq [\psi(x)]$ so $[\phi(x)] \cup [\psi(x)] = \S_I(T)$ which means that $T^\pc \models \forall x(\phi(x) \vee \psi(x))$, as required.
\end{proof}
We now consider some examples to show that none of the implications in \thref{hausdorff-implies-thick} are reversible and that non-thick theories exist.
\begin{example}
\thlabel{hausdorff-not-boolean}
We give an example of a Hausdorff non-Boolean theory. Write $\Q_{(0,1)} = \{q \in \Q : 0 < q < 1\}$. Let $\L$ be the language that has a constant symbol for each element of $\Q_{(0,1)}$ and an order symbol $\leq$. Considering the obvious $\L$-structure on $\Q_{(0,1)}$ we let $T$ be the set of h-inductive sentences that are true in $\Q_{(0,1)}$.

Let $[0,1]$ be the real unit interval. We claim that any model $M$ of $T$ admits a unique homomorphism into $[0,1]$. For a singleton $a \in M$ we let $L_a = \{q \in \Q_{(0,1)} : M \models q \leq a \}$ and $R_a = \{q \in \Q_{(0,1)} : M \models a \leq q\}$. As $T$ specifies that $\leq$ is a linear order, this determines a Dedekind cut in $[0, 1]$ and so there is a unique $r_a \in [0, 1]$ such that $q_1 \leq r_a \leq q_2$ for all $q_1 \in L_a$ and $q_2 \in R_a$. We can thus define a homomorphism by $M \to [0,1]$ be sending $a$ to $r_a$, and clearly every homomorphism must send $a$ to $r_a$.

There are two consequences of the above claim. Firstly, it means that the real unit interval $[0,1]$ is a maximal p.c.\ model for this theory. This means that $T$ is not Boolean because there is an infinite maximal p.c.\ model (cf.\ \thref{bounded-set-compactness}).

Secondly, it means that $T$, and hence $T^\pc$, has APh. So $T$ is Hausdorff by \thref{hausdorff-equivalences}. Indeed, any span of homomorphisms $M \leftarrow M_0 \to M'$ between models of $T$ can be completed to a square by composing with the homomorphisms $M \to [0,1] \leftarrow M'$. This square then commutes by uniqueness of the homomorphism $M_0 \to [0,1]$.
\end{example}
\begin{example}
\thlabel{semi-hausdorff-not-hausdorff}
We give an example of a semi-Hausdorff non-Hausdorff theory. For this we use the theory $T$ from \thref{bounded-set-compactness}, which has $\omega$ many constants $c_n$ and declares them to be distinct. There is then only one p.c.\ model (up to isomorphism), namely $\omega$ with $c_n$ interpreted as $n$. Then for any tuples $a$ and $b$ we have $a \equiv b$ if and only if $a$ and $b$ are equal to the same tuple of constants if and only if $a = b$. So $T$ is semi-Hausdorff. To show that $T$ is not Hausdorff we show that APh fails for $T^\pc$. Using \thref{regular-formulas-eliminate-quantifiers-implies-positive-quantifier-elimination} one quickly checks that $T$, and hence $T^\pc$, has positive quantifier elimination, from which it quickly follows that $T^\pc$ does not specify anything more than $T$ does. That is, $T$ and $T^\pc$ are logically equivalent. Let $M$ be $\omega$ together with one extra point $*$, which is then a model of $T^\pc$. We define a homomorphism $f_1: M \to \omega$ by taking the identity on $\omega$ and setting $f_1(*) = 1$. Similarly we define $f_2: M \to \omega$ with $f_2(*) = 2$. The span $\omega \xleftarrow{f_1} M \xrightarrow{f_2} \omega$ cannot be amalgamated.
\end{example}
\begin{example}
\thlabel{thick-not-semi-hausdorff}
We give an example of a thick theory that is not semi-Hausdorff. Consider the signature $\L$ with unary relation symbols $P_n$ and $P'_n$ for all $n < \omega$, and a binary relation symbol $R$. We define the $\L$-structure $M = \{a_n, b_n : n < \omega\}$ as follows. The interpretation of $P_n$ is $\{a_n, b_n\}$ and $P'_n$ is the complement of $P_n$. We take $R$ to be the symmetric anti-reflexive relation $\{(a_n, b_n), (b_n, a_n) : n < \omega\}$, so $R$ is the inequality relation on each $P_n$. Let $T$ be the h-inductive theory of $M$. Then $M$ is a p.c.\ model of $T$. There is a maximal p.c.\ model $N$ of $T$ given by $N = M \cup \{a_\omega, b_\omega\}$, where $N \models P'_n(a_\omega) \wedge P'_n(b_\omega)$ for all $n < \omega$ and also $N \models R(a_\omega, b_\omega) \wedge R(b_\omega, a_\omega)$.

Since $N$ is maximal, the only indiscernible sequences are the constant ones. So $T$ is a thick theory. However, $T$ is not semi-Hausdorff. To see this, suppose for a contradiction that $\Omega(x_1, x_2, y_1, y_2)$ is a partial type such that for any $c_1, c_2, d_1, d_2 \in N$ we have $N \models \Omega(c_1, c_2, d_1, d_2)$ if and only if $\tp(c_1, c_2; N) = \tp(d_1, d_2; N)$. For $n < \omega$ define the set of formulas
\[
\Sigma_n(x, z_1, z_2) = \Omega(x, z_1, x, z_2) \cup \{R(z_1, z_2)\} \cup \{P'_k(x) \wedge P'_k(z_1) \wedge P'_k(z_2) : k < n\},
\]
and set $\Sigma(x, z_1, z_2) = \bigcup_{n < \omega} \Sigma_n(x, z_1, z_2)$. For any $n < \omega$ we have that $N \models \Sigma_n(a_\omega, a_n, b_n)$, so by compactness there is a p.c.\ model of $T$ that contains a realisation of $\Sigma(x, z_1, z_2)$. By maximality of $N$ such a realisation must exist in $N$. That is, there are $c, d, e \in N$ such that $N \models \Sigma(c, d, e)$. By construction of $\Sigma(x, z_1, z_2)$ we have $\{d, e\} = \{a_\omega, b_\omega\}$ and $c$ is one of $a_\omega$ or $b_\omega$. Without loss of generality we can thus assume $N \models \Sigma(a_\omega, a_\omega, b_\omega)$. Writing $p(x, y) = \tp(a_\omega, a_\omega; N)$ and $q(x, y) = \tp(a_\omega, b_\omega; N)$ we must then have $p(x, y) = q(x, y)$. This is a contradiction, because $p(x, y)$ contains the formula $x = y$, while $q(x, y)$ contains an obstruction of this formula, namely $R(x, y)$. We conclude that $T$ is not semi-Hausdorff.
\end{example}
\begin{example}
\thlabel{not-thick}
We show that the theory from \thref{compactness-failure} is not thick. Recall that our signature contains unary relation symbols $P_n$ for all $n < \omega$ and an inequality symbol $\neq$. The theory $T$ expresses that $P_n$ and $P_k$ are disjoint for all $n \neq k$ and that the inequality symbol is indeed inequality. The p.c.\ models of $T$ then consist of an infinite set for each $P_n$.

We claim that $T$ has positive quantifier elimination. We will use \thref{regular-formulas-eliminate-quantifiers-implies-positive-quantifier-elimination}, so let $\exists y \phi(x_1, \ldots, x_n, y)$ be a positive formula where $\phi(x_1, \ldots, x_n, y)$ is a conjunction of atomic formulas and $y$ is a single variable. We distinguish two cases.
\begin{itemize}
\item One of the atomic formulas in $\phi$ is of the form $x_i = y$ for some $1 \leq i \leq n$. Then we may replace all occurrences of $y$ by $x_i$ and thus eliminate the quantifier.
\item The variable $y$ does not appear in any equality in $\phi$. If the atomic formulas involving $y$ contain a contradiction (e.g., $y \neq y$ or $P_n(y) \wedge P_k(y)$ for $n \neq k$) then the entire formula is equivalent to $\bot$. Otherwise we can safely remove all atomic formulas involving $y$ from $\phi$ and thus eliminate the quantifier.
\end{itemize}

Now suppose for a contradiction that $T$ is thick. Let $\Theta((x_i)_{i < \omega})$ express (in p.c.\ models) that $(x_i)_{i < \omega}$ is an indiscernible sequence. Then $\{P_0(x_0) \wedge P_1(x_1)\} \cup \Theta((x_i)_{i < \omega})$ cannot be realised in a p.c.\ model, so by compactness there is some $\phi(x_1, \ldots, x_N)$ that is a finite conjunction of formulas in $\Theta((x_i)_{i < \omega})$, such that $\{P_0(x_0) \wedge P_1(x_1), \phi(x_1, \ldots, x_N)\}$ is inconsistent with $T$. That is, $\exists (x_i)_{2 \leq i \leq N} \phi(x_0, \ldots, x_N)$ is an obstruction of $P_0(x_0) \wedge P_1(x_1)$. By positive quantifier elimination, this obstruction is equivalent to a formula
\[
\bigvee_{i \in I} \phi_i(x_0, x_1),
\]
where each $\phi_i(x_0, x_1)$ is a conjunction of atomic formulas. As every $\phi_i(x_0, x_1)$ must be an obstruction of $P_0(x_0) \wedge P_1(x_1)$, it must contain at least one of the following atomic formulas:
\begin{itemize}
\item $P_n(x_0)$ for some $n < \omega$,
\item $P_n(x_1)$ for some $n < \omega$,
\item $x_0 = x_1$.
\end{itemize}
Let $k < \omega$ be such that $P_k$ is not mentioned in any $\phi_i(x_0, x_1)$. Let $M$ be some p.c.\ model of $T$ and let $(a_i)_{i < \omega}$ be distinct elements in $P_k(M)$. Then $M \models \Theta((a_i)_{i < \omega})$. However, by choice of $k$ we have $M \not \models \phi_i(a_0, a_1)$ for all $i \in I$ and so $M \not \models \phi(a_0, \ldots, a_N)$, contradicting that this is a finite conjunction of formulas in $\Theta((x_i)_{i < \omega})$.
\end{example}
\section{Bibliographic remarks}
\label{sec:bibliographic-remarks-basics}
The basics of positive logic (i.e., the contents of Sections \ref{sec:very-basics} and \ref{sec:properties-of-category-of-models}) appear in, for example, \cite{ben-yaacov_positive_2003, ben-yaacov_fondements_2007, poizat_positive_2018}. There is also \cite[Chapter 8]{hodges_model_1993}, which treats what we called Pillay theories (see \thref{common-morleyisations}). So the set up there is slightly less general, but the proofs are really the same and are easily adapted to our more general setting. A lot of the current terminology is based on this less general setting (e.g., ``joint continuation property'' versus ``joint embedding property'' and ``p.c.\ model''  versus ``e.c.\ model''). Some literature actually uses the older terminology, even in the more general setting for positive logic (see also \thref{pc-model-vs-ec-model}). By working in a $\Delta_\qf$-Morleyised theory the older notions are all obtained as a special case of the newer notions. For example: homomorphisms are precisely embeddings of structures, h-inductive and h-universal sentences are precisely inductive and universal sentences respectively (also called $\forall \exists$-sentences and $\forall$-sentences respectively) and a continuation of structures is just an extension. Finally, because homomorphisms can be viewed as just inclusions of structures, directed unions are genuine unions.

Our treatment of type spaces follows \cite{ben-yaacov_positive_2003}. Another approach is possible, which we discuss in the remark below.
\begin{remark}
\thlabel{prime-types}
In \cite{haykazyan_spaces_2019} a different kind of type for positive logic is considered. There all sets of formulas of the form $\tp(a; M)$ are taken, where $M$ is just some model of a fixed theory $T$. So $M$ is not necessarily p.c. This corresponds to taking prime filters on the distributive lattice of definable sets of $T$. We will call these \emph{prime types}.

This way another topological space can be defined, for which we will write $\S_n'(T)$ (to distinguish it from the space $\S_n(T)$ of maximal types). The points are prime types, and a basis of open sets is given by the subsets of the form
\[
[\phi(x)] = \{ p(x) \in \S_n'(T) : \phi(x) \in p(x) \}.
\]
This will yield a \emph{spectral space}. This is precisely the same approach as is standard in full first-order logic. The latter can be viewed as an instance of the duality Boolean algebras and Stone spaces, namely with definable sets in $n$ variables on the Boolean algebra side and the Stone space of types on the topological side. This duality generalises to one between distributive lattices and spectral spaces, of which the distributive lattice of positively definable sets together with the space $\S_n'(T)$ forms an instance.

There is a general philosophy in model theory that a theory is `the same' as the collection of its type spaces. For both $\S(T)$ and $\S'(T)$ this can be made precise.
\begin{itemize}
\item In \cite[Theorem 2.38]{ben-yaacov_positive_2003} it is made precise how a theory can be recovered from the collection of type spaces `like' ones of the form $\S_n(T)$. If we start with a theory $T$, consider its type spaces and then recover a theory $T'$ from that we can never hope that $T$ and $T'$ are exactly the same. This is because $\S_n(T)$ only detects types in p.c.\ models, and there can be many different theories with the same p.c.\ models (see e.g.\ \thref{universal-kaiser-hull-ec-models}). Another problem is that all that the type space $\S_n(T)$ sees are type-definable sets, because these correspond to the closed sets in the type space and there is no way to distinguish between formulas and sets of formulas. However, non of these things are relevant for the model-theoretic properties of the theories involved. For example, $T$ and $T'$ will have monster models with the same automorphism groups.
\item In \cite[Theorem 1.1]{kamsma_type_2023} a duality is described between positive theories and collections of types spaces `like' ones of the form $\S_n'(T)$. It turns out that this approach does offer enough detail to recover a theory up to logical equivalence (if the language is fixed, otherwise up to some appropriate isomorphism).
\end{itemize}
\end{remark}

In \thref{common-morleyisations}(i) we defined Pillay and Robinson theories. This is a very common setup in model theory and is thus deserving of its own name (see also \cite[Chapter 8]{hodges_model_1993}). In 1998, Hrushovski introduced the name Robinson theory \cite{hrushovski_simplicity_1998}, based on the substantial amount of work that Robinson had done in this setup. Later, Pillay developed simplicity theory in a similar setting \cite{pillay_forking_2000}, but an important difference was that the amalgamation property was no longer required (something that we have seen implies quite good behaviour, namely that the theory is Hausdorff, see \thref{separate-types-equivalences}). Therefore, it seemed fitting to use the name Pillay theories for Robinson theories without the amalgamation property.

The properties Hausdorff and semi-Hausdorff were defined in \cite[Definition 1.41]{ben-yaacov_positive_2003}. They were defined in topological terms, hence the terminology, and a topological proof is given that Hausdorff implies semi-Hausdorff. Our presentation focuses on the logical aspect. Thickness is introduced in \cite{ben-yaacov_thickness_2003}. The fact that being Hausdorff is equivalent to APh (\thref{hausdorff-equivalences}) is \cite[Theorem 8]{poizat_positive_2018}.

The examples of a non-Boolean Hausdorff theory (\thref{hausdorff-not-boolean}) and a thick theory that is not semi-Hausdorff (\thref{thick-not-semi-hausdorff}) are taken from \cite[Example 3.12]{dmitrieva_dividing_2023} and \cite[Section 4]{poizat_quelques_2010} respectively. Although, in both cases they served a completely different purpose. The example of a semi-Hausdorff theory that is not Hausdorff (\thref{semi-hausdorff-not-hausdorff}) is essentially \cite[Example 4]{poizat_positive_2018}. These examples, and the other examples we gave in Section \ref{sec:boolean-hausdorff-semi-hausdorff-thick}, are elementary but rather artificial. More natural examples are known.
\begin{itemize}
\item Existentially closed exponential fields are Hausdorff and non-Boolean. These are studied in \cite{haykazyan_existentially_2021}, with not being Boolean being established in \cite[Corollary 3.8]{haykazyan_existentially_2021}. The fact that this example is Hausdorff is established in \cite[Proposition 10.4]{dobrowolski_kim-independence_2022}, see also the discussion after that proposition.
\item Bilinear spaces over a fixed field can be studied as a positive theory. This is done in \cite{kamsma_bilinear_2023}. If the field is finite then the theory is Boolean, and we are in the well-known setting of studying bilinear spaces over a finite field in full first-order logic. However, if the field is infinite then the theory is semi-Hausdorff and not Hausdorff \cite[Proposition 4.14]{kamsma_bilinear_2023}.
\item Ultrametric spaces with distances in $\N$ are known to be not thick. These are first discussed in \cite[Example 4.3]{ben-yaacov_simplicity_2003}, and non-thickness is established as a consequence of the general theory of simplicity in positive logic. That is (referring forward to material from later on in these notes), \cite[Example 4.3]{ben-yaacov_simplicity_2003} establishes that the type of a single element over the empty set has no non-dividing extensions. So dividing independence does not satisfy \textsc{full existence}. At the same time, the theory is shown to be stable and thus in particular simple (\thref{stable-implies-simple}), but in thick simple theories dividing independence satisfies \textsc{full existence} (\thref{simple-thick-implies-full-existence}).
\end{itemize}

For more information about geometric logic, as mentioned in \thref{geometric-axiomatisation}, see for example \cite[Chapter D1]{johnstone_sketches_2002}. There one can also find more on the topos-theoretic approach to positive logic, under the name coherent logic, as briefly discussed after \thref{positive-formulas-and-theory}.

Another use for positive logic in the context categorical logic is in accessible categories. For example, accessible categories can be characterised as the categories of models of infinitary positive theories. More can be found in \cite[Chapter 5]{adamek_locally_1994}, where positive theories are called basic theories.

\chapter{Countable categoricity}
\label{ch:countable-categoricity}

We provide a characterisation of theories that have one countable p.c.\ model (up to isomorphism), i.e.\ \emph{countably categorical theories}, see \thref{countably-categorical-characterisation}. In full first-order logic this is usually referred to as the Ryll-Nardzewski theorem, and all the usual characterisations go through except for having finite type spaces (see \thref{finite-type-spaces}).

Along the way we establish a positive version of multiple model-theoretically important notions and results, such as omitting types (\thref{omitting-types}), back-and-forth (\thref{back-and-forth}), saturation (Section \ref{sec:saturation}), atomic models (Section \ref{sec:atomic-models}) and prime models (Section \ref{sec:prime-models}).

\section{Omitting types}
\label{sec:omitting-types}
We give a criterion for when a subset of an arbitrary structure is in fact a p.c.\ model of some given theory. This reminds of Tarski's test for full first-order logic and the positive logic version is due to Haykazyan, hence the name.
\begin{lemma}[Haykazyan's test]
\thlabel{haykazyans-test}
Let $T$ be a theory and $M \models T$. Suppose that $A \subseteq M$ is a subset such that for every $a \in A$ and all quantifier-free $\phi(x, y)$ one of the following holds:
\begin{itemize}
\item there is $b \in A$ such that $M \models \phi(a, b)$,
\item there is quantifier-free $\psi(x, z)$ and $c \in A$ such that $M \models \psi(a, c)$ and $T \models \neg \exists x y z(\phi(x, y) \wedge \psi(x, z))$.
\end{itemize}
Then $A$ is a p.c.\ model of $T$.
\end{lemma}
\begin{proof}
We claim that for every quantifier-free formula $\phi(x, y)$ and all $a \in A$ we have that if $M \models \exists y \phi(a, y)$ then there is $b \in A$ such that $M \models \phi(a, b)$. Suppose for a contradiction that this is not the case. Then there is quantifier-free $\psi(x, z)$ and $c \in A$ such that $M \models \psi(a, c)$ and $T \models \neg \exists x y z(\phi(x, y) \wedge \psi(x, z))$. As $M \models T$ this in particular implies that $M \not \models \exists x y \phi(x, y)$, contradicting that $M \models \exists y \phi(a, y)$.

Using the claim, we can now easily check that $A$ is a substructure. For that we need to check that it is closed under constant symbols and function symbols\footnote{If one insists that structures are non-empty then $M$ is non-empty, so $M \models \exists y(y = y)$ and hence by the claim there is $b \in A$ (with $M \models b = b$). If we allow empty structures then $M$ can be empty, and so $A = M$ will be empty. This will still be a p.c.\ model, and this case does actually not require special treatment, but we still explain what happens. The assumption implies there will be some $\psi$, which now has to be propositional, such that $M \models \psi$ and $T \models \psi \to \neg \exists x(x = x)$. So any continuation of $M$ will be empty.}. So let $c$ be a constant symbol. Then $M \models \exists y(y = c)$, and so there is $b \in A$ with $M \models b = c$, thus $c \in A$. Similarly, if $f$ is an $n$-ary function symbol and $a_1, \ldots, a_n \in A$ then $y = f(a_1, \ldots, a_n)$ is a formula with parameters in $A$ and because $M \models \exists y(y = f(a_1, \ldots, a_n))$, there is $b \in A$ with $M \models b = f(a_1, \ldots, a_n)$ and hence $f(a_1, \ldots, a_n) \in A$.

So by construction $A \subseteq M$ is an embedding of structures. By \thref{models-of-universal-consequences} we have that $A \models T^u$. We will show that $A$ is a p.c.\ model of $T^\u$, which by \thref{universal-kaiser-hull-ec-models} implies that $A$ is a p.c.\ model of $T$. For this, we verify \thref{pc-model}(iii). Let $\chi(x)$ be any formula and write it as $\exists y \phi(x, y)$, where $\phi(x, y)$ is quantifier-free. Let $a \in A$ be such that $A \not \models \chi(a)$. Then there is no $b \in A$ such that $M \models \phi(a, b)$, as that would imply $A \models \phi(a, b)$. There must thus be quantifier-free $\psi(x, z)$ and some $c \in A$ such that $M \models \psi(a, c)$ and $T \models \neg \exists xyz(\phi(x, y) \wedge \psi(x, z))$. Therefore $A \models \psi(a, c)$ and $T^\u \models \neg \exists xyz(\phi(x, y) \wedge \psi(x, z))$. So $\exists z \psi(x, z)$ is an obstruction of $\chi(x)$ modulo $T^\u$ and $A \models \exists z \psi(a, z)$, as required.
\end{proof}
In full first-order logic we call a type \emph{isolated} if there is a formula implying the entire type. The term comes from the fact that this corresponds to the type being an isolated point in the type space. In positive logic we can use the same idea, but such types are no longer necessarily isolated points, so we change the terminology.
\begin{definition}
\thlabel{supported-type}
Let $T$ be a theory. We call a partial type $\Sigma(x)$ in finitely many variables \emph{supported}\index{supported partial type} if there is $\phi(x)$ such that $T \cup \{ \exists x \phi(x) \}$ is consistent and for all $\chi(x) \in \Sigma(x)$ we have
\[
T^\pc \models \forall x(\phi(x) \to \chi(x)).
\]
In this case we call $\phi(x)$ the \term{support} of $\Sigma(x)$.
\end{definition}
\begin{lemma}
\thlabel{supported-type-equivalent}
Let $T$ be a theory. Let $p(x)$ be a type. Then $p(x)$ is supported by $\phi(x)$ if and only if $\phi(x) \in p(x)$ and for every $\psi(x) \not \in p(x)$ we have $T \models \neg \exists x (\phi(x) \wedge \psi(x))$.
\end{lemma}
\begin{proof}
We first prove the left to right direction. Let $\psi(x) \not \in p(x)$ and assume for a contradiction that there is a model $M$ of $T$ with $a \in M$ such that $M \models \phi(a) \wedge \psi(a)$. We may assume that $M$ is a p.c.\ model. Write $q(x) = \tp(a; M)$. As $\phi(x)$ implies every formula in $p(x)$ modulo $T^\pc$ we have $p(x) \subseteq q(x)$, and hence $p(x) = q(x)$ by maximality. We arrive at a contradiction, as then $\psi(x) \in p(x)$. To see that $\phi(x) \in p(x)$ we again assume for a contradiction that this is not the case. By what we have just established we must then have $T \models \neg \exists x \phi(x)$, contradicting that $T \cup \{ \exists x \phi(x) \}$ is consistent.

We now prove the converse. Firstly, $T \cup \{ \exists x \phi(x) \}$ is consistent as $\phi(x) \in p(x)$ and $p(x)$ is realised in some model of $T$. Now let $\chi(x) \in p(x)$. Suppose for a contradiction that $T^\pc \not \models \forall x(\phi(x) \to \chi(x))$. Then there is a p.c.\ model $M$ of $T$ and some $a \in M$ with $M \models \phi(a)$ and $M \not \models \chi(a)$. There is thus an obstruction $\psi(x)$ of $\chi(x)$ such that $M \models \psi(a)$. As $\chi(x) \in p(x)$, we must have $\psi(x) \not \in p(x)$. So by our assumption on $\phi(x)$ we must have $T \models \neg \exists x(\phi(x) \wedge \psi(x))$, contradicting $M \models \phi(a) \wedge \psi(a)$.
\end{proof}
\begin{proposition}
\thlabel{supported-types-must-be-realised}
If a theory $T$ has JCP then every p.c.\ model of $T$ realises all supported types.
\end{proposition}
\begin{proof}
Let $M$ be any p.c.\ model and let $p(x)$ be a type that is supported by some $\phi(x)$. By \thref{supported-type-equivalent} we have $\phi(x) \in p(x)$. Let $M'$ be a p.c.\ model of $T$ that realises $p(x)$, so in particular $M' \models \exists x \phi(x)$. As $T$ has JCP there is a model $N$ of $T$ with immersions $M \to N \leftarrow M'$. So $N \models \exists x \phi(x)$ and hence $M \models \exists x \phi(x)$. Let $a \in M$ be such that $M \models \phi(a)$. As $M \models T^\pc$ and because $\phi(x)$ supports $p(x)$ we have $M \models p(a)$, and so $p(x)$ is realised in $M$.
\end{proof}
By the above proposition, supported types have to be realised (if the theory has JCP). As is the case with isolated types in full first-order logic, these are the only types that are necessarily realised by p.c.\ models (if $T$ is countable), as the following theorem shows.
\begin{theorem}[Omitting types]
\thlabel{omitting-types}
Let $T$ be a countable theory. If $\Sigma(x)$ is a partial type in finitely many variables that is not supported then there is a p.c.\ model $M$ of $T$ that omits $\Sigma(x)$. That is, for all $a \in M$ we have $M \not \models \Sigma(a)$.
\end{theorem}
\begin{proof}
Write $\Sigma(x) = \Sigma(x_1, \ldots, x_n)$, where $x_1, \ldots, x_n$ are all single variables.

Let $C = \{c_i\}_{i < \omega}$ be a set of new constant symbols. We will construct a set of sentences $\Gamma$ in this extended language that is consistent with $T$ and satisfies the following properties.
\begin{enumerate}[label=(\arabic*)]
\item For every quantifier-free formula $\phi(y, z)$ in the original language and any tuple $a \in C$ one of the following holds:
\begin{enumerate}[label=(\roman*)]
\item there is a tuple $b \in C$ for which we have $\phi(a, b) \in \Gamma$,
\item there is a quantifier-free $\psi(y, w)$ in the original language and a tuple $c \in C$ such that $\psi(a, c) \in \Gamma$ and $T \models \neg \exists y z w(\phi(y, z) \wedge \psi(y, w))$.
\end{enumerate}
\item For all $\{i_1, \ldots, i_n\} \subseteq \omega$ there is an obstruction $\psi(x_1, \ldots, x_n)$ of $\Sigma(x_1, \ldots, x_n)$ such that $\psi(c_{i_1}, \ldots, c_{i_n}) \in \Gamma$.
\end{enumerate}
We only defined what it means to be ``an obstruction'' of a formula, but this straightforwardly extends to sets of formulas: $\psi(x)$ is an obstruction of $\Sigma(x)$ if $\{ \psi(x) \} \cup \Sigma(x)$ is inconsistent with $T$. Equivalently, using compactness, if there are $\phi_1(x), \ldots, \phi_k(x) \in \Sigma(x)$ such that $\psi(x)$ is an obstruction of $\phi_1(x) \wedge \ldots \wedge \phi_k(x)$ modulo $T$.

We will construct $\Gamma$ as the union of a countable chain $\emptyset = \Gamma_0 \subseteq \Gamma_1 \subseteq \ldots$, such that each $\Gamma_i$ is finite. We alternate constructions between even and odd stages. For this we let $( \phi_i(a_i, z) )_{i < \omega}$ be an enumeration of all quantifier-free formulas in the language extended by the constant symbols in $C$. We also let $(\bar{c}_i)_{i < \omega}$ be an enumeration of $C^n$.

Having constructed $\Gamma_{2i}$, we let $C' \subseteq C$ be the finite subset of constants that appear in $\Gamma_{2i}$. Let $M$ be a p.c.\ model of $T$ that realises $\Gamma_{2i}$, and interpret any constants from $a_i$ that do not appear in $\Gamma_{2i}$ arbitrarily in $M$. To construct $\Gamma_{2i+1}$ we distinguish two cases.
\begin{enumerate}[label=(\roman*)]
\item If $M \models \exists z \phi_i(a_i, z)$ then let $b \in M$ be such that $M \models \phi_i(a_i, b)$. For each element in $b$ that does not correspond to the interpretation of a constant symbol from $C$, we pick a constant symbol from $C \setminus C' a_i$ so that $b$ becomes the interpretation of constants from $C$. Set $\Gamma_{2i+1} = \Gamma_{2i} \cup \{\phi_i(a_i, b)\}$, which is consistent with $T$ as $M$ is a model.
\item If $M \not \models \exists z \phi_i(a_i, z)$ then, because $M$ is a p.c.\ model, there is an obstruction $\theta(y)$ of $\exists z \phi_i(y, z)$ such that $M \models \theta(a_i)$. Write $\theta(y)$ as $\exists w \psi(y, w)$ with $\psi(y, w)$ quantifier-free, so $T \models \neg \exists y z w (\phi(y, z) \wedge \psi(y, w))$. Let $c \in M$ be such that $M \models \psi(a_i, c)$. For each element in $c$ that does not correspond to the interpretation of a constant symbol from $C$, we pick a constant symbol from $C \setminus C' a_i$ so that $c$ becomes the interpretation of constants from $C$. Set $\Gamma_{2i+1} = \Gamma_{2i} \cup \{\psi(a_i, c)\}$, which is consistent with $T$ as $M$ is a model.
\end{enumerate}
This completes the construction of $\Gamma_{2i+1}$.

Now assume that we have constructed $\Gamma_{2i+1}$. Then there is a formula $\theta(x, y)$ such that $\Gamma_{2i+1}$ is equivalent to $\theta(\bar{c}_i, d)$, where $d$ is a tuple of constants from $C$ that is disjoint from $\bar{c}_i$. As $\exists y \theta(x, y)$ does not support $\Sigma(x)$, there must be $\chi(x) \in \Sigma(x)$ and some p.c.\ model $M$ of $T$ with $\bar{a} \in M$ such that $M \models \exists y \theta(\bar{a}, y)$ and $M \not \models \chi(\bar{a})$. As $M$ is a p.c.\ model, there is an obstruction $\psi(x)$ of $\chi(x)$ (and hence of $\Sigma(x)$) such that $M \models \psi(\bar{a})$. Let $b \in M$ be such that $M \models \theta(\bar{a}, b)$. Interpret the constant symbols in $\bar{c}_i$ and $d$ as $\bar{a}$ and $b$ respectively. Setting $\Gamma_{2i+2} = \Gamma_{2i+1} \cup \{\psi(\bar{c}_i)\}$ we then have $M \models \Gamma_{2i+2}$, and so $\Gamma_{2i+2}$ is consistent with $T$.

This completes the construction of $\Gamma$. Now let $N$ be a model of $T \cup \Gamma$. Let $M \subseteq N$ be the set enumerated by constant symbols in $C$. Property (1) of $\Gamma$ tells us that \thref{haykazyans-test} applies, and so we have that $M$ is a p.c.\ model of $T$. At the same time, property (2) of $\Gamma$ tells us that no tuple in $M$ can realise $\Sigma(x)$. So $M$ is indeed the required p.c.\ model that omits $\Sigma(x)$.
\end{proof}
\section{Back-and-forth}
\label{sec:back-and-forth}
It often happens in model theory that we build isomorphisms using a back-and-forth argument. Such arguments work just as well in positive logic. In this section we give a technical setup and work out the details so that we can directly apply it in many situations.
\begin{definition}
\thlabel{partial-immersion}
Let $M$ and $N$ be two structures in the same signature. A partial function $f: M \to N$ is called a \term{partial immersion}\index{immersion!partial} if for all $a \in \dom(f)$ and for every formula $\phi(x)$ we have
\[
M \models \phi(a) \Longleftrightarrow N \models \phi(f(a)).
\]
\end{definition}
Note that if $f$ is the empty function then $f$ is a partial immersion precisely when $M \models \phi$ if and only if $N \models \phi$ for every sentence $\phi$ (see also the discussion after \thref{immersion}).
\begin{definition}
\thlabel{back-and-forth-correspondence}
Let $M$ and $N$ be two structures in the same signature, and let $\kappa$ be a cardinal. We say that $M$ and $N$ are in $\kappa$-\term{back-and-forth correspondence} if:
\begin{enumerate}[label=(\roman*)]
\item $M \models \phi$ if and only if $N \models \phi$ for every sentence $\phi$,
\item for every partial immersion $f: M \to N$ with $|\dom(f)| < \kappa$ and any singleton $a \in M$ there is $b \in N$ such that $f$ can be extended to a partial immersion that sends $a$ to $b$,
\item for every partial immersion $f: M \to N$ with $|\dom(f)| < \kappa$ and any singleton $b \in N$ there is $a \in M$ such that $f$ can be extended to a partial immersion that sends $a$ to $b$.
\end{enumerate}
\end{definition}
The difference between (ii) and (iii) in the above definition is that that in (ii) we start with $a \in M$ and find $b \in N$, and in (iii) this is the other way around.
\begin{remark}
\thlabel{pc-models-of-jcp-theory-satisfy-same-sentences}
It is often the case that two structures $M$ and $N$, which we want to prove are in $\kappa$-back-and-forth correspondence, are in fact p.c.\ models of some theory $T$ with JCP. In this case, (i) in \thref{back-and-forth-correspondence} is automatic. Indeed, the characterisation in \thref{jcp} tells us that $T^\pc$ is precisely the set of h-inductive sentences true in $M$, and the same for $N$. The claim then follows because any sentence can in particular be viewed as an h-inductive sentence.
\end{remark}
\begin{theorem}[Back-and-forth]
\thlabel{back-and-forth}
Suppose that $M$ and $N$ are two structures of cardinality at most $\kappa$ that are in $\kappa$-back-and-forth correspondence. Then $M$ and $N$ are isomorphic.
\end{theorem}
\begin{proof}
Choose enumerations $(a_i)_{i < \kappa}$ and $(b_i)_{i < \kappa}$ of $M$ and $N$ respectively (possibly repeating elements if their cardinalities are less than $\kappa$). We will inductively construct bijections $f_i: A_i \to B_i$, such that for all $i \leq \kappa$:
\begin{enumerate}[label=(\arabic*)]
\item $f_i$ extends $f_j$ for all $j < i$;
\item $(a_j)_{j < i} \subseteq A_i \subseteq M$;
\item $(b_j)_{j < i} \subseteq B_i \subseteq N$;
\item $f_i$ is a partial immersion $M \to N$;
\item $|A_i| = |B_i| < \kappa$ (except for $i = \kappa$).
\end{enumerate}
Clearly, $f_\kappa$ would then be the desired isomorphism. At limit stages we take unions. For the base case we take the empty function, which trivially satisfies (1)--(3) and satisfies (4) because $M$ and $N$ are in $\kappa$-back-and-forth correspondence (see \thref{back-and-forth-correspondence}(i)).

That leaves the successor case. Let $f_i$ be constructed. Then $A_i = \dom(f_i)$ has cardinality less than $\kappa$ by the induction hypothesis. So using the $\kappa$-back-and-forth correspondence of $M$ and $N$ there is $b \in N$ such that we can extend $f_i$ to a partial immersion $g: A_i \cup \{a_i\} \to B_i \cup \{b\}$ by setting $g(a_i) = b$. It follows immediately that $g$ is surjective. Injectivity follows from being a partial immersion, and so $g$ is a bijection. Again using $\kappa$-back-and-forth correspondence of $M$ and $N$, this time applied to $g$, we find $a \in M$ such that we can extend $g$ to a partial immersion $f_{i+1}: A_i \cup \{a_i, a\} \to B_i \cup \{b, b_i\}$ by setting $f_{i+1}(a) = b_i$, which is again a bijection. We set $A_{i+1} = A_i \cup \{a_i, a\}$ and $B_{i+1} = B_i \cup \{b, b_i\}$. This finishes the construction and hence the proof.
\end{proof}
\section{Positive saturation}
\label{sec:saturation}
\begin{definition}
\thlabel{finitely-satisfiable}
Let $M$ be a structure and $A \subseteq M$ be any subset. A \term{formula over $A$} is a formula $\phi(x, a)$, where $a \in A$. A set $\Sigma(x)$ of formulas over $A$ is called \term{finitely satisfiable} in $M$ if for any finite subset of formulas $\{\phi_1(x, a_1), \ldots, \phi_n(x, a_n)\} \subseteq \Sigma(x)$ we have that $M \models \exists x(\phi_1(x, a_1) \wedge \ldots \wedge \phi_n(x, a_n))$. We say that $\Sigma(x)$ is \emph{satisfiable in $M$} if there is $b \in M$ such that $M \models \Sigma(b)$.
\end{definition}
Technically what is happening with the set $A$ is that we temporarily add constant symbols to our signature for the elements of $A$, and expand $M$ by interpreting these constant symbols as their corresponding elements.
\begin{lemma}
\thlabel{continue-to-realise-type}
Let $M$ be a p.c.\ model of some theory $T$ and let $\Sigma(x)$ be a set of formulas over $M$ in any number of variables. Then $\Sigma(x)$ is finitely satisfiable in $M$ if and only if there is a continuation of $M$ to a p.c.\ model $N$ of $T$ in which $\Sigma(x)$ is satisfiable.
\end{lemma}
\begin{proof}
For the left to right direction we consider the set of formulas $\Diag(M) \cup \Sigma(x)$. By assumption this is finitely satisfiable in a model of $T$, namely in $M$. By compactness (\thref{compactness}) there is a p.c.\ model $N$ of $T$ and $a \in N$ such that $N \models \Diag(M) \cup \Sigma(a)$. As $N \models \Diag(M)$ there is a homomorphism $M \to N$ and so $N$ is the required continuation of $M$.

Conversely, let $a \in N$ be such that $N \models \Sigma(x)$ and let $\Sigma_0(x) \subseteq \Sigma(x)$ be a finite subset. Then $N \models \exists x \bigwedge \Sigma_0(x)$. So as $M$ is p.c.\ we have $M \models \exists x \bigwedge \Sigma_0(x)$, and we conclude that $\Sigma(x)$ is finitely satisfiable in $M$.
\end{proof}
In practice we usually care about satisfiability and realisations of sets of formulas, and so we may as well assume them to be closed under finite conjunctions. Similar to how we used $\bigwedge \Sigma_0(x)$ in the above proof. This simplifies $\Sigma(x)$ being finitely satisfiable in $M$ to: for every $\phi(x, a) \in \Sigma(x)$ we have $M \models \exists x \phi(x, a)$. Throughout, we will implicitly use this notational convenience.
\begin{definition}
\thlabel{saturated-model}
Let $\kappa$ be an infinite cardinal. A structure $M$ is called \emph{positively $\kappa$-saturated}\index{positively saturated structure} if, for every $A \subseteq M$ with $|A| < \kappa$, every set $\Sigma(x)$ of formulas over $A$ in a single variable $x$ that is finitely satisfiable in $M$ is satisfiable in $M$.
\end{definition}
\begin{lemma}
\thlabel{saturation-extend-variables}
Let $\kappa$ be an infinite cardinal. A structure $M$ is positively $\kappa$-saturated if and only if for every $A \subseteq M$ with $|A| < \kappa$, every set $\Sigma(x)$ of formulas over $A$ and with $|x| \leq \kappa$ that is finitely satisfiable in $M$ is satisfiable in $M$.
\end{lemma}
The difference between \thref{saturated-model} and \thref{saturation-extend-variables} is that in the latter we allow $x$ to be of length $\kappa$.
\begin{proof}
Let $A \subseteq M$ with $|A| < \kappa$ and let $\Sigma((x_i)_{i < \kappa})$ be a set of formulas over $A$ in $\kappa$ many variables, where each $x_i$ is a single variable. We may assume that for any $\phi(y, z) \in \Sigma((x_i)_{i < \kappa})$, where $y$ is a single variable and $z$ potentially a tuple of variables, we also have $\exists y \phi(y, z) \in \Sigma((x_i)_{i < \kappa})$. Indeed, this does not change satisfiability (or finite satisfiability) of $\Sigma((x_i)_{i < \kappa})$.

For $\delta < \kappa$ we let $\Sigma_\delta((x_i)_{i < \delta})$ be the restriction of $\Sigma((x_i)_{i < \kappa})$ to the variables $(x_i)_{i < \delta}$. That is, we take only those formulas that mention those variables. We inductively build a sequence $(b_i)_{i < \kappa}$ such that for all $\delta < \kappa$ we have $M \models \Sigma_\delta((b_i)_{i < \delta})$. The base case and the limit stages are trivial. So now assume we have constructed $(b_i)_{i < \delta}$. Note that $\Sigma_{\delta+1}(x_\delta, (b_i)_{i < \delta})$ is a set of formulas in a single variable over $A(b_i)_{i < \delta}$, and this set of parameters has cardinality $< \kappa$. Therefore, we only need to check that $\Sigma_{\delta+1}(x_\delta, (b_i)_{i < \delta})$ is finitely satisfiable in $M$, as saturation then gives us a realisation, which is exactly the required $b_\delta$.

So let $\phi(x_\delta, (b_i)_{i < \delta}) \in \Sigma_{\delta+1}(x_\delta, (b_i)_{i < \delta})$, where $\phi(x_\delta, (x_i)_{i < \delta})$ is some formula over $A$ (only really using finitely many variables from $(x_i)_{i < \delta}$). Then $\exists x_\delta \phi(x_\delta, (x_i)_{i < \delta}) \in \Sigma((x_i)_{i < \kappa})$ by our earlier assumption. We conclude that $\exists x_\delta \phi(x_\delta, (b_i)_{i < \delta}) \in \Sigma_\delta((b_i)_{i < \delta})$ and so $M \models \exists x_\delta \phi(x_\delta, (b_i)_{i < \delta})$, as required.
\end{proof}
\begin{example}
\thlabel{saturated-does-not-imply-pc}
Positive saturation does not imply being p.c. For example, take the empty theory in the language of pure equality. So the p.c.\ models are the singletons (see also \thref{pc-model-vs-ec-model}). Let $M = \{a, b\}$ be a two element set. Then any set $\Sigma(x)$ of formulas with parameters in $M$ can only say that $x$ is equal to one of $a$ and $b$ or to both. The latter is not finitely satisfiable in $M$, and the other options are clearly satisfiable in $M$. So $M$ is positively $\kappa$-saturated for all $\kappa$, but $M$ is not p.c.
\end{example}
We will generally only be interested in positive saturatedness in p.c.\ models. In Section \ref{sec:building-positively-saturated-models} we will see that we can always construct such p.c.\ models. However, the lemma below already gives us a criterion for when every p.c.\ model of a theory is positively $\omega$-saturated, and so in particular establishing their existence.
\begin{lemma}
\thlabel{supported-n-types-implies-omega-saturated}
Let $T$ be a theory and suppose that every type in finitely many variables is supported. Then every p.c.\ model of $T$ is positively $\omega$-saturated.
\end{lemma}
\begin{proof}
Let $M$ be a p.c.\ model and let $\Sigma(x, b)$ be finitely satisfiable in $M$, where $x$ and $b \in M$ are finite. Then by \thref{continue-to-realise-type} there is a realisation $a$ in some p.c.\ model $N$ that is a continuation of $M$. Set $p(x, y) = \tp(a, b; M)$ and let $\phi(x, y)$ be the support of $p(x, y)$. Then $N \models \exists x \phi(x, b)$ so because $M$ is p.c.\ we find $a' \in M$ with $M \models \phi(a', b)$. As $\phi$ supports $p$ we have that $M \models p(a', b)$ and hence $M \models \Sigma(a', b)$ because $\Sigma(x, b) \subseteq p(x, b)$ by construction.
\end{proof}
As usual, we get that positively saturated structures are isomorphic.
\begin{theorem}
\thlabel{saturated-models-isomorphic}
Let $T$ be a theory with JCP. Suppose that $M$ and $N$ are positively $\kappa$-saturated p.c.\ models of $T$ with $|M|, |N| \leq \kappa$. Then $M$ and $N$ are isomorphic.
\end{theorem}
\begin{proof}
We will prove that $M$ and $N$ are in $\kappa$-back-and-forth correspondence. Then the result follows from \thref{back-and-forth}. Following \thref{pc-models-of-jcp-theory-satisfy-same-sentences}, (i) from \thref{back-and-forth-correspondence} is automatic, so we prove (ii), and (iii) follows by symmetry.

Let $f: M \to N$ be a partial immersion with $|\dom(f)| < \kappa$ and let $a \in M$ be a singleton. Write $C = \dom(f)$, and set $p(x) = \tp(a/C; M) = \{ \phi(x, c) : c \in C \text{ and } M \models \phi(a, c) \}$. Consider the type
\[
f(p)(x) = \{ \phi(x, f(c)) : \phi(x, c) \in p(x) \}.
\]
As $p(x)$ is satisfiable in $M$, we have that $f(p)(x)$ is finitely satisfiable in $N$ because $f$ is a partial immersion. By positive $\kappa$-saturatedness of $N$ there is a realisation $b \in N$ of $f(p)(x)$. Let $g$ extend $f$ by setting $g(a) = b$. We prove that $g$ is a partial immersion. For this, we let $\phi(x, y)$ be any formula ($x$ may not appear in $\phi$) and let $c \in C$ match the length of $y$. By construction we have that $M \models \phi(a, c)$ implies $N \models \phi(b, f(c))$. It is the converse that requires an argument. We prove the contrapositive. Suppose that $M \not \models \phi(a, c)$. As $M$ is a p.c.\ model, there is an obstruction $\psi(x, y)$ of $\phi(x, y)$ modulo $T$ such that $M \models \psi(a, c)$. By the already established direction, we have $N \models \psi(b, f(c))$ and hence $N \not \models \phi(b, f(c))$. This concludes the proof.
\end{proof}
\begin{example}
\thlabel{example-saturated-models-but-not-iso}
The analogous version of the above theorem in full first-order logic is often stated as: any two $\kappa$-saturated elementary equivalent structures of cardinality $\kappa$ are isomorphic. The assumption that $M$ and $N$ are p.c.\ models of the same theory with JCP is then the analogue of being elementary equivalent, but it is slightly more subtle. The argument does not go through if we let $M$ and $N$ be two positively $\kappa$-saturated structures of cardinality at most $\kappa$ that satisfy the same h-inductive sentences. The issue is that they may not be p.c.\ models, and the $g$ constructed in the proof of \thref{saturated-models-isomorphic} will generally only be a partial homomorphism, but not a partial immersion.

As a concrete example we consider the theory from \thref{bounded-set-compactness}, which has $\omega$ many constant symbols $\{c_i\}_{i < \omega}$ and declares them all to be distinct. Then the model $M$ that consists only of interpretations for the constant symbols is the unique p.c.\ model. By the usual compactness theorem and downward L\"owenheim-Skolem, there is a countable model $N$ that is elementary equivalent to $M$ and which contains an element $a \in N$ that is not the interpretation of any of the constant symbols. One quickly checks that both $M$ and $N$ are positively $\omega$-saturated. However, they are clearly not isomorphic.
\end{example}
\section{Atomic models}
\label{sec:atomic-models}
\begin{definition}
\thlabel{atomic-prime-model}
Let $T$ be a theory. We call $M$ an \term{atomic model} of $T$ if it is a p.c.\ model of $T$ in which only supported types are realised.
\end{definition}
\begin{theorem}
\thlabel{countable-atomic-isomorphic}
Let $T$ be a theory with JCP. Then any two countable atomic models of $T$ are isomorphic.
\end{theorem}
\begin{proof}
Let $M$ and $N$ be countable atomic models of $T$. We show that $M$ and $N$ are in $\omega$-back-and-forth correspondence, so the isomorphism follows from \thref{back-and-forth}. Following \thref{pc-models-of-jcp-theory-satisfy-same-sentences}, (i) from \thref{back-and-forth-correspondence} is automatic, so we prove (ii), and (iii) follows by symmetry.

Let $f: M \to N$ be a partial immersion with $|\dom(f)|$ finite and let $a \in M$ be a singleton. Enumerate $\dom(f)$ as a finite tuple $a'$. By assumption, $\tp(a,a'; M)$ is supported. So let $\phi(x, y)$ be the support of $\tp(a,a'; M)$. Then $M \models \exists x \phi(x, a')$ and so $N \models \exists x \phi(x, f(a'))$. Let $b$ be such that $N \models \phi(b, f(a'))$. Then because $N$ is a p.c.\ model we have by definition of being a support that $\tp(a,a'; M) \subseteq \tp(b, f(a'); N)$. By maximality of types in p.c.\ models (\thref{pc-model-iff-types-are-maximal}) we have $\tp(a,a'; M) = \tp(b, f(a'); N)$, and so $f$ can be extended to a partial immersion that sends $a$ to $b$, as required.
\end{proof}
\section{Prime models}
\label{sec:prime-models}
\begin{definition}
\thlabel{prime-model}
Let $T$ be a theory. We call $M$ a \term{prime model} if it is a p.c.\ model of $T$ and every p.c.\ model of $T$ is a continuation of $M$.
\end{definition}
\begin{proposition}
\thlabel{prime-model-implies-jcp}
If a theory $T$ has a prime model then $T$ has JCP.
\end{proposition}
\begin{proof}
Let $M$ be a prime model and let $N_1$ and $N_2$ be two p.c.\ models. Then there are immersions $N_1 \leftarrow M \to N_2$. By \thref{amalgamation-bases} we can amalgamate this span to obtain $N_1 \to N \leftarrow N_2$, where $N$ is a model of $T$.
\end{proof}
\begin{proposition}
\thlabel{prime-model-is-small}
If a theory $T$ has a prime model $M$ then $|M| \leq |T|$.
\end{proposition}
\begin{proof}
By \thref{lowenheim-skolem} there is a p.c.\ model $N$ of $T$ with $|N| \leq |T|$. Since $M$ is prime, there is a homomorphism $M \to N$. As $M$ is a p.c.\ model, this is an immersion and so in particular it is an injection. Hence $|M| \leq |N| \leq |T|$.
\end{proof}
\begin{theorem}
\thlabel{prime-iff-countable-and-atomic}
Let $T$ be a countable theory with JCP, and let $M$ be a p.c.\ model of $T$. Then $M$ is prime if and only if it is countable and atomic.
\end{theorem}
\begin{proof}
Countability follow from \thref{prime-model-is-small}. Next we show that $M$ only realises supported types. Let $p(x)$ be a type that is not supported. Then by \thref{omitting-types} there is a p.c.\ model $N$ of $T$ that omits $p(x)$. As $M$ is prime there is a homomorphism $f: M \to N$. So if $p(x)$ were to be realised in $M$, say by $a \in M$, then $f(a)$ would realise $p(x)$ in $N$. Therefore, $p(x)$ cannot be realised in $M$.

For the right to left direction we let $N$ be any p.c.\ model of $T$. Enumerate $M$ as $(a_i)_{i < \omega}$. We will inductively construct an increasing chain of functions $f_i: \{a_j\}_{j < i} \to N$ such that for all tuples $a$ in $\{a_j\}_{j < i}$ and all $\phi(x)$ we have
\[
M \models \phi(a) \quad \implies \quad N \models \phi(f(a)).
\]
Then $f = \bigcup_{i < \omega} f_i$ will be the desired homomorphism $M \to N$. For the base case we take the empty function, for which we need to check that every sentence satisfied by $M$ is also satisfied by $N$ (see also \thref{homomorphisms-and-sentences}). This is indeed true, because both $M$ and $N$ are p.c.\ models of the same theory with JCP. So, by \thref{jcp}, $N$ is a model of $T^\pc$, which can be computed as the set of h-inductive sentences satisfied by $M$ (and any positive sentence can be viewed as an h-inductive sentence).

Having constructed $f_i$, we let $p(x, y) = \tp(a_i, (a_j)_{j < i}; M)$. As $M$ is atomic, there is a support $\phi(x, y)$ of $p(x, y)$. By \thref{supported-type-equivalent} we have $\phi(x, y) \in p(x, y)$ and so $M \models \exists x \phi(x,(a_j)_{j < i})$. By the induction hypothesis we then have $N \models \exists x \phi(x, f((a_j)_{j < i}))$. Let $b \in N$ be such that $N \models \phi(b, f((a_j)_{j < i}))$. As $\phi(x, y)$ supports $p(x, y)$ and $N$ is a p.c.\ model, we have $N \models p(b, f((a_j)_{j < i}))$. We can thus set $f_{i+1}(a_i) = b$. This completes the construction and thus the proof.
\end{proof}
\begin{corollary}
\thlabel{prime-models-isomorphic}
Let $T$ be a countable theory. Any two prime models of $T$ are isomorphic.
\end{corollary}
\begin{proof}
Let $M$ and $N$ be prime models of $T$. By \thref{prime-model-implies-jcp}, $T$ has JCP, so \thref{prime-iff-countable-and-atomic} applies. Hence $M$ and $N$ are both countable and atomic. The result now follows from \thref{countable-atomic-isomorphic}.
\end{proof}
\section{Characterising countably categorical theories}
\label{sec:characterising-countably-categorical-theories}
\begin{definition}
\thlabel{categoricity}
Let $\kappa$ be a cardinal. A theory $T$ is called \emph{$\kappa$-categorical}\index{categorical theory} if it has only one p.c.\ model of cardinality $\kappa$, up to isomorphism.
\end{definition}
\begin{theorem}
\thlabel{countably-categorical-characterisation}
Let $T$ be a countable theory with JCP. Then the following are equivalent:
\begin{enumerate}[label=(\roman*)]
\item $T$ is $\omega$-categorical,
\item every type in finitely many variables is supported,
\item all p.c.\ models are atomic,
\item all countable p.c.\ models are atomic,
\item every p.c.\ model is positively $\omega$-saturated,
\item there is a positively $\omega$-saturated prime model,
\item there is a positively $\omega$-saturated atomic model.
\end{enumerate}
\end{theorem}
\begin{proof}
We prove the following implications.
\[
\begin{tikzcd}
                                                               & \text{(v)} \arrow[ld, Rightarrow, bend right] \arrow[rd, dashed]          &                                      &                                    \\
\text{(i)} \arrow[r, Rightarrow] \arrow[rr, dashed, bend left] & \text{(ii)} \arrow[d, Rightarrow] \arrow[u, Rightarrow] \arrow[r, dashed] & \text{(i)-(v)} \arrow[r, Rightarrow] & \text{(iv)} \arrow[ld, Rightarrow] \\
\text{(vi)} \arrow[u, Rightarrow] \arrow[rru, dashed]          & \text{(iii)} \arrow[l, Rightarrow] \arrow[ru, dashed]                     & \text{(vii)} \arrow[lu, Rightarrow]  &                                   
\end{tikzcd}
\]
\underline{(i) $\Rightarrow$ (ii)} Suppose for a contradiction that there is a type $p(x)$ in finitely many variables that is not supported. Let $M$ be a p.c.\ model in which $p(x)$ is realised, say by $a \in M$. By \thref{lowenheim-skolem} there is a countable p.c.\ model $M' \subseteq M$ with $a \in M'$. In particular, $p(x)$ is realised in $M'$ by $a$. By \thref{omitting-types} there is also a p.c.\ model $N$ that omits $p(x)$. Again, using \thref{lowenheim-skolem}, there is a countable p.c.\ model $N' \subseteq N$. In particular, $N'$ omits $p(x)$. However, by assumption $M'$ and $N'$ must be isomorphic, which is our desired contradiction.

\underline{(ii) $\Rightarrow$ (iii)} By definition.

\underline{(iii) $\Rightarrow$ (iv)} Trivial.

\underline{(iv) $\Rightarrow$ (i)} By \thref{countable-atomic-isomorphic}.

\underline{(ii) $\Rightarrow$ (v)} By \thref{supported-n-types-implies-omega-saturated}.

\underline{(v) $\Rightarrow$ (i)} By \thref{saturated-models-isomorphic}.

Now that we have established the equivalence of (i)--(v), we prove the equivalence with the final two properties.

\underline{(i)--(v) $\Rightarrow$ (vi)} Let $M$ be the unique countable p.c.\ model. By (iii) $M$ is atomic and, so $M$ is prime by \thref{prime-iff-countable-and-atomic}. Finally, $M$ is positively $\omega$-saturated by (v). 

\underline{(vi) $\Rightarrow$ (vii)} Let $M$ be prime and positively $\omega$-saturated. By \thref{prime-model-is-small} we have that $M$ is countable. Therefore, $M$ is atomic by \thref{prime-iff-countable-and-atomic}.

\underline{(vii) $\Rightarrow$ (ii)} Let $M$ be a positively $\omega$-saturated atomic model. Let $p(x)$ be a type in finitely many variables and let $M'$ be a p.c.\ model in which $p(x)$ is realised. Using JCP of $T$ we find immersions $M \to N \leftarrow M'$ into some model $N$ of $T$. Then $N$ realises $p(x)$. So for any $\phi(x) \in p(x)$ we have $N \models \exists x \phi(x)$, and hence $M \models \exists x \phi(x)$. We thus see that $p(x)$ is finitely satisfiable in $M$. By positive $\omega$-saturation we thus have that $p(x)$ is realised in $M$, and since $M$ is atomic this means that $p(x)$ is supported.
\end{proof}
\begin{remark}
\thlabel{finite-type-spaces}
\thref{countably-categorical-characterisation} provides several equivalent characterisations of being $\omega$-categorical for positive theories. However, compared to the analogous theorem for full first-order logic one important characterisation is missing: namely that every space of types in finitely many variables is finite. This is no longer an equivalent condition in positive logic. In fact, one easily sees that having finite type spaces is equivalent to being $\omega$-categorical and Boolean, where being Boolean follows because the complement of any positively definable set is positively definable using a finite disjunction.

As a counterexample we consider the theory $T$ from \thref{bounded-set-compactness} with constant symbols $\{c_i\}_{i < \omega}$, and which asserts that $c_i \neq c_j$ for all $i \neq j$. Then $T$ has a unique p.c.\ model, namely the model that consists only of interpretations for the constant symbols. We thus see that $T$ is $\omega$-categorical, but each constant yields a different type, so we have infinitely many $1$-types.
\end{remark}
\section{Bibliographic remarks}
\label{sec:bibliographic-remarks-omega-categorical}
In \cite{haykazyan_spaces_2019} a stronger omitting types theorem is proved. Namely that any meagre set of types (in a different type space than we defined in Section \ref{sec:types-and-type-spaces}) can be omitted. The proof is based on the Baire Category Theorem. Our proof is more elementary and closer to the usual proof of the omitting types theorem for full first-order logic (e.g., \cite[Theorem 4.1.2]{tent_course_2012}). Once the omitting types theorem is proved, the proof of \thref{countably-categorical-characterisation} (characterising countable categoricity) goes as usual, using atomic models and prime models, which are also present in \cite{haykazyan_spaces_2019}. However, \cite{haykazyan_spaces_2019} does not treat positive saturation, so items (v)--(vii) in \thref{categoricity} are new compared to \cite{haykazyan_spaces_2019} and first appeared in \cite[Theorem 5.8]{kamsma_bilinear_2023}.

As explained before \thref{haykazyans-test}, the name ``Haykazyan's test'' is because that proposition was first proved by Haykazyan. The exact reference for that is \cite[Proposition 5.1]{haykazyan_spaces_2019}.

The notion of a supported type appears already in \cite[page 844]{haykazyan_spaces_2019}. That is, \cite[page 844]{haykazyan_spaces_2019} gives a topological definition of what it means for a set of formulas to be supported. We use the translation from \cite[Definition 5.2]{kamsma_bilinear_2023}.

The fact that countable categoricity in positive logic does not necessarily correspond to finite type spaces (\thref{finite-type-spaces}) can also be found in \cite[Example 6.6]{haykazyan_spaces_2019}.

The notion of positive saturation is taken from \cite[Section 2.4]{poizat_positive_2018}.

\chapter{Saturated, homogeneous and monster models}
\chaptermark{Saturation, homogeneity and the monster} 
\label{ch:saturated-homogeneous}

Just like in full first-order logic, it will be convenient to work in a monster model. In this chapter we make precise what that means and how such a model can be constructed. Compared to the full first-order setting there is really nothing new going on, as the usual constructions go through in positive logic. So if the reader wishes they can just skip to the description of the monster model and the accompanying conventions in Section \ref{sec:monster-models}. The only other thing that will be used in other places is the existence of positively saturated p.c.\ models, which is \thref{building-saturated-model}.

After having set up the conventions and notation for the monster model, we discuss some standard model-theoretic tools in Section \ref{sec:the-toolbox}. This mainly concerns indiscernible sequences. We then continue in Section \ref{sec:lascar-strong-types} with the basics concerning Lascar strong types and some subtleties that arise in positive logic.

\section{Building positively saturated models}
\label{sec:building-positively-saturated-models}
Earlier we defined what a positively saturated p.c.\ model is (\thref{saturated-model}). Now we show that they can be constructed as usual.
\begin{proposition}
\thlabel{building-saturated-model}
Let $M$ be a p.c.\ model of a theory $T$. Then for all $\kappa \geq |M| + |T|$ there is a positively $\kappa^+$-saturated p.c.\ model $N$ of $T$ with $|N| \leq 2^\kappa$, which is a continuation of $M$.
\end{proposition}
\begin{proof}
We inductively construct a continuous chain $(M_i)_{i < \kappa^+}$ of p.c.\ models of cardinality at most $2^\kappa$ with $M_0 = M$ such that: for every $i < \kappa^+$ we have that, for any $A \subseteq M_i$ with $|A| \leq \kappa$, any set of formulas $\Sigma(x)$ over $A$ in a single variable, that is finitely satisfiable in $M_i$, has a realisation in $M_{i+1}$. This can be done because for each such an $M_i$ there are at most $2^\kappa$ many sets of formulas over subsets of cardinality at most $\kappa$. The union $N = \bigcup_{i < \kappa^+} M_i$ will then be the required p.c.\ model (it is p.c.\ by \thref{pc-models-closed-under-directed-unions}). We spell this out in detail below.

Having constructed $M_i$, we note that $M_i$ has $ 2^\kappa$ subsets of cardinality at most $\kappa$. As $\kappa \geq |T|$, there are $\kappa$ many formulas over a set of cardinality at most $\kappa$ and therefore there are $2^\kappa$ many sets of formulas in a single variable over such a set. There are thus $2^\kappa \times 2^\kappa = 2^\kappa$ many sets of formulas in a single variable with parameters in some $A \subseteq M_i$ with $|A| \leq \kappa$. Enumerate all such sets of formulas that are finitely satisfiable in $M_i$ as $\{ \Sigma_i(x_i) \}_{i \leq 2^\kappa}$. Set $\Sigma((x_i)_{i \leq 2^\kappa}) = \bigcup_{i \leq 2^\kappa} \Sigma_i(x_i)$, which is then finitely satisfiable in $M_i$. By \thref{continue-to-realise-type} we find a p.c.\ model $M'$ of $T$ that is a continuation of $M_i$ such that there are $(a_i)_{i \leq 2^\kappa} \in M'$ with $M \models \Sigma((a_i)_{i \leq 2^\kappa})$. As $M'$ is a continuation of $M_i$ and $M_i$ is p.c., we may view $M_i$ as a subset of $M'$. By downward L\"owenheim-Skolem (\thref{lowenheim-skolem}) we then find a p.c.\ model $M_{i+1} \subseteq M'$ such that $M_i (a_i)_{i \leq 2^\kappa} \subseteq M_{i+1}$ and $|M_{i+1}| \leq |M_i (a_i)_{i \leq 2^\kappa}| + |T| \leq 2^\kappa$, as required.

We are left to verify that $N = \bigcup_{i < \kappa^+} M_i$ is positively $\kappa^+$-saturated. So let $A \subseteq N$ with $A < \kappa^+$ and let $\Sigma(x)$ be a set of formulas in a single variable with parameters in $A$, that is finitely satisfiable in $N$. Then there must be some $i < \kappa^+$ such that $A \subseteq M_i$. As $M_i$ is p.c., we have that $\Sigma(x)$ is finitely satisfiable in $M_i$. By construction then, there is a realisation of $\Sigma(x)$ in $M_{i+1}$, and hence in $N$.
\end{proof}
\begin{proposition}
\thlabel{saturation-implies-universal}
Let $M$ be a positively $\kappa$-saturated p.c.\ model of a theory $T$ with JCP. Then for any model $N$ of $T$ with $|N| \leq \kappa$ there is a homomorphism $f: N \to M$.
\end{proposition}
\begin{proof}
Let $a$ be a tuple of length $|N|$ that enumerates $N$, and write $\Sigma(x) = \tp(a; N)$. So $\Sigma(x)$ is essentially the positive diagram of $N$, but with variables enumerating $N$ instead of constant symbols. In particular, realisations of $\Sigma(x)$ in $M$ correspond to homomorphisms $N \to M$ (see also \thref{model-of-diagram-is-homomorphism}). As $M$ is positively $\kappa$-saturated, it is enough to show that $\Sigma(x)$ is finitely satisfiable in $M$ by \thref{saturation-extend-variables}.

As $T$ has JCP, there is a model $M'$ of $T$ with homomorphisms $N \to M' \leftarrow M$. Since $M$ is a p.c.\ model, the homomorphism $M \to M'$ is an immersion. Thus, for any $\phi(x) \in \Sigma(x)$ we have
\[
N \models \exists x \phi(x)
\implies
M' \models \exists x \phi(x)
\implies
M \models \exists x \phi(x),
\]
which establishes that $\Sigma(x)$ is finitely satisfiable in $M$, as required.
\end{proof}
\section{Homogeneity}
\label{sec:homogeneity}
\begin{definition}
\thlabel{positively-homogeneous}
Let $\kappa$ be an infinite cardinal. A structure $M$ is called \emph{strongly positively $\kappa$-homogeneous}\index{strongly positively homogeneous structure} if the following equivalent conditions hold:
\begin{enumerate}[label=(\roman*)]
\item every partial immersion $f: M \to M$ with a domain of cardinality less than $\kappa$ can be extended to an automorphism,
\item for every two tuples $a, b \in M$ of length less than $\kappa$ such that $\tp(a; M) = \tp(b; M)$ there is an automorphism $h: M \to M$ that sends $a$ to $b$, so $h(a) = b$.
\end{enumerate}
\end{definition}
We can build strongly positively homogeneous p.c.\ models as usual. The main takeaway is the following, and the rest of the section is devoted to proving it.
\begin{theorem}
\thlabel{building-saturated-homogeneous-model}
Let $M$ be a p.c.\ model of a theory $T$. Then for all $\kappa$ there is a positively $\kappa$-saturated and strongly positively $\kappa$-homogeneous p.c.\ model $N$ of $T$, which is a continuation of $M$.
\end{theorem}
\begin{definition}
\thlabel{positively-special-model}
We call a structure $M$ of cardinality $\kappa = \aleph_\alpha$ \emph{positively special}\index{positively special structure} if it is the union of a chain $(M_i)_{i < \alpha}$ of immersions such that $M_i$ is $\aleph_{i+1}$-saturated. We call such a chain $(M_i)_{i < \alpha}$ a \term{positively specialising chain}.
\end{definition}
Another formulation of \thref{positively-special-model} is saying that $M$ is a union of a chain $(M_\lambda)$, where $\lambda$ ranges over the cardinals $< \kappa$, such that each $M_\lambda$ is $\lambda^+$-saturated.
\begin{proposition}
\thlabel{building-positively-special-models}
Let $M$ be a p.c.\ model of a theory $T$, and let $\kappa$ be an uncountable cardinal such that $\kappa > |M| + |T|$ and $\lambda < \kappa$ implies $2^\lambda < \kappa$. Then there is a continuation $N$ of $M$ which is a p.c.\ model of $T$ and which is a positively special structure of cardinality $\kappa$.
\end{proposition}
\begin{proof}
Let $\alpha$ be such that $\kappa = \aleph_\alpha$. By induction we build a chain $(M_i)_{i < \alpha}$ of p.c.\ models of $T$ such that for each $i < \alpha$ we have that $M_i$ is $\aleph_{i+1}$-saturated and $|M_i| < \kappa$. Furthermore, we make it so that $M_0$ is a continuation of $M$.

Every stage $i$---whether it is the base case, successor step or limit stage---is essentially done in the same way: starting with some p.c.\ model $M'$ of $T$, with $|M'| < \kappa$, we apply \thref{building-saturated-model} to find a positively $\aleph_{i+1}$-saturated p.c.\ model $M_i$ of $T$ that is a continuation of $M'$ with $|M_i| \leq 2^{|M'| + |T| + \aleph_i} < \kappa$. Based on which case we are in, we make different choices for $M'$:
\begin{itemize}
\item in the base case we take $M' = M$,
\item for successor steps we take $M' = M_{i-1}$,
\item for limit stages we take $M' = \bigcup_{j < i} M_j$.
\end{itemize}
As the chain $(M_i)_{i < \alpha}$ consists of p.c.\ models, all maps between the structures are immersions and so it is a positively specialising chain. We can thus take $N = \bigcup_{i < \alpha} M_i$.
\end{proof}
\begin{corollary}
\thlabel{building-arbitrarily-large-positively-special-models}
For any structure $M$ and any infinite cardinal $\mu$ there exists a continuation $N$ of $M$ which is a positively special structure with $\cf(|N|) \geq \mu$.
\end{corollary}
\begin{proof}
Take $\kappa = \beth_\mu(|M| + |T|)$. Then $\kappa$ satisfies the assumptions of \thref{building-positively-special-models}, while we also have $\cf(\kappa) \geq \mu$.
\end{proof}
\begin{proposition}
\thlabel{positively-special-structures-isomorphic}
Let $T$ be a theory with JCP, and suppose $M$ and $N$ are p.c.\ models of $T$ of the same cardinality that are also positively special structures. Then $M$ and $N$ are isomorphic.
\end{proposition}
\begin{proof}
Write $\kappa = |M| = |N|$ and let $\alpha$ be such that $\kappa = \aleph_\alpha$. Let $(M_i)_{i < \alpha}$ and $(N_i)_{i < \alpha}$ be positively specialising chains with unions $M$ and $N$ respectively.
\begin{claim}
\thlabel{positively-special-structure-isomorphic:enumerate-structures}
There are enumerations $(a_j)_{j < \kappa}$ and $(b_j)_{j < \kappa}$ of $M$ and $N$ respectively, such that for each $j < \kappa$ we have that $a_j \in M_i$ and $b_j \in N_i$, where $i$ is such that $|j| = \aleph_i$.
\end{claim}
The enumerations $(a_j)_{j < \kappa}$ and $(b_j)_{j < \kappa}$ will possibly allow repetitions.
\begin{proof}[Proof of claim]
We prove the claim for $M$, and the enumeration of $N$ is then completely analogous. Let $(c_j)_{j < \kappa}$ be an enumeration of $M$ without repetitions. We define $(a_j)_{j < \kappa}$ by induction on $j$. Having defined $(a_k)_{k < j}$ we let $i$ be such that $|j| = \aleph_i$. Then let $j'$ be the least such that $c_{j'} \in M_i \setminus \{a_k : k < j\}$, or else take any $j'$ such that $c_{j'} \in M_0$. We set $a_j = c_{j'}$.

The constructed sequence $(a_j)_{j < \kappa}$ clearly satisfies the desired property, so we are left to verify that it actually enumerates all elements of $M$. Suppose for a contradiction that it does not. Let $j' < \kappa$ be minimal such that $c_{j'} \not \in \{ a_j : j < \kappa \}$ and let $i$ be minimal such that $c_{j'} \in M_i$. We consider the sequence $(a_j)_{\aleph_i \leq j < \kappa}$. This sequence can only contain elements from $\{c_k : k < j'\}$. It can also never repeat any elements, as that only happens in the second case of the definition of the $a_j$'s. This implies that $\kappa = |\{ a_j : \aleph_i \leq j < \kappa \}| \leq |\{c_k : k < j'\}| = |j'| < \kappa$, a contradiction.
\end{proof}

We now finish the proof by a back-and-forth argument. We cannot just apply \thref{back-and-forth}, because we need a special induction hypothesis (see (iii) below). We inductively construct an increasing chain of partial bijections $(f_j: M \to N)_{i < \kappa}$ such that for each $j < \kappa$:
\begin{enumerate}[label=(\roman*)]
\item $a_j \in \dom(f_j)$;
\item $b_j \in \cod(f_j)$;
\item $\dom(f_j) \subseteq M_i$ and $\cod(f_j) \subseteq N_i$, where $i$ is such that $|j| = \aleph_i$;
\item $f_j$ is a partial immersion.
\end{enumerate}
We first construct $f_0$. Write $p(x) = \tp(a_0; M)$. We claim that $p(x)$ is finitely satisfiable in $N_0$. Let $\phi(x) \in p(x)$, then $M \models \exists x \phi(x)$. By \thref{jcp}(vi), $T$ having JCP means that $M$ and $N$ satisfy the same positive sentences, so we also have $N \models \exists x \phi(x)$. Since $N_0 \subseteq N$ is an immersion we have that $N_0 \models \exists x \phi(x)$, as required. As $N_0$ is $\aleph_1$-saturated, there is a realisation $b$ of $p(x)$ in $N_0$. Maximality of types in p.c.\ models (\thref{pc-model-iff-types-are-maximal}) guarantees that $\tp(b; N) = \tp(b; N_0) = p$. Set $f_0(a_0) = b$. To make sure that $b_0 \in \cod(f_0)$ we proceed as in the inductive step below.

For the inductive step we assume that $(f_k)_{k < j}$ is constructed. Write $A = \bigcup_{k < j} \dom(f_k)$ and $B = \bigcup_{k < j} \cod(f_k)$. We define $f_j$ on $A$ by extending the $f_k$'s constructed so far, so for now $f_j$ is a bijective partial immersion (by (iv)) $M \to N$ with domain $A$ and codomain $B$. Let $p(x) = \tp(a_j/A; M) = \{\phi(x, a) : a \in A \text{ and } M \models \phi(a_j, a)\}$ and set $p'(x) = f_j(p(x)) = \{ \phi(x, f_j(a)) : \phi(x, a) \in p(x) \}$. Let $i$ be such that $|j| = \aleph_i$, and note that by (iii) we have that $A \subseteq M_i$ and $B \subseteq N_i$. We claim that $p'(x)$ is finitely realisable in $N_i$. Let $\phi(x, a) \in p(x)$, then $M \models \exists x \phi(x, a)$ and so $N \models \exists x \phi(x, f_j(a))$. Since $N_i \subseteq N$ is an immersion we have that $N_i \models \exists x \phi(x, f_j(a))$, as required. As $|B| \leq |j| = \aleph_i$ and $N_i$ is positively $\aleph_{i+1}$-saturated, we have that $p'(x)$ has a realisation $b$ in $N_i$. By maximality of types in p.c.\ models (\thref{pc-model-iff-types-are-maximal}) we have $\tp(b/B; N) = \tp(b/B; N_i) = p'(x)$. Set $f_j(a_j) = b$. Now let $q(x) = \tp(b_j/B b; N)$. Then by a similar argument we find a realisation $a$ of ${f_j}^{-1}(q(x))$ in $M_i$, and we set $f_j(a) = b_j$. This completes the inductive construction.

Set $f = \bigcup_{j < \kappa} f_j$. Then $\dom(f) = M$ by (i), $\cod(f) = N$ by (ii) and it is an isomorphism by (iv). We also note that the possible repetitions in the enumerations of $M$ and $N$ are no problem, as the equality symbol is always part of the language.
\end{proof}
\begin{proof}[Proof of \thref{building-saturated-homogeneous-model}]
By \thref{building-arbitrarily-large-positively-special-models} there is a continuation $N$ of $M$ which is a p.c.\ model of $T$ and which is a positively special structure with $\cf(|N|) \geq \kappa$. We claim that $N$ is the required p.c.\ model. For this we fix a positively specialising chain $(N_i)_{i < \alpha}$ for $N$, where $\alpha$ is such that $|N| = \aleph_\alpha$.

We first prove that $N$ is positively $\kappa$-saturated. Let $\Sigma(x)$ be a set of formulas over $A \subseteq N$, with $|A| < \kappa$, that is finitely satisfiable in $N$. As $\kappa \leq \cf(|N|) \leq \alpha$, there is $i < \alpha$ such that $A \subseteq N_i$. As $|A| < \kappa \leq \aleph_\alpha$, we may assume that $i$ is such that $\aleph_{i+1} \geq |A|$. Then $\Sigma(x)$ is finitely satisfiable in $N_i$ because $N_i \subseteq N$ is an immersion and as $N_i$ is positively $\aleph_{i+1}$-saturated it contains a realisation $a$ of $\Sigma(x)$. Then $a$ is also a realisation of $\Sigma(x)$ in $N$, which proves that $N$ is positively $\kappa$-saturated.

Now we prove that $N$ is strongly $\kappa$-homogeneous. Let $f: N \to N$ be a partial  immersion with domain $A \subseteq N$, where $|A| < \kappa$. As before, there is $i < \alpha$ such that $A \subseteq N_i$ and such that $\aleph_{i+1} \geq |A|$. Let $\L$ be the signature of $T$, and extend it to $\L_A$ by adding a constant symbol for each element of $A$. For each $j \geq i$ we let $(N_j, A)$ be the $\L_A$-structure where each $a \in A$ is interpreted as itself and we let $(N_j, f(A))$ be the $\L_A$-structure where $a \in A$ is interpreted as $f(a)$. By our assumption on $i$, each of $(N_j, A)$ and $(N_j, f(A))$ is positively $\aleph_{j+1}$-saturated. So we can form two positively specialising chains $(N'_j)_{j < \alpha}$ and $(N''_j)_{j < \alpha}$ by setting
\[
N'_j = \begin{cases}
(N_j, A) & \text{if } i \leq j \\
(N_i, A) & \text{else}
\end{cases}
\quad \text{and} \quad
N''_j = \begin{cases}
(N_j, f(A)) & \text{if } i \leq j \\
(N_i, f(A)) & \text{else}
\end{cases}
\]
This makes $(N, A)$ and $(N, f(A))$ into positively special structures. They are also p.c.\ models of the same theory with JCP, namely the set of h-inductive sentences in $\L_A$ that are true in $(N, A)$, which is the same as when taking those true in $(N, f(A))$. By \thref{positively-special-structures-isomorphic} we then have that $(N, A)$ and $(N, f(A))$ are isomorphic. That is, there is an automorphism of $N$ that extends $f$.
\end{proof}
\section{Monster models}
\label{sec:monster-models}
It is common in model theory to work in a so-called monster model, which is a very saturated and very homogeneous model. This is mainly a notational convenience, as it allows us to view types as automorphism orbits and to find realisations of sets of formulas we no longer need to move to a bigger model.

To give a precise definition of a monster model, we first need to fix some notion of smallness. That is, we declare when a cardinal is considered ``small'', and then we want our monster model to be saturated and homogeneous with respect to all ``small'' sets. There are various ways to make this precise, of which we name a few.
\begin{enumerate}[label=(\arabic*)]
\item Assume inaccessible cardinals exist. Fix some inaccessible cardinal $\kappa$ and let ``small'' mean $< \kappa$.
\item Work in a set theory that allows for classes and class-sized models (e.g., von Neumann-Bernays-G\"odel set theory), and let ``small'' mean ``not a proper class''.
\item At the start of every proof we fix a cardinal $\kappa$ so that everything that we need in that proof is of cardinality $< \kappa$, and we let ``small'' mean $< \kappa$.
\end{enumerate}
Each of these approaches has their own advantages and disadvantages.
\begin{enumerate}[label=(\arabic*)]
\item This approach requires us to assume the existence of large cardinals, and so we are no longer within ZFC set theory.
\item Bernays-G\"odel set theory is conservative over ZFC, meaning that everything that is proved in Bernays-G\"odel set theory is provable in ZFC. So compared to approach (1) we do not need to assume extra strength of our set theory. However, sometimes we will want to move to a bigger monster model, for which our original monster model is considered ``small'' (e.g., when considering global types, see \thref{global-type}). This would be problematic, as there is nothing `bigger' than classes in Bernays-G\"odel set theory.
\item This approach stays within ZFC and we can clearly move to bigger and bigger monster models. It only requires us to trust that at the start of every proof we could indeed guarantee that there is a big enough cardinal $\kappa$ that is bigger than anything we wish to consider.
\end{enumerate}
In practice it turns out that approach (3) is indeed viable, so that is the author's preferred approach. However, we stress once more that the monster model is purely a notational convenience, and so it does not matter which notion of smallness one prefers. We thus invite the reader to pick their favourite notion of smallness in the following definition.
\begin{definition}
\thlabel{monster-model}
Let $T$ be a theory with JCP and fix a notion of ``small'' as explained above. A \term{monster model} of $T$ is a model $\MM$\nomenclature[MM]{$\MM$}{Monster model} of $T$ that is:
\begin{itemize}
\item \underline{Positively closed}: $\MM$ is a p.c.\ model of $T$.
\item \underline{Very homogeneous}: any partial immersion $f: \MM \to \MM$ with small domain extends to an automorphism on all of $\MM$. Equivalently for any two small tuples $a$ and $b$ in $\MM$ we have $\tp(a; \MM) = \tp(b; \MM)$ if and only if there is an automorphism $f$ of $\MM$ such that $f(a) = b$.
\item \underline{Very saturated}: any small set of formulas with parameters in $\MM$ that is finitely satisfiable in $\MM$ is satisfiable in $\MM$.
\end{itemize}
\end{definition}
Whatever notion of smallness we take, \thref{building-saturated-homogeneous-model} shows that monster models exist for every theory with JCP. The point of assuming JCP is that then by \thref{saturation-implies-universal} every small model admits a homomorphism into the monster model. In particular, every small p.c.\ model admits an immersion into the monster model and may thus be viewed as a submodel.
\begin{remark}
\thlabel{maximal-pc-model-is-monster}
If $T$ has a maximal p.c.\ model $M$ then this is the monster model $\MM = M$. The preferred formalism for monster models then does not matter. In this case, the monster model is sometimes said to be bounded.

To see this, we first claim that any homomorphism $f: M \to M$ is an automorphism. As $M$ is a p.c.\ model it is an immersion, so we only need to show that $f$ is surjective. Suppose for a contradiction that it is not. We inductively build a chain $(M_i)_{i < |M|^+}$ such that $M_i = M$ for all $i < |M|^+$. For successors we let the link $M_i \to M_{i+1}$ be given by $f$. At limit stages $\ell < |M|^+$ we let $M'$ be the union of $(M_i)_{i < \ell}$, which is a p.c.\ model by \thref{pc-models-closed-under-directed-unions}, so there is an immersion $M' \to M$. For each $i < \ell$ the link $M_i \to M_\ell$ is then the composition $M_i \to M' \to M$. Let $N$ be the union of $(M_i)_{i < |M|^+}$. Since $f$ is not surjective we can pick some $a_i \in M_{i+1} \setminus f(M_i)$ for every $i < |M|^+$. We thus obtain a set $\{a_i : i < |M|^+\}$ of cardinality $|M|^+$ in $N$, but at the same time $N$ is a p.c.\ model (again, by \thref{pc-models-closed-under-directed-unions}), so it must admit an immersion $N \to M$ and hence $|N| < |M|$. We arrive at a contradiction and conclude that $f$ must be surjective.

We now show that any positively $(\aleph_0 + |M|)$-saturated p.c.\ model $N$ is (isomorphic to) $M$. Indeed, there is an immersion $g: N \to M$. By positive saturation (or more precisely, \thref{saturation-extend-variables}), $N$ also realises $\Diag(M)$, and so there is an immersion $f: M \to N$. By the above claim $gf$ is an automorphism, so $g$ is surjective and hence an isomorphism.

The claim that $M$ is the monster model then follows from the fact that positively $(\aleph_0 + |M|)$-saturated and strongly positively $(\aleph_0 + |M|)$-homogeneous p.c.\ models always exist (\thref{building-saturated-homogeneous-model}).
\end{remark}
\begin{mdframed}
\Large{\begin{convention}
\thlabel{monster-convention}
From now on we work in a monster model $\MM$, so all p.c.\ models, tuples and sets are assumed to be small and to live in $\MM$.
\end{convention}}
\end{mdframed}
We finish this section by establishing some (standard) notation for working in the monster model.
\begin{convention}
\thlabel{monster-notation}
We generally omit the monster model $\MM$ from the notation. So for example, we would write $\tp(a)$ and $\models \phi(a)$ instead of $\tp(a; \MM)$ and $\MM \models \phi(a)$. We also fix the following notation. Everything is small unless explicitly mentioned otherwise.
\begin{itemize}
\item We use lowercase Latin letters $a, b, c, \ldots$ for (possibly infinite) tuples of elements in $\MM$.
\item We use uppercase Latin letters $A, B, C, \ldots$ for arbitrary subsets of the monster. We use the letters $M$ and $N$ when these subsets are p.c.\ models.
\item For a tuple $a$ and a set $B$ we write\nomenclature[tpaB]{$\tp(a/B)$}{Type of $a$ over $B$}
\[
\tp(a/B) = \{ \phi(x, b) : b \in B \text{ and } \models \phi(a, b)  \}
\]
for the set of formulas over $B$ that are satisfied by $a$, and we call this the \emph{type of $a$ over $B$}\index{type over a parameter set}.
\item We write $a \equiv_B a'$ to mean $\tp(a/B) = \tp(a'/B)$\nomenclature[Bequiv]{$\equiv_B$}{Same type over $B$}.
\item We write $\Aut(\MM/B)$\nomenclature[Aut]{$\Aut(\MM/A)$}{Group of automorphisms over $A$} for the set of autmorphisms of $\MM$ that fix $B$ pointwise. So by homogeneity we have $a \equiv_B a'$ if and only if there is $f \in \Aut(\MM/B)$ with $f(a) = a'$.
\item For a set of formulas $\Sigma(x)$ and a tuple $a$ we write $a \models \Sigma$ to mean that $a$ satisfies $\Sigma(x)$ in $\MM$, that is $\models \Sigma(a)$.
\item For a set of formulas $\Sigma$ with parameters in $C$ and some $B \subseteq C$ we write $\Sigma|_{B}$ for the subset of $\Sigma$ consisting of those formulas with parameters from $B$.
\end{itemize}
\end{convention}
Similar to \thref{type-space} we can define type spaces, but now over fixed parameter sets.
\begin{definition}
\thlabel{type-space-with-parameters}
Let $B$ be a set of parameters and let $I$ be an index set. Then the \emph{type space of $I$-types over $B$}\index{type space over $B$}, written as $\S_I(B)$,\nomenclature[SBI]{$\S_I(B)$}{Type space of $I$-types over $B$} is defined as follows:
\[
\S_I(B) = \{ \tp(a/B) : a \text{ is indexed by } I \}.
\]
\end{definition}
Note that $\S_I(B)$ could be topologised similarly to how we topologised $\S_I(T)$, by having closed sets correspond to sets of formulas (with parameters in $B$). Though we will have no use for this.
\section{The toolbox}
\label{sec:the-toolbox}
In this section we collect some tools that are essential for advanced model theory. These are tools that we know and love from the full first-order setting, and we see that in positive logic we do not have to give up any strength.

\begin{proposition}
\thlabel{combine-partial-types}
Type-definable sets are closed under finite disjunction, infinite conjunction and existential quantification over any string of variables. More precisely, we can perform the following constructions on sets of formulas, where the tuples of variables involved can be infinite.
\begin{enumerate}[label=(\roman*)]
\item Given sets of formulas $\Sigma_1(x)$ and $\Sigma_2(x)$ with parameters, we define
\[
\Sigma(x) = \{ \phi_1(x) \vee \phi_2(x) : \phi_1(x) \in \Sigma_1(x) \text{ and } \phi_2(x) \in \Sigma_2(x) \}.
\]
Then for all $a$ we have:
\[
\models \Sigma_1(a) \text{ or } \models \Sigma_2(a) \quad \Longleftrightarrow \quad \models \Sigma(a).
\]
\item Given a (potentially infinite) family of sets of formulas $\{ \Sigma_i(x)\}_{i \in I}$ with parameters, we define
\[
\Sigma(x) = \bigcup_{i \in I} \Sigma_i(x).
\]
Then for all $a$ we have:
\[
\models \Sigma_i(a) \text{ for all } i \in I \quad \Longleftrightarrow \quad \models \Sigma(a).
\]
\item Given a set of formulas $\Sigma_0(x, y)$ with parameters, we define
\[
\Sigma(x) = \{ \exists y \phi(x, y) : \phi(x, y) \in \Sigma_0(x, y) \},
\]
where the existential quantification is each time really only over the variables that are mentioned in $\phi(x, y)$. Then for all $a$ we have:
\[
\text{there is } b \text{ with } \models \Sigma_0(a, b) \quad \Longleftrightarrow \quad \models \Sigma(a).
\]
\end{enumerate}
\end{proposition}
In light of the above proposition it makes sense to treat sets of formulas as infinitary formulas and apply positive connectives to them.
\begin{convention}
\thlabel{apply-connectives-to-sets-of-formulas}
We will often apply disjunctions, conjunctions and existential quantification to sets of formulas to form a new set of formulas. The conjunctions and existential quantification are allowed to be infinite, but the disjunctions are not. The new set of formulas is then given by the corresponding item from \thref{combine-partial-types}.

Singleton sets will be abbreviated by the formula that they contain. So we write $\phi(x) \vee \Sigma(x)$ instead of $\{\phi(x)\} \vee \Sigma(x)$.
\end{convention}
\begin{proof}
Item (ii) is immediate from the definitions. We prove the other two.

\underline{(i)} For the left to right direction we may assume, without loss of generality, that $\models \Sigma_1(a)$. So for any $\phi_1(x) \in \Sigma_1(x)$ and $\phi_2(x) \in \Sigma_2(x)$ we have $\models \phi_1(a) \vee \phi_2(a)$, and so $\models \Sigma(a)$. For the converse we prove the contrapositive. So assume that $\not \models \Sigma_1(a) \vee \Sigma_2(a)$. Then there are $\phi_1(x) \in \Sigma_1(x)$ and $\phi_2(x) \in \Sigma_2(x)$ such that $\not \models \phi_1(a)$ and $\not \models \phi_2(a)$, hence $\not \models \phi_1(a) \vee \phi_2(a)$. We conclude that $\not \models \Sigma(a)$, as required.

\underline{(iii)} The left to right direction is immediate from the definitions. For the converse, we let $a$ be such that $\models \Sigma(a)$. Let $\phi(x, y) \in \Sigma_0(x, y)$ then by assumption $\models \exists y \phi(a, y)$, and so there is $b$ such that $\models \phi(a, b)$. By compactness there must thus be $b$ such that $\models \Sigma_0(a, b)$, that is $\models \exists y \Sigma_0(a, y)$.
\end{proof}

We already gave a definition of an indiscernible sequence in \thref{indiscernible-sequence}. Now that we are working in a monster model, we can define what it means to be indiscernible over some parameter set.
\begin{definition}
\thlabel{indiscernible-sequence-over-parameters}
Let $B$ be a set of parameters. An \emph{indiscernible sequence over $B$}\index{indiscernible sequence!over parameters} is an infinite sequence $(a_i)_{i \in I}$ such that for any $i_1 < \ldots < i_n$ and $j_1 < \ldots < j_n$ in $I$ we have
\[
a_{i_1} \ldots a_{i_n} \equiv_B a_{j_1} \ldots a_{j_n}.
\]
We will also abbreviate this as a \emph{$B$-indiscernible sequence}.
\end{definition}
Indiscernible sequences are often constructed by first constructing a very long sequence and then using the lemma below to find some indiscernible sequence that is based on the very long sequence in the following sense.
\begin{definition}
\thlabel{based-on}
Let $(a_i)_{i \in I}$ and $(b_j)_{j \in J}$ be two infinite sequences and let $C$ be some parameter set. We say that $(b_j)_{j \in J}$ is \term{based on}\emph{ $(a_i)_{i \in I}$ over $C$} if for any $j_1 < \ldots < j_n$ in $J$ there are $i_1 < \ldots < i_n$ in $I$ such that $b_{j_1} \ldots b_{j_n} \equiv_C a_{i_1} \ldots a_{i_n}$.
\end{definition}
\begin{definition}
\thlabel{lambda-t}
Write $\lambda_\kappa = \beth_{(2^\kappa)^+}$ for any cardinal $\kappa$ and $\lambda_T = \lambda_{|T|}$.\nomenclature[lambdakappa]{$\lambda_\kappa$}{Abbreviation for $\beth_{(2^\kappa)^+}$}\nomenclature[lambdaT]{$\lambda_T$}{Abbreviation for $\lambda_{|T|} = \beth_{(2^{|T|})^+}$}
\end{definition}
\begin{lemma}
\thlabel{base-indiscernible-sequence-on-long-sequence}
Let $B$ be any parameter set and let $\kappa$ be any cardinal. Then for any sequence $(a_i)_{i \in I}$ of $\kappa$-tuples with $|I| \geq \lambda_{|T| + |B| + \kappa}$ there is a $B$-indiscernible sequence $(a_i')_{i < \omega}$ that is based on $(a_i)_{i \in I}$ over $B$.
\end{lemma}
The proof of \thref{base-indiscernible-sequence-on-long-sequence} uses the Erd\H{o}s-Rado theorem, which we will state here after first recalling the necessary notation. For cardinals $\kappa, \lambda, \mu$ and $n < \omega$ we write $\kappa \to (\lambda)^n_\mu$ if for every function $f: [\kappa]^n \to \mu$ we can find a subset $X \subseteq \kappa$ with $|X| = \lambda$ such that $f$ is constant on $[X]^n$. Here $[\kappa]^n$ and $[X]^n$ are the sets of subsets of size $n$ of $\kappa$ and $X$ respectively.
\begin{theorem}[{Erd\H{o}s-Rado}]
\thlabel{erdos-rado}
For all infinite cardinals $\mu$ we have
\[
\beth_n^+(\mu) \to (\mu^+)^{n+1}_\mu.
\]
\end{theorem}
\begin{proof}[Proof of \thref{base-indiscernible-sequence-on-long-sequence}]
For convenience, write $\lambda = \lambda_{|T| + |B| + \kappa} = \beth_{(2^{|T| + |B| + \kappa})^+}$ and $\tau = |\S_\kappa(B)|$. Then $\lambda$ has the following properties, where the second one follows from the Erd\H{o}s-Rado theorem (\thref{erdos-rado}):
\begin{enumerate}[label=(\roman*)]
\item the cofinality of $\lambda$ is strictly greater than $\tau$ and $\lambda$ is a limit cardinal;
\item for all $\mu < \lambda$ and $n < \omega$, there is some $\mu' < \lambda$ such that $\mu' \to (\mu)^n_\tau$.
\end{enumerate}
We will inductively construct a sequence of types $p_0, p_1, \ldots$ over $B$, such that for all $n < \omega$:
\begin{enumerate}[label=(\arabic*)]
\item $p_n$ has free variables $x_0, \ldots, x_{n-1}$, each of which has length $\kappa$;
\item for any $m < n$ and any $i_1 < \ldots < i_m < n$ we have that $p_n(x_0, \ldots, x_{n-1}) \models p_m(x_{i_1}, \ldots, x_{i_m})$;
\item for every $\mu < \lambda$ there is $I' \subseteq I$ with $|I'| = \mu$ such that for any $i_1 < \ldots < i_n$ in $I'$ we have $\models p_n(a_{i_1}, \ldots, a_{i_n})$.
\end{enumerate}
Property (2) can also be phrased semantically as follows: for any $m < n$ and any $a_0, \ldots, a_{n-1}$ realising $p_n$ we have that any subsequence of $a_0, \ldots, a_{n-1}$ of length $m$ realises $p_m$.

For $n = 0$ there is nothing to do. So we assume that $p_n$ has been chosen and we will find $p_{n+1}$. Let $\mu < \lambda$ be arbitrary. Then by (ii) there is $\mu' < \lambda$ such that $\mu' \to (\mu)^{n+1}_\tau$. By (3) we then find $I' \subseteq I$ with $|I'| = \mu'$ such that for any $i_1 < \ldots < i_n$ in $I'$ we have $\models p_n(a_{i_1}, \ldots, a_{i_n})$. We define $f: [I']^{n+1} \to \S_{(n+1) \times \kappa}(B)$ by
\[
f(\{ i_1 < \ldots < i_{n+1} \}) = \tp(a_{i_1} \ldots a_{i_{n+1}} / B).
\]
We thus find a subset $I_\mu \subseteq I'$ with $|I_\mu| = \mu$ such that for any $i_1 < \ldots < i_{n+1}$ and $j_1 < \ldots < j_{n+1}$ in $I_\mu$ we have $\tp(a_{i_1} \ldots a_{i_{n+1}} / B) = \tp(a_{j_1} \ldots a_{j_{n+1}} / B)$. Set $q_\mu = \tp(a_{i_1} \ldots a_{i_{n+1}} / B)$, where  $i_1 < \ldots < i_{n+1}$ is some subsequence in $I_\mu$ of length $n+1$ (by the above $q_\mu$ does not depend on the choice of the subsequence). As $I_\mu \subseteq I'$ we have that for any $i_1 < \ldots < i_n < n+1$:
\begin{equation}
\label{eq:q_mu-implies-p_n}
q_\mu(x_0, \ldots, x_n) \models p_n(x_{i_1}, \ldots, x_{i_n}).\tag{$*$}
\end{equation}
Since $\mu < \lambda$ was arbitrary, we have such an $I_\mu$ and associated $q_\mu$ for every cardinal $\mu < \lambda$. By (i) there must be a cofinal subset $J \subseteq \lambda$ of cardinals such that for any $\mu, \mu' \in J$ we have $q_\mu = q_{\mu'}$. Set $p_{n+1} = q_\mu$. The induction hypothesis is quickly verified: (1) holds by construction, (2) follows from ($*$) together with the induction hypothesis for $p_n$ and for (3) we note that for any $\mu < \lambda$ there is $\mu' \in J$ with $\mu < \mu'$ so a ($\mu$-sized subset of) $I_{\mu'}$ will then be the required $I'$.

This finishes the construction of the sequence $(p_n)_{n < \omega}$. In particular, from (2) it follows that this is an increasing sequence. So $\bigcup_{n < \omega} p_n$ has a realisation $(a'_i)_{i < \omega}$. Then $B$-indiscernibility follows from (2), while being based on $(a_i)_{i \in I}$ over $B$ follows from (3).
\end{proof}
\begin{lemma}
\thlabel{extend-base-set-of-indiscernible-sequence}
Let $(a_i)_{i \in I}$ be a $B$-indiscernible sequence and let $C$ be any parameter set. Then there is $C'$ with $C' \equiv_B C$ such that $(a_i)_{i \in I}$ is $BC'$-indiscernible.
\end{lemma}
\begin{proof}
By compactness we may assume $|I|$ is large enough to apply \thref{base-indiscernible-sequence-on-long-sequence}. We then base a $BC$-indiscernible sequence $(a'_i)_{i \in I}$ on $(a_i)_{i \in I}$, where we applied compactness again to assume that $(a'_i)_{i \in I}$ is indexed by $I$. As $(a_i)_{i \in I}$ was already $B$-indiscernible we have $(a_i)_{i \in I} \equiv_B (a'_i)_{i \in I}$. So we find our required $C'$ by taking one such that $C'(a_i)_{i \in I} \equiv_B C(a'_i)_{i \in I}$.
\end{proof}

\section{Lascar strong types}
\label{sec:lascar-strong-types}
We discuss Lascar strong types, and the subtleties involving them compared to full first-order logic. These types will be relevant for the independence theorem in simple theories (\thref{independence-theorem}) and stationarity in stable theories (\thref{stable-theory-stationary-iff-lstp}).
\begin{definition}
\thlabel{invariant-bounded-equivalence-relation}
Let $E(x, y)$ be an equivalence relation (on the monster model), and let $B$ be a parameter set. We call it:
\begin{itemize}
\item \emph{bounded}\index{equivalence relation!bounded} if there is a bounded number of equivalence classes (i.e., small with respect to the monster);
\item \emph{$B$-invariant}\index{equivalence relation!invariant} if for every $a_1, a_2, a'_1, a'_2$ with $a_1 a_2 \equiv_B a'_1 a'_2$ we have $E(a_1, a_2)$ if and only if $E(a'_1, a'_2)$.
\end{itemize}
\end{definition}
We briefly note that the definition of bounded equivalence relation above is not very precise, as we never made precise what small means in Section \ref{sec:monster-models}. So this should really be read as follows. There is a cardinal $\kappa$ such that in every p.c.\ model (containing the parameter set $B$) there are at most $\kappa$ many equivalence classes of $E$.
\begin{definition}
\thlabel{lascar-distance-and-strong-type}
Let $a$ and $a'$ be two tuples of the same length and let $B$ be any parameter set. We say that $a$ and $a'$ have \term{Lascar distance} \emph{at most $n$ (over $B$)}, and write $\d_B(a, a') \leq n$\nomenclature[dBaan]{$\d_B(a, a') \leq n$}{Lascar distance at most $n$}, if there are $a = a_0, a_1, \ldots, a_n = a'$ such that $a_i$ and $a_{i+1}$ are on a $B$-indiscernible sequence for all $0 \leq i < n$.

We say that $a$ and $a'$ have the same \term{Lascar strong type} \emph{(over $B$)}, and write $a \equivls_B a'$\nomenclature[BequivLS]{$\equivls_B$}{Same Lascar strong type over $B$}, if the following equivalent conditions hold:
\begin{enumerate}[label=(\roman*)]
\item $\d_B(a, a') \leq n$ for some $n < \omega$;
\item for each bounded $B$-invariant equivalence relation $E(x, y)$ we have $E(a, b)$.
\end{enumerate}
If $B = \emptyset$ we omit it as a subscript from the notation.
\end{definition}
\begin{remark}
\thlabel{on-indiscernible-sequeunce-is-same-as-starting-one}
We note that the condition that $a$ and $a'$ are on some $B$-indiscernible sequence is equivalent to $a$ and $a'$ starting a $B$-indiscernible sequence. This follows from compactness. We spell the argument out below.

Let $(a_i)_{i \in I}$ be a $B$-indiscernible sequence, such that there are $j,j' \in I$ with $a = a_j$ and $a' = a_{j'}$. For each $n < \omega$ we let $\Sigma_n(x_1, \ldots, x_n)$ be the type $\tp(a_{i_1} \ldots a_{i_n}/B)$, where $i_1 < \ldots < i_n \in I$. Note that the choice of the $i_1, \ldots, i_n$ does not matter, due to $B$-indiscernibility. Define the following type:
\[
\Sigma((x_k)_{k \in \Z}) = \bigcup \{ \Sigma_n(x_{k_1}, \ldots, x_{k_n}) : k_1 < \ldots < k_n \in \Z \}.
\]
By construction and $B$-indiscernibility this type is finitely satisfiable, namely by finite subsequences of $(a_i)_{i \in I}$. Let $(a'_k)_{k \in \Z}$ be a realisation of $\Sigma((x_k)_{k \in \Z})$, then this is a $B$-indiscernible sequence. To conclude, we distinguish two cases.
\begin{itemize}
\item If $j < j'$ then $a a' = a_j a_{j'} \equiv_B a'_0 a'_1$ and so we can let $(a''_k)_{k < \omega}$ be such that $a a' (a''_i)_{i < \omega} \equiv_B a'_0 a'_1 (a'_i)_{i < \omega}$. Then $a''_0 = a$ and $a''_1 = a'$ and so $(a''_i)_{i < \omega}$ is a $B$-indiscernible sequence starting with $a, a'$.
\item If $j' < j$ then  $a a' = a_j a_{j'} \equiv_B a'_0 a'_{-1}$ and so we can let $(a''_k)_{k < \omega}$ be such that $a a' (a''_i)_{i < \omega} \equiv_B a'_0 a'_{-1} (a'_{-i})_{i < \omega}$. Then $a''_0 = a$ and $a''_1 = a'$ and so $(a''_i)_{i < \omega}$ is a $B$-indiscernible sequence starting with $a, a'$ (because a sequence remains indiscernible after inverting its order).
\end{itemize}
\end{remark}
\begin{lemma}
\thlabel{equivalence-lascar-strong-type-conditions}
The conditions in \thref{lascar-distance-and-strong-type} are indeed equivalent.
\end{lemma}
\begin{proof}
\underline{(i) $\Rightarrow$ (ii)} It suffices to prove that for any $B$-indiscernible sequence $(a_i)_{i < \omega}$ we have that $E(a_0, a_1)$ for any $B$-invariant bounded equivalence relation. Let $\kappa$ be the number of equivalence classes of $E$. Using compactness we elongate the sequence to $(a_i)_{i < \kappa^+}$. Then there must be $i < j < \kappa^+$ such that $E(a_i, a_j)$. Hence, by $B$-indiscernibility and $B$-invariance we get $E(a_0, a_1)$.

\underline{(ii) $\Rightarrow$ (i)} Clearly, the relation described in (i) is $B$-invariant, so we need to show that it is bounded. Suppose not, then for $\lambda = \lambda_{|T| + |Ba|}$ there are $(a_i)_{i < \lambda}$ such that $a_i$ and $a_j$ cannot be connected by $B$-indiscernible sequences as in (i), for all $i < j < \lambda$. Base a $B$-indiscernible sequence $(a'_i)_{i < \omega}$ on $(a_i)_{i < \lambda}$. Let $i < j < \lambda$ be such that $a'_0 a'_1 \equiv_B a_i a_j$. After applying an automorphism over $B$ we find a $B$-indiscernible sequence with $a_i$ and $a_j$ on it, a contradiction.
\end{proof}
\begin{proposition}
\thlabel{same-type-over-saturated-model-lascar-distance-2}
Assume thickness. Let $M$ be a positively $\lambda_T$-saturated p.c.\ model. Then $a' \equiv_M a$ implies $\d_M(a, a') \leq 2$. If we assume semi-Hausdorffness, we can drop the assumption of positive $\lambda_T$-saturatedness.
\end{proposition}
\begin{proof}
Let $a \equiv_M a'$ with $M$ a $\lambda_T$-saturated p.c.\ model. Using thickness, we let $\Sigma(x_0, x_1)$ be the partial type expressing that there are $(x_i)_{2 \leq i < \omega}$ such that $(x_i)_{i < \omega}$ is $M$-indiscernible. We show that $\Sigma(x_0, a) \cup \Sigma(x_0, a')$ is finitely satisfiable, which is enough. Let $\phi(x_0, x_1) \in \Sigma(x_0, x_1)$ and let $m$ denote the finite part of $M$ that appears in $\phi(x_0, x_1)$. As $M$ is positively $\lambda_T$-saturated, we can inductively find $(a_i)_{i < \lambda_T}$ in $M$ such that $a_i (a_j)_{j < i} \equiv_m a (a_j)_{j < i}$. Base an $m$-indiscernible sequence $(a'_i)_{i < \omega}$ on $(a_i)_{i < \lambda_T}$. Then $\models \phi(a'_0, a'_1)$. There are $i < j < \lambda_T$ with $a_i a_j \equiv_m a'_0 a'_1$, hence $\models \phi(a_i, a_j)$. By construction then $\models \phi(a_i, a)$. As $a' \equiv_M a$ and $a_i \in M$ we also have $\models \phi(a_i, a')$. So we have a realisation of $\phi(x_0, a) \wedge \phi(x_0, a')$, and we conclude that $\Sigma(x_0, a) \cup \Sigma(x_0, a')$ is finitely realisable, as required.

For the claim about semi-Hausdorff theories we refer to \thref{semi-hausdorff-same-type-over-pc-model-lascar-distance-2}.
\end{proof}
\begin{corollary}
\thlabel{same-lstp-iff-same-types-over-sequence-of-models}
Assume thickness. Then we have that $a \equivls_B a'$ if and only if there are positively $\lambda_T$-saturated p.c.\ models $M_1, \ldots, M_n$ and $a = a_0, a_1, \ldots, a_n = a'$ such that $a_i \equiv_{M_{i+1}} a_{i+1}$ for all $0 \leq i < n$.

If we assume semi-Hausdorffness, we can drop the requirement that the p.c.\ models are positively $\lambda_T$-saturated.
\end{corollary}
\begin{proof}
The right to left direction follows immediately from \thref{same-type-over-saturated-model-lascar-distance-2}. For the other direction, we let $a = a_0, a_1, \ldots, a_n = a'$ be such that $a_i$ and $a_{i+1}$ are on a $B$-indiscernible sequence for all $0 \leq i < n$. Let $M$ be a positively $\lambda_T$-saturated p.c.\ model containing $B$ (\thref{building-saturated-model}). Then by \thref{extend-base-set-of-indiscernible-sequence} there is $M_{i+1}$ for each $0 \leq i < n$ with $M_{i+1} \equiv_B M$ such that the $B$-indiscernible sequence connecting $a_i$ and $a_{i+1}$ is $M_{i+1}$-indiscernible, hence $a_i \equiv_{M_{i+1}} a_{i+1}$.
\end{proof}
\begin{definition}
\thlabel{lascar-strong-automorphisms}
Assume thickness. Let $A$ be a parameter set. We define $\Aut_f(\MM/A)$\nomenclature[Autf]{$\Aut_f(\MM/A)$}{Group of Lascar strong automorphisms over $A$}, the \emph{group of Lascar strong automorphisms over $A$}\index{Lascar strong automorphism} as the subgroup of $\Aut(\MM/A)$ generated by
\[
\bigcup \{ \Aut(\MM / M) : M \text{ is a positively $\lambda_T$-saturated p.c.\ model containing } A \}.
\]
\end{definition}
\begin{corollary}
\thlabel{same-lstp-iff-lascar-strong-automorphism}
Assume thickness. Then we have that $a \equivls_B a'$ if and only if $f(a) = b$ for some $f \in \Aut_f(\MM/B)$.
\end{corollary}
\begin{proposition}
\thlabel{lascar-distance-characterisation-of-thickness}
The following are equivalent for a theory $T$:
\begin{enumerate}[label=(\roman*)]
\item $T$ is thick,
\item the property $\d(x, y) \leq n$ is type-definable for all $n \geq 1$,
\item the property $\d(x, y) \leq n$ is type-definable for some $n \geq 1$,
\item the property $\d_B(x, y) \leq n$ is type-definable (over $B$) for all sets of parameters $B$ and all $n \geq 1$,
\item the property $\d_B(x, y) \leq n$ is type-definable (over $B$) for all sets of parameters $B$ some $n \geq 1$.
\end{enumerate}
\end{proposition}
\begin{proof}
The equivalences (ii) $\Leftrightarrow$ (iv) and (iii) $\Leftrightarrow$ (v) are immediate, because $\d_B(x, y) \leq n$ is the same as $\d(xb, yb) \leq n$, where $b$ is a tuple that enumerates $B$. The implication (ii) $\Rightarrow$ (iii) is trivial. We prove (i) $\Rightarrow$ (ii) and (iii) $\Rightarrow$ (i).

\underline{(i) $\Rightarrow$ (ii)} Let $\Theta((x_i)_{i < \omega})$ be the partial type that expresses that $(x_i)_{i < \omega}$ is an indiscernible sequence. Then $\d(x_0, x_1) \leq 1$ is expressed by
\[
\exists (x_i)_{2 \leq i < \omega} \Theta((x_i)_{i < \omega}),
\]
which can be expressed by a partial type. Then $\d(x, y) \leq n$ is expressed by
\[
\exists z_0 \ldots z_n \left( x = z_0 \wedge y = z_n \wedge \bigwedge_{i < n} \d(z_i, z_{i+1}) \leq 1 \right),
\]
which can again be expressed by a partial type.

\underline{(iii) $\Rightarrow$ (i)} By assumption we can define a partial type $\Theta((x_i)_{i < \omega})$ as follows:
\[
\bigcup \{ \d(x_{i_1} \ldots x_{i_k}, x_{j_1} \ldots x_{j_k}) \leq n : i_1 < \ldots < i_k < j_1 < \ldots < j_k < \omega \}.
\]
We claim that $\Theta((x_i)_{i < \omega})$ expresses that $(x_i)_{i < \omega}$ is an indiscernible sequence.

Let $(a_i)_{i < \omega}$ be such that $\models \Theta((a_i)_{i < \omega})$, and let $i_1 < \ldots < i_k < \omega$ and $j_1 < \ldots < j_k < \omega$. Define $h_1 = \max(i_k, j_k) + 1$ and $h_i = h_1 + i$ for $1 < i \leq k$. The point is that then $i_1 < \ldots < i_k < h_1 < \ldots < h_k$ and $j_1 < \ldots < j_k < h_1 < \ldots < h_k$. So by definition of $\Theta((x_i)_{i < \omega})$ we have
\[
\d(a_{i_1} \ldots a_{i_k}, a_{h_1} \ldots a_{h_k}) \leq n,
\]
which implies $a_{i_1} \ldots a_{i_k} \equiv a_{h_1} \ldots a_{h_k}$. Similarly, we find $a_{j_1} \ldots a_{j_k} \equiv a_{h_1} \ldots a_{h_k}$. We thus have $a_{i_1} \ldots a_{i_k} \equiv a_{j_1} \ldots a_{j_k}$, and we conclude that $(a_i)_{i < \omega}$ is indiscernible.

Conversely, suppose that $(a_i)_{i < \omega}$ is indiscernible. Let $i_1 < \ldots < i_k < j_1 < \ldots < j_k < \omega$. Define a sequence $(b_m)_{m < \omega}$ by $b_0 = (a_{i_1}, \ldots, a_{i_k})$, $b_1 = (a_{j_1}, \ldots, a_{j_k})$ and for $m \geq 2$ we set $b_m = (a_{j_k + mk}, \ldots, a_{j_k + mk + k-1})$. Then $(b_m)_{m < \omega}$ is sequence of $k$-tuples in $(a_i)_{i < \omega}$ that respect the original order. In particular, $(b_m)_{m < \omega}$ is indiscernible. So we have $\d(b_0, b_1) \leq 1$, and in particular $\d(b_0, b_1) \leq n$, which is just saying that
\[
\d(a_{i_1} \ldots a_{i_k}, a_{j_1} \ldots a_{j_k}) \leq n.
\]
We thus conclude $\models \Theta((a_i)_{i < \omega})$, as required.
\end{proof}

\section{Bibliographic remarks}
\label{sec:bibliographic-remarks-saturated-homogeneous}
The construction of positively saturated homogeneous models in \thref{building-saturated-homogeneous-model} is standard, we follow \cite[Section 6.1]{tent_course_2012}.

For Bernays-G\"odel set theory, as mentioned at in Section \ref{sec:monster-models}, see for example \cite[page 70]{jech_set_2003}. A reference for the Erd\H{o}s-Rado theorem (\thref{erdos-rado}) can be found in the same book \cite[Theorem 9.6]{jech_set_2003}.

Basing indiscernible sequences on very long sequences, such as in \thref{base-indiscernible-sequence-on-long-sequence}, is considered standard. A proof of this lemma in positive logic appears for example in \cite[Lemma 1.2]{ben-yaacov_simplicity_2003}, but the proof is really not different from the full first-order setting (e.g., \cite[Lemma 7.2.12]{tent_course_2012}).

The definition of Lascar strong types (\thref{lascar-distance-and-strong-type}) in positive logic is taken from \cite[Definition 1.39]{ben-yaacov_simplicity_2003}. The additional equivalent characterisation for thick theories in \thref{same-lstp-iff-same-types-over-sequence-of-models} is due to \cite{dobrowolski_kim-independence_2022}.

\chapter{Simple theories}
\label{ch:simple-theories}

In this chapter we develop dividing independence for simple theories. Much of this is similar to the treatment in full first-order logic. The main extra difficulty is proving what we call ``full existence'' for dividing independence, see Section \ref{sec:thickness-implies-full-existence}. We finish this chapter with a version of the Kim-Pillay theorem for positive logic (\thref{kim-pillay}), which summarises the results concerning dividing independence in simple theories.

Those familiar with the usual treatment in full first-order logic may wonder why there is no mention of forking. This is because the definition of forking does not generalise so well to positive logic, and there is no actual need to work with forking instead of dividing in simple theories (see also \thref{forking-vs-dividing}).

\section{Dividing}
\label{sec:dividing}
\begin{definition}
\thlabel{dividing}
Let $\Sigma(x, b)$ be a set of formulas over $Cb$. We say that $\Sigma(x, b)$ \emph{divides over $C$}\index{dividing} if there is a $C$-indiscernible sequence $(b_i)_{i < \omega}$ with $b_i \equiv_C b$ for all $i < \omega$ such that $\bigcup_{i < \omega} \Sigma(x, b_i)$ is inconsistent.

If $\Sigma(x, b)$ contains just one formula $\phi(x, b)$ then we will also say that $\phi(x, b)$ divides over $C$.
\end{definition}
Note that in the above definition, the condition $b_i \equiv_C b$ for all $i < \omega$ on the sequence $(b_i)_{i < \omega}$ can be replaced by $b_0 = b$, and we obtain an equivalent definition. We will often use this implicitly.
\begin{definition}
\thlabel{formula-holding-along-sequence}
Let $\phi(x_1, \ldots, x_n)$ be a formula, where $x_1, \ldots, x_n$ are tuples of variables of the same length. Given a sequence $(a_i)_{i \in I}$, of tuples of the same length matching that of the $x_1, \ldots, x_n$, we say that \emph{$\phi$ holds along $(a_i)_{i \in I}$}\index{formula holds along a sequence} if for any $i_1 < \ldots < i_n$ in $I$ we have $\models \phi(a_{i_1}, \ldots, a_{i_n})$.
\end{definition}
\begin{definition}
\thlabel{psi-dividing}
Let $\phi(x, b)$ be a formula over $Cb$ and let $y$ be a finite tuple of variables matching those elements of $b$ that appear in $\phi(x, b)$. Let $\psi(y_1, \ldots, y_k)$ be a formula over $C$, where the length of each of $y_1, \ldots, y_k$ matches $y$. Then $\phi(x, b)$ is said to \emph{$\psi$-divide}\index{dividing!psi-dividing@$\psi$-dividing} over $C$ if:
\begin{enumerate}[label=(\roman*)]
\item $\psi(y_1, \ldots, y_k)$ is an obstruction of $\exists x(\phi(x, y_1) \wedge \ldots \wedge \phi(x, y_k))$,
\item there is a sequence $(b_i)_{i < \omega}$ with $b_i \equiv_C b$ for all $i < \omega$, such that $\psi$ holds along $(b_i)_{i < \omega}$.
\end{enumerate}
\end{definition}
Note that in (ii) in the above definition we may equivalently require there to be a sequence $(b'_i)_{i < \omega}$ of tuples with $b'_i \equiv_C b'$ for all $i < \omega$, where $b'$ is the part of $b$ that matches $y$, such that $\psi$ holds along $(b'_i)_{i < \omega}$. So $\phi(x, b)$ $\psi$-divides over $C$ if and only if $\phi(x, b')$ $\psi$-divides over $C$.
\begin{lemma}
\thlabel{psi-dividing-lemma}
A set of formulas $\Sigma(x, b)$ over $Cb$ divides over $C$ if and only if it contains a formula $\phi(x, b)$ that $\psi$-divides over $C$ for some formula $\psi$.
\end{lemma}
\begin{proof}
We first prove the left to right direction. Let $(b_i)_{i < \omega}$ be a $C$-indiscernible sequence that witnesses that $\Sigma(x, b)$ divides over $C$. By compactness there are $\phi_1(x, z), \ldots, \phi_k(x, z) \in \Sigma(x, z)$ such that $\not \models \exists x(\phi_1(x, b_1) \wedge \ldots \wedge \phi_k(x, b_k))$. We thus find some $\psi(z_1, \ldots, z_k)$ that implies $\neg \exists x(\phi_1(x, z_1) \wedge \ldots \wedge \phi_k(x, z_k))$ and $\models \psi(b_1, \ldots, b_k)$. As these formulas only mention a finite number of the variables in $z_1, \ldots, z_k$, we may omit unused variables to obtain finite subtuples of variables $y_1, \ldots, y_k$, while we keep some unused variables to guarantee that these tuples all have the same length and match the same subtuple $y \subseteq z$. Let $\phi(x, y)$ be the formula $\phi_1(x, y) \wedge \ldots \wedge \phi_k(x, y)$, so $\phi(x, y) \in \Sigma(x, z)$ and $\psi(y_1, \ldots, y_k)$ is an obstruction of $\exists x (\phi(x, y_1) \wedge \ldots \wedge \phi(x, y_k))$. At the same time we have $\models \psi(b_1, \ldots, b_k)$, which by indiscernibility implies that $\models \psi(b_{i_1}, \ldots, b_{i_k})$ for any $i_1 < \ldots < i_k < \omega$, as required.

For the other direction we let $\phi(x, b) \in \Sigma(x, b)$ be a formula that $\psi$-divides. Let $(b_i)_{i < \omega}$ be as in \thref{psi-dividing}. By compactness we may elongate the sequence $(b_i)_{i < \omega}$ to $(b_i)_{i < \lambda}$ for some big enough $\lambda$. We can then base a $C$-indiscernible sequence $(b^*_i)_{i < \omega}$ on $(b_i)_{i < \lambda}$. As $\psi$ holds along $(b_i)_{i < \lambda}$ it will hold along $(b^*_i)_{i < \omega}$, which in turn implies the inconsistency of $\{ \phi(x, b^*_i) : i < \omega \}$ and hence of $\bigcup_{i < \omega} \Sigma(x, b^*_i)$.
\end{proof}
\begin{corollary}
\thlabel{psi-dividing-indiscernible-sequence}
Let $\phi(x, y)$ and $\psi(y_1, \ldots, y_k)$ be formulas over $C$ such that $\psi(y_1, \ldots, y_k)$ is inconsistent with $\phi(x, y_1) \wedge \ldots \wedge \phi(x, y_k)$. Then $\phi(x, b)$ $\psi$-divides over $C$ if and only if there is a $C$-indiscernible sequence $(b_i)_{i < \omega}$ with $b_0 = b$ such that $\psi$ holds along it.
\end{corollary}
\begin{proof}
The right to left direction is immediate. The left to right follows the same proof as the second half in \thref{psi-dividing-lemma}, after which we apply an automorphism over $C$ to the indiscernible sequence to get $b_0 = b$.
\end{proof}
\begin{remark}
\thlabel{psi-dividing-remark}
In full first-order logic there is the notion of ``$k$-dividing'', which says that a set of formulas is inconsistent along every $k$-subsequence of some infinite sequence. So this is very similar to $\psi$-dividing. In fact, if the set of formulas in question contains only one formula $\phi(x, b)$, then it is exactly the same if we take $\psi(y_1, \ldots, y_k)$ to be $\neg \exists x(\phi(x, y_1) \wedge \ldots \wedge \phi(x, y_k))$. This is a common theme in positive logic. In full first-order logic it suffices to specify how many things are inconsistent with one another and we just use a formula that says ``there does not exist ...''. In positive logic we need a positive formula witnessing this.
\end{remark}
\begin{proposition}
\thlabel{dividing-in-terms-of-automorphic-indiscernible-sequences}
The following are equivalent:
\begin{enumerate}[label=(\roman*)]
\item $\tp(a/Cb)$ does not divide over $C$;
\item for every $C$-indiscernible sequence $(b_i)_{i < \omega}$ with $b_0 = b$ there is a $Ca$-indiscernible sequence $(b'_i)_{i < \omega}$ with $(b'_i)_{i < \omega} \equiv_{Cb} (b_i)_{i < \omega}$;
\item for every $C$-indiscernible sequence $(b_i)_{i < \omega}$ with $b_0 = b$ there is $a' \equiv_{Cb} a$ such that $(b_i)_{i < \omega}$ is $Ca'$-indiscernible.
\end{enumerate}
\end{proposition}
\begin{proof}
\underline{(i) $\Rightarrow$ (ii) and (iii)} Let $(b_i)_{i < \omega}$ be a $C$-indiscernible sequence with $b_0 = b$. Set $p(x, b) = \tp(a/Cb)$, so $\bigcup_{i < \omega} p(x, b_i)$ is consistent. Let $a^*$ be a realisation of this set of formulas. So we have $a^* b_i \equiv_C ab$ for all $i < \omega$. By compactness we can elongate this sequence to $(b_i)_{i < \lambda}$, where $\lambda = \lambda_{|T| + |Cba|}$, with that same property. Base a $Ca^*$-indiscernible sequence $(b^*_i)_{i < \omega}$ on $(b_i)_{i < \lambda}$ over $Ca^*$. Then $a^* b^*_0 \equiv_C ab$, and we obtain (ii) by letting $(b'_i)_{i < \omega}$ be such that $a^* b^*_0 (b^*_i)_{i < \omega} \equiv_C ab (b'_i)_{i < \omega}$. To obtain (iii) we note that $(b^*_i)_{i < \omega} \equiv_C (b_i)_{i < \omega}$ and we let $a'$ be such that $a^* (b^*_i)_{i < \omega} \equiv_C a' (b_i)_{i < \omega}$, so that $a' b = a' b_0 \equiv_C a^* b^*_0 \equiv_C ab$.

\underline{(ii) $\Rightarrow$ (i) and (iii) $\Rightarrow$ (i)} Write $p(x, b) = \tp(a/Cb)$ and let $(b_i)_{i < \omega}$ be a $C$-indiscernible sequence with $b_0 = b$. We show in both cases that $\bigcup_{i < \omega} p(x, b_i)$ is consistent.
\begin{itemize}
\item Assuming (ii) we find $Ca$-indiscernible $(b'_i)_{i < \omega}$ with $(b'_i)_{i < \omega} \equiv_{Cb} (b_i)_{i < \omega}$. Then $\models p(a, b'_0)$ because $b'_0 = b$. So by $Ca$-indiscernibility, $a$ is a realisation of $\bigcup_{i < \omega} p(x, b'_i)$, and the claim follows as $(b_i)_{i < \omega} \equiv_C (b'_i)_{i < \omega}$.
\item Assuming (iii), we let $a'$ be such that $a' \equiv_{Cb} a$ and $(b_i)_{i < \omega}$ is $Ca'$-indiscernible. Then $a'b_0 = a'b \equiv_C ab$, and so $\models p(a', b_0)$. So $a'$ realises $\bigcup_{i < \omega} p(x, b_i)$. \qedhere
\end{itemize}
\end{proof}

\section{Independence relations}
\label{sec:independence-relations}
\begin{definition}
\thlabel{independence-relation}
An \term{independence relation} $\ind$\nomenclature[AAAindependence]{$\ind$}{Independence relation} is a ternary relation on small subsets of the monster model. If $A$, $B$ and $C$ are in the relation we write
\[
A \ind_C B,
\]
which should be read as ``$A$ is independent from $B$ over $C$''. We also allow tuples in the relation, which are then interpreted as the set they enumerate. For example, if $a$ is a tuple enumerating $A$ then $a \ind_C B$ means the same as $A \ind_C B$.
\end{definition}
\begin{example}
\thlabel{independence-relation-examples}
As an example of a very nicely behaved independence relation we consider the theory of vector spaces over some fixed field $K$, in the signature $(\neq, 0, +, \{s \cdot (-)\}_{s \in K})$, where $s \cdot (-)$ for $s \in K$ is a unary function symbol for scalar multiplication by $s$. So this is the usual Boolean theory of $K$-vector spaces. We define an independence relation $\ind$ based on linear independence as follows
\[
A \ind_C B \quad \Longleftrightarrow \quad \linspan{(AC)} \cap \linspan{(BC)} = \linspan{(C)}.
\]
This generalises linear independence in the following sense: a set $A$ of vectors is linearly independent precisely when $a \ind_\emptyset A \setminus \{a\}$ for all $a \in A$.
\end{example}
We give a list of all the properties that an independence relation can have (and we are interested in). We will not encounter all these properties straight away, but we mention them here anyway so that this definition can also serve as a reference.
\begin{definition}
\thlabel{independence-properties}
Let $\ind$ be an independence relation. We define the following properties for $\ind$, where $a$ and $b$ are arbitrary tuples and $C$ is an arbitrary set.
\begin{description}
\item[\textsc{invariance}] For any $f \in \Aut(\MM)$ we have that $a \ind_C b$ implies $f(a) \ind_{f(C)} f(b)$.
\item[\textsc{monotonicity}] For any $a' \subseteq a$ and $b' \subseteq b$ we have that $a \ind_C b$ implies $a' \ind_C b'$.
\item[\textsc{normality}] If $a \ind_C b$ then $Ca \ind_C Cb$.
\item[\textsc{existence}] We always have $a \ind_C C$.
\item[\textsc{full existence}] There is always $b'$ with $b' \equiv_C b$ such that $a \ind_C b'$.
\item[\textsc{base monotonicity}] If $C \subseteq C' \subseteq b$ then $a \ind_C b$ implies $a \ind_{C'} b$.
\item[\textsc{extension}] If $a \ind_C b$ then for any $d$ there is $d'$ with $d' \equiv_{Cb} d$ and $a \ind_C bd'$.
\item[\textsc{symmetry}] If $a \ind_C b$ then $b \ind_C a$.
\item[\textsc{transitivity}] If $C \subseteq C'$ with $a \ind_C C'$ and $a \ind_{C'} b$ then $a \ind_C b$.
\item[\textsc{finite character}] If for all finite $a' \subseteq a$ and all finite $b' \subseteq b$ we have $a' \ind_C b'$ then $a \ind_C b$.
\item[\textsc{local character}] For every cardinal $\kappa$ there is a cardinal $\lambda$ such that for all $a$ with $|a| < \kappa$ and any $C$ there is $C' \subseteq C$ with $|C'| < \lambda$ and $a \ind_{C'} C$.
\item[\textsc{independence theorem}] If $a \ind_C b$, $a' \ind_C c$ and $b \ind_C c$ with $a \equivls_C a'$ then there is $a''$ with $a'' \equivls_{Cb} a$ and $a'' \equivls_{Cc} a'$ such that $a'' \ind_C bc$.
\item[\textsc{stationarity}] For any $C$ such that $a \equiv_C a'$ implies $a \equivls_C a'$ for all $a, a'$, we have that $a \ind_C b$, $a' \ind_C b$ and $a \equiv_C a$ implies $a \equiv_{Cb} a'$.
\end{description}
\end{definition}
We will often use the properties \textsc{invariance}, \textsc{monotonicity} and \textsc{normality} of an independence relation implicitly.

The point of dividing is that it will give us an independence relation with some nice properties. However, it is actually the negation, non-dividing, that will mean that things are independent. Compare this for example to linear independence: we first define when vectors are linearly dependent (i.e., some non-trivial linear equation between them holds) and then we say that they are linearly independent if this does not happen (i.e., no non-trivial linear equation between them holds).
\begin{definition}
\thlabel{non-dividing-independence}
Let $A, B, C$ be sets and let $a$ and $b$ enumerate $A$ and $B$ respectively. Then we write\nomenclature[AAAindependence-dividing]{$\ind^d$}{Dividing independence}
\[
A \ind^d_C B
\]
if $\tp(a/Cb)$ does not divide over $C$. We call this relation \emph{dividing independence}.
\end{definition}
It should be clear from the definition of dividing that if $a$ and $a'$ are tuples enumerating the same set in possibly different ways then we have that $\tp(a/Cb)$ divides over $C$ if and only if $\tp(a'/Cb)$ divides over $C$. Similarly, nothing changes when changing the enumeration of $b$. So $\ind^d$ is indeed a relation on subsets, and hence an independence relation.
\begin{theorem}
\thlabel{dividing-basic-properties}
Dividing independence satisfies the following properties: \textsc{invariance}, \textsc{monotonicity}, \textsc{normality}, \textsc{existence}, \textsc{base monotonicity}, \textsc{finite character} and \textsc{left transitivity}. This final property is the same as \textsc{transitivity} with the sides of the independence relation swapped: if $C \subseteq C'$ then $C' \ind^d_C b$ and $a \ind^d_{C'} b$ implies $a \ind^d_C b$.
\end{theorem}
\begin{proof}
We prove each property separately.
\begin{description}
\item[\textsc{invariance}] By contrapositive: suppose that $p(x, f(b)) = \tp(f(a)/f(C)f(b))$ divides over $f(C)$. Then there is an $f(C)$-indiscernible sequence $(b'_i)_{i < \omega}$ with $b'_i \equiv_{f(C)} f(b)$ for all $i < \omega$, such that $\bigcup_{i < \omega} p(x, b'_i)$ is inconsistent. Then $(f^{-1}(b'_i))_{i < \omega}$ witnesses dividing over $C$ for $\tp(a/Cb)$.

\item[\textsc{monotonicity}] By contrapositive: suppose that $p'(x', b') = \tp(a'/Cb')$ divides over $C$, as witnessed by a $C$-indiscernible sequence $(b'_i)_{i < \omega}$. Elongate this to a $C$-indiscernible sequence $(b'_i)_{i < \lambda}$ for $\lambda = \lambda_{|T| + |Cb|}$. For each $i < \lambda$ use $b'_i \equiv_C b'$ to find $b_i \supseteq b'_i$ with $b_i b_i' \equiv_C b b'$. Let $(b^*_i)_{i < \omega}$ be a $C$-indiscernible sequence based on $(b_i)_{i < \lambda}$ over $C$. The restriction of $(b^*_i)_{i < \omega}$ to the subtuples matching $b'$ has the same type over $C$ as $(b'_i)_{i < \omega}$ and so after applying an automorphism we may as well assume that $b'_i$ is the restriction of $b^*_i$ to the subtuple matching $b'$. Set $p(x, b) = \tp(a/Cb)$, then $\bigcup_{i < \omega} p'(x', b_i) \subseteq \bigcup_{i < \omega} p(x, b^*_i)$, and so inconsistency of the former implies inconsistency of the latter.

\item[\textsc{normality}] Let $c$ enumerate $C$ and set $p(x, b) = \tp(a/Cb)$ and $q(xz, bc) = \tp(ac/Cbc)$. Let $(b_i c_i)_{i < \omega}$ be any $C$-indiscernible sequence with $b_i c_i \equiv_C bc$ for all $i < \omega$. As $(b_i)_{i < \omega}$ is a $C$-indiscernible sequence with $b_i \equiv_C b$ for all $i < \omega$ and $p(x, b)$ does not divide over $C$, there is a realisation $a'$ of $\bigcup_{i < \omega} p(x, b_i)$. For each $i < \omega$ we have that $c_i = c$, and so $\models q(a'c, b_i c_i)$. We conclude that $a'c$ realises $\bigcup_{i < \omega} q(xz, b_i c_i)$, showing that this set is consistent and hence that $q(xz, bc)$ does not divide over $C$.

\item[\textsc{existence}] Let $c$ enumerate $C$ and set $p(x, c) = \tp(a/Cc)$. Let $(c_i)_{i < \omega}$ be a $C$-indiscernible sequence with $c_i \equiv_C c$ for all $i < \omega$. That is, $c_i = c$ for all $i < \omega$. Hence $\bigcup_{i < \omega} p(x, c_i) = p(x, c)$, which is consistent.

\item[\textsc{base monotonicity}] We will use the characterisation of dividing in \thref{dividing-in-terms-of-automorphic-indiscernible-sequences}(iii). Let $(b_i)_{i < \omega}$ be a $C'$-indiscernible sequence with $b_0 = b$. Then $(b_i)_{i < \omega}$ is in particular $C$-indiscernible, because $C \subseteq C'$. Since $\tp(a/Cb)$ does not divide over $C$ there is $a'$ with $a' \equiv_{Cb} a$ such that $(b_i)_{i < \omega}$ is $Ca'$-indiscernible. As $C' \subseteq b$ we have $a' \equiv_{C'b} a$, and we conclude that $\tp(a/C'b)$ does not divide over $C'$.

\item[\textsc{finite character}] By contrapositive: suppose that $p(x, b) = \tp(a/Cb)$ divides over $C$. Then by \thref{psi-dividing-lemma} there is $\phi(x, b) \in p(x, b)$ and some $\psi$ such that $\phi(x, b)$ $\psi$-divides over $C$. Let $x'$ and $b'$ be the finite parts of $x$ and $b$ respectively that appear in $\phi(x, b)$, then $\phi(x', b')$ $\psi$-divides over $C$. As $\phi(x', b') \in \tp(a'/Cb')$, where $a' \subseteq a$ matches $x'$, we can apply \thref{psi-dividing-lemma} again (now in the other direction) to see that $\tp(a'/Cb')$ divides over $C$.

\item[\textsc{left transitivity}] We will use the characterisation of dividing in \thref{dividing-in-terms-of-automorphic-indiscernible-sequences}(ii). Let $(b_i)_{i < \omega}$ be a $C$-indiscernible sequence with $b_0 = b$. As $C' \ind^d_C b$ there is a $C'$-indiscernible sequence $(b'_i)_{i < \omega}$ with $(b'_i)_{i < \omega} \equiv_{Cb} (b_i)_{i < \omega}$. In particular, $b'_0 = b$, and so because $a \ind^d_{C'} b$ there is a $C'a$-indiscernible sequence $(b''_i)_{i < \omega}$ with $(b''_i)_{i < \omega} \equiv_{C'b} (b'_i)_{i < \omega}$. In particular, we have $(b''_i)_{i < \omega} \equiv_{Cb} (b'_i)_{i < \omega} \equiv_{Cb} (b_i)_{i < \omega}$, and so we conclude that $\tp(a/Cb)$ does not divide over $C$. \qedhere
\end{description}
\end{proof}
\section{Definition of simplicity: local character and NTP}
\label{sec:definition-of-simplicity}
\begin{definition}
\thlabel{tree-notation}
Let $\alpha$ and $\beta$ be ordinals. Write $\alpha^{< \beta}$\nomenclature[alphabeta]{$\alpha^{<\beta}$}{Tree of functions $\gamma \to \alpha$ for $\gamma < \beta$} for the tree of functions $\gamma \to \alpha$ for $\gamma < \beta$. The tree structure is given by setting $\nu \unlhd \eta$\nomenclature[nueta]{$\nu \unlhd \eta$}{Tree ordering: $\nu$ is below $\eta$} if $\eta$ is an extension of $\nu$.

For a function $\eta: \gamma \to \alpha$ and some $i < \alpha$, we write $\eta^\frown i$ for the function $\gamma+1 \to \alpha$ that appends $i$ to $\eta$. Formally:
\[
\eta^\frown i(x) = \begin{cases}
\eta(x) & \text{if } x < \gamma, \\
i & \text{if } x = \gamma.
\end{cases}
\]
\end{definition}
The above notation is in line with viewing a function $\eta: \gamma \to \alpha$ as a sequence of length $\gamma$ of elements in $\alpha$. We will often take this view and as such we will often refer to the domain of $\eta$ as its \emph{length}\index{length!of a function of ordinals}.
\begin{definition}
\thlabel{tree-property}
Let $k \geq 2$ be a natural number. A formula $\phi(x,y)$ is said to have the \emph{$k$-tree property}\index{tree property} ($k$-\TP)\nomenclature[kTP]{$k$-\TP}{The $k$-tree property} if there are parameters $(a_\eta)_{\eta \in \omega^{< \omega}}$ and an obstruction $\psi(y_1, \ldots, y_k)$ of the formula $\exists x (\phi(x, y_1) \wedge \ldots \wedge \phi(x, y_k))$ such that:
\begin{enumerate}[label=(\roman*)]
\item for all $\sigma \in \omega^\omega$ the set $\{ \phi(x, a_{\sigma|_n}): n < \omega\}$ is consistent,
\item for all $\eta \in \omega^{< \omega}$ and $i_1 < \ldots < i_k < \omega$ we have $\models \psi(a_{\eta^\frown i_1}, \ldots, a_{\eta^\frown i_k})$.
\end{enumerate}
A formula $\phi(x,y)$ has the \emph{tree property} (\TP) if there exists a natural number $k \geq 2$ such that $\phi(x,y)$ has $k$-\TP.

A theory has the \emph{tree property} (\TP) if there is a formula that has the tree property, and otherwise it is \NTP.\nomenclature[NTP]{\NTP}{Not the tree property}
\end{definition}
\begin{theorem}
\thlabel{simplicity-equivalences}
The following are equivalent for a theory $T$.
\begin{enumerate}[label=(\roman*)]
\item Dividing independence $\ind^d$ satisfies \textsc{local character}.
\item For any finite $a$ and any $C$ there is $C' \subseteq C$ with $|C'| < |T|^+$ such that $a \ind^d_{C'} C$.
\item For every cardinal $\kappa$ there is a cardinal $\lambda$ such that for all $a$ with $|a| < \kappa$ and any sequence $(b_i)_{i < \mu}$ there is $i_0 < \lambda$ with $a \ind^d_{(b_j)_{j < i_0}} (b_j)_{i_0 \leq j < \mu}$.
\item For any finite $a$ and any sequence $(b_i)_{i < \mu}$ there is some $i_0 < |T|^+$ such that $a \ind^d_{(b_j)_{j < i_0}} (b_j)_{i_0 \leq j < \mu}$.
\item The theory $T$ is \NTP.
\end{enumerate}
\end{theorem}
\begin{definition}
\thlabel{simplicity}
We call a theory $T$ \emph{simple}\index{simple theory} if the equivalent conditions from \thref{simplicity-equivalences} hold.
\end{definition}
We prove \thref{simplicity-equivalences} in the remainder of this section. A substantial part of it works for arbitrary independence relations.
\begin{proposition}
\thlabel{local-character-in-terms-of-sequences}
Suppose that $\ind$ satisfies \textsc{base monotonicity}, \textsc{normality} and \textsc{monotonicity}. Then the following are equivalent:
\begin{enumerate}[label=(\roman*)]
\item $\ind$ satisfies \textsc{local character}, i.e.\ for every cardinal $\kappa$ there is a cardinal $\lambda$ such that for all $a$ with $|a| < \kappa$ and any $C$ there is $C' \subseteq C$ with $|C'| < \lambda$ and $a \ind_{C'} C$;
\item for every cardinal $\kappa$ there is a cardinal $\lambda$ such that for all $a$ with $|a| < \kappa$, any $C$ and any sequence $(b_i)_{i < \mu}$ there is $i_0 < \lambda$ with $a \ind_{C(b_j)_{j < i_0}} (b_j)_{i_0 \leq j < \mu}$.
\end{enumerate}
Furthermore, for every $\kappa$ we can take the $\lambda$ in (i) and (ii) to be the same.

If $\ind$ also satisfies \textsc{finite character} then the above statements are further equivalent to the case where $\kappa$ is fixed to be $\omega$. That is, if there is such a $\lambda$ for $\kappa = \omega$ then there is such a $\lambda$ for every $\kappa$.
\end{proposition}
\begin{proof}
Fix $\kappa$. We first prove (i) $\Rightarrow$ (ii). Let $a$ and $(b_i)_{i < \mu}$ be as in the statement of (ii). Set $B = \{b_i : i < \mu\}$ and apply \textsc{local character} to $BC$. Then we find $B' \subseteq B$ and $C' \subseteq C$ with $|B'C'| < \lambda$ such that $a \ind_{B'C'} BC$. As $|B'| < \lambda$ there is $i_0 < \lambda$ such that $B' \subseteq \{b_j : j < i_0\}$. The result then follows by \textsc{base-monotonicity} and \textsc{monotonicity}.

For the converse we let $a$ be such that $|a| < \kappa$ and $C$ be any set. Let $(c_i)_{i < \mu}$ enumerate $C$. By assumption there is $i_0 < \lambda$ such that $a \ind_{(c_j)_{j < i_0}} (c_j)_{i_0 \leq j < \mu}$. Setting $C' = \{c_j : j < i_0\}$ then yields the required independence, after an application of \textsc{normality}.

Finally, assuming \textsc{finite character}, we will show that if (i) holds for $\kappa = \omega$ then it holds for all $\kappa$. So let $\kappa$ be arbitrary and let $a$ be such that $|a| < \kappa$. Let $\lambda$ be as in (i) for the $\omega$ case and set $\lambda' = |\lambda \times [\kappa]^{<\omega}|^+$. We claim that for any $C$ there is $C' \subseteq C$ with $|C'| < \lambda'$ such that $a \ind_{C'} C$. For each finite $a' \subseteq a$ we let $C_{a'} \subseteq C$ be such that $|C_{a'}| < \lambda$ and $a' \ind_{C_{a'}} C$. Set $C' = \bigcup \{C_{a'} : a' \subseteq a \text{ finite}\}$, then by construction $|C'| < \lambda'$. For all finite $a' \subseteq a$ we have $a' \ind_{C'} C$ by \textsc{base monotonicity}. Hence by \textsc{finite character} we have $a \ind_{C'} C$, as required.
\end{proof}
\begin{proof}[Proof of \thref{simplicity-equivalences}]
By \thref{local-character-in-terms-of-sequences} we have (i) $\Leftrightarrow$ (iii) and (ii) $\Leftrightarrow$ (iv), in the latter it is crucial that the $\lambda$ in \thref{local-character-in-terms-of-sequences} can be taken to be the same in both cases. By that same result we have that (ii) and (iv) imply (i) and (iii). It thus suffices to prove (v) $\Rightarrow$ (ii) and (iii) $\Rightarrow$ (v), both of which we prove by contrapositive.

\underline{(v) $\Rightarrow$ (ii)} Let $a$ be finite and let $C$ be such that $\tp(a/C)$ divides over $C'$ for all $C' \subseteq C$ with $|C'| < |T|^+$. We will construct an instance of \TP. We first construct a tree $(c_\eta)_{\eta \in \omega^{< |T|^+}}$ by induction on its height. Let $\zeta_\alpha \in \omega^\alpha$ denote the constant zero function. Our induction hypothesis at $\delta < \lambda$, where we have constructed $(c_\eta)_{\eta \in \omega^{\leq \delta}}$, will be as follows:
\begin{enumerate}[label=(\arabic*)]
\item $c_{\zeta_\alpha}$ is a finite tuple of elements from $C$ for all $\alpha \leq \delta$;
\item $(c_{\eta|_\alpha})_{\alpha \leq \delta} \equiv (c_{\zeta_\alpha})_{\alpha \leq \delta}$ for all $\eta \in \omega^\delta$;
\item if $\delta = \gamma + 1$ is a successor then there are $\phi_\delta(x, y)$ and an obstruction $\psi_\delta(y_1, \ldots, y_{k_\delta})$ of $\exists x(\phi_\delta(x, y_1) \wedge \ldots \wedge \phi_\delta(x, y_{k_\delta}))$ such that $\models \phi_\delta(a, c_{\zeta_\delta})$ and for any $\eta \in \omega^\gamma$ we have that $\psi_\delta$ holds along $(c_{\eta^\frown i})_{i < \omega}$.
\end{enumerate}
For $\delta < |T|^+$ a limit or zero we let all $c_\eta$, where $\eta \in \omega^\delta$, be the empty tuple. Now suppose that $(c_\eta)_{\eta \in \omega^{\leq \delta}}$ has been constructed. By (1) we have that $|\{c_{\zeta_\alpha} : \alpha \leq \delta \}| < |T|^+$. So $\tp(a/C)$ divides over $\{c_{\zeta_\alpha} : \alpha \leq \delta \}$. Let then $\phi_{\delta+1}(x, c) \in \tp(a/C)$ and $\psi(y_1, \ldots, y_k)$ be such that $\phi_{\delta+1}(x, c)$ $\psi$-divides over $\{c_{\zeta_\alpha} : \alpha \leq \delta \}$. Then there is a sequence $(d_i)_{i < \omega}$ with $d_i \equiv_{(c_{\zeta_\alpha})_{\alpha \leq \delta}} c$ for all $i < \omega$ such that $\psi$ holds along $(d_i)_{i < \omega}$. After applying an automorphism, we may assume that $d_0 = c$. Furthermore, we may assume that the part of $\{c_{\zeta_\alpha} : \alpha \leq \delta \}$ that is mentioned in $\phi_{\delta+1}(x, c)$ and $\psi(y_1, \ldots, y_k)$ is contained in $c$ (and hence in each $d_i$), and so $\phi_{\delta+1}(x, y)$ and $\psi(y_1, \ldots, y_k)$ do not contain any parameters. We now define $c_{\zeta_\delta^\frown i} = d_i$ for all $i < \omega$. This ensures (1). For $\eta \in \omega^\delta$ we have $(c_{\eta|_\alpha})_{\alpha \leq \delta} \equiv (c_{\zeta_\alpha})_{\alpha \leq \delta}$ by (2) from the induction hypothesis. We then let $(c_{\eta^\frown i})_{i < \omega}$ be such that $(c_{\eta^\frown i})_{i < \omega}(c_{\eta|_\alpha})_{\alpha \leq \delta} \equiv (d_i)_{i < \omega} (c_{\zeta_\alpha})_{\alpha \leq \delta}$. Then (2) follows because $d_i \equiv_{(c_{\zeta_\alpha})_{\alpha \leq \delta}} c = d_0 = c_{\zeta_{\delta+1}}$ for all $i < \omega$. Finally, (3) follows because for any $\eta \in \omega^\delta$ the sequence $(c_{\eta^\frown i})_{i < \omega}$ is an automorphic copy of $(d_i)_{i < \omega}$, along which $\psi$ holds. This completes the inductive construction of the tree $(c_\eta)_{\eta \in \omega^{<|T|^+}}$.

There are $|T|$ possible pairs for formulas $\phi(x, y)$ and $\psi(y_1, \ldots, y_k)$. We have $|T|^+$ successor levels in our tree, each of which is assigned a pair $\phi_\delta(x, y)$ and $\psi_\delta(y_1, \ldots, y_{k_\delta})$. Hence, by the pigeonhole principle there is an infinite set $l_0 < l_1 < l_2 < \ldots < |T|^+$ of levels to which the same $\phi_\delta$ and $\psi_\delta$ are assigned. We write just $\phi$ and $\psi$ for these formulas. We consider the following subtree $(f_\mu)_{\mu \in \omega^{<\omega}}$ that consists of the chosen levels (with the root being the leftmost point on level $l_0$). To be precise, for $\mu \in \omega^{<\omega}$ of length $n$ we define $\eta_\mu \in \omega^{l_n}$ of length $l_n$ as
\[
\eta_\mu(l) = \begin{cases}
\mu(i) & \text{if } l = l_{i+1} - 1, \\
0 & \text{otherwise.}
\end{cases}
\]
Note that $l_{i+1} - 1$ makes sense because we have only chosen successor levels. We then set $f_\mu = c_{\eta_\mu}$.

We claim that $(f_\mu)_{\mu \in \omega^{<\omega}}$ witnesses \TP for $\phi(x, y)$. Let $\sigma \in \omega^\omega$. By construction and (2), we have $(f_{\sigma|_n})_{n < \omega} \equiv (f_{\zeta_n})_{n < \omega}$. Then by (1) and (3) we have that $a$ realises $\{\phi(x, f_{\zeta_n}) : n < \omega\}$, and so we conclude that $\{\phi(x, f_{\sigma|_n}) : n < \omega\}$ is consistent. Finally, $\psi$ holds along $(f_{\eta^\frown i})_{i < \omega}$ for every $\eta \in \omega^{<\omega}$ by (3) and because we only chose successor levels for $(l_n)_{n < \omega}$.

\underline{(iii) $\Rightarrow$ (v)} We will show that (iii) is violated for $\kappa = \omega$. Let $\phi(x, y)$ have \TP as witnessed by $\psi(y_1, \ldots, y_k)$ and $(c_\eta)_{\eta \in \omega^{<\omega}}$. Let $\lambda$ be any cardinal. Set $\mu = (2^{|T| + \lambda})^+$ and, using compactness, enlarge our tree to $(c_\eta)_{\eta \in \mu^{< \lambda}}$.

We construct some $\sigma \in \mu^{\lambda}$ by induction on its length. Let $\sigma|_\gamma$ be defined for $\gamma < \lambda$. Write $C = \{c_{\sigma|_i} : i \leq \gamma\}$ and $\eta = \sigma|_\gamma$. There are at most $2^{|T| + \lambda}$  different types over $C$. So by our choice of $\mu$, there exists an infinite $I_\gamma \subseteq \mu$ such that for any $i,j \in I_\gamma$ we have $c_{\eta^\frown i} \equiv_C c_{\eta^\frown j}$. Let $i_0$ be the least element of $I_\gamma$ and set $\sigma(\gamma) = i_0$.

Having finished the construction of $\sigma$ we define $b_i = c_{\sigma|_{i+1}}$ for all $i < \lambda$. Then there is a realisation $a$ of $\{\phi(x, b_i) : i < \lambda\}$. Let $\gamma < \lambda$ then by construction we have that $c_{{\sigma|_\gamma}^\frown i} \equiv_{(b_j)_{j < \gamma}} b_\gamma$ for all $i \in I_\gamma$, while $\psi$ holds along $(c_{{\sigma|_\gamma}^\frown i})_{i \in I_\gamma}$. In particular, this means that $\phi(x, b_\gamma)$ $\psi$-divides over $(b_j)_{j < \gamma}$ and since $\phi(x, b_\gamma) \in \tp(a/(b_i)_{i < \lambda})$ we have that $a \nind^d_{(b_j)_{j < \gamma}} (b_j)_{\gamma \leq j < \lambda}$. As $\lambda$ and $\gamma$ were arbitrary, and $a$ is finite, we conclude that (iii) fails.
\end{proof}
\section{Thickness implies full existence for dividing independence}
\label{sec:thickness-implies-full-existence}
The aim of this section is to prove that dividing independence satisfies full existence in simple theories. It turns out that we also need to assume thickness (see \thref{full-existence-without-thickness}). We start with the statement of the main result of this section, and the remainder is devoted to the rather technical proof (whose tools will have no further use to us).
\begin{theorem}
\thlabel{simple-thick-implies-full-existence}
Assume thickness. If $T$ is simple then dividing independence has \textsc{full existence}. That is, for any $a, b, C$, there is $b'$ with $b' \equiv_C b$ such that $a \ind^d_C b'$.
\end{theorem}
\begin{definition}
\thlabel{dividing-sequence}
A \term{sequence of dividing witnesses} \emph{over $C$} is a sequence $\zeta = (\phi_i(x, y^i), \psi_i(y^i_1, \ldots, y^i_{k_i}))_{i \in I}$ of pairs for formulas over $C$ such that $\psi_i(y^i_1, \ldots, y^i_{k_i})$ is an obstruction of $\exists x(\phi_i(x, y^i_1) \wedge \ldots \wedge \phi_i(x, y^i_{k_i}))$.

Given such a sequence of $\zeta$ of dividing witnesses, a \emph{$\zeta$-dividing sequence}\index{dividing!zeta-dividing sequence@$\zeta$-dividing sequence} \emph{(over $C$)} is a sequence of tuples $(b_i)_{i \in I}$ such that $\{\phi_i(x, b_i) : i \in I\}$ is consistent and for all $i \in I$ we have that $\phi_i(x, b_i)$ $\psi_i$-divides over $C(b_j)_{j < i}$.
\end{definition}
It is clear from the definition that dividing sequences are stable under taking subsequences. More precisely, suppose we are given a sequence of dividing witnesses $\zeta$ and a $\zeta$-dividing sequence $(b_i)_{i \in I}$. Let $I' \subseteq I$, and let $\zeta' \subseteq \zeta$ be the corresponding subsequence of dividing witnesses. Then $(b_i)_{i \in I'}$ is a $\zeta'$-dividing sequence.
\begin{lemma}
\thlabel{dividing-sequence-type-definable}
Assume thickness. Let $\zeta = (\phi_i(x, y^i), \psi_i(y^i_1, \ldots, y^i_{k_i}))_{i \in I}$ be a sequence of dividing witnesses over $C$. Then being a $\zeta$-dividing sequence over $C$ is type-definable over $C$. Moreover, the defining set of formulas is entirely determined by its restriction to finite subsequences.

More precisely, for any $I' \subseteq I$ there is a set of formulas $\Sigma_{I'}((y^i)_{i \in I'})$ over $C$ such that $\models \Sigma_{I'}((b_i)_{i \in I'})$ if and only if $(b_i)_{i \in I'}$ is a $(\phi_i(x, y^i), \psi_i(y^i_1, \ldots, y^i_{k_i}))_{i \in I'}$-dividing sequence over $C$. Moreover, for infinite $I' \subseteq I$ this set of formulas is given by $\Sigma_{I'} = \bigcup \{\Sigma_{I_0} : I_0 \subseteq I' \text{ is finite}\}$.
\end{lemma}
\begin{proof}
For $i \in I$ we let $\Gamma_i(y^i, z)$ be the set of formulas over $C$ expressing
\[
\exists (y_j)_{j < \omega}(
\text{``$(y_j)_{j < \omega}$ is $Cz$-indiscernible''} \wedge
\text{``$\psi_i$ holds along  $(y_j)_{j < \omega}$''} \wedge
y_0 = y^i).
\]
Note that we used the thickness assumption here to have the first conjunct in the above be type-definable.

For $I' \subseteq I$ we define $\Sigma_{I'}((y^i)_{i \in I'})$ to be the set of formulas expressing
\[
\exists x \bigwedge_{i \in I'} \phi(x, y^i) \wedge \bigwedge_{i \in I'} \Gamma_i(y^i, (y^j)_{j \in I', j < i}).
\]
Then $\Sigma_{I'}$ defines being a $(\phi_i(x, y^i), \psi_i(y^i_1, \ldots, y^i_{k_i}))_{i \in I'}$-dividing sequence over $C$, as is seen by writing out definitions, where we used the characterisation of $\psi$-dividing from \thref{psi-dividing-indiscernible-sequence}.

To see the claim about finite subsequences, we note that for any infinite $I' \subseteq I$ we have that $\Gamma_i(y^i, (y^j)_{j \in I', j < i})$ is equivalent to $\bigcup \{ \Gamma_i(y^i, (y^j)_{j \in I_0, j < i}) : I_0 \subseteq I' \text{ is finite}\}$.
\end{proof}
\begin{lemma}
\thlabel{finite-dividing-sequence-extend-base-set}
Let $\zeta$ be some finite sequence of dividing witnesses and let $(b_1, \ldots, b_n)$ be a $\zeta$-dividing sequence, both over $C$. Then for any tuple $d$ there is $d'$ with $d' \equiv_C d$ such that $(b_1, \ldots, b_n)$ is a $\zeta$-dividing sequence over $Cd'$.
\end{lemma}
\begin{proof}
Write $\zeta = (\phi_i, \psi_i)_{1 \leq i \leq n}$. We will prove by induction on $0 \leq k \leq n$ that there is $d_k$ with $d_k \equiv_C d$ such that for all $i \geq 1$ with $i \leq k$ there is a $C(b_j)_{j < i}d_k$-indiscernible sequence starting with $b_i$ such that $\psi_i$ holds along it. Then taking $d' = d_n$ yields the desired result.

For $k = 0$ we take $d$. Now assume we have constructed $d_k$. Let $s$ be a $C(b_j)_{j \leq k}$-indiscernible sequence starting with $b_{k+1}$ such that $\psi_{k+1}$ holds along it. We apply \thref{extend-base-set-of-indiscernible-sequence} to $s$ and $d_k$ to find $d_{k+1}$ such that $d_{k+1} \equiv_{C(b_j)_{j \leq k}} d_k$ and $s$ is $C(b_j)_{j \leq k}d_{k+1}$-indiscernible. Then $d_{k+1}$ satisfies the induction hypothesis: for $i = k+1$ by construction and for $i \leq k$ because that part of the induction hypothesis is invariant under automorphisms over $C(b_j)_{j \leq k}$.
\end{proof}
\begin{lemma}
\thlabel{forking-type-yields-arbitrarily-long-dividing-sequence}
Assume thickness. Suppose that there are $a, b, C$ such that any extension of $\tp(a/C)$ to a type over $Cb$ divides over $C$. Then for any ordinal $\delta$ there is a sequence of dividing witnesses $\zeta$ and a $\zeta$-dividing sequence $(b_i)_{i \in \delta^\op}$, where $\delta^\op$ carries the opposite order of $\delta$.
\end{lemma}
\begin{proof}
We construct a sequence $\zeta = (\phi_i(x, y), \psi_i(y^i_1, \ldots, y^i_{k_i}))_{i \in \delta^\op}$ and $(b_i)_{i \in \delta^\op}$ such that $a$ realises $\{\phi_i(x, b_i) : i \in \delta^\op\}$ by induction on $\delta$. The base case is trivial and the limit stage follows from type-definability of dividing sequences (\thref{dividing-sequence-type-definable}).

So let us assume that $\zeta$ and $(b_i)_{i \in \delta^\op}$ have been constructed. We will construct $b_\delta$ and a pair $(\phi(x, y), \psi(y_1, \ldots, y_k))$ such that $(b_i)_{i \in (\delta+1)^\op}$ is a $(\phi, \psi)^\frown \zeta$-dividing sequence. Let $\Sigma((y_i)_{i \in \delta^\op})$ be the set of formulas from \thref{dividing-sequence-type-definable} that expresses that $(y_i)_{i \in \delta^\op}$ is a $\zeta$-dividing sequence over $Cb$. We claim that the set of formulas
\[
\tp((b_i)_{i \in \delta^\op}/C) \cup \Sigma((y_i)_{i \in \delta^\op})
\]
over $Cb$ is consistent. Indeed, for any finite subsequence $(b_i)_{i \in I_0} \subseteq (b_i)_{i \in I}$ we can apply \thref{finite-dividing-sequence-extend-base-set} to find $b'$ with $b' \equiv_C b$ such that $(b_i)_{i \in I_0}$ is a $\zeta_0$-dividing sequence over $Cb'$ (where $\zeta_0 \subseteq \zeta$ matches $I_0 \subseteq I$). Hence there is $(b^*_i)_{i \in I_0}$ with $b (b^*_i)_{i \in I_0} \equiv_C b' (b_i)_{i \in I_0}$, which realises the corresponding finite part of the above set of formulas.

Let $(b^*_i)_{i \in \delta^\op}$ realise the above set of formulas. Let $b_\delta$ then be such that
\[
b_\delta (b_i)_{i \in \delta^\op} \equiv_C b (b^*_i)_{i \in \delta^\op},
\]
so $(b_i)_{i \in \delta^\op}$ is a $\zeta$-dividing sequence over $Cb_\delta$ and $\tp(a/Cb_\delta)$ divides over $C$, because it is the image of an extension of $\tp(a/C)$ to $Cb$ under an automorphism over $C$. Therefore, there is some $\phi(x, b_\delta) \in \tp(a/Cb_\delta)$ and a $\psi$ such that $\phi(x, b_\delta)$ $\psi$-divides over $C$. By \thref{psi-dividing-indiscernible-sequence} these two facts together say precisely that $(b_i)_{i \in (\delta+1)^\op}$ is a $(\phi, \psi)^\frown \zeta$-dividing sequence over $C$. This completes the inductive construction, and hence the proof.
\end{proof}
\begin{proof}[Proof of \thref{simple-thick-implies-full-existence}]
Suppose for a contradiction that \textsc{full existence} fails for dividing independence. Then there is $p(x) = \tp(a/C)$ and some $b$ such that any extension of $p(x)$ to a type over $Cb$ divides over $C$. Let $\kappa = (|C| + |T|)^+$ and apply \thref{forking-type-yields-arbitrarily-long-dividing-sequence} to find a sequence $\zeta = (\phi_i, \psi_i)_{i \in \kappa^\op}$ of dividing witnesses and a $\zeta$-dividing sequence $(b_i)_{i \in \kappa^\op}$. As $\phi_i$ and $\psi_i$ are formulas over $C$ for each $i < \kappa$, we can apply the pigeonhole principle to find infinite $I \subseteq \kappa$ such that for all $i,j \in I$ we have $(\phi_i, \psi_i) = (\phi_j, \psi_j)$. Let us call this tuple of formulas $(\phi, \psi)$ and let $\zeta'$ be the constant sequence of length $|T|^+$ whose entries are $(\phi, \psi)$. Let $\Sigma((y_i)_{i < |T|^+})$ be the set of formulas from \thref{dividing-sequence-type-definable} that expresses that $(y_i)_{i < |T|^+}$ is a $\zeta'$-dividing sequence over $C$. This type is finitely satisfiable by the finite subsequences of $(b_i)_{i \in I^\op}$, and so we find a $\zeta'$-dividing sequence $(b'_i)_{i < |T|^+}$ over $C$. Let $a'$ be a realisation of $\{\phi(x, b'_i) : i < |T|^+\}$. For all $i < |T|^+$ we have that $\phi(x, b'_i)$ $\psi$-divides over $C (b'_j)_{j < i}$. As $\phi(x, b'_i) \in \tp(a'/C(b'_j)_{j \leq |T|^+})$, we have that this type divides over $C (b'_j)_{j < i}$. However, this contradicts the assumption that $T$ is simple by \thref{simplicity-equivalences}.
\end{proof}
\section{Morley sequences and Kim's lemma}
\label{sec:morley-sequences-and-kims-lemma}
\begin{definition}
\thlabel{independent-sequence}
Let $\ind$ be an independence relation. A \emph{$\ind_C$-independent sequence}\index{independent sequence!with respect to $\ind$} is a sequence $(a_i)_{i \in I}$ such that $a_i \ind_C (a_j)_{j < i}$ for all $i \in I$.
\end{definition}
\begin{proposition}
\thlabel{full-existence-yields-independent-sequences}
If $\ind$ satisfies \textsc{invariance}, \textsc{monotonicity} and \textsc{full existence} then for any $a$ and $C$ and any cardinal $\kappa$ there is a $\ind_C$-independent sequence $(a_i)_{i < \kappa}$ with $a_i \equiv_C a$ for all $i < \kappa$.
\end{proposition}
\begin{proof}
Let $N' \supseteq C$ be a positively $(|Ca|^+ + \kappa)$-saturated p.c.\ model (\thref{building-saturated-model}). By \textsc{full existence} there is $N$ with $N \equiv_C N'$ such that $a \ind_C N$. Inductively and by saturation we find $(a_i)_{i < \kappa}$ in $N$ such that $a_i \equiv_{C(a_j)_{j < i}} a$ for all $i < \kappa$. This is then the required sequence, as for any $i < \kappa$ we have that $a \ind_C N$ implies $a \ind_C (a_j)_{j < i}$ by \textsc{monotonicity} and so $a_\delta \ind_C (a_i)_{i < \delta}$ by \textsc{invariance}.
\end{proof}
See \thref{dividing-basic-properties} for a precise statement of \textsc{left transitivity}.
\begin{lemma}
\thlabel{independent-sequence-cuts}
Suppose that $\ind$ satisfies \textsc{monotonicity}, \textsc{normality}, \textsc{base monotonicity}, \textsc{left transitivity} and \textsc{finite character}. If $(a_i)_{i \in I}$ is a $\ind_C$-independent sequence then for any $I_0, I_1 \subseteq I$ with $i_0 < i_1$ for all $i_0 \in I_0$ and $i_1 \in I_1$ we have $(a_i)_{i \in I_1} \ind_B (a_i)_{i \in I_0}$.
\end{lemma}
\begin{proof}
By \textsc{finite character} it is enough to prove this for finite $I_1$, and we proceed by induction on $n = |I_1|$. For $n = 1$ this follows immediately from being $\ind_C$-independent and \textsc{monotonicity}. For the induction step, assume that $I_1$ is $i_1 < \ldots < i_n < i_{n+1}$. By the induction hypothesis we have $a_{i_1} \ldots a_{i_n} \ind_C (a_i)_{i \in I_0}$ and thus by \textsc{normality}:
\[
C a_{i_1} \ldots a_{i_n} \ind_C (a_i)_{i \in I_0}.
\]
As $(a_i)_{i \in I}$ is a $\ind_C$-independent sequence we have, after an application of \textsc{monotonicity}, that $a_{i_{n+1}} \ind_C (a_i)_{i \in I_0} a_{i_1} \ldots a_{i_n}$. Then by \textsc{base monotonicity}, \textsc{monotonicity} and \textsc{normality} we have
\[
a_{i_{n+1}} a_{i_1} \ldots a_{i_n} \ind_{C a_{i_1} \ldots a_{i_n}} (a_i)_{i \in I_0}.
\]
Applying \textsc{left transitivity} to those two instances of independence then yields $a_{i_{n+1}} a_{i_1} \ldots a_{i_n} \ind_C (a_i)_{i \in I_0}$, as required.
\end{proof}
\begin{definition}
\thlabel{morley-sequence}
Let $\ind$ be an independence relation. A \emph{$\ind$-Morley sequence}\index{Morley sequence!with respect to $\ind$} \emph{(over $C$)} is an infinite $C$-indiscernible $\ind_C$-independent sequence.
\end{definition}
\begin{lemma}
\thlabel{basing-morley-sequence-on-long-independent-sequence}
Suppose that $\ind$ satisfies \textsc{invariance}, \textsc{monotonicity} and \textsc{finite character}. If $(a_i)_{i \in I}$ is a $\ind_C$-independent sequence and $(b_j)_{j \in J}$ is a sequence that is based on $(a_i)_{i \in I}$ over $C$ then $(b_j)_{j \in J}$ is $\ind_C$-independent.
\end{lemma}
\begin{proof}
Let $j \in J$ and let $j_1 < \ldots < j_n < j$. Then there are $i_1 < \ldots < i_n < i$ in $I$ such that $b_{j_1} \ldots b_{j_n} b_j \equiv_C a_{i_1} \ldots a_{i_n} a_i$. As $(a_i)_{i \in I}$ is $\ind_C$-independent we have by \textsc{monotonicity} that $a_i \ind_C a_{i_n} \ldots a_{i_1}$ and hence $b_j \ind_C b_{j_n} \ldots b_{j_1}$ by \textsc{invariance}. We conclude that $b_j \ind_C (b_k)_{k < j}$ by \textsc{finite character}, and so $(b_j)_{j \in J}$ is $\ind_C$-independent.
\end{proof}
\begin{remark}
\thlabel{reshape-morley-sequence}
An example application of \thref{basing-morley-sequence-on-long-independent-sequence} that we will often use is when we reshape a $\ind$-Morley sequence using compactness. Let us thus assume that $\ind$ satisfies \textsc{invariance}, \textsc{monotonicity} and \textsc{finite character}.

If $(a_i)_{i \in I}$ is a $\ind$-Morley sequence over $C$ and $J$ is any infinite linear order then by compactness we can find a sequence $(a'_j)_{j \in J}$ such that for any $j_1 < \ldots < j_n$ in $J$ we have $a'_{j_1} \ldots a'_{j_n} \equiv_C a_{i_1} \ldots a_{i_n}$, where the choice of $i_1 < \ldots < i_n$ in $I$ does not matter due to $C$-indiscernibility (as long as they are ordered in the same way). So $(a'_j)_{j \in J}$ is $C$-indiscernible and is based on $(a_i)_{i \in I}$. It follows by \thref{basing-morley-sequence-on-long-independent-sequence} that $(a'_j)_{j \in J}$ is a $\ind$-Morley sequence over $C$. Furthermore, if $I \subseteq J$ then $(a_i)_{i \in I} \equiv_C (a'_i)_{i \in I}$ and so by applying an automorphism over $C$ we may assume $a'_i = a_i$ for all $i \in I$.

In the other direction, if $\ind$ satisfies \textsc{monotonicity} then any infinite subsequence of a $\ind$-Morley sequence is still a $\ind$-Morley sequence.
\end{remark}
\begin{proposition}
\thlabel{building-morley-sequences}
Suppose that $\ind$ satisfies \textsc{invariance}, \textsc{monotonicity}, \textsc{finite character} and \textsc{full existence}. Then for any $a$ and $C$ there is a $\ind$-Morley sequence $(a_i)_{i < \omega}$ over $C$ with $a_0 = a$.
\end{proposition}
\begin{proof}
By \thref{full-existence-yields-independent-sequences} we find a $\ind_C$-independent sequence $(a'_i)_{i < \lambda_{|T|+|Ca|}}$ with $a'_i \equiv_C a$ for all $i < \lambda_{|T|+|Ca|}$. Let $(a''_i)_{i < \omega}$ be a $C$-indiscernible sequence based on $(a'_i)_{i < \lambda_{|T|+|Ca|}}$ over $C$. Then by \thref{basing-morley-sequence-on-long-independent-sequence} this sequence is a Morley sequence over $C$. We also have $a''_0 \equiv_C a$, so let $(a_i)_{i < \omega}$ be such that $a (a_i)_{i < \omega} \equiv_C a''_0 (a''_i)_{i < \omega}$. Then $a_0 = a$ and $(a_i)_{i < \omega}$ is the required Morley sequence.
\end{proof}
\begin{lemma}
\thlabel{general-kims-lemma}
Suppose that $\ind$ satisfies \textsc{invariance}, \textsc{monotonicity}, \textsc{normality}, \textsc{base monotonicity}, \textsc{left transitivity}, \textsc{finite character} and \textsc{local character}. If $(b_i)_{i < \omega}$ is a $\ind$-Morley sequence over $C$ and $\Sigma(x, y)$ is a set of formulas over $C$ such that $\bigcup_{i < \omega} \Sigma(x, b_i)$ is consistent then there is $a$ with $\models \Sigma(a, b_0)$ and $a \ind_C b_0$.
\end{lemma}
\begin{proof}
Let $\kappa = |x|^+$ and let $\lambda$ be the corresponding cardinal from \textsc{local character}. By compactness (see \thref{reshape-morley-sequence}) there is a $\ind$-Morley sequence $(b'_i)_{i \in \lambda^\op}$ over $C$ that is based on $(b_i)_{i < \omega}$ over $C$, where $\lambda^\op$ carries the opposite order $<^\op$ of the order $<$ on $\lambda$. As $(b'_i)_{i \in \lambda^\op}$ is based on $(b_i)_{i < \omega}$ over $C$, $\bigcup_{i \in \lambda^\op} \Sigma(x, b'_i)$ is consistent. So let $a'$ be a realisation of this set. Applying \thref{local-character-in-terms-of-sequences}, whose crucial assumption is \textsc{local character} for $\ind$, to the sequence $(b'_i)_{i < \lambda}$ we find $i_0 < \lambda$ such that (after an application of \textsc{monotonicity})
\[
a' \ind_{C(b_i)_{i < i_0}} b_{i_0}.
\]
By \thref{independent-sequence-cuts} and \textsc{normality} we also have $C(b_i)_{i >^\op i_0} \ind_C b_{i_0}$ and thus
\[
C(b_i)_{i < i_0} \ind_C b_{i_0}.
\]
So by \textsc{left transitivity} we find $a' \ind_C b_{i_0}$. Using $b_{i_0} \equiv_C b_0$ we find $a$ such that $a b_0 \equiv_C a' b_{i_0}$. Then this $a$ is as required, because $a \ind_C b_0$ follows from \textsc{invariance} and $\models \Sigma(a, b_0)$ follows from the fact that $\models \Sigma(a', b_{i_0})$.
\end{proof}
We mostly want to talk about independent sequences and Morley sequences with respect to $\ind^d$. So much so, that in those cases we drop the independence relation from the notation. Besides, this way we match the traditional use of the term Morley sequence as well as possible. The only discrepancy being that traditionally they are defined with respect to forking independence, which we avoid in positive logic (see also \thref{forking-vs-dividing}).
\begin{definition}
\thlabel{dividing-morley-sequence}
We call a $\ind^d_C$-independent sequence (respectively a $\ind^d$-Morley sequence over $C$) simply an \term{independent sequence} (respectively a \term{Morley sequence} \emph{over $C$}).
\end{definition}
\begin{corollary}
\thlabel{morley-sequences-exist}
Assume thickness. If $T$ is simple then for any $a$ and $C$ there is a Morley sequence $(a_i)_{i < \omega}$ over $C$ with $a_0 = a$.
\end{corollary}
\begin{proof}
By \thref{simple-thick-implies-full-existence} dividing independence has \textsc{full existence} (this is where we use simplicity and thickness) and thus satisfies all the assumptions of \thref{building-morley-sequences}, from which the result immediately follows.
\end{proof}
\begin{remark}
\thlabel{dividing-independent-sequence-cuts-remark}
Note that dividing independence satisfies all the properties necessary for \thref{independent-sequence-cuts}, which therefore applies to independent sequences, and in particular to Morley sequences.
\end{remark}
We finish with the main result of this section, nowadays known as \emph{Kim's lemma}. Even though it has ``lemma'' in the name, we will state it as a theorem due to its importance. Its main use can be described as follows: to show that a type $p(x, b) = \tp(a/Cb)$ does not divide over $C$ we would have to test consistency of $p(x, y)$ along \emph{every} $C$-indiscernible sequence in $\tp(b/C)$. Kim's lemma tells us that it is in fact enough to check only \emph{one} Morley sequence in $\tp(b/C)$. The main trick we actually saw before, in \thref{general-kims-lemma}.
\begin{theorem}[Kim's lemma]
\thlabel{kims-lemma}
Suppose that $T$ is simple and let $\Sigma(x, b)$ be a set of formulas over $Cb$. If $\bigcup_{i < \omega} \Sigma(x, b_i)$ is consistent for some Morley sequence $(b_i)_{i < \omega}$ over $C$ with $b_0 = b$ then $\Sigma(x, b)$ does not divide over $C$.

In particular, assuming thickness, we have that $\Sigma(x, b)$ divides over $C$ if and only if there is a Morley sequence $(b_i)_{i < \omega}$ with $b_0 = b$ such that $\bigcup_{i < \omega} \Sigma(x, b_i)$ is inconsistent.
\end{theorem}
\begin{proof}
As $T$ is simple, we can apply \thref{general-kims-lemma} to find $a$ with $\models \Sigma(a, b)$ and $a \ind^d_C b$. So $\tp(a/Cb)$ does not divide over $C$ and contains $\Sigma(x, b)$, from which we conclude that $\Sigma(x, b)$ does not divide over $C$. The ``in particular'' claim then follows from \thref{morley-sequences-exist}, which guarantees the existence of Morley sequences.
\end{proof}
\section{Extension and symmetry}
\label{sec:extension-and-symmetry}
\begin{theorem}
\thlabel{dividing-extension}
Assume thickness. If $T$ is simple then dividing independence satisfies \textsc{extension}. That is, for any $a, b, d, C$, if $a \ind^d_C b$ then there is $d'$ with $d' \equiv_{Cb} d$ such that $a \ind^d_C bd'$.
\end{theorem}
We actually prove something more general, with a more technical statement.
\begin{lemma}
\thlabel{dividing-extension-partial-types}
Assume thickness. If $T$ is simple then given a partial type $\Sigma(x, b)$ that does not divide over $C$ there is a type $p(x, b) \supseteq \Sigma(x, b)$ that does not divide over $C$.
\end{lemma}
\begin{proof}
Let $(b_i)_{i < \omega}$ be a Morley sequence over $C$ with $b_0 = b$, which exists by \thref{morley-sequences-exist}. Let $\lambda$ be the number of types over $Cb$ in variables $x$ that match $a$. By compactness we may elongate $(b_i)_{i < \omega}$ to $(b_i)_{i < \lambda^+}$. As $\Sigma(x, b)$ does not divide over $C$, there is a realisation $a$ of $\bigcup_{i < \lambda^+} \Sigma(x, b_i)$. By the pigeonhole principle we find an infinite $I \subseteq \lambda^+$ such that $ab_i \equiv_C ab_j$ for all $i,j \in I$. Pick $i_0 \in I$ and set $p(x, b_{i_0}) = \tp(a/Cb_{i_0})$. Then $a$ realises $\bigcup_{i \in I} p(x, b_i)$, and since $(b_i)_{i \in I}$ is a Morley sequence over $C$ we have by Kim's lemma (\thref{kims-lemma}) that $p(x, b_{i_0})$ does not divide over $C$. We conclude by noting that by construction $\Sigma(x, b_{i_0}) \subseteq p(x, b_{i_0})$ and by using the fact that $b_{i_0} \equiv_C b$.
\end{proof}
\begin{proof}[Proof of \thref{dividing-extension}]
Set $p(x, b) = \tp(a/Cb)$. Viewing $p(x, b)$ as a partial type over $Cbd$, that just happens to not mention the parameters in $d$ we can apply \thref{dividing-extension-partial-types} to find a type $q(x, bd) \supseteq p(x, b)$ over $Cbd$ such that $q(x, bd)$ does not divide over $C$. Let $a'$ realise $q(x, bd)$ then $a' \equiv_{Cb} a$. So we find $d'$ with $ad'  \equiv_{Cb} a'd$, which implies $a \ind^d_C bd'$ by \textsc{invariance}, as required.
\end{proof}
\begin{remark}
\thlabel{forking-vs-dividing}
In full first-order logic one often considers the notion of \term{forking}, which is defined as follows. A type $p(x, b)$ forks over $C$ if it implies a finite disjunction $\psi_1(x, d_1) \vee \ldots \vee \psi_n(x, d_n)$, where $\psi_i(x, d_i)$ divides over $C$ for each $1 \leq i \leq n$. The point of this definition is to enforce the \textsc{extension} property. That is, forking and dividing coincide exactly when $\ind^d$ satisfies \textsc{extension}.

We could consider a similar definition in positive logic. However, we would have to work with infinite disjunctions (see below), which makes the definition less practical to work with. Instead we proved directly that in simple theories $\ind^d$ satisfies \textsc{extension}. From this point on, even in full first-order logic, one uses the easier notion of dividing anyway (as forking and dividing now coincide). There is thus no need for us to even define a notion of forking.

It is instructive to see how the finite disjunction arises and why this does not work in positive logic. Fix a type $p(x, b) = \tp(a/Cb)$ and suppose that $d \supseteq b$ is such that any $q(x, d) \supseteq p(x, b)$ divides over $C$. That is, we have a failure of \textsc{extension}. Let $I$ be the set of all types in free variables $x$ over $Cd$ that contain $p(x, b)$. By assumption, for each $q \in I$, there is $\psi_q(x, d) \in q(x, d)$ such that $\psi_q(x, d)$ divides over $C$ (cf.\ \thref{psi-dividing-lemma}). By construction $p(x, d)$ implies $\bigvee_{q \in I} \psi_q(x, d)$, so $p(x, d) \cup \{\neg \psi_q(x, d) : q \in I\}$ is inconsistent. Hence there is finite $I' \subseteq I$ such that $p(x, d) \cup \{\neg \psi_q(x, d) : q \in I'\}$ is inconsistent, which means that $p(x, d)$ implies $\bigvee_{q \in I'} \psi_q(x, d)$. That is, $p(x, d)$ forks over $C$. The step using compactness to get the finite $I'$ heavily relies on being able to negate the formulas $\psi_q(x, d)$.
\end{remark}
\begin{theorem}
\thlabel{dividing-symmetry}
Assume thickness. If $T$ is simple then dividing independence satisfies \textsc{symmetry}. That is, for any $a, b, C$, if $a \ind^d_C b$ then $b \ind^d_C a$.
\end{theorem}
\begin{proof}
We start by assuming $a \ind^d_C b$. Let $N' \supseteq Cb$ be a positively $\lambda_{|T|+|Cba|}$-saturated p.c.\ model. By \textsc{extension} there is $N$ with $N \equiv_{Cb} N'$ such that $a \ind^d_C N$. Inductively and by saturation we find $(a'_i)_{i < \lambda_{|T|+|Cba|}}$ in $N$ such that $a'_i \equiv_{Cb(a'_j)_{j < i}} a$ for all $i < \lambda_{|T|+|Cba|}$. For all $i < \lambda_{|T|+|Cba|}$ we have that $a \ind^d_C N$ implies $a \ind^d_C (a'_j)_{j < i}$ by \textsc{monotonicity} and so $a'_i \ind^d_C (a'_j)_{j < i}$ by \textsc{invariance}, so $(a'_i)_{i < \lambda_{|T|+|Cba|}}$ is a $\ind^d_C$-independent sequence. Let $(a_i)_{i < \omega}$ be a $Cb$-indiscernible sequence based on $(a'_i)_{i < \lambda_{|T|+|Cba|}}$ over $Cb$. By \thref{basing-morley-sequence-on-long-independent-sequence}, $(a_i)_{i < \omega}$ is a Morley sequence over $C$. Furthermore, for every $i < \omega$ we have that $a_i \equiv_{Cb} a$. So letting $p(y, a) = \tp(b/Ca)$, we have that $\bigcup_{i < \omega} p(y, a_i)$ is consistent, as it is realised by $b$. We conclude by Kim's lemma (\thref{kims-lemma}) that $p(y, a_0)$, and hence $p(y, a)$, does not divide over $C$. So $b \ind^d_C a$, as required.
\end{proof}
\begin{corollary}
\thlabel{dividing-transitivity}
Assume thickness. If $T$ is simple then dividing independence satisfies \textsc{transitivity}. That is, for any $a, b, C, C'$ with $C \subseteq C' \subseteq b$ we have that $a \ind^d_C C'$ and $a \ind^d_{C'} b$ implies $a \ind^d_C b$.
\end{corollary}
\begin{proof}
This is just \textsc{left transitivity} with the sides of the independence relation $\ind^d$ swapped, so the result follows from \textsc{symmetry}.
\end{proof}
\section{The independence theorem}
\label{sec:independence-theorem}
\begin{lemma}
\thlabel{indiscernible-sequence-becomes-first-in-morley}
Assume thickness. If $T$ is simple then for any $B$-indiscernible sequence $(a_i)_{i < \omega}$ there is a sequence $(a'_i)_{1 \leq i < \omega}$ such that $a_i, a'_1, a'_2, \ldots$ is a Morley sequence over $C$ for all $i < \omega$.
\end{lemma}
\begin{proof}
By \thref{morley-sequences-exist} there is a Morley sequence $(a''_i)_{i < \omega}$ with $a''_0 = a_0$. By \thref{independent-sequence-cuts} (see \thref{dividing-independent-sequence-cuts-remark}) we have that $(a''_i)_{1 \leq i < \omega} \ind^d_B a_0$. So by \thref{dividing-in-terms-of-automorphic-indiscernible-sequences} there is $(a'_i)_{1 \leq i < \omega}$ with $(a'_i)_{1 \leq i < \omega} \equiv_{C a_0} (a''_i)_{1 \leq i < \omega}$ such that $(a_i)_{i < \omega}$ is $B(a'_i)_{1 \leq i < \omega}$-indiscernible. In particular, for all $i < \omega$, we have that
\[
a_i (a'_j)_{1 \leq j < \omega} \equiv_B a_0 (a'_j)_{1 \leq j < \omega} \equiv_B a_0 (a''_i)_{1 \leq i < \omega}.
\]
The result then follows because $a_0 (a''_i)_{1 \leq i < \omega}$ is just the Morley sequence $(a''_i)_{i < \omega}$ over $B$.
\end{proof}
\begin{corollary}
\thlabel{lascar-strong-type-morley-sequences}
Assume thickness. If $T$ is simple then we have that $a \equivls_B a'$ if and only if there are $a = a_0, a_1, \ldots, a_n = a'$ such that $a_i$ and $a_{i+1}$ are on a Morley sequence over $B$ for all $0 \leq i < n$.
\end{corollary}
Note also that by the same argument as in \thref{on-indiscernible-sequeunce-is-same-as-starting-one}, the condition of being on a Morley sequence is equivalent to starting one.
\begin{proof}
As Morley sequences over $B$ are in particular $B$-indiscernible sequences the right to left direction is immediate. For the other direction it is enough to note that \thref{indiscernible-sequence-becomes-first-in-morley} implies that $\d_B(a, a') \leq 1$ implies that there is $a^*$ such that $a, a^*$ and $a', a^*$ start a Morley sequence over $B$. Indeed, let $(a_i)_{i < \omega}$ be a $B$-indiscernible sequence with $a_0 = a$ and $a_1 = a'$. Then by \thref{indiscernible-sequence-becomes-first-in-morley} there is a sequence $(a'_i)_{1 \leq i < \omega}$ such that both $a, a'_1, a'_2, \ldots$ and $a', a'_1, a'_2, \ldots$ are Morley sequences over $B$. So we can take $a^* = a'_1$.
\end{proof}
\begin{lemma}
\thlabel{strong-extension}
Assume thickness. Suppose that $\ind$ is an independence relation that satisfies \textsc{monotonicity} and \textsc{extension} then it satisfies \textsc{strong extension}. That is, for any $a, b, d, C$, if $a \ind_C b$ then there is $d'$ with $d' \equivls_{Cb} d$ such that $a \ind_C bd'$. In particular, if $T$ is simple then $\ind^d$ satisfies \textsc{strong extension}.
\end{lemma}
\begin{proof}
Let $M \supseteq Cb$ be some positively $\lambda_T$-saturated p.c.\ model (\thref{building-saturated-model}). By \textsc{extension} there is $M'$ with $M' \equiv_{Cb} M$ such that $a \ind_C M'$. Applying \textsc{extension} again we find $d'$ with $d' \equiv_{M'} d$ such that $a \ind_C M'd'$. As $Cb \subseteq M'$ and $M'$ is positively $\lambda_T$-saturated, we have $d' \equivls_{Cb} d$ by \thref{same-lstp-iff-same-types-over-sequence-of-models}. So by \textsc{monotonicity} we have $a \ind_C bd'$, as required. The final line follows because $\ind^d$ always satisfies \textsc{monotonicity} (\thref{dividing-basic-properties}) and it satisfies \textsc{extension} in thick simple theories (\thref{dividing-extension}).
\end{proof}
\begin{lemma}
\thlabel{independence-theorem-swaperoo}
Suppose that $\ind$ is an independence relation that satisfies \textsc{invariance}, \textsc{monotonicity}, \textsc{normality}, \textsc{base monotonicity}, \textsc{transitivity}, \textsc{symmetry} and \textsc{strong extension}. If $a \ind_C b$ and $a \ind_C c$ then there is $c'$ with $c' \equivls_{Ca} c$ such that $a \ind_C bc'$ and $b \ind_C c'$.
\end{lemma}
\begin{proof}
By \textsc{symmetry} we have $b \ind_C a$ and so by \textsc{strong extension} there is $c'$ with $c' \equivls_{Ca} c$ and $b \ind_C ac'$. Thus $b \ind_C c'$ by \textsc{monotonicity}.

By \textsc{monotonicity}, \textsc{base monotonicity} and \textsc{normality} we also have $bc' \ind_{Cc'} a$ and so by \textsc{symmetry} we find
\[
a \ind_{Cc'} bc'.
\]
By \textsc{invariance} we have $a \ind_C c'$, so by \textsc{normality} we have
\[
a \ind_C Cc',
\]
and so $a \ind_C bc'$ by \textsc{transitivity}.
\end{proof}
\begin{lemma}
\thlabel{independence-theorem-helper}
Suppose that $T$ is simple and that we are given $b$, $C$, a Morley sequence $(c_i)_{i < \omega}$ over $C$ and types $p(x, b)$ and $q(x, c_0)$ such that $b \ind^d_C c_0 c_1$ and $p(x, b) \cup q(x, c_0)$ does not divide over $C$. Then $p(x, b) \cup q(x, c_1)$ does not divide over $C$.
\end{lemma}
\begin{proof}
By \textsc{base monotonicity} we have $b \ind^d_{C c_0} c_1$. So by \thref{dividing-in-terms-of-automorphic-indiscernible-sequences} there is a $C c_0 b$-indiscernible sequence $(c'_i)_{1 \leq i < \omega}$ with $(c'_i)_{1 \leq i < \omega} \equiv_{C c_0 c_1} (c_i)_{1 \leq i < \omega}$, so $c'_1 = c_1$. Hence, after replacing $(c_i)_{1 \leq i < \omega}$ with $(c'_i)_{1 \leq i < \omega}$, we may as well assume that $(c_i)_{1 \leq i < \omega}$ is $C c_0 b$-indiscernible. Let $\lambda = \lambda_{|T| + |C b c_0|}$ and use compactness to elongate this sequence to $(c_i)_{i < \lambda}$ (still such that $(c_i)_{1 \leq i < \lambda}$ is $C c_0 b$-indiscernible). For each $i < \lambda$ we let $b_i$ be such that $b_i (c_j)_{i \leq j < \lambda} \equiv_C b (c_j)_{j < \lambda}$. Then, for all $i < j < \lambda$:
\begin{enumerate}[label=(\roman*)]
\item $b_i c_i \equiv_C b c_0$;
\item $b_i c_j \equiv_C b c_1$, by $Cb$-indiscernibility of $(c_j)_{1 \leq j < \lambda}$.
\end{enumerate}
Base a $C$-indiscernible sequence $(b^*_i c^*_i)_{i < \omega}$ on $(b_i c_i)_{i < \lambda}$. Then properties (i) and (ii) are carried over to this new sequence, and $(c^*_i)_{i < \omega} \equiv_C (c_i)_{i < \omega}$.

By (i) and the assumption that $p(x, b) \cup q(x, c_0)$ does not divide over $C$, there is a realisation $a$ of $p(x, b^*_0) \cup q(x, c^*_0)$ with $a \ind^d_C b^*_0 c^*_0$ (see \thref{dividing-extension-partial-types}) and so by \thref{dividing-in-terms-of-automorphic-indiscernible-sequences} there is $a' \equiv_{C b^*_0 c^*_0} a$ such that $(b^*_i c^*_i)_{i < \omega}$ is $Ca'$-indiscernible. Therefore $(c^*_i)_{1 \leq i < \omega}$ is a Morley sequence over $C$ that is $Ca'b^*_0$-indiscernible. By Kim's lemma (\thref{kims-lemma}), we have $a' b^*_0 \ind^d_C c^*_1$. Indeed, set $\Sigma(x, y, c^*_1) = \tp(a' b^*_0/Cc^*_1)$ and note that $\bigcup_{1 \leq i < \omega} \Sigma(x, y, c^*_i)$ is consistent as it is realised by $a' b^*_0$. By (ii) we have $b^*_0 c^*_1 \equiv_C b c_1$ and so there is $a''$ with $a'' b c_1 \equiv_C a' b^*_0 c^*_1$. In particular $a'' b \equiv_C a' b^*_0 \equiv_C a b^*_0$ and $a'' c_1 \equiv_C a' c^*_1 \equiv_C a' c^*_0 \equiv_C a c^*_0$, and so $a'' \models p(x, b) \cup q(x, c_1)$. We have that $a''b \ind^d_C c_1$ implies $a'' \ind^d_{Cb} b c_1$ by \textsc{symmetry}, \textsc{base monotonicity}, \textsc{monotonicity} and \textsc{normality}. As $p(x, b)$ does not divide over $C$ we also have $a'' \ind^d_C Cb$ (after an application of \textsc{normality}), and so $a'' \ind^d_C b c_1$ by \textsc{transitivity}. We thus conclude that $\tp(a''/C b c_1)$, which contains $p(x, b) \cup q(x, c_1)$, does not divide over $C$.
\end{proof}
\begin{theorem}
\thlabel{independence-theorem}
Assume thickness. Suppose that $T$ is a simple theory, then dividing independence satisfies the \textsc{independence theorem}. That is, suppose we are given $a, a', b, c, C$, such that $a \ind^d_C b$, $a' \ind^d_C c$ and $b \ind^d_C c$ with $a \equivls_C a'$. Then there is $a''$ with $a'' \equivls_{Cb} a$ and $a'' \equivls_{Cc} a'$ such that $a'' \ind^d_C bc$.
\end{theorem}
\begin{proof}
We first argue that we may assume $b$ and $c$ to enumerate positively $\lambda_T$-saturated p.c.\ models containing $C$. Let $M$ be a positively $\lambda_T$-saturated p.c.\ model containing $Cb$ (\thref{building-saturated-model}). By \textsc{extension} (and \textsc{symmetry}) applied to $b \ind^d_C c$ there is $M'$ with $M' \equiv_{Cb} M$ such that $M' \ind^d_C c$. Applying \textsc{strong extension} (\thref{strong-extension}), this time to $a \ind^d_C b$, together with \textsc{invariance} yields $a^*$ such that $a^* \equivls_{Cb} a$ and $a^* \ind^d_C M'$. We now replace $a$ with $a^*$ and $b$ with $M'$. Analogously, we can replace $a'$ with a tuple that has the same Lascar strong type over $Cc$ and replace $c$ with some positively $\lambda_T$-saturated p.c.\ model containing $Cc$. Finding an $a''$ as in the conclusion of the theorem now also works for our original $a, a', b$ and $c$.

Set $p(x, b) = \tp(a/Cb)$ and $q(x, c) = \tp(a'/Cc)$. By \thref{dividing-extension-partial-types}, and our assumption on $b$ and $c$, it then suffices to prove that $p(x, b) \cup q(x, c)$ does not divide over $C$. Using $a \equivls_C a'$ we let $c'$ be such that $ac' \equivls_C a'c$. Applying \thref{independence-theorem-swaperoo} to $a \ind^d_C b$ and $a \ind^d_C c'$ we find $c''$ with $c'' \equivls_{Ca} c'$, $a \ind^d_C bc''$ and $b \ind^d_C c''$. Applying \thref{independence-theorem-swaperoo} again, this time to $b \ind^d_C c$ and $b \ind^d_C c''$, we find $c^*$ with $c^* \equivls_{Cb} c''$ and $b \ind^d_C c c^*$.

Since $c^* \equivls_C c'' \equivls_C c' \equivls_C c$, there are $c^* = c_0, \ldots, c_n = c$ such that $c_i$ and $c_{i+1}$ start a Morley sequence over $C$ for all $i < n$ by \thref{lascar-strong-type-morley-sequences}. By \textsc{extension} applied to $b \ind^d_C c c^*$ we may assume that $b \ind^d_C c_0 \ldots c_n$. It then follows by induction on $i \leq n$ that $p(x, b) \cup q(x, c_i)$ does not divide over $C$. Indeed, for the base case we have that $a \ind^d_C bc''$, so $p(x, b) \cup q(x, c'')$ does not divide over $C$, and it follows that $p(x, b) \cup q(x, c_0)$ does not divide over $C$ because $b c_0 = b c^* \equiv_C b c''$. The induction step is precisely \thref{independence-theorem-helper}. This concludes our proof, because $c_n = c$.
\end{proof}
\section{The Kim-Pillay style theorem}
\label{sec:kim-pillay}
\begin{theorem}[Kim-Pillay style theorem]
\thlabel{kim-pillay}
Assume thickness. A theory $T$ is simple if and only if there is an independence relation $\ind$ satisfying \textsc{invariance}, \textsc{monotonicity}, \textsc{normality}, \textsc{existence}, \textsc{full existence}, \textsc{base monotonicity}, \textsc{extension}, \textsc{symmetry}, \textsc{transitivity}, \textsc{finite character}, \textsc{local character} and \textsc{independence theorem}. Furthermore, in this case, $\ind = \ind^d$.
\end{theorem}
\begin{remark}
\thlabel{kim-pillay-weaken-independence-theorem}
One direction of \thref{kim-pillay} can be strengthened as follows. To conclude that $T$ is simple we can assume a weaker version of \textsc{independence theorem}. More precisely, suppose that $T$ is thick and that we are given an independence relation $\ind$ satisfying all the properties in \thref{kim-pillay}, except that instead of \textsc{independence theorem} we assume the following. For any positively $\lambda_T$-saturated p.c.\ model $M$ and any $a,a',b,c$ such that $a \ind_M b$, $a' \ind_M c$ and $b \ind_M c$ with $a \equiv_M a'$ there is $a''$ with $a'' \equiv_{Mb} a$ and $a'' \equiv_{Mc} a'$ such that $a'' \ind_M bc$. Then $T$ is simple and $\ind = \ind^d$. So in particular, the full version of \textsc{independence theorem} then holds for $\ind$.

We also note in the weakened version of the \textsc{independence theorem} we only assume $a \equiv_M a'$ and not $a \equivls_M a'$. However, the latter is automatic by \thref{same-lstp-iff-same-types-over-sequence-of-models}, so this is truly a weakening of \textsc{independence theorem}.
\end{remark}
\begin{remark}
\thlabel{kim-pillay-leave-out-properties}
The properties \textsc{normality} and \textsc{full existence} could be left out from \thref{kim-pillay}. This is because they follow from the remaining properties.

If $\ind$ satisfies \textsc{extension} and \textsc{symmetry} then it satisfies \textsc{normality}. Indeed, if $a \ind_C b$ then we apply \textsc{extension} to find $c'$ with $C' \equiv_{Cb} C$ such that $a \ind_C C'b$, but then $C' = C$ and so $a \ind_C Cb$. We get $Ca \ind_C Cb$ by applying \textsc{symmetry}, repeating the argument and then applying \textsc{symmetry} again.

If $\ind$ satisfies \textsc{existence} and \textsc{extension} then it satisfies \textsc{full existence}. Indeed, let $a$, $b$ and $C$ be arbitrary. Then by \textsc{existence} we have $a \ind_C C$. So by \textsc{extension} we find the required $b'$ with $b' \equiv_C b$ and $a \ind_C b'$.
\end{remark}
\begin{proof}[Proof of \thref{kim-pillay}]
Earlier in this chapter, we have shown that $\ind^d$ satisfies all the properties listed in the theorem, provided that $T$ is simple. We will now prove that given an independence relation $\ind$ satisfying the above list of properties (with \textsc{independence theorem} adjusted as in \thref{kim-pillay-weaken-independence-theorem}), we have that $\ind = \ind^d$. Simplicity of $T$ then follows, because this means in particular that $\ind^d$ satisfies \textsc{local character}.

We first prove that $a \ind^d_C b$ implies $a \ind_C b$. By \thref{building-morley-sequences} there is a $\ind$-Morley sequence $(b_i)_{i < \omega}$ over $C$ with $b_0 = b$. In particular, $(b_i)_{i < \omega}$ is a $C$-indiscernible sequence, so, writing $p(x, b) = \tp(a/Cb)$, we have that $\bigcup_{i < \omega} p(x, b_i)$ is consistent. Applying \thref{general-kims-lemma} we find $a'$ with $\models p(a', b)$ and $a' \ind_C b$. So $a'b \equiv_C ab$ and $a \ind_C b$ follows from \textsc{invariance}.

Now assume $a \ind_C b$. We prove that $a \ind^d_C b$. Let $(b_i)_{i < \omega}$ be a $C$-indiscernible sequence with $b_0 = b$, and write $p(x, b) = \tp(a/Cb)$. We will show that $\bigcup_{i < \omega} p(x, b_i)$ is consistent.

By compactness we can elongate $(b_i)_{i < \omega}$ to $(b_i)_{i \leq \lambda}$, taking $\lambda$ big enough to apply \textsc{local character} for $\ind$ with respect to $\kappa = |b|$. We inductively construct a chain $(M_i)_{i < \lambda}$ of positively $\lambda_T$-saturated p.c.\ models such that for all $i < \lambda$:
\begin{enumerate}[label=(\arabic*)]
\item $C(b_j)_{j < i} \subseteq M_i$,
\item $(b_j)_{i \leq j \leq \lambda}$ is $M_i$-indiscernible.
\end{enumerate}
The base case and limit stage follow the same argument. Let $\ell < \lambda$ be a limit (or $\ell = 0$). Let $M$ be any $\lambda_T$-saturated p.c.\ model containing $C(M_i)_{i < \ell}$. By the induction hypothesis $(b_i)_{\ell \leq i \leq \lambda}$ is $C(M_i)_{i < \ell}$-indiscernible. So by \thref{extend-base-set-of-indiscernible-sequence} we find $M_\ell$ with $M_\ell \equiv_{C(M_i)_{i < \ell}} M$ such that $(b_i)_{\ell \leq i \leq \lambda}$ is $M_\ell$-indiscernible.

For the successor step we assume $M_i$ has been constructed. Let $M$ be any $\lambda_T$-saturated p.c.\ model containing $M_i b_i$. By the induction hypothesis $(b_j)_{i \leq j \leq \lambda}$ is $M_i$-indiscernible, and so $(b_j)_{i+1 \leq j \leq \lambda}$ is $M_i b_i$-indiscernible. So by \thref{extend-base-set-of-indiscernible-sequence} we find $M_{i+1}$ with $M_{i+1} \equiv_{M_i b_i} M$ such that $(b_j)_{i+1 \leq j \leq \lambda}$ is $M_{i+1}$-indiscernible. This finishes the construction of the chain $(M_i)_{i < \lambda}$.

Set $M = \bigcup_{i < \lambda} M_i$ and apply \thref{local-character-in-terms-of-sequences} to $(M_i)_{i < \lambda}$ viewed as a sequence to find $i_0 < \lambda$ with $b_\lambda \ind_{M_{i_0}} M$. Let $(b'_n)_{n < \omega}$ be the sequence defined by $b'_n = b_{i_0 + n}$, then $(b'_n)_{n < \omega}$ is a $\ind$-Morley sequence over $M_{i_0}$. Indeed, this sequence is $M_{i_0}$-indiscernible by construction. To see that it is a $\ind_{M_{i_0}}$-independent sequence we use that for any $n < \omega$ we have $b_\lambda \ind_{M_{i_0}} (b_j)_{i_0 \leq j < i_0 + n}$ by \textsc{monotonicity} and conclude by $b_\lambda (b_j)_{i_0 \leq j < i_0 + n} \equiv_{M_{i_0}} b_{i_0 + n} (b_j)_{i_0 \leq j < i_0 + n}$ and \textsc{invariance}. By $C$-indiscernibility of $(b_i)_{i < \lambda}$, of which $(b'_n)_{n < \omega}$ is a subsequence, it suffices now to prove that $\bigcup_{n < \omega} p(x, b'_n)$ is consistent.

As $b'_0 \equiv_C b$ there is $a'$ with $a'b'_0 \equiv_C ab$ and so $a' \ind_C b'_0$. Applying \textsc{extension} to this last independence, and after applying an automorphism, we find $a''$ with $a'' \equiv_{Cb'_0} a'$ with $a'' \ind_C M_{i_0} b'_0$. By \textsc{base monotonicity} and \textsc{monotonicity} we then have $a'' \ind_{M_{i_0}} b'_0$. By induction on $n < \omega$ we will prove that there is $a_n$ that realises $\bigcup_{i \leq n} p(x, b'_i)$ with $a_n \equiv_{M_{i_0}} a''$ and $a_n \ind_{M_{i_0}} b'_0 \ldots b'_n$. This suffices, because then by compactness $\bigcup_{n < \omega} p(x, b'_n)$ is consistent.

For the base case we take $a_0 = a''$, which can be done because $\models p(a', b'_{i_0})$ and so $\models p(a'', b'_{i_0})$. For the inductive step we assume that we have constructed $a_n$. As $b'_{n+1}$ is on some $M_{i_0}$-indiscernible sequence that $b'_0$ is on as well, we have $b'_{n+1} \equiv_{M_{i_0}} b'_0$. Let $a^*$ be such that $a^* b'_{n+1} \equiv_{M_{i_0}} a'' b'_0$, so $a^* \ind_D b'_{n+1}$. Using $a_n \equiv_{M_{i_0}} a'' \equiv_{M_{i_0}} a^*$ we apply the weakened version of \textsc{independence theorem} for $\ind$ to $a_n \ind_{M_{i_0}} b'_0 \ldots b'_n$, $a^* \ind_{M_{i_0}} b'_{n+1}$ and $b'_{n+1} \ind_{M_{i_0}} b'_0 \ldots b'_n$ to find $a_{n+1}$ with $a_{n+1} \ind_{M_{i_0}} b'_0 \ldots b'_{n+1}$, as well as $a_{n+1} \equiv_{M_{i_0} b'_0 \ldots b'_n} a_n$ and $a_{n+1} \equiv_{M_{i_0} b'_{n+1}} a^*$. Either of these last two equalities of types then implies $a_{n+1} \equiv_{M_{i_0}} a''$, and together they imply that $a_{n+1}$ realises $\bigcup_{i \leq n+1} p(x, b'_i)$. This completes the inductive construction and thus the argument.
\end{proof}

\section{Bibliographic remarks}
\label{sec:bibliographic-remarks-simple}
Simplicity in positive logic was first developed by Pillay in 2000 for what we call Pillay theories \cite{pillay_forking_2000}. Later, Ben-Yaacov generalised this in 2003 to the full generality of positive logic \cite{ben-yaacov_simplicity_2003}. In our treatment we have emphasised the semantic side of independence and work as much as possible with abstract independence relations $\ind$ and general properties of $\ind^d$.

In Pillay's approach, simplicity was defined as forking independence having \textsc{local character} (so he does treat forking in positive logic, cf.\ \thref{forking-vs-dividing}), whereas Ben-Yaacov defined it as dividing independence having \textsc{local character}. Since both forking independence and dividing independence always satisfy \textsc{base monotonicity}, Pillay's approach automatically gives \textsc{full existence}. In Ben-Yaacov's approach \textsc{full existence} is not automatic, and as we have seen it requires a lot of work in thick theories (see Section \ref{sec:thickness-implies-full-existence}). There is even a stable positive theory where \textsc{full existence} for dividing independence fails \cite[Example 4.3]{ben-yaacov_simplicity_2003} (see also the discussion of ultrametric spaces with distances in $\N$ in Section \ref{sec:bibliographic-remarks-basics}), so this theory is not simple in Pillay's sense. We therefore view Ben-Yaacov's definition as the right notion of simplicity, which is furthermore confirmed by it being equivalent to \NTP (\thref{simplicity-equivalences}). Of course, for thick theories, the two definitions coincide because assuming either it can be proved that forking and dividing coincide.

We have decided to only treat simplicity for thick theories, which simplifies the treatment and allows us to stay closer to the treatment in full first-order logic. This still captures a large class of positive theories, see for example \cite[Section 2]{kamsma_positive_2024} for a list of (classes of) examples of thick theories.
\begin{remark}
\thlabel{full-existence-without-thickness}
The proof strategy for \thref{simple-thick-implies-full-existence}, \textsc{full existence} for dividing independence in a thick simple theory, is very similar to the usual approach in full first-order logic. In both cases one constructs dividing sequences, which have to keep track of some witness of dividing. In full first-order logic this is the $k$ in $k$-dividing, but in positive logic this becomes a formula, namely the $\psi$ in $\psi$-dividing (see also \thref{psi-dividing-remark}). This makes it so that we have to consider longer sequences. However, the real use for thickness is \thref{dividing-sequence-type-definable}, ensuring that being a dividing sequences is type-definable, which is also heavily used in the full first-order approach. This allows us to reshape the dividing sequences, and to reduce to finite dividing sequences, which is important for \thref{finite-dividing-sequence-extend-base-set}.

Without thickness not all is lost, but the theory becomes more subtle and complicated. For example, we no longer get \textsc{full existence} for dividing over any set, even in stable theories (see also the discussion of ultrametric spaces with distances in $\N$ in Section \ref{sec:bibliographic-remarks-basics}). However, we do get it over positively $\lambda$-saturated p.c.\ models for big enough $\lambda$. This is all worked out in \cite{ben-yaacov_simplicity_2003}.
\end{remark}

Adapting the definition of $k$-\TP to positive logic (\thref{tree-property}) is now a standard trick that is due to \cite{haykazyan_existentially_2021}, where they give positive versions of \TP[2] and \SOP[1]. The particular definition of $k$-\TP first appears in \cite[Definition 4.3]{dmitrieva_dividing_2023}, where its equivalence to simplicity (in terms of local character for dividing) was also proved \cite[Theorem 6.14]{dmitrieva_dividing_2023}.

The reason for calling \thref{kim-pillay} a ``Kim-Pillay style theorem'' is because Kim and Pillay first characterised simple theories in terms of the existence of a (unique) independence relation \cite[Theorem 4.2]{kim_simple_1997}. Their theorem is for full first-order logic. \thref{kim-pillay} can be pieced together from Ben-Yaacov's work \cite{ben-yaacov_simplicity_2003,ben-yaacov_thickness_2003}, but the addition of the thickness assumption allows for a simpler statement similar to the original \cite[Theorem 4.2]{kim_simple_1997}.

In the next chapter we treat stability in positive logic and connect it to simplicity as treated in this chapter. Another direction would be to treat the development of Kim-independence in thick \NSOP[1] theories, for which we refer the reader to \cite{dobrowolski_kim-independence_2022} (with a correction to the proof of the independence theorem in \cite{dobrowolski_correction_2024}).

\chapter{Stable theories}
\label{ch:stable-theories}
In this chapter we view stable theories through the lens of dividing independence. We build on the results for simple theories from Chapter \ref{ch:simple-theories}, and use this to show that the main difference between stability and simplicity is the \textsc{stationarity} property for dividing independence. This results in a Kim-Pillay style theorem for stable theories (\thref{stable-kim-pillay}).

Classically there are many equivalent definitions of stability. We use the type counting definition (\thref{stability}), but we establish the usual equivalence with definability of types and the non-existence of binary trees (\thref{stable-characterisations}). Notably, we do not consider the order property. This is solely because we have no use for it in these notes, and not because it causes problems in positive logic. In fact, when correctly formulated, not having the order property is further equivalent to the conditions we give here (see also \thref{stable-iff-no-order-property}).

\section{Invariant types}
\label{sec:invariant-types}
\begin{definition}
\thlabel{global-type}
A \term{global type}\index{type!global} is a type $q(x)$ over the monster model. That is, it is a maximally consistent set of formulas over the monster model.
\end{definition}
A global type will generally not have a realisation in the monster model. However, it will often be convenient to work with realisations of global types. This can be done in roughly two different ways:
\begin{enumerate}
\item the realisation lives in some bigger monster model;
\item we restrict the global type to some sufficiently saturated p.c.\ submodel of the monster, containing all the other parameters we are interested in, and realise this restriction.
\end{enumerate}
We will take the first view, but if the reader is not comfortable with this due to their preferred formalism for the monster model then using the second approach does not change the arguments whatsoever.
\begin{convention}
\thlabel{global-type-realisations}
We generally use Greek lowercase letters $\alpha, \beta, \ldots$ for realisations of global types.
\end{convention}
\begin{definition}
\thlabel{invariant-type}
Let $p(x) = \tp(a/B)$ be a type and let $C \subseteq B$. We say that $p(x)$ is \emph{$C$-invariant}\index{invariant type}\index{type!invariant} if the following equivalent conditions hold:
\begin{enumerate}[label=(\roman*)]
\item for any $b, b' \in B$ with $b \equiv_C b'$ we have $ab \equiv_C ab'$;
\item for any $b, b' \in B$ with $b \equiv_C b'$ and any formula $\phi(x, y)$ we have $\phi(x, b) \in p(x)$ if and only if $\phi(x, b') \in p(x)$.
\end{enumerate}
\end{definition}
The equivalence of the conditions in the above definition should be obvious.

We are particularly interested in global invariant types. In full first-order logic there are many of these (as also follows from \thref{semi-hausdorff-invariant-types}). However, in positive logic they may not exist.
\begin{example}
\thlabel{non-invariant-type}
The theory from \thref{thick-not-semi-hausdorff} is also an example of a theory where a type over a p.c.\ model $M$ does not necessarily extend to a global $M$-invariant type. In light of \thref{semi-hausdorff-invariant-types} this yields another proof that this theory is not semi-Hausdorff. We use the same notation as in \thref{thick-not-semi-hausdorff}.

We claim that the type $p(x) = \tp(a_\omega/M)$ does not extend to a global $M$-invariant type. Suppose that there is a global $M$-invariant extension $q(x) \supseteq p(x)$. Then either $q(x) = \tp(a_\omega / N)$ or $q(x) = \tp(b_\omega / N)$, as $N$ is the monster model of this theory (see \thref{maximal-pc-model-is-monster}). So since $a_\omega \equiv_M b_\omega$ we must have by $M$-invariance that $R(x, a_\omega)$ and $R(x, b_\omega)$ are both in $q(x)$. The only possible realisations of $q(x)$ are $a_\omega$ and $b_\omega$, so this contradicts the fact that $T \models \neg \exists x R(x, x)$.
\end{example}
\begin{proposition}
\thlabel{semi-hausdorff-invariant-types}
Assume semi-Hausdorffness. Then any type over any p.c.\ model $M$ extends to a global $M$-invariant type.
\end{proposition}
\begin{proof}
Let $p(x) = \tp(a/M)$ be a type over a p.c.\ model $M$. Recall that semi-Hausdorffness means that equality of types is type-definable. In particular, there is a partial type $\Omega(xy, xy')$ over $M$ such that $\models \Omega(a'b, a'b')$ if and only if $a' b \equiv_M a' b'$ for any $a', b, b'$. Let $\Sigma(x)$ be the (large) set of formulas
\[
p(x) \cup \bigcup \{\Omega(xb, xb') : b, b' \text{ finite tuples in the monster with } b \equiv_M b' \}.
\]
We claim that this set is finitely satisfiable. Indeed, let $\phi(x, m) \in p(x)$. As $\models \exists x \phi(x, m)$ and $M$ is p.c., there must be $a' \in M$ with $M \models \phi(a', m)$. Now for any $b, b'$ with $b \equiv_M b'$ we have in particular $a'b \equiv_M a'b'$. So $a'$ realises
\[
\{\phi(x, m)\} \cup \bigcup \{\Omega(xb, xb') : b, b' \text{ finite tuples in the monster with } b \equiv_M b' \},
\]
which establishes finite realisability of $\Sigma(x)$. By compactness there is a realisation $\alpha$ of $\Sigma(x)$. So $\tp(\alpha/\MM)$ extends $p(x)$ and is a global $M$-invariant type\footnote{If the reader insists on not allowing a bigger monster model then alternatively Zorn's lemma can be used to get a maximally consistent set of formulas containing $\Sigma(x)$, which will be the desired global $M$-invariant type extending $p(x)$.}.
\end{proof}
\begin{lemma}
\thlabel{morley-sequence-in-invariant-type-is-indiscernible}
Let $p(x) = \tp(a/B)$ be a type and let $C \subseteq B$. Suppose that $p(x)$ is $C$-invariant and $(a_i)_{i < \omega} \subseteq B$ is a sequence such that $a_i \models p|_{C (a_j)_{j < i}}$ for all $i < \omega$. Then $(a_i)_{i < \omega}$ is $C$-indiscernible.
\end{lemma}
\begin{proof}
We prove by induction on $n < \omega$ that for any $i_0 < \ldots < i_n < \omega$ we have $a_{i_0} \ldots a_{i_n} \equiv_C a_0 \ldots a_n$. For $n = 0$ this is immediate as $a_{i_0}$ and $a_0$ both realise $p|_C$. Now assume the induction hypothesis for $n$. Then by $C$-invariance and the fact that $a_{i_{n+1}}$ and $a_{n+1}$ both realise $p|_{C a_0 \ldots a_n}$ we have that
\[
a_{i_0} \ldots a_{i_n} a_{i_{n+1}} \equiv_C a_0 \ldots a_n a_{i_{n+1}} \equiv_C a_0 \ldots a_n a_{n+1},
\]
as required.
\end{proof}
\begin{proposition}
\thlabel{type-extends-to-global-invariant-implies-lstp}
Suppose that a type $p(x) = \tp(a/B)$ extends to a global $B$-invariant type. Then $a' \equiv_B a$ implies $\d_B(a', a) \leq 2$.
\end{proposition}
\begin{proof}
By assumption there is a global $B$-invariant $q(x) \supseteq p(x)$. Inductively construct $(a_i)_{1 \leq i < \omega}$ such that $a_i \models q|_{M a a' (a_j)_{j < i}}$. Then by \thref{morley-sequence-in-invariant-type-is-indiscernible} we have that the sequences $a, a_1, a_2, \ldots$ and $a', a_1, a_2, \ldots$ are both $B$-indiscernible, and so $\d_B(a, a') \leq 2$.
\end{proof}
We now obtain as a corollary what was already claimed in \thref{same-type-over-saturated-model-lascar-distance-2}.
\begin{corollary}
\thlabel{semi-hausdorff-same-type-over-pc-model-lascar-distance-2}
Assume semi-Hausdorffness. Then for any p.c.\ model $M$ we have that $a \equiv_M a'$ implies $\d_M(a, a') \leq 2$.
\end{corollary}
\begin{proof}
Combine \thref{semi-hausdorff-invariant-types} and \thref{type-extends-to-global-invariant-implies-lstp}.
\end{proof}
Invariant types are non-dividing types, in the following sense.
\begin{lemma}
\thlabel{invariant-types-do-not-divide}
Suppose that $p(x) = \tp(a/M)$ is a $C$-invariant type and $M \supseteq C$ is a positively $(\aleph_0 + |C|)^+$-saturated p.c.\ model. Then $p(x)$ does not divide over $C$. In particular, any global $C$-invariant type does not divide over $C$.
\end{lemma}
\begin{proof}
By \textsc{finite character} of dividing it is enough to show that for every finite $b \in M$, the type $p'(x, b) = \tp(a/Cb)$ does not divide over $C$. Let $(b_i)_{i < \omega}$ be a $C$-indiscernible sequence with $b_i \equiv_C b$ for all $i < \omega$. By saturation of $M$ we find $(b'_i)_{i < \omega} \subseteq M$ such that $(b'_i)_{i < \omega} \equiv_C (b_i)_{i < \omega}$. For each $i < \omega$ we have $b'_i \equiv_C b$ and so by $C$-invariance $a b'_i \equiv_C ab$. Hence $\models p'(a, b'_i)$ for all $i < \omega$, and so $\bigcup_{i < \omega} p'(x, b'_i)$ is realised by $a$. We conclude that $\bigcup_{i < \omega} p'(x, b_i)$ is consistent, as required.
\end{proof}
\section{Stability}
\label{sec:stability}
\begin{definition}
\thlabel{stability}
Let $\lambda$ be an infinite cardinal. A theory $T$ is called \emph{$\lambda$-stable}\index{stable!lambda-stable theory@$\lambda$-stable theory} if for all parameter sets $B$ with $|B| \leq \lambda$ and all finite index sets $I$ we have $|\S_I(B)| \leq \lambda$. We call $T$ \emph{stable}\index{stable!stable theory} if it is $\lambda$-stable for some $\lambda$.
\end{definition}
If $T$ is single-sorted then $\lambda$-stability is simply saying that $|\S_n(B)| \leq \lambda$ for all $n < \omega$ and all $B$ with $|B| \leq \lambda$.

The following is a more practical condition to check.
\begin{proposition}
\thlabel{stable-iff-1-types-stable}
Let $\lambda$ be an infinite cardinal. A theory $T$ is $\lambda$-stable if and only if for all $B$ with $|B| \leq \lambda$ and all singleton index sets $I$ we have $|\S_I(B)| \leq \lambda$.
\end{proposition}
If $T$ is single-sorted then the above simplifies to: $T$ is $\lambda$-stable if and only if for all $B$ with $|B| \leq \lambda$ we have $|\S_1(B)| \leq \lambda$. For notational convenience, we give the proof for the single-sorted case. It should be clear that the same proof works in the multi-sorted case, and the only difference is purely notational.
\begin{proof}
The variables in this proof are all single variables. The left to right direction is trivial. For the converse we prove by induction on $n < \omega$ that $|\S_n(B)| \leq \lambda$ for all $B$ with $|B| \leq \lambda$. The base case $n = 0$ follows because there is only one type in no variables over $B$: the set of formulas over $B$ with no variables that hold in the monster model.

For the successor step we assume that the induction hypothesis holds for types in at most $n$ variables. Let $B$ be a parameter set with $|B| \leq \lambda$. By the induction hypothesis there are at most $\lambda$ many types in $n$ variables over $B$. Enumerate those types as $(p_i(x_1, \ldots, x_n))_{i < \lambda}$. For each $i < \lambda$ let $b_i$ be a realisation of $p_i(x_1, \ldots, x_n)$. So each $b_i$ is a finite tuple, hence $B' = B \cup \bigcup_{i < \lambda} b_i$ has cardinality at most $\lambda$. By assumption, there are at most $\lambda$ many types in one variable over $B'$. Enumerate those types as $(q_i(x))_{i < \lambda}$. Set $P = \{p_i(x_1, \ldots, x_n) : i < \lambda\}$ and $Q = \{q_i(x) : i < \lambda\}$. We will define a surjection $f: P \times Q \to \S_{n+1}(B)$, from which $|\S_{n+1}(B)| \leq \lambda$ follows, thus concluding the induction and hence the proof.

Let $(p_i(x_1, \ldots, x_n), q_j(x)) \in P \times Q$. Recall that $q_j|_{B b_i}(x)$ is the restriction of $q_j(x)$ to the parameters $B b_i$. Let $q(x, x_1, \ldots, x_n)$ be the result of replacing the tuple $b_i$ in $q_j|_{B b_i}(x)$ by the tuple of variables $(x_1, \ldots, x_n)$. Then $q(x, x_1, \ldots, x_n)$ is the type over $B$ such that $q(x, b_i) = q_j|_{B b_i}(x)$. We let $f$ send $(p_i(x_1, \ldots, x_n), q_j(x))$ to $q(x, x_1, \ldots, x_n)$.

Having defined $f$, we now show that it is a surjection. Let $q(x, x_1, \ldots, x_n)$ be any type over $B$ in $n+1$ variables. Let $p(x_1, \ldots, x_n)$ be the restriction of $q(x, x_1, \ldots, x_n)$ to the variables $x_1, \ldots, x_n$. Then $p(x_1, \ldots, x_n)$ is a type over $B$ and so there is $i < \lambda$ such that $p(x_1, \ldots, x_n) = p_i(x_1, \ldots, x_n)$. There is then also $j < \lambda$ such that $q(x, b_i) = q_j(x)$. By construction $f$ sends $(p_i(x_1, \ldots, x_n), q_j(x))$ to $q(x, x_1, \ldots, x_n)$, and we conclude that $f$ is surjective.
\end{proof}
\begin{proposition}
\thlabel{maximal-model-implies-stable}
Any theory with a maximal p.c.\ model is stable.
\end{proposition}
\begin{proof}
Let $M$ be the maximal p.c.\ model. We claim that $T$ is $|M| + \aleph_0$-stable. By \thref{maximal-pc-model-is-monster} we have that $M$ is the monster model. There are at most $|M| + \aleph_0$ many distinct finite tuples in $M$. So for any set $B$ there are at most $|M| + \aleph_0$ many types over $B$ in finitely many free variables, since each of these types is realised in $M$.
\end{proof}
\begin{definition}
\thlabel{definable-type}
Let $p(x)$ be a type over $B$ and let $\phi(x, y)$ be a formula without parameters. A \emph{$\phi$-definition over $C$}\index{definition@$\phi$-definition} of $p(x)$ is a set of formulas $\d_p \phi(y)$ over $C$ with $|\d_p \phi(y)| \leq |T|$ such that for all $b \in B$ we have
\[
\phi(x, b) \in p(x) \quad \Longleftrightarrow \quad \models \d_p \phi(b).
\]
We say that $p(x)$ is \emph{definable over $C$}\index{definable type}\index{type!definable} if it has a $\phi$-definition over $C$ for every formula $\phi(x, y)$. If $p(x)$ is definable over $B$ then we just say that $p(x)$ is \emph{definable}.
\end{definition}
In full first-order logic the $\phi$-definition $\d_p \phi(y)$ is only a formula, not a set of formulas. Even after this relaxation of the definition of a definable type to allow for sets of formulas, we can still prove many of the important results that we know from full first-order logic. For example, that a theory is stable if and only if every type is definable (see \thref{stable-characterisations}). At the same time it is really necessary to relax the definition, see \thref{definable-type-needs-set-of-formulas}.
\begin{example}
\thlabel{definable-type-needs-set-of-formulas}
We give an example of a stable theory $T$, a type $p(x)$ and a formula $\phi(x, y)$ such that the $\phi$-definition $\d_p \phi(y)$ has to be an infinite set of formulas and can thus not be given by a single formula.

We recall the theory from \thref{hausdorff-not-boolean}. Our language $\L$ has a constant symbol for each element of $\Q_{(0,1)} = \{q \in \Q : 0 < q < 1\}$, as well as an order symbol $\leq$. The theory $T$ is the set of h-inductive sentences true in $\Q_{(0,1)}$, viewed as an obvious $\L$-structure. There is a maximal model given by the real unit interval $[0,1]$. So $T$ is stable by \thref{maximal-model-implies-stable}.

We claim that for $p(x) = \tp(\sqrt{2}/\Q_{(0,1)})$ and $\phi(x, y)$ the formula $x \leq y$ the $\phi$-definition $\d_p \phi(y)$ cannot be finite. Using \thref{regular-formulas-eliminate-quantifiers-implies-positive-quantifier-elimination} it is straightforward to see that $T$ has positive quantifier elimination. So if $\d_p \phi(y)$ were to be a single formula then it would be a finite union of closed intervals with rational endpoints. At the same time we have $\d_p \phi([0,1]) = [\sqrt{2}, 1]$, which cannot be written as such a finite union. We conclude that $\d_p \phi(y)$ cannot be finite.

The $\phi$-definition that is constructed in \thref{definable-type-needs-set-of-formulas} is rather abstract. Just to give some extra intuition we give an explicit construction for the above $p(x)$ and $\phi(x, y)$. Let $(q_i)_{i < \omega}$ be a sequence in $\Q_{(0,1)}$ that approaches $\sqrt{2}$ from below. Then we can take
\[
\d_p \phi(y) = \{ q_i \leq y : i < \omega \},
\]
and we clearly have $\d_p \phi([0,1]) = [\sqrt{2}, 1]$, as required.
\end{example}
\begin{definition}
\thlabel{binary-tree-rank}
For contradictory formulas $\phi(x, y)$ and $\psi(x, y)$ we define the \emph{$(\phi, \psi)$-rank}\index{rank@$(\phi, \psi)$-rank} $R_{\phi, \psi}(-)$\nomenclature[Rphipsi]{$R_{\phi, \psi}(-)$}{The $(\phi, \psi)$-rank} as follows. Its input is a set of formulas, possibly with parameters, in free variables $x$. Then $R_{\phi, \psi}(-)$ is the least function into the ordinals (augmented by $-1$ and a ``biggest ordinal'' $\infty$) such that:
\begin{itemize}
\item $R_{\phi,\psi}(\Sigma) \geq 0$ if $\Sigma(x)$ is consistent;
\item $R_{\phi,\psi}(\Sigma) \geq \alpha + 1$ if there is some $b$ such that $R_{\phi,\psi}(\Sigma \cup \{ \phi(x, b) \} ) \geq \alpha$ and $R_{\phi,\psi}(\Sigma \cup \{ \psi(x, b) \} ) \geq \alpha$;
\item $R_{\phi,\psi}(\Sigma) \geq \ell$ if $R_{\phi,\psi}(\Sigma) \geq \alpha$ for all $\alpha < \ell$, where $\ell$ is a limit ordinal.
\end{itemize}
\end{definition}
So in the above definition we have $R_{\phi,\psi}(\Sigma) = -1$ if and only if $\Sigma(x)$ is inconsistent, while $R_{\phi,\psi}(\Sigma) = \infty$ means that $R_{\phi,\psi}(\Sigma) \geq \alpha$ for all ordinals $\alpha$.
\begin{lemma}
\thlabel{binary-tree-rank-lemma}
Let $\phi(x, y)$ and $\psi(x, y)$ be contradictory formulas.
\begin{enumerate}[label=(\roman*)]
\item If $\Sigma(x)$ implies $\Sigma'(x)$ then $R_{\phi,\psi}(\Sigma) \leq R_{\phi,\psi}(\Sigma')$.
\item The property $R_{\phi, \psi}(\Sigma) \geq n$ is type-definable by
\[
\exists (y_\eta)_{\eta \in 2^{< n}} \left( \bigwedge_{\sigma \in 2^n} \exists x \left( \Sigma(x) \wedge \bigwedge_{k < n} \chi_{\sigma(k)}(x, y_{\sigma|_k}) \right) \right),
\]
where $\chi_0$ and $\chi_1$ are $\phi$ and $\psi$ respectively. In particular, if $\Sigma$ is finite (i.e.\ a formula), then this is just a formula.
\end{enumerate}
\end{lemma}
Note that \thref{binary-tree-rank-lemma}(ii) makes sense, even for infinite $\Sigma$, by \thref{combine-partial-types}. What this is really saying is that $R_{\phi, \psi}(\Sigma) \geq n$ is witnessed by a binary tree of parameters, represented by the variables $(y_\eta)_{\eta \in 2^{< n}}$. For each branch we then form a set of formulas as follows. At every node along the branch, we pick either $\phi$ or $\psi$ based on how the branch continues, and we plug in the variable corresponding to that node. We then require this set of formulas to be consistent, for every branch. This is depicted in figure \ref{fig:binary-tree} for $n = 3$.
\begin{figure}[ht]
\centering
\begin{tikzpicture}[x=0.75pt,y=0.75pt,yscale=-1,xscale=1]

\draw  [draw opacity=0][fill={rgb, 255:red, 0; green, 0; blue, 0 }  ,fill opacity=1 ] (155,300) .. controls (155,297.24) and (157.24,295) .. (160,295) .. controls (162.76,295) and (165,297.24) .. (165,300) .. controls (165,302.76) and (162.76,305) .. (160,305) .. controls (157.24,305) and (155,302.76) .. (155,300) -- cycle ;
\draw  [draw opacity=0][fill={rgb, 255:red, 0; green, 0; blue, 0 }  ,fill opacity=1 ] (75,240) .. controls (75,237.24) and (77.24,235) .. (80,235) .. controls (82.76,235) and (85,237.24) .. (85,240) .. controls (85,242.76) and (82.76,245) .. (80,245) .. controls (77.24,245) and (75,242.76) .. (75,240) -- cycle ;
\draw  [draw opacity=0][fill={rgb, 255:red, 0; green, 0; blue, 0 }  ,fill opacity=1 ] (235,240) .. controls (235,237.24) and (237.24,235) .. (240,235) .. controls (242.76,235) and (245,237.24) .. (245,240) .. controls (245,242.76) and (242.76,245) .. (240,245) .. controls (237.24,245) and (235,242.76) .. (235,240) -- cycle ;
\draw    (80,240) -- (160,300) ;
\draw    (39.63,180) -- (80,240) ;
\draw    (80,240) -- (120,180) ;
\draw    (160,300) -- (240,240) ;
\draw  [draw opacity=0][fill={rgb, 255:red, 0; green, 0; blue, 0 }  ,fill opacity=1 ] (34.63,180) .. controls (34.63,177.24) and (36.87,175) .. (39.63,175) .. controls (42.39,175) and (44.63,177.24) .. (44.63,180) .. controls (44.63,182.76) and (42.39,185) .. (39.63,185) .. controls (36.87,185) and (34.63,182.76) .. (34.63,180) -- cycle ;
\draw  [draw opacity=0][fill={rgb, 255:red, 0; green, 0; blue, 0 }  ,fill opacity=1 ] (115,180) .. controls (115,177.24) and (117.24,175) .. (120,175) .. controls (122.76,175) and (125,177.24) .. (125,180) .. controls (125,182.76) and (122.76,185) .. (120,185) .. controls (117.24,185) and (115,182.76) .. (115,180) -- cycle ;
\draw    (20,120) -- (39.63,180) ;
\draw    (39.63,180) -- (60,120) ;
\draw    (100,120) -- (120,180) ;
\draw    (120,180) -- (140,120) ;
\draw    (199.63,180) -- (240,240) ;
\draw    (240,240) -- (280,180) ;
\draw  [draw opacity=0][fill={rgb, 255:red, 0; green, 0; blue, 0 }  ,fill opacity=1 ] (194.63,180) .. controls (194.63,177.24) and (196.87,175) .. (199.63,175) .. controls (202.39,175) and (204.63,177.24) .. (204.63,180) .. controls (204.63,182.76) and (202.39,185) .. (199.63,185) .. controls (196.87,185) and (194.63,182.76) .. (194.63,180) -- cycle ;
\draw  [draw opacity=0][fill={rgb, 255:red, 0; green, 0; blue, 0 }  ,fill opacity=1 ] (275,180) .. controls (275,177.24) and (277.24,175) .. (280,175) .. controls (282.76,175) and (285,177.24) .. (285,180) .. controls (285,182.76) and (282.76,185) .. (280,185) .. controls (277.24,185) and (275,182.76) .. (275,180) -- cycle ;
\draw    (180,120) -- (199.63,180) ;
\draw    (199.63,180) -- (220,120) ;
\draw    (260,120) -- (280,180) ;
\draw    (280,180) -- (300,120) ;

\draw (171,292.4) node [anchor=north west][inner sep=0.75pt]  [font=\small]  {$y_{\emptyset }$};
\draw (53,232.4) node [anchor=north west][inner sep=0.75pt]  [font=\small]  {$y_{0}$};
\draw (251,232.4) node [anchor=north west][inner sep=0.75pt]  [font=\small]  {$y_{1}$};
\draw (131,172.4) node [anchor=north west][inner sep=0.75pt]  [font=\small]  {$y_{01}$};
\draw (7,172.4) node [anchor=north west][inner sep=0.75pt]  [font=\small]  {$y_{00}$};
\draw (291,172.4) node [anchor=north west][inner sep=0.75pt]  [font=\small]  {$y_{11}$};
\draw (167,172.4) node [anchor=north west][inner sep=0.75pt]  [font=\small]  {$y_{10}$};
\draw (12.07,104.58) node [anchor=north west][inner sep=0.75pt]  [font=\scriptsize,rotate=-330]  {$\Sigma ( x) \land \varphi ( x,y_{\emptyset }) \land \varphi ( x,y_{0}) \land \varphi ( x,y_{00})$};
\draw (52.12,105.08) node [anchor=north west][inner sep=0.75pt]  [font=\scriptsize,rotate=-330]  {$\Sigma ( x) \land \varphi ( x,y_{\emptyset }) \land \varphi ( x,y_{0}) \land \psi ( x,y_{00})$};
\draw (92.12,105.08) node [anchor=north west][inner sep=0.75pt]  [font=\scriptsize,rotate=-330]  {$\Sigma ( x) \land \varphi ( x,y_{\emptyset }) \land \psi ( x,y_{0}) \land \varphi ( x,y_{01})$};
\draw (132.12,105.08) node [anchor=north west][inner sep=0.75pt]  [font=\scriptsize,rotate=-330]  {$\Sigma ( x) \land \varphi ( x,y_{\emptyset }) \land \psi ( x,y_{0}) \land \psi ( x,y_{01})$};
\draw (172.12,105.08) node [anchor=north west][inner sep=0.75pt]  [font=\scriptsize,rotate=-330]  {$\Sigma ( x) \land \psi ( x,y_{\emptyset }) \land \varphi ( x,y_{1}) \land \varphi ( x,y_{10})$};
\draw (212.12,105.08) node [anchor=north west][inner sep=0.75pt]  [font=\scriptsize,rotate=-330]  {$\Sigma ( x) \land \psi ( x,y_{\emptyset }) \land \varphi ( x,y_{1}) \land \psi ( x,y_{10})$};
\draw (252.12,105.08) node [anchor=north west][inner sep=0.75pt]  [font=\scriptsize,rotate=-330]  {$\Sigma ( x) \land \psi ( x,y_{\emptyset }) \land \psi ( x,y_{1}) \land \varphi ( x,y_{11})$};
\draw (292.12,105.08) node [anchor=north west][inner sep=0.75pt]  [font=\scriptsize,rotate=-330]  {$\Sigma ( x) \land \psi ( x,y_{\emptyset }) \land \psi ( x,y_{1}) \land \psi ( x,y_{11})$};
\draw (101,272.4) node [anchor=north west][inner sep=0.75pt]  [font=\small]  {$\varphi $};
\draw (206,272.4) node [anchor=north west][inner sep=0.75pt]  [font=\small]  {$\psi $};
\draw (41,202.4) node [anchor=north west][inner sep=0.75pt]  [font=\small]  {$\varphi $};
\draw (11,142.4) node [anchor=north west][inner sep=0.75pt]  [font=\small]  {$\varphi $};
\draw (266,202.4) node [anchor=north west][inner sep=0.75pt]  [font=\small]  {$\psi $};
\draw (296,142.4) node [anchor=north west][inner sep=0.75pt]  [font=\small]  {$\psi $};
\draw (216,142.4) node [anchor=north west][inner sep=0.75pt]  [font=\small]  {$\psi $};
\draw (136,142.4) node [anchor=north west][inner sep=0.75pt]  [font=\small]  {$\psi $};
\draw (56,142.4) node [anchor=north west][inner sep=0.75pt]  [font=\small]  {$\psi $};
\draw (201,202.4) node [anchor=north west][inner sep=0.75pt]  [font=\small]  {$\varphi $};
\draw (91,142.4) node [anchor=north west][inner sep=0.75pt]  [font=\small]  {$\varphi $};
\draw (171,142.4) node [anchor=north west][inner sep=0.75pt]  [font=\small]  {$\varphi $};
\draw (251,142.4) node [anchor=north west][inner sep=0.75pt]  [font=\small]  {$\varphi $};
\draw (106,202.4) node [anchor=north west][inner sep=0.75pt]  [font=\small]  {$\psi $};

\end{tikzpicture}
\caption{Picturing $R_{\phi,\psi}(\Sigma) \geq 3$ as a binary tree.}
\label{fig:binary-tree}
\end{figure}
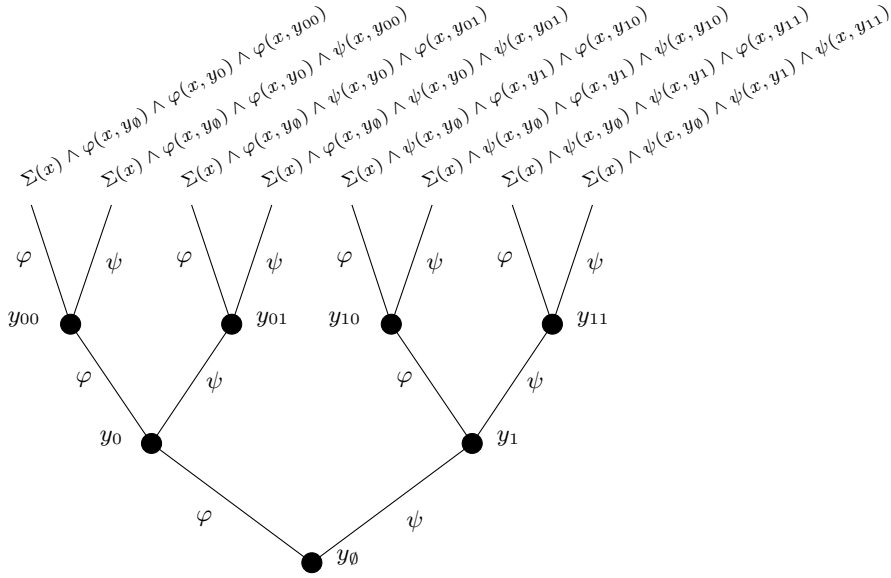
\begin{proof}
We prove (i) by induction: if $R_{\phi, \psi}(\Sigma) \geq \alpha$ then $R_{\phi, \psi}(\Sigma') \geq \alpha$. The base step and limit step are trivial. Now let $R_{\phi, \psi}(\Sigma) \geq \alpha + 1$, so there is $b$ with $R_{\phi, \psi}(\Sigma \cup \{\phi(x, b)\}) \geq \alpha$ and $R_{\phi, \psi}(\Sigma \cup \{\psi(x, b)\}) \geq \alpha$. As $\Sigma(x) \cup \{\phi(x, b)\}$ implies $\Sigma'(x) \cup \{\phi(x, b)\}$ we have by the induction hypothesis that $R_{\phi, \psi}(\Sigma' \cup \{\phi(x, b)\}) \geq \alpha$. Similarly we get $R_{\phi, \psi}(\Sigma' \cup \{\psi(x, b)\}) \geq \alpha$. We conclude that $R_{\phi, \psi}(\Sigma') \geq \alpha + 1$, as required.

For (ii) we first prove that $R_{\phi, \psi}(\Sigma) \geq n$ implies the given set of formulas, by induction on $n$. For $n = 0$ the type just says $\exists x \Sigma(x)$. Now if $R_{\phi, \psi}(\Sigma) \geq n+1$ then there is $b$ such that $R_{\phi, \psi}(\Sigma \cup \{\phi(x, b)\}) \geq n$ and $R_{\phi, \psi}(\Sigma \cup \{\psi(x, b)\}) \geq n$. By the induction hypothesis we then find trees of parameters $(b'_\eta)_{\eta \in 2^{<n}}$ and $(b''_\eta)_{\eta \in 2^{<n}}$. Then we define a new tree of parameters $(b_\eta)_{\eta \in 2^{<n+1}}$ by making $b$ the root, so $b_\emptyset = b$. We then set $b_{0^\smallfrown \eta} = b'_\eta$ and $b_{1^\smallfrown \eta} = b''_\eta$ for all $\eta \in 2^{<n}$. Now $(b_\eta)_{\eta \in 2^{<n+1}}$ is the required realisation of $(y_\eta)_{\eta \in 2^{<n+1}}$.

For the converse of (ii) we again proceed by induction on $n$. The base case is trivial. For the induction step, let $(b_\eta)_{\eta \in 2^{<n+1}}$ realise $(y_\eta)_{\eta \in 2^{<n+1}}$. Then by the induction hypothesis $(b_{0^\smallfrown \eta})_{\eta \in 2^{<n}}$ witnesses $R_{\phi, \psi}(\Sigma \cup \{\phi(x, b_\emptyset)\}) \geq n$, as it realises $(y_\eta)_{\eta \in 2^{<n}}$. Similarly $(b_{1^\smallfrown \eta})_{\eta \in 2^{<n}}$ witnesses $R_{\phi, \psi}(\Sigma \cup \{\psi(x, b_\emptyset)\}) \geq n$. So we conclude that indeed $R_{\phi, \psi}(\Sigma) \geq n+1$.
\end{proof}
\begin{theorem}
\thlabel{stable-characterisations}
The following are equivalent for a theory $T$:
\begin{enumerate}[label=(\roman*)]
\item $T$ is stable,
\item $R_{\phi,\psi}(x=x) < \omega$ for all contradictory formulas $\phi(x, y)$ and $\psi(x, y)$,
\item every type is definable,
\item $T$ is $\lambda$-stable for every $\lambda$ with $\lambda^{|T|} = \lambda$.
\end{enumerate}
\end{theorem}
\begin{proof}
We prove (i) $\Rightarrow$ (ii) $\Rightarrow$ (iii) $\Rightarrow$ (iv), and (iv) $\Rightarrow$ (i) is immediate as there is at least one such $\lambda$ (e.g., $\lambda = 2^{|T|}$).

\underline{(i) $\Rightarrow$ (ii)} We prove the contrapositive. So let $\phi(x, y)$ and $\psi(x, y)$ be contradictory formulas such that $R_{\phi,\psi}(x=x) \geq \omega$. Let $\lambda$ be any cardinal, we will prove that $T$ is not $\lambda$-stable. Let $\mu$ be minimal such that $2^\mu > \lambda$. Then $|2^{<\mu}| \leq \lambda$. Write $\chi_0$ and $\chi_1$ for $\phi$ and $\psi$ respectively. Following \thref{binary-tree-rank-lemma}(ii), we can use compactness to find $(b_\eta)_{\eta \in 2^{< \mu}}$ such that for all $\sigma \in 2^\mu$ set
\[
\Sigma_\sigma(x) = \{ \chi_{\sigma(i)}(x, b_{\sigma|_i}) : i < \mu \}
\]
is consistent. For each such $\sigma$ we thus find a type $p_\sigma(x) \supseteq \Sigma_\sigma(x)$ over $B = (b_\eta)_{\eta \in 2^{< \mu}}$. By construction $p_\sigma \neq p_{\sigma'}$ whenever $\sigma \neq \sigma'$. We conclude that there are more than $2^\mu > \lambda$ types over $B$ in variables $x$ (and $x$ is a finite tuple). At the same time, $|B| \leq \lambda$, and so $T$ is not $\lambda$-stable.

\underline{(ii) $\Rightarrow$ (iii)} Let $p(x)$ be a type over some parameter set $B$ and let $\phi(x, y)$ be any formula. We will show that there is a $\phi$-definition of $p(x)$ over $B$. Let $\psi(x, y)$ be an obstruction of $\phi(x, y)$. By \thref{binary-tree-rank-lemma}(i) we have $R_{\phi,\psi}(p) \leq R_{\phi,\psi}(x = x) < \omega$, and so there is $n_\psi < \omega$ such that $R_{\phi,\psi}(p) = n_\psi$. In particular, $R_{\phi,\psi}(p) \not \geq n_\psi + 1$. So the corresponding set of formulas in \thref{binary-tree-rank-lemma}(ii) is inconsistent and by compactness we find $\chi_\psi(x) \in p(x)$ such that $R_{\phi,\psi}(\chi_\psi) \not \geq n_\psi + 1$. Again, using \thref{binary-tree-rank-lemma}(ii), we let $\theta_\psi(y)$ be a formula equivalent to $R_{\phi,\psi}(\chi_\psi(x) \wedge \phi(x, y)) \geq n_\psi$. Note that this formula does indeed have a free variable $y$ and it has the same parameters as $\chi_\psi$. Set $\d_p \phi(y) = \{ \theta_\psi(y) : \psi \text{ is an obstruction of } \phi \}$, so clearly $|\d_p \phi(y)| \leq |T|$ and it only has parameters in $B$.

We verify that $\d_p \phi(y)$ is indeed a $\phi$-definition of $p(x)$. First assume $\phi(x, b) \in p(x)$. Let $\psi$ be an obstruction of $\phi$. We have $\chi_\psi(x) \wedge \phi(x, b) \in p(x)$, so $R_{\phi, \psi}(\chi_\psi(x) \wedge \phi(x, b)) \geq R_{\phi, \psi}(p) = n_\psi$ and thus $\models \theta_\psi(b)$. As $\psi$ was an arbitrary obstruction we have $\models \d_p \phi(b)$. We prove the contrapositive of the converse. So assume that $\phi(x, b) \not \in p(x)$. Then there must be some $\psi(x, b) \in p(x)$, such that $\psi(x, y)$ is an obstruction of $\phi(x, y)$. So we have $R_{\phi, \psi}(\chi_\psi \wedge \psi(x, b)) \geq R_{\phi, \psi}(p) = n_\psi$. We must thus have $R_{\phi, \psi}(\chi_\psi \wedge \phi(x, b)) < n_\psi$ as otherwise $R_{\phi, \psi}(\chi_\psi) \geq n_\psi + 1$. Hence we have $\not \models \theta_\psi(b)$ and thus $\not \models \d_p \phi(b)$, as required.

\underline{(iii) $\Rightarrow$ (iv)} Let $\lambda$ be such that $\lambda^{|T|} = \lambda$, and note that this implies $\lambda > |T|$. We will prove that $T$ is $\lambda$-stable. Let $B$ be any set of parameters with $|B| \leq \lambda$. There are at most $|B| + |T|$ many formulas over $B$ and so there are at most $(|B| + |T|)^{|T|} \leq \lambda^{|T|} = \lambda$ many sets of formulas of cardinality $\leq |T|$ over $B$. By assumption, every type $p(x)$ over $B$ is definable and is thus fully determined by its $\phi$-definitions, where $\phi(x, y)$ ranges over all formulas. As there are at most $\lambda$ many different possibilities for $\phi$-definitions, we have that there are at most $|T| \times \lambda = \lambda$ many types over $B$, as required.
\end{proof}
\begin{lemma}
\thlabel{definable-types-are-invariant}
Let $p(x)$ be a type over $B$ and suppose that $p(x)$ is definable over $C \subseteq B$. Then $p(x)$ is $C$-invariant.
\end{lemma}
\begin{proof}
Let $\phi(x, y)$ be any formula without parameters and let $b,b' \in B$ with $b \equiv_C b'$. By assumption there is a $\phi$-definition $\d_p \phi(y)$ of $p(x)$ over $C$. Then
\[
\phi(x, b) \in p(x) \quad \Longleftrightarrow \quad
\models \d_p \phi(b) \quad \Longleftrightarrow \quad
\models \d_p \phi(b') \quad \Longleftrightarrow \quad
\phi(x, b') \in p(x),
\]
where the middle equivalence follows from $b \equiv_C b'$.
\end{proof}
\begin{corollary}
\thlabel{definable-types-do-not-divide}
Let $M \supseteq C$ be a positively $(\aleph_0 + |C|)^+$-saturated p.c.\ model. Suppose that $p(x)$ is a type over $M$ that is definable over $C$.  Then $p(x)$ does not divide over $C$.
\end{corollary}
\begin{proof}
Combine \thref{definable-types-are-invariant,invariant-types-do-not-divide}.
\end{proof}
\begin{lemma}
\thlabel{failure-of-simplicity-dividing-chain-of-saturated-models}
Suppose that $T$ is not simple. Then there is a chain $(M_i)_{i < |T|^+}$ of positively $|T|^+$-saturated p.c.\ models and some $a$ such that, for $M = \bigcup_{i < |T|^+} M_i$, the type $\tp(a/M)$ divides over $M_i$ for all $i < |T|^+$.
\end{lemma}
\begin{proof}
The construction is the same as in the proof of (iii) $\Rightarrow$ (v) in \thref{simplicity-equivalences}, with a bit more care taken to end up with the desired chain of models.

Let $\phi(x, y)$ have \TP as witnessed by $\psi(y_1, \ldots, y_k)$ and $(c_\eta)_{\eta \in \omega^{<\omega}}$. Let $\lambda = \beth_{|T|^+}$ and, using compactness, enlarge our tree to $(c_\eta)_{\eta \in (\lambda^+)^{< |T|^+}}$.

We construct some $\sigma \in (\lambda^+)^{< |T|^+}$ by induction on its length, at the same time as a chain $(M_i)_{i < |T|^+}$ of positively $|T|^+$-saturated p.c.\ models, such that for each $\gamma < |T|^+$ we have:
\begin{enumerate}[label=(\arabic*)]
\item $\{c_{\sigma|_i} : i < \gamma\} \subseteq M_\gamma$,
\item $|M_\gamma| < \lambda$,
\item there is an infinite $I_\gamma \subseteq (\lambda^+)$ such that $\{c_{{\sigma|_\gamma}^\frown i} : i \in I_\gamma\}$ all have the same type over $M_\gamma$ and $\sigma(\gamma) \in I_\gamma$.
\end{enumerate}
Suppose that we have constructed $\sigma|_\gamma$ and $(M_i)_{i < \gamma}$, for $\gamma < |T|^+$. Let $M_\gamma$ be any positively $|T|^+$-saturated p.c.\ model of cardinality $< \lambda$ containing $\{c_{\sigma|_i} : i < \gamma\} \cup \bigcup_{i < \gamma} M_i$. As $|\{c_{\sigma|_i} : i < \gamma\} \cup \bigcup_{i < \gamma} M_i| < \lambda$ by the induction hypothesis, such an $M_\gamma$ exists by \thref{building-saturated-model}. There are at most $2^{|M_\gamma|} < \lambda$ many types over $M_\gamma$. So there is infinite $I_\gamma \subseteq \lambda^+$ as in property (3). We finish the construction by letting $\sigma(\gamma)$ be any element (say, the least one) in $I_\gamma$.

Let $a$ be a realisation of $\{ \phi(x, \sigma|_i) : i < |T|^+ \}$, and set $M = \bigcup_{i < |T|^+} M_i$. Fix $\gamma < |T|^+$, it remains to prove that $\tp(a/M)$ divides over $M_\gamma$. Let $I_\gamma$ be as in (3) of the induction hypothesis. Then $c_{\sigma|_{\gamma+1}} \in \{c_{{\sigma|_\gamma}^\frown i} : i \in I_\gamma\}$, and so $c_{{\sigma|_\gamma}^\frown i} \equiv_{M_\gamma} c_{\sigma|_{\gamma+1}}$ for all $i \in I_\gamma$. At the same time, by definition of \TP, we have that $\psi$ holds along $(c_{{\sigma|_\gamma}^\frown i})_{i \in I_\gamma}$. Therefore, $\phi(x, c_{\sigma|_{\gamma+1}})$ $\psi$-divides over $M_\gamma$. As $c_{\sigma|_{\gamma+1}} \in M_{\gamma+1} \subseteq M$ we have by choice of $a$ that $\phi(x, c_{\sigma|_{\gamma+1}}) \in \tp(a/M)$, and we conclude that $\tp(a/M)$ divides over $M_\gamma$.
\end{proof}
\begin{theorem}
\thlabel{stable-implies-simple}
Every stable theory is simple.
\end{theorem}
\begin{proof}
Suppose for a contradiction that $T$ is stable and not simple. Let $a$ and $(M_i)_{i < |T|^+}$ be as in \thref{failure-of-simplicity-dividing-chain-of-saturated-models}. Set $M = \bigcup_{i < |T|^+} M_i$. By \thref{stable-characterisations} $p(x) = \tp(a/M)$ is definable. So for each formula $\phi(x, y)$ there is a $\phi$-definition $\d_p \phi(y)$ of $p(x)$ over $M$. Let $C$ be the union of all parameters mentioned in $\d_p \phi(y)$, as $\phi(x, y)$ ranges over all possible formulas. Then $C \subseteq M$ and $|C| \leq |T|$ as $|\d_p \phi(y)| \leq |T|$ for all $\phi(x, y)$. So there is $i < |T|^+$ such that $C \subseteq M_i$. As $\tp(a/M)$ divides over $M_i$, we have by finite character of dividing that there is some $i < j < |T|^+$ such that $\tp(a/M_j)$ divides over $M_i$. In particular, by \textsc{base monotonicity} of dividing, we have that $\tp(a/M_j)$ divides over $C$. At the same time, $\tp(a/M_j)$ is by construction definable over $C$, and so \thref{definable-types-do-not-divide} implies that it does not divide over $C$. We have arrived at our contradiction and conclude that $T$ must be simple.
\end{proof}

\section{Stationarity}
\label{sec:stationarity}
\begin{definition}
\thlabel{stationary-type}
A \term{stationary type}\index{type!stationary} is a type $p(x) = \tp(a/C)$ that admits exactly one non-dividing extension to any parameter set. That is, for any $B \supseteq C$, there is a type $p'(x) \supseteq p(x)$ over $B$ such that:
\begin{enumerate}[label=(\roman*)]
\item $p'(x)$ does not divide over $C$;
\item for any type $r(x) \supseteq p(x)$ over $B$ that does not divide over $C$ we have $r(x) = p'(x)$.
\end{enumerate}
\end{definition}
We can reformulate the above definition in terms of the $\ind^d$ notation. A type $p(x) = \tp(a/C)$ is stationary if for any $B$ there is some $a'$ with $a' \equiv_C a$ and $a' \ind^d_C B$. Furthermore, for any $B$ and any $a'$ and $a''$ with $a' \equiv_C a'' \equiv_C a$, $a' \ind^d_C B$ and $a'' \ind^d_C B$ we have that $a' \equiv_{CB} a''$.
\begin{theorem}
\thlabel{stable-theory-stationary-iff-lstp}
Assume thickness. If $T$ is a stable theory then $\tp(a/C)$ is stationary if and only if we have for all $a'$ that $a \equiv_C a'$ implies $a \equivls_C a'$. In particular, $\ind^d$ satisfies \textsc{stationarity} in stable theories.
\end{theorem}
\begin{proof}
First we note that $T$ is simple, by \thref{stable-implies-simple}. So we can, and will, use the properties for $\ind^d$ summarised in \thref{kim-pillay}. That being said, the left to right direction actually goes through for any theory $T$.

We first prove the left to right direction, so assume that $\tp(a/C)$ is stationary. Let $q(x) \supseteq p(x)$ be a global non-dividing extension and let $\alpha$ be a realisation of $q(x)$. Let $b, b'$ be any two tuples with $b \equiv_C b'$ and let $\alpha^*$ be such that $\alpha^* b \equiv_C \alpha b'$. As $q$ does not divide over $C$ we have that $\alpha \ind^d_C b b'$. Hence $\alpha \ind^d_C b$ and $\alpha^* \ind^d_C b$. So by stationarity we have that $\alpha \equiv_{Cb} \alpha^*$, and thus
\[
\alpha b \equiv_C \alpha^* b \equiv_C \alpha b'.
\]
As $b$ and $b'$ were arbitrary with $b \equiv_C b'$, we see that $q(x)$ is global $C$-invariant. So $p(x)$ extends to a global $C$-invariant type, and we conclude by \thref{type-extends-to-global-invariant-implies-lstp}.

We prove the contrapositive of the right to left direction. So suppose that there is a non-stationary type $p(x) = \tp(a/C)$ such that $a' \models p$ implies $a' \equivls_C a$. As $p(x)$ is not stationary and every type has non-dividing extensions in a thick simple theory (\thref{simple-thick-implies-full-existence}), there must be two distinct non-dividing extensions of $p(x)$. That is, there is $b$ and distinct extensions $p_0(x) = \tp(a_0/Cb)$ and $p_1(x) = \tp(a_1/Cb)$ of $p(x)$, so that both $p_0(x)$ and $p_1(x)$ do not divide over $C$. We may assume that $b$ is finite.

By \thref{stable-characterisations}(iv) it is enough to prove that $T$ is not $\lambda$-stable for $\lambda = (|C| + 2)^{|T|}$, as this satisfies $\lambda^{|T|} = \lambda$. Let $(b_i)_{i < \lambda}$ be a Morley sequence over $C$ with $b_0 = b$, which exists by simplicity (\thref{morley-sequences-exist}). We will construct types $(q_\eta(x))_{\eta \in 2^{\leq \lambda}}$ by induction on the length of $\eta$, such that for $\eta \in 2^\gamma$ with $\gamma \leq \lambda$:
\begin{enumerate}[label=(\arabic*)]
\item for all $\nu \unlhd \eta$ we have $q_\nu(x) \subseteq q_\eta(x)$,
\item $q_\eta(x)$ is a type over $C (b_i)_{i < \gamma}$,
\item $q_\eta(x) \supseteq p_{\eta(i)}(x, b_i)$ for all $i < \gamma$,
\item $q_\eta(x)$ does not divide over $C$.
\end{enumerate}
For the base case, $\gamma = 0$, we simply set $q_\emptyset(x) = p(x)$. Now assume that we have constructed $(q_\eta(x))_{\eta \in 2^{< \gamma}}$. If $\gamma$ is a limit then for $\eta \in 2^\gamma$ we set $q_\eta(x) = \bigcup_{i < \gamma} q_{\eta|_i}(x)$. Then (1)--(3) are immediate, and (4) follows from \textsc{finite character} of dividing independence.

This leaves the successor step. So suppose that $\gamma = \delta + 1$ and let $\eta \in 2^\gamma$. As $b_\delta \equiv_C b_0 = b$, we find $a'$ with $a' b_\delta \equiv_C a_{\eta(\delta)} b$. Let $a''$ be a realisation of $q_{\eta|_\delta}$. Then $a'' \ind^d_C (b_i)_{i < \delta}$, $a' \ind^d_C b_\delta$ and $b_\delta \ind^d_C (b_i)_{i < \delta}$. The first independence is (4) from the induction hypothesis, the second follows from our choice of $a'$ and the fact that $p_{\eta(\delta)}(x)$ does not divide over $C$ and the third is immediate from $(b_i)_{i < \lambda}$ being a Morley sequence over $C$. Furthermore, we have $a' \equiv_C a \equiv_C a''$, where the second equivalence follows because $p(x) = q_\emptyset(x) \subseteq q_{\eta|_\delta}(x)$. So by assumption $a' \equivls_C a''$. We can thus apply the \textsc{independence theorem} for dividing independence to find $a^*$ with $a^* \ind^d_C (b_i)_{i < \gamma}$, $a^* \equivls_{C b_\delta} a'$ and $a^* \equivls_{C (b_i)_{i < \delta}} a''$. We set $q_\eta(x) = \tp(a^*/C(b_i)_{i < \gamma})$, which immediately takes care of property (2). Properties (4), (3) and (1) follow from the respective properties that the \textsc{independence theorem} gives for $a^*$ (as well as the induction hypothesis).

This finishes the construction. Now for any distinct $\eta, \eta' \in 2^\lambda$, property (3) of the inductive construction guarantees that $q_\eta(x) \neq q_{\eta'}(x)$. So $\{ q_\eta(x) : \eta \in 2^\lambda \}$ is a set of $2^\lambda$ distinct types over $C(b_i)_{i < \lambda}$, while $|C(b_i)_{i < \lambda}| \leq \lambda$. Therefore, $T$ is not $\lambda$-stable.
\end{proof}
\begin{corollary}
\thlabel{stable-stationarity-over-lambda-t-saturated-models}
Assume thickness. If $T$ is stable then any type over any positively $\lambda_T$-saturated p.c.\ model is stationary. If we assume semi-Hausdorffness then any type over any p.c.\ model is stationary.
\end{corollary}
\begin{proof}
By \thref{same-lstp-iff-same-types-over-sequence-of-models} having the same type over a positively $\lambda_T$-saturated p.c.\ model $M$ implies having the same Lascar strong type over $M$, and so the result follows from \thref{stable-theory-stationary-iff-lstp}. If we assume semi-Hausdorffness then we can drop the saturatedness assumption in \thref{same-lstp-iff-same-types-over-sequence-of-models} and hence in this result.
\end{proof}
\begin{example}
\thlabel{non-stationary-type-over-pc-model}
The assumption in \thref{stable-stationarity-over-lambda-t-saturated-models} that the p.c.\ model is positively $\lambda_T$-saturated is necessary. Consider the theory from \thref{thick-not-semi-hausdorff}. We will use the same notation. As that theory has a maximal model $N$ (which is thus the monster model, see \thref{maximal-pc-model-is-monster}), the theory is clearly $|N|$-stable. Furthermore, it is thick. However, the type $p(x) = \tp(a_\omega/M)$ is not stationary. As $N$ is a maximal model, all indiscernible sequences are constant sequences and so no type divides (over any base set). In particular, $p_1(x) = \tp(a_\omega/N)$ and $p_2(x) = \tp(b_\omega/N)$ are two distinct non-dividing extensions of $p(x)$, showing that $p(x)$ is not stationary (cf.\ \thref{non-invariant-type}).
\end{example}
\begin{theorem}
\thlabel{simple-is-stable-iff-stationarity}
Assume thickness. If $T$ is simple then the following are equivalent:
\begin{enumerate}[label=(\roman*)]
\item $T$ is stable;
\item $\ind^d$ satisfies \textsc{stationarity};
\item $\ind^d$ satisfies \textsc{stationarity} over positively $\lambda_T$-saturated p.c.\ models: for any positively $\lambda_T$-saturated p.c.\ model $M$ and any $a, a', b$ with $a \ind^d_M b$, $a' \ind^d_M b$ and $a \equiv_M a'$ we have $a \equiv_{Mb} a'$.
\end{enumerate}
\end{theorem}
The proof of \thref{simple-is-stable-iff-stationarity} rests mainly on the following notion and lemma, that allows us to reduce checking $\lambda$-stability from types over arbitrary parameter sets to types over certain p.c.\ models.
\begin{definition}
\thlabel{lambda-directed-system}
Fix an infinite cardinal $\lambda$. A \emph{$\lambda$-directed poset}\index{directed poset!lambda-directed poset@$\lambda$-directed poset} is a poset $I$ such that any $I' \subseteq I$ with $|I'| < \lambda$ has an upper bound in $I$. A \emph{$\lambda$-directed system}\index{directed system!lambda-directed system@$\lambda$-directed system} is a functor from a $\lambda$-directed poset $I$ into the category of $\L$-structures.
\end{definition}
Note that this generalises \thref{directed-system} to higher cardinals in the sense that a directed poset/system is just an $\omega$-directed poset/system. This then immediately gives rise to the notion of \emph{union of a $\lambda$-directed system} or \emph{$\lambda$-directed union}\index{directed union!lambda-directed union@$\lambda$-directed union} as in that definition, because every $\lambda$-directed poset is in particular directed. However, we will only be interested in the case where all structures in the system are p.c.\ models and live in the monster model, resulting in the following simplified definition.
\begin{definition}
\thlabel{lambda-directed-union}
Let $I$ be a $\lambda$-directed poset and let $(M_i)_{i \in I}$ be a family of p.c.\ submodels of the monster model, such that for all $i \leq j$ in $I$ we have $M_i \subseteq M_j$. We then say that $M = \bigcup_{i \in I} M_i$ is the \emph{$\lambda$-directed union} of the $\lambda$-directed system $(M_i)_{i \in I}$.
\end{definition}
Note that in the above definition $M$ is in particular a p.c.\ model, because it is the directed union of a directed system.
\begin{lemma}
\thlabel{set-contained-in-lambda-directed-union-of-lambda-saturated}
Let $\lambda$ be an infinite cardinal with $\lambda \geq |T|$, and let $\kappa$ be any cardinal such that $\kappa^{<\lambda} = \kappa$ and $\kappa \geq 2^\lambda$. Then any parameter set $B$ with $|B| \leq \kappa$ is contained in a p.c.\ model $M$ that is a $\lambda$-directed union of positively $\lambda$-saturated p.c.\ models, each of cardinality $\leq 2^\lambda$, such that $|M| \leq \kappa$.
\end{lemma}
We also note that $M$ in the above lemma is also positively $\lambda$-saturated. This follows from the general fact that the $\lambda$-directed union of a system of positively $\lambda$-saturated models is itself always positively $\lambda$-saturated. However, we will have no further use for these facts.
\begin{proof}
For $B' \subseteq B$ with $|B'| < \lambda$ we let $M_{B'}$ be some positively $\lambda$-saturated p.c.\ model containing $B'$ with $|M_{B'}| \leq 2^\lambda$. This exists by \thref{building-saturated-model}. Let $[B]^{< \lambda}$ be the set of subsets of $B$ of cardinality $< \lambda$. Then $[B]^{< \lambda}$ is a $\lambda$-directed poset, ordered by inclusion. Thus $M = \bigcup_{B' \in [B]^{< \lambda}} M_{B'}$ is a $\lambda$-directed union. It remains to check that $M$ satisfies the required cardinality bound. Indeed, $|[B]^{<\lambda}| \leq \kappa^{< \lambda} = \kappa$ and so $|M| \leq \kappa \times 2^\lambda = \kappa$.
\end{proof}
\begin{corollary}
\thlabel{reduce-checking-stability-to-certain-models}
Let $\lambda$ be an infinite cardinal with $\lambda \geq |T|$, and let $\kappa$ be any cardinal such that $\kappa^{<\lambda} = \kappa$ and $\kappa \geq 2^\lambda$. Suppose that for every p.c.\ model $M$, which is a $\lambda$-directed union of positively $\lambda$-saturated p.c.\ models of cardinality $\leq 2^\lambda$ and with $|M| \leq \kappa$, we have that $|\S_I(M)| \leq \kappa$ for all finite index sets $I$. Then $T$ is $\kappa$-stable.
\end{corollary}
\begin{proof}
Let $B$ be any parameter set of cardinality at most $\kappa$. By \thref{set-contained-in-lambda-directed-union-of-lambda-saturated} there is a p.c.\ model $M$ containing $B$ that satisfies the description in the statement. Therefore $|\S_I(B)| \leq |\S_I(M)| \leq \kappa$, and we conclude that $T$ is $\kappa$-stable.
\end{proof}
\begin{proof}[{Proof of \thref{simple-is-stable-iff-stationarity}}]
The implications (i) $\Rightarrow$ (ii) $\Rightarrow$ (iii) are \thref{stable-theory-stationary-iff-lstp} and \thref{stable-stationarity-over-lambda-t-saturated-models} respectively. So we prove (iii) $\Rightarrow$ (i).

Let $\kappa = 2^{2^{\lambda_T}}$, then $\kappa^{<\lambda_T} = \kappa$ and $\kappa \geq 2^{\lambda_T}$. So we can, and will, use \thref{reduce-checking-stability-to-certain-models} to prove that $T$ is $\kappa$-stable. So let $M$ be a p.c.\ model of cardinality $\leq \kappa$ that is a $\lambda_T$-directed union of positively $\lambda_T$-saturated p.c.\ models, each of cardinality $\leq 2^{\lambda_T}$. Let $(M_j)_{j \in J}$ be this $\lambda_T$-directed system. For any type $p(x) \in \S_I(M)$, where $I$ is a finite index set, there is by \textsc{local character} some $C \subseteq M$ with $|C| \leq |T|$ such that $p(x)$ does not divide over $C$. For each $c \in C$ there is $j_c \in J$ such that $c \in M_{j_c}$. As $|C| \leq |T| < \lambda_T$, there is an upper bound $j_p \in J$ of $\{j_c : c \in C\}$, and so $C \subseteq M_{j_p}$. By \textsc{base monotonicity} we have that $p(x)$ does not divide over $M_{j_p}$. As $M_{j_p}$ is a positively $\lambda_T$-saturated p.c.\ model, we have by the stationarity assumption that $p(x)$ is completely determined by its restriction to $M_{j_p}$. We have thus shown that the following assignment is an injection:
\begin{align*}
\S_I(M) &\to \coprod_{j \in J} \S_I(M_j), \\
p & \mapsto (j_p, p|_{M_{j_p}}).
\end{align*}
As $|M_j| \leq 2^{\lambda_T}$ for all $j \in J$, we have that $|J| \leq |M|^{2^{\lambda_T}} \leq \kappa^{2^{\lambda_T}} = \kappa$, as well as $|\S_I(M_j)| \leq 2^{|M_j|} \leq 2^{2^{\lambda_T}} = \kappa$. We thus conclude that the above disjoint union has cardinality at most $\kappa$ and so $|\S_I(M)| \leq \kappa$, as required.
\end{proof}
\section{Stable independence}
\label{sec:stable-independence}
\begin{definition}
\thlabel{stationarity-over-fixed-set}
Let $\ind$ be an independence relation and let $C$ be a parameter set. We say that $\ind$ satisfies \emph{stationarity over $C$}\index{stationarity over parameter set} if for all $a, a', b$ we have that if $a \equiv_C a'$ and $a \ind_C b$ and $a' \ind_C b$ then $a \equiv_{Cb} a'$.
\end{definition}
With this terminology \thref{simple-is-stable-iff-stationarity}(iii) can be rephrased as ``$\ind^d$ satisfies stationarity over positively $\lambda_T$-saturated p.c.\ models'' and the \textsc{stationarity} property becomes stationarity over those sets $C$ over which types and Lascar-strong types coincide.
\begin{lemma}
\thlabel{stationary-over-c-implies-independence-theorem-over-c}
Let $\ind$ be an independence relation satisfying \textsc{invariance}, \textsc{monotonicity} and \textsc{extension}. Suppose that $\ind$ satisfies stationarity over some parameter set $C$. Then for any $a, a', b, c$ with $a \ind_C b$, $a' \ind_C c$ and $a \equiv_C a'$ there is $a''$ with $a'' \equiv_{Cb} a$ and $a'' \equiv_{Cc} a'$ such that $a'' \ind_C bc$.
\end{lemma}
\begin{proof}
Let $a, a', b, c$ be as in the statement. By \textsc{extension} there is $c'$ with $c' \equiv_{Cb} c$ and $a \ind_C bc'$. Let $a''$ be such that $ac' \equiv_{Cb} a''c$. We claim that this is the desired $a''$. By \textsc{invariance} we indeed have $a'' \ind_C bc$, and by construction $a'' \equiv_{Cb} a$. For the final equality of types we apply \textsc{monotonicity} to see $a'' \ind_C c$ and together with $a' \ind_C c$ this implies $a'' \equiv_{Cc} a'$ by stationarity over $C$.
\end{proof}
\begin{corollary}
\thlabel{stationarity-implies-independence-theorem}
Let $\ind$ be an independence relation satisfying \textsc{invariance}, \textsc{monotonicity} and \textsc{extension}. If $\ind$ satisfies stationarity over positively $\lambda_T$-saturated p.c.\ models then it satisfies independence theorem over positively $\lambda_T$-saturated p.c.\ models.
\end{corollary}
By ``independence theorem over positively $\lambda_T$-saturated p.c.\ models'' we mean the weakening described in \thref{kim-pillay-weaken-independence-theorem}.
\begin{theorem}[Kim-Pillay style characterisation of stable theories]
\thlabel{stable-kim-pillay}
Assume thickness. A theory $T$ is stable if and only if there is an independence relation $\ind$ satisfying \textsc{invariance}, \textsc{monotonicity}, \textsc{normality}, \textsc{existence}, \textsc{full existence}, \textsc{base monotonicity}, \textsc{extension}, \textsc{symmetry}, \textsc{transitivity}, \textsc{finite character}, \textsc{local character} and \textsc{stationarity}. Furthermore, in this case, $\ind = \ind^d$.
\end{theorem}
\begin{remark}
\thlabel{stable-kim-pillay-weaken-stationarity}
Similar to \thref{kim-pillay-weaken-independence-theorem}, we can strengthen one direction of \thref{stable-kim-pillay}. To conclude stability of $T$ and that $\ind = \ind^d$ we can replace \textsc{stationarity} by stationarity over positively $\lambda_T$-saturated p.c.\ models.
\end{remark}
\begin{proof}
If $T$ is stable then it is in particular simple (\thref{stable-implies-simple}). So from \thref{kim-pillay} we know that $\ind^d$ satisfies all the listed properties, except for \textsc{stationarity}, which follows from \thref{stable-theory-stationary-iff-lstp}. Conversely, suppose that $\ind$ is an arbitrary independence relation, satisfying the listed properties (with stationarity only over positively $\lambda_T$-saturated p.c.\ models, as per \thref{stable-kim-pillay-weaken-stationarity}). By \thref{stationarity-implies-independence-theorem}, $\ind$ satisfies \textsc{independence theorem} over positively $\lambda_T$-saturated p.c.\ models, and so following \thref{kim-pillay-weaken-independence-theorem} we have that \thref{kim-pillay} applies. We thus have that $T$ is simple and $\ind = \ind^d$. As $\ind^d$ satisfies stationarity over positively $\lambda_T$-saturated p.c.\ models we conclude $T$ is stable by \thref{simple-is-stable-iff-stationarity}.
\end{proof}
\begin{remark}
\thlabel{stable-kim-pillay-left-out-properties}
Note that \thref{stable-kim-pillay} does not list \textsc{independence theorem}, as that has been replaced by the stronger \textsc{stationarity} (see \thref{stationary-over-c-implies-independence-theorem-over-c,stationarity-implies-independence-theorem}). Following \thref{kim-pillay-leave-out-properties} we could also have left out \textsc{normality} and \textsc{full existence}, but as in \thref{kim-pillay} we chose to include all basic properties.
\end{remark}

\section{Bibliographic remarks}
\label{sec:bibliographic-remarks-stable}
Shelah already proved results for stable theories in positive logic \cite{shelah_lazy_1975}. Later, Ben-Yaacov established the connection of stability with simplicity in positive logic \cite[Section 2]{ben-yaacov_simplicity_2003}, similar to the contents of this chapter. Just like in Chapter \ref{ch:simple-theories}, we assume thickness in various places to simplify the treatment.

The reader might be familiar with the definition of stability in terms of the order property. This also works in positive logic, after adjusting the definition of the order property similarly to how we adjusted the definition of the tree property (\thref{tree-property}). We chose not to treat the order property, because we have no use for it in these notes and the proof of its equivalence to the other characterisations of stability (\thref{stable-characterisations}) is long and technical. Instead, we just state what is true here and give references to \cite{dmitrieva_dividing_2023}, which is the first place where they appear in print in the modern terminology of positive logic, but is hardly the original source of the arguments.
\begin{definition}[{\cite[Definition 3.5]{dmitrieva_dividing_2023}}]
\thlabel{order-property}
A formula $\phi(x, y)$ has the \term{order property} (\OP)\nomenclature[OP]{\OP}{The order property} if there are sequences $(a_i)_{i < \omega}$ and $(b_i)_{i < \omega}$ and an obstruction $\psi(x, y)$ of $\phi(x, y)$ such that for all $i,j < \omega$, we have
\begin{align*}
&\models \phi(a_i, b_j) &\text{if } i < j, \\
&\models \psi(a_i, b_j) &\text{if } i \geq j.
\end{align*}
\end{definition}
In fact, the usual results for local stability go through in positive logic. That is, we can call a formula stable if that particular formula does not have the order property. This then has many equivalent conditions, such as type counting, which are listed in \cite[Theorem 3.11]{dmitrieva_dividing_2023}. Linking this back to our treatment of stability, we have the following.
\begin{theorem}
\thlabel{stable-iff-no-order-property}
A theory $T$ is stable if and only if no formula has the order property.
\end{theorem}

The observation that the theory in \thref{non-invariant-type} contains a type over a p.c.\ model $M$ that does not extend to a global $M$-invariant type is due to Mennuni.

We defined invariant types in Section \ref{sec:invariant-types} and proved that these exist over p.c.\ models in semi-Hausdorff theories (\thref{semi-hausdorff-invariant-types}). One can also define a notion of \emph{Lascar-invariant type}, which can be proved to exist over p.c.\ models in thick theories. We refer the reader to \cite[Section 3]{dobrowolski_kim-independence_2022} for more details.

Just as was the case for the Kim-Pillay style theorem for simple theories (\thref{kim-pillay}), the stable version (\thref{stable-kim-pillay}) can be pieced together from Ben-Yaacov's work \cite{ben-yaacov_simplicity_2003,ben-yaacov_thickness_2003}. Though, as before, the addition of the thickness assumption allows for a much simpler statement and easier treatment.

\chapter{Examples}
\label{ch:examples}

We consider two classes of examples that can be studied in the framework of positive logic, but generally not in the framework of full first-order logic.

The first class of examples (Section \ref{sec:hyperimaginaries}) describes how hyperimaginaries can be added to any positive theory (and so in particular to a full first-order theory) while preserving the important properties of the theory. So this yields an $(-)^\heq$ construction, analogous to the $(-)^\eq$ construction we know from full first-order logic.

The second class of examples (Section \ref{sec:continuous-logic}) describes how continuous logic can be studied in positive logic. More precisely, it describes how to turn a monster model of a continuous theory into a monster model of a positive theory in such a way that the automorphisms remain the same. In particular, this means that both perspectives agree on types (which correspond to automorphisms orbits). This means that the model theory in both perspectives remains the same, in the sense that both perspectives agree on things such as dividing, stability, simplicity, etc. In fact, because of the explicit description of the construction, one obtains an explicit dictionary to translate model-theoretic definitions and results from positive logic to continuous logic (e.g., \thref{continuous-tree-property,continuous-tp-iff-positive-tp,continuous-kim-pillay}).

\section{Hyperimaginaries}
\label{sec:hyperimaginaries}
We still work in a monster model, which we recall is denoted by $\MM$, as per \thref{monster-convention}. We will extend the monster model with new sorts to a new monster model that contains hyperimaginary elements. To distinguish between these two different structures we will no longer omit them from the notation.
\begin{definition}
\thlabel{type-definable-equivalence-relation}
A \term{type-definable equivalence relation}\index{equivalence relation!type-definable} is a set of formulas $E(x, y)$ without parameters, where $x$ and $y$ are (possibly infinite, but small) tuples of variables, such that $E$ defines an equivalence relation in $\MM$.
\end{definition}
The idea is that we fix a set of such type-definable equivalence relations and add a new sort for each type-definable equivalence relation in this set. Then we will extend the monster model so that the elements of each of these new sorts are the equivalence classes of the corresponding equivalence relation. To capture the interaction between these equivalence classes and their representatives, we also need to add further symbols to the language. However, to stay within a finitary first-order language we cannot add projection function symbols, as is commonly done in the $(-)^\eq$-construction, because the representatives of these equivalence classes can be infinite. We will thus add relation symbols that capture all the possible finitary interactions between the original sort(s) and the new sorts. We then recover (the graphs of) the projection functions as partial types (see \thref{hyperimaginary-projection}).
\begin{definition}
\thlabel{hyperimaginary-language}
Given a set $\E$ of type-definable equivalence relations, we define the \term{hyperimaginary signature} $\L_\E$ as a multisorted extension of $\L$. The sort(s) already in $\L$ will be called the \term{real sort}\emph{(s)}. Then for each $E \in \E$, we add a sort $S_E$, called a \term{hyperimaginary sort}. For a hyperimaginary sort $S_E$, we write $S_{E,r}$ for the tuple of real sorts, matching the sorts of the representatives of the $E$-equivalence classes. For a variable $y$ of sort $S_E$, we write $y_r$ for a tuple of variables of sort $S_{E,r}$.

Furthermore, $\L_\E$ contains the following relation symbols. Let $E_1, \ldots, E_n \in \E$ and let $y_i$ be a variable of sort $S_{E_i}$ for each $1 \leq i \leq n$. Let $\phi(x, y_{1,r}, \ldots, y_{n,r})$ be an $\L$-formula, and write $S_x$ for the sort of $x$. Then we add a relation symbol $R_\phi(x, y_1, \ldots, y_n)$ of sort $S_x \times S_{E_1} \times \ldots \times S_{E_n}$.
\end{definition}
We note that in the above definition, not all variables in $\phi(x, y_{1,r}, \ldots, y_{n,r})$ need to actually appear in the formula. So it is not a problem for the $y_{i,r}$ to be infinite tuples. Similarly, when we write something like $\exists y_r \phi(y_r)$, then we really only quantify over the variables that actually appear in $\phi$, so this is still a finitary formula. We also point out that $x$ can be a tuple of variables, so that $S_x$ is a tuple of the corresponding real sorts (which is what we mean by ``the sort of $x$'').
\begin{definition}
\thlabel{hyperimaginary-extension-of-monster}
Let $\E$ be a set of type-definable equivalence relations.
We extend $\MM$ to an $\L_\E$-structure $\MM^\E$\nomenclature[MME]{$\MM^\E$}{Hyperimaginary extension of the monster model} as follows. For each $E \in \E$ the sort $S_E$ is interpreted as the collection of $E$-equivalence classes in $\MM$, and its elements are called \term{hyperimaginary element}\emph{s}. For $E_1, \ldots, E_n \in \E$ and $\phi(x, y_{1,r}, \ldots, y_{n,r})$, where $y_i$ is a variable of sort $E_i$ for all $1 \leq i \leq n$, we interpret the symbol $R_\phi$ as follows. We let $\MM^\E \models R_\phi(a, c_1, \ldots, c_n)$ if and only if there are representatives $b_1, \ldots, b_n$ of $c_1, \ldots, c_n$ respectively such that $\MM \models \phi(a, b_1, \ldots, b_n)$.

We define the $\L_\E$-theory $T^\E$\nomenclature[T]{$T^\E$}{Hyperimaginary extension of the theory $T$} to be the set of all h-inductive $\L_\E$-sentences that are true in $\MM^\E$, where $T$ is the theory for which $\MM$ is a monster model.
\end{definition}
The current setup allows for flexibility in which hyperimaginary sorts are being added. Often one wants to add hyperimaginary sorts for all type-definable equivalence relations. However, we cannot take $\E$ to be all type-definable equivalence relations, as that would be large with respect to the monster $\MM$. For example, for every $\lambda$ there is the equality relation in variables $(x_i)_{i < \lambda}$. To solve this, we show that we can effectively restrict to hyperimaginaries of length $\leq |T|$.
\begin{lemma}
\thlabel{set-of-formulas-implies-formula-then-small-set-implies-it}
Let $\Sigma(x)$ be any set of formulas, and suppose that $\phi(x)$ is such that $\models \Sigma(a)$ implies $\models \phi(a)$. Then there is $\Sigma'(x) \subseteq \Sigma(x)$ with $|\Sigma'(x)| \leq |T|$ such that $\models \Sigma'(a)$ implies $\models \phi(a)$.
\end{lemma}
\begin{proof}
Let $\psi(x)$ be any obstruction of $\phi(x)$. Then $\Sigma(x) \cup \{\psi(x)\}$ is inconsistent, so by compactness there is $\chi_\psi(x) \in \Sigma(x)$ such that $\chi_\psi(x)$ is an obstruction of $\psi(x)$. Define
\[
\Sigma'(x) = \{ \chi_\psi(x) : \psi(x) \text{ is an obstruction of } \phi(x) \}.
\]
As $\phi(x)$ has $\leq |T|$ obstructions, we see that $|\Sigma'(x)| \leq |T|$. Now suppose for a contradiction that $\models \Sigma'(a)$ and $\not \models \phi(a)$. Then there is an obstruction $\psi(x)$ of $\phi(x)$ such that $\models \psi(a)$. However, $\models \Sigma'(a)$ implies $\models \chi_\psi(a)$, contradicting that $\chi_\psi(x)$ is an obstruction of $\psi(x)$.
\end{proof}
\begin{lemma}
\thlabel{type-definable-equivalence-relation-containing-formula}
Let $E(x, y)$ be a type-definable equivalence relation and let $\phi(x, y) \in E(x, y)$. Then there is a type-definable equivalence relation $E_\phi(x, y)$, such that $\phi(x, y) \in E_\phi(x, y)$, $E_\phi(x, y) \subseteq E(x, y)$ and $|E_\phi(x, y)| \leq |T|$.
\end{lemma}
\begin{proof}
For any $\psi(x, y) \in E(x, y)$ we have that $\models E(a, b) \wedge E(b, c)$ implies $\models \psi(a, c)$. We can thus apply \thref{set-of-formulas-implies-formula-then-small-set-implies-it} to $E(x, y) \cup E(y, z)$ and $\psi(x, z)$ to find $\Sigma'(x, y, z)$ with $|\Sigma'(x, y, z)| \leq |T|$ and $\models \Sigma'(a, b, c)$ implies $\models \psi(a, c)$. Pick $E'_\psi(x, y) \subseteq E(x, y)$ such that $|E'_\psi(x, y)| \leq |T|$ and $\Sigma'(x, y, z) \subseteq E'_\psi(x, y) \cup E'_\psi(y, z)$.

We now inductively define $E^0_\phi(x, y) = \{ \phi(x, y) \}$ and
\[
E^{n+1}_\phi(x, y) = E^n_\phi(x, y) \cup E^n_\phi(y, x) \cup \bigcup_{\psi \in E^n_\phi} E'_\psi(x, y),
\]
and we set $E_\phi(x, y) = \bigcup_{n < \omega} E^n_\phi(x, y)$. We claim that this is the desired set of formulas. By construction we have $\phi(x, y) \in E_\phi(x, y) \subseteq E(x, y)$ and $|E_\phi(x, y)| \leq |T|$, so we are left to check that it defines an equivalence relation. Reflexivity is immediate from $E_\phi(x, y) \subseteq E(x, y)$. Symmetry follows because step $E^n_\phi(y, x) \subseteq E^{n+1}_\phi(x, y) \subseteq E_\phi(x, y)$ for all $n < \omega$. We check transitivity, so suppose that $\models E_\phi(a, b) \wedge E_\phi(b, c)$ and let $\psi(x, y) \in E_\phi(x, y)$ be arbitrary. Then $\psi(x, y) \in E^n_\phi(x, y)$ for some $n < \omega$, and so $E'_\psi(x, y) \subseteq E^{n+1}_\phi(x, y) \subseteq E_\phi(x, y)$. We thus have $\models E'_\psi(a, b) \wedge E'_\psi(b, c)$, which by construction of $E'_\psi(x, y)$ implies that $\models \psi(a, c)$, as required.
\end{proof}
The following corollary can be summarised as ``every hyperimaginary is interdefinable with a set of hyperimaginaries whose representing tuples have length $\leq |T|$''.
\begin{corollary}
\thlabel{hyperimaginary-definable-with-small-hyperimaginaries}
Let $E(x, y)$ be a type-definable equivalence relation. Then there is a set $\{E_i(x_i, y_i)\}_{i \in I}$ of type-definable equivalence relations such that for each $i \in I$ the tuples $x_i$ and $y_i$ have length $\leq |T|$ and are subtuples of $x$ and $y$ respectively, and such that $\models E(a, b)$ if and only if $\models E_i(a, b)$ for all $i \in I$.
\end{corollary}
\begin{proof}
Using \thref{type-definable-equivalence-relation-containing-formula}, we take $\{E_i(x_i, y_i)\}_{i \in I}$ to be an enumeration of $\{ E_\phi(x, y) : \phi(x, y) \in E(x, y) \}$, where we restrict the variables each time to those that are actually mentioned. Then by construction $E_i(x_i, y_i) \subseteq E(x, y)$ for all $i \in I$, so $\models E(a, b)$ implies $\models E_i(a, b)$ for all $i \in I$. For the converse we let $\phi(x, y) \in E(x, y)$ and let $i \in I$ be such that $E_i = E_\phi$. So $\phi \in E_i$, therefore $\models E_i(a, b)$ implies $\models \phi(a, b)$. As $\phi$ was arbitrary, we conclude $\models E(a, b)$.
\end{proof}
\begin{definition}
\thlabel{heq-construction}
Define\nomenclature[heq]{$\heq$}{Hyperimaginaries of length $\leq |T|$}
\[
\heq = \{ E(x, y) \text{ a type-definable equivalence relation} : |x| = |y| \leq |T| \}.
\]
\end{definition}
\begin{convention}
\thlabel{fixed-hyperimaginaries}
For the remainder of this section we fix a set $\E$ of hyperimaginaries.
\end{convention}
Taking $\E = \heq$ we thus will be considering $T^\heq$ and $\MM^\heq$, which by \thref{hyperimaginary-definable-with-small-hyperimaginaries} effectively means we have added all hyperimaginaries.

Many arguments will be easier and smoother if we can treat real elements and hyperimaginary elements notationally in the same way. We formalise this in the following definition. In particular, we extend the notation $S_{E, r}$ and $y_r$ for hyperimaginary sorts and variables to tuples, and we introduce notation for the projection functions that send tuples to the classes they represent. In doing so, one might prefer to think of the real sorts as hyperimaginary sorts themselves (namely modulo the equivalence relation $x = y$).
\begin{definition}
\thlabel{hyperimaginary-projection-notation}
Let $S = (S_i)_{i \in I}$ be a tuple of sorts and let $x = (x_i)_{i \in I}$ be a tuple of corresponding variables. We write $S_r = (S_{i,r})_{i \in I}$ and $x_r = (x_{i,r})_{i \in I}$, where $S_{i,r} = S_i$ and $x_{i,r} = x_i$ whenever $S_i$ is already a real sort.

Given a tuple $a = (a_i)_{i \in I}$ in $\MM$ of sort $S_r$, we write $[a]$ for the corresponding tuple of equivalence classes of sort $S$. That is, $[a] = ([a_i])_{i \in I}$, where $[a_i]$ is defined as follows: if $S_i = S_E$ is a hyperimaginary sort then $[a_i]$ is the $E$-equivalence class represented by $a_i$, otherwise $[a_i] = a_i$.
\end{definition}
With the above definition the description of the new relation symbols in $\L_\E$ becomes simpler: for every tuple of variables $x$ and every $\L$-formula $\phi(x_r)$ we have a relation symbol $R_\phi(x)$, and we have $\MM^\E \models R_\phi([a])$ if and only if there is $b$ such that $[b] = [a]$ and $\MM \models \phi(b)$.
\begin{lemma}
\thlabel{hyperimaginary-formulas-as-real-partial-types}
Let $\phi(x)$ be an $\L_\E$-formula. Then there is a set of $\L$-formulas $\Sigma_\phi(x_r)$ such that $\MM \models \Sigma_\phi(a)$ if and only if $\MM^\E \models \phi([a])$.
\end{lemma}
\begin{proof}
We first assume that $\phi(x)$ is of the form
\[
\exists y \left( \varepsilon(x, y) \wedge \bigwedge_{i \in I} R_{\chi_i}(x, y) \right),
\]
where $\varepsilon(x, y)$ is a conjunction of equalities and $\chi_i(x_r, y_r)$ is an $\L$-formula for each $i \in I$.

We define a set of formulas $\Gamma_\phi$ as follows. For each $i \in I$ we let $x_i$ and $y_i$ be copies of $x_r$ and $y_r$ respectively. We let $E_\varepsilon(x_r, y_r)$ be the union of partial types in $\E$ expressing $\varepsilon([x_r], [y_r])$. Finally, let $E_x(x_r, x'_r)$ and $E_y(y_r, y'_r)$ be the partial types expressing $[x_r] = [x'_r]$ and $[y_r] = [y'_r]$ respectively. Then we let $\Gamma_\phi(x_r, y_r, (x_i)_{i \in I}, (y_i)_{i \in I})$ be
\[
\left\{ \bigwedge_{i \in I} \chi_i(x_i, y_i) \right\} \cup E_\varepsilon(x_r, y_r) \cup \bigcup \{E_x(x_r, x_i) : i \in I\} \cup \bigcup \{E_y(y_r, y_i) : i \in I\}.
\]
Let $\Sigma_\phi(x_r)$ express the following (see also \thref{combine-partial-types})
\[
\exists y_r (x_i)_{i \in I} (y_i)_{i \in I} \Gamma_\phi(x_r, y_r, (x_i)_{i \in I}, (y_i)_{i \in I}).
\]
We claim that $\Sigma_\phi$ is as required. So suppose that $\MM \models \Sigma_\phi(a)$, then we find $b$, $(a_i)_{i \in I}$ and $(b_i)_{i \in I}$ such that $\MM \models \Gamma_\phi(a, b, (a_i)_{i \in I}, (b_i)_{i \in I})$. Let $i \in I$, then by construction $\MM^\E \models R_{\chi_i}([a_i], [b_i])$. As $\MM \models E_x(a, a_i)$ and $\MM \models E_y(b, b_i)$, we have $[a] = [a_i]$ and $[b] = [b_i]$, and so $\MM^\E \models R_{\chi_i}([a], [b])$. Since $\MM \models E_\varepsilon(a, b)$, we also have $\MM^\E \models \varepsilon([a], [b])$, and so indeed $\MM^\E \models \phi([a])$.

For the converse, we assume that $\MM^\E \models \phi([a])$. Then there is $b$ such that $\MM^\E \models \varepsilon([a], [b]) \wedge \bigwedge_{i \in I} R_{\chi_i}([a], [b])$. So $\MM \models E_\varepsilon(a, b)$. Furthermore, for every $i \in I$ there are $a_i$ and $b_i$ such that $\MM \models \chi_i(a_i, b_i)$ with $[a_i] = [a]$ and $[b_i] = [b]$. The latter means that $\MM \models E_x(a, a_i)$ and $\MM \models E_y(b, b_i)$. Hence $\MM \models \Gamma_\phi(a, b, (a_i)_{i \in I}, (b_i)_{i \in I})$ and so $\MM \models \Sigma_\phi(a)$.

We assumed $\phi$ to be of a particular form. Recall that a regular formula is one that is built from atomic formulas, conjunction and existential quantification. Any regular $\L_\E$-formula is logically equivalent to a $\phi$ of the assumed form. This is quickly seen because any atomic $\L$-formula $\chi(x, y)$ is equivalent to $R_\chi(x, y)$. As every positive formula is logically equivalent to a finite disjunction of regular formulas, we only need to define $\Sigma_{\phi_1 \vee \phi_2}(x)$, where $\phi_1(x)$ and $\phi_2(x)$ are of the above form, which can be done using \thref{combine-partial-types}(i).
\end{proof}
\begin{lemma}
\thlabel{hyperimaginary-partial-types-as-real-partial-types}
Let $\Gamma(x)$ be a set of $\L_\E$-formulas. Then there is a set of $\L$-formulas $\Sigma_\Gamma(x_r)$ such that $\MM \models \Sigma_\Gamma(a)$ if and only if $\MM^\E \models \Gamma([a])$.
\end{lemma}
\begin{proof}
Define
\[
\Sigma_\Gamma(x_r) = \bigcup_{\phi \in \Gamma} \Sigma_\phi(x_r),
\]
where $\Sigma_\phi$ is as in \thref{hyperimaginary-formulas-as-real-partial-types}.
\end{proof}
\begin{lemma}
\thlabel{hyperimaginary-types-equal-upgrades-to-real-types-equal}
We have $\tp_{\L_\E}([a]; \MM^\E) = \tp_{\L_\E}([b]; \MM^\E)$ if and only if there is $b'$ such that $\tp_{\L}(a; \MM) = \tp_{\L}(b'; \MM)$ and $[b'] = [b]$. In particular, if $a$ and $b$ consist only of real elements then $\tp_{\L_\E}(a; \MM^\E) = \tp_{\L_\E}(b; \MM^\E)$ if and only if $\tp_{\L}(a; \MM) = \tp_{\L}(b; \MM)$.
\end{lemma}
\begin{proof}
We first prove the left to right direction. Write $p(x_r) = \tp_\L(a; \MM)$ and
\[
\Sigma(x_r) = p(x_r) \cup E(x_r, b),
\]
where $E(x_r, x'_r)$ is the set of formulas expressing $[x_r] = [x'_r]$. It is enough to show that $\Sigma(x_r)$ is finitely satisfiable. Let $\phi(x_r) \in p(x_r)$. Then $\MM^\E \models R_\phi([a])$, so $\MM^\E \models R_\phi([b])$. So there is $b'$ with $[b'] = [b]$, that is $\MM \models E(b', b)$, and $\MM \models \phi(b')$, as required.

Conversely, let $b'$ be such that $\tp_{\L}(a; \MM) = \tp_{\L}(b'; \MM)$ and $[b'] = [b]$. It suffices to prove that $\tp_{\L_\E}([a]; \MM^\E) = \tp_{\L_\E}([b']; \MM^\E)$. Write $q(x) = \tp_{\L_\E}([a]; \MM^\E)$ and let $\Sigma_q(x_r)$ be as in \thref{hyperimaginary-types-equal-upgrades-to-real-types-equal}. Then $\Sigma_q(x_r) \subseteq \tp_{\L}(a; \MM) = \tp_{\L}(b'; \MM)$, and so $\MM^E \models q([b'])$. We thus have $\tp_{\L_\E}([a]; \MM^\E) = q(x) \subseteq \tp_{\L_\E}([b']; \MM^\E)$. A symmetric argument shows that $\tp_{\L_\E}([b']; \MM^\E) \subseteq \tp_{\L_\E}([a]; \MM^\E)$, which concludes our proof.
\end{proof}
\begin{lemma}
\thlabel{hyperimaginary-projection}
For every tuple of variables $x$ there is a partial $\L_\E$-type $\Xi(x_r, x)$ such that $\MM^\E \models \Xi(a, [a'])$ if and only if $[a'] = [a]$.
\end{lemma}
\begin{proof}
Let $E(x_r, x'_r)$ be the set of formulas that expresses $[x_r] = [x'_r]$. Define
\[
\Xi(x_r, x) = \{ R_\varepsilon(x_r, x) : \varepsilon \in E \}.
\]
We prove that $\Xi$ is as required. Suppose that $[a'] = [a]$. Then $\MM \models E(a, a')$. Let $\varepsilon \in E$, then $\MM \models \varepsilon(a, a')$, and so $\MM^\E \models R_\varepsilon(a, [a'])$. Thus indeed $\MM^\E \models \Xi(a, [a'])$.

Conversely, suppose that $\MM^\E \models \Xi(a, [a'])$. Consider the partial type
\[
\Gamma(x_r) = E(a, x_r) \cup E(x_r, a').
\]
For any $\varepsilon(a, x_r) \in E(a, x_r)$ we have $\MM^\E \models R_\varepsilon(a, [a'])$. So there is $a^*$ with $[a^*] = [a']$ and $\MM \models \varepsilon(a, a^*)$. Therefore, $\MM \models \varepsilon(a, a^*) \wedge E(a^*, a')$. We thus see that $\Gamma(x_r)$ is finitely satisfiable, so there is a realisation $a''$, and we conclude that $[a] = [a''] = [a']$.
\end{proof}
\begin{lemma}
\thlabel{automorphism-extends-to-hyperimaginaries-uniquely}
Any automorphism $f: \MM \to \MM$ extends uniquely to an automorphism $f^\E: \MM^\E \to \MM^\E$ by setting $f^\E([a]) = [f(a)]$.
\end{lemma}
\begin{proof}
We first prove that $f^\E$ is well defined and preserves the new relation symbols in $\L_\E$. The fact that $f^\E$ is an automorphism then follows from applying the same construction to $f^{-1}$, yielding an inverse $(f^{-1})^\E$ to $f^\E$ that also preserves the new relation symbols in $\L_\E$.
\begin{itemize}
\item \underline{Well defined}. Let $E \in \E$ and let $a$ and $a'$ be tuples in $\MM$ matching the variables in $E$ such that $[a] = [a']$. Then $\MM \models E(a, a')$, and so $\MM \models E(f(a), f(a'))$. Therefore $[f(a)] = [f(a')]$, showing that $f^\E$ is well defined.
\item \underline{Preservation of relation symbols}. Preservation of equality is just being well defined. Suppose that $\MM^\E \models R_\phi([a])$. Then there is $a'$ such that $[a'] = [a]$ and $\MM \models \phi(a')$. So $\MM \models \phi(f(a'))$ and thus $\MM^\E \models R_\phi([f(a')])$. We conclude by $[f(a')] = f^\E([a']) = f^\E([a])$.
\end{itemize}
We are left to prove that $f^\E$ is unique. Suppose that $g: \MM^\E \to \MM^\E$ is an automorphism extending $f$. For any tuple $a$ we have that $\MM^\E \models \Xi(a, [a])$, by \thref{hyperimaginary-projection}. So since $g$ is an automorphism, we must have $\MM^\E \models \Xi(g(a), g([a]))$. That is, $g([a]) = [g(a)] = [f(a)] = f^\E([a])$, as required.
\end{proof}
\begin{theorem}
\thlabel{hyperimaginary-monster}
The structure $\MM^\E$ is a monster model of $T^\E$.
\end{theorem}
\begin{proof}
We prove that $\MM^\E$ is a p.c.\ model of $T^\E$, and that it is just as saturated and homogeneous as $\MM$. So whatever formalism one prefers for the monster model (see \thref{monster-model}), $\MM^\E$ will be a monster model of $T^\E$ according to the same formalism. Let $\kappa$ (possibly not `small') be such that $\MM$ is positively $\kappa$-saturated and strongly positively $\kappa$-homogeneous. We may assume that $\kappa$ is bigger than the length of any tuple representing a hyperimaginary.

\underline{Positively closed}. We will prove (iii) in \thref{pc-model}. Suppose that $\MM^\E \not \models \phi([a])$. Then $\MM \not \models \Sigma_\phi(a)$, where $\Sigma_\phi$ is as in \thref{hyperimaginary-formulas-as-real-partial-types}. There is thus $\psi(x_r) \in \Sigma_\phi(x_r)$ such that $\MM \not \models \psi(a)$. Because $\MM$ is p.c.\ we find $\chi(x_r)$ with $T \models \neg \exists x_r(\psi(x_r) \wedge \chi(x_r))$ and $\MM \models \chi(a)$. Thus $\MM^\E \models R_\chi([a])$. We will conclude by proving that $T^\E \models \neg \exists x(\phi(x) \wedge R_\chi(x))$. Suppose for a contradiction that there is $[b]$ such that $\MM^\E \models \phi([b]) \wedge R_\chi([b])$. Then there is $b'$ with $[b'] = [b]$ and $\MM \models \chi(b')$. So $\MM^\E \models \phi([b'])$ and thus $\MM \models \Sigma_\phi(b')$. We then get $\MM \models \psi(b') \wedge \chi(b')$, contradicting our choice of $\chi$.

\underline{Saturation}. Let $\Gamma(x, y)$ be a set of $\L_\E$-formulas, and let $[b]$ be a tuple with $|[b]| < \kappa$ such that $\Gamma(x, [b])$ is finitely satisfiable in $\MM^\E$. Let $\Sigma_\Gamma(x_r, y_r)$ be as in \thref{hyperimaginary-partial-types-as-real-partial-types}. By the construction there we have
\[
\Sigma_\Gamma(x_r, y_r) = \bigcup_{\phi \in \Gamma} \Sigma_\phi(x_r, y_r),
\]
where $\Sigma_\phi$ is as in \thref{hyperimaginary-formulas-as-real-partial-types}. So finite satisfiability of $\Gamma(x, [b])$ implies finite satisfiability of $\Sigma_\Gamma(x_r, b)$. Since $|b| < \kappa$ we find a realisation $a$ in $\MM$ with $\MM \models \Sigma_\Gamma(a, b)$, and hence $\MM^\E \models \Gamma([a], [b])$.

\underline{Homogeneity}. Suppose that $\tp_{\L_\E}([a]) = \tp_{\L_\E}([b])$, with $|[a]| = |[b]| < \kappa$. By \thref{hyperimaginary-types-equal-upgrades-to-real-types-equal} there is $b'$ such that $\tp_{\L}(a) = \tp_{\L}(b')$ and $[b'] = [b]$. Let $f: \MM \to \MM$ be an automorphism such that $f(a) = b'$. By \thref{automorphism-extends-to-hyperimaginaries-uniquely} we find an automorphism $f^\E: \MM^\E \to \MM^\E$ such that $f^\E([a]) = [f(a)] = [b'] = [b]$, as required.
\end{proof}
By \thref{hyperimaginary-types-equal-upgrades-to-real-types-equal} the $\L_\E$-type of a tuple of real elements is determined by its $\L$-type. In particular, any sequence of (tuples of) real elements is indiscernible in $\MM$ if and only if it is indiscernible in $\MM^\E$, so there is no ambiguity in the statement (and proof) below.
\begin{lemma}
\thlabel{hyperimaginary-indiscernible-transfer}
A sequence $([a_i])_{i \in I}$ is indiscernible if and only if there is an indiscernible sequence $(b_i)_{i \in I}$ such that $[b_i] = [a_i]$ for all $i \in I$.
\end{lemma}
\begin{proof}
We first prove the left to right direction. By compactness we may assume $I$ to be as long as we need, and at the end we simply forget about the extra elements. We thus find an indiscernible sequence $(a'_i)_{i \in I}$ based on $(a_i)_{i \in I}$. Let $p((x_{i,r})_{i \in I}) = \tp((a'_i)_{i \in I})$ and define
\[
\Gamma((x_{i,r})_{i \in I}) = p((x_{i,r})_{i \in I}) \cup \{ \Xi(x_{i,r}, [a_i]) : i \in I \}.
\]
Then a realisation of $\Gamma$ is precisely what we need. So we prove that $\Gamma$ is finitely satisfiable. That is, for $i_1 < \ldots < i_n \in I$, we will produce a realisation of $\Gamma$ restricted to the variables $x_{i_1,r}, \ldots, x_{i_n, r}$ and parameters $[a_{i_1}], \ldots, [a_{i_n}]$. By construction there are $j_1 < \ldots < j_n \in I$ such that $a'_{i_1} \ldots a'_{i_n} \equiv a_{j_1} \ldots a_{j_n}$. As $[a_{i_1}] \ldots [a_{i_n}] \equiv [a_{j_1}] \ldots [a_{j_n}]$, we have by \thref{hyperimaginary-types-equal-upgrades-to-real-types-equal} that there are $a''_{i_1}, \ldots, a''_{i_n}$ such that $a''_{i_1}, \ldots, a''_{i_n} \equiv a_{j_1} \ldots a_{j_n}$, while also $[a''_{i_k}] = [a_{i_k}]$ for all $1 \leq k \leq n$. So $a''_{i_1}, \ldots, a''_{i_n}$ is the desired realisation of the restriction of $\Gamma$.

For the right to left direction, we have for any $i_1 < \ldots < i_n \in I$ and $j_1 < \ldots < j_n \in I$ that $b_{i_1} \ldots b_{i_n} \equiv b_{j_1} \ldots b_{j_n}$. By \thref{hyperimaginary-types-equal-upgrades-to-real-types-equal} this implies
\[
[a_{i_1}] \ldots [a_{i_n}] = [b_{i_1} \ldots b_{i_n}] \equiv [b_{j_1} \ldots b_{j_n}] = [a_{j_1}] \ldots [a_{j_n}],
\]
and we conclude that $([a_i])_{i \in I}$ is indeed indiscernible.
\end{proof}
\begin{theorem}
\thlabel{hyperimaginary-hausdorff-thick-transfer}
The following properties are preserved under $(-)^\E$:
\begin{itemize}
\item Hausdorff,
\item semi-Hausdorff,
\item thick.
\end{itemize}
That is, if $T$ has the property then $T^\E$ has it as well.
\end{theorem}
\begin{proof}
\underline{Hausdorff}. Let $p(x) = \tp_{\L_\E}([a]; \MM^\E)$ and $q(x) = \tp_{\L_\E}([b]; \MM^\E)$ be distinct types. Let $\Sigma_p(x_r)$ and $\Sigma_q(x_r)$ be as in \thref{hyperimaginary-partial-types-as-real-partial-types}.

Let $s(x_r)$ be any $\L$-type such that $\Sigma_q(x_r) \subseteq s(x_r)$. We will produce formulas $\alpha_s(x_r)$ and $\beta_s(x_r)$ such that $\Sigma_p(x_r) \cup \{\alpha_s(x_r)\}$ is inconsistent, $\beta_s(x_r) \not \in s(x_r)$ and $T \models \forall x_r(\alpha_s(x_r) \vee \beta_s(x_r))$. Let $t(x_r) \supseteq \Sigma_p(x_r)$ be an $\L$-type. Then $t(x_r) \neq s(x_r)$, because $\Sigma_p(x_r) \cup \Sigma_q(x_r)$ is inconsistent. As $T$ is Hausdorff, there are $\chi_t(x_r)$ and $\theta_t(x_r)$ such that $\chi_t(x_r) \not \in t(x_r)$ and $\theta_t(x_r) \not \in s(x_r)$, while $T \models \forall x_r(\chi_t(x_r) \vee \theta_t(x_r))$. Then $\Sigma_p(x_r) \cup \{\chi_t(x_r) : t(x_r) \supseteq \Sigma_p(x_r)\}$ is inconsistent. There are thus $t_1(x_r), \ldots, t_n(x_r)$ such that $\Sigma_p(x_r) \cup \{\chi_{t_1}(x_r) \wedge \ldots \wedge \chi_{t_n}(x_r)\}$ is inconsistent. We can now take $\alpha_s(x_r)$ to be $\chi_{t_1}(x_r) \wedge \ldots \wedge \chi_{t_n}(x_r)$ and $\beta_s(x_r)$ to be $\theta_{t_1}(x_r) \vee \ldots \vee \theta_{t_n}(x_r)$.

By construction $\Sigma_q(x_r) \cup \{\beta_s(x_r) : s(x_r) \supseteq \Sigma_q(x_r)\}$ is inconsistent. So there are $s_1(x_r), \ldots, s_k(x_r)$ such that $\Sigma_q(x_r) \cup \{\beta_{s_1}(x_r) \wedge \ldots \wedge \beta_{s_n}(x_r)\}$ is inconsistent. Let $\beta(x_r)$ be the formula $\beta_{s_1}(x_r) \wedge \ldots \wedge \beta_{s_k}(x_r)$ and let $\alpha(x_r)$ be the formula $\alpha_{s_1}(x_r) \vee \ldots \vee \alpha_{s_k}(x_r)$. Then $\Sigma_p(x_r) \cup \{\alpha(x_r)\}$ and $\Sigma_q(x_r) \cup \{\beta(x_r)\}$ are both inconsistent and $T \models \forall x_r(\alpha(x_r) \vee \beta(x_r))$.

We now consider the formulas $R_\alpha(x)$ and $R_\beta(x)$. By construction we have $T^\E \models \forall x(R_\alpha(x) \vee R_\beta(x))$. We claim that $R_\alpha(x) \not \in p(x)$. Suppose for a contradiction that $\MM^\E \models R_\alpha([a])$. Then there is $a'$ with $[a'] = [a]$ and $\MM \models \alpha(a')$. At the same time $\MM^\E \models p([a'])$ and so $\MM \models \Sigma_p(a')$, contradicting the inconsistency of $\Sigma_p(x_r) \cup \{\alpha(x_r)\}$. Analogously we have that $R_\beta(x) \not \in q(x)$, and so we conclude that $T^\E$ is Hausdorff.\\
\\
\noindent \underline{Semi-Hausdorff}. Suppose that equality of $\L$-types is type-definable by a partial $\L$-type $\Omega$. Then for any tuple $x$ of variables in $\L_\E$, we let $\Omega^\E(x, x')$ be the set of $\L_\E$-formulas that expresses
\[
\exists x_r x'_r(\Xi(x_r, x) \wedge \Xi(x'_r, x') \wedge \Omega(x_r, x'_r)).
\]
We claim that $\Omega^\E$ expresses equality of $\L_\E$-types.

If $\MM^\E \models \Omega^\E([a], [b])$ then this is saying that there are $a'$ and $b'$ such that $[a'] = [a]$, $[b'] = [b]$ and $\MM \models \Omega(a', b')$. So we have $a' \equiv b'$ and hence $[a] = [a'] \equiv [b'] = [b]$ by \thref{hyperimaginary-types-equal-upgrades-to-real-types-equal}.

Conversely, suppose that $[a] \equiv [b]$. Then by \thref{hyperimaginary-types-equal-upgrades-to-real-types-equal} there is $b'$ such that $[b'] = [b]$ and $a \equiv b'$. In particular $\MM \models \Xi(a, [a]) \wedge \Xi(b', [b]) \wedge \Omega(a, b')$, as required.\\
\\
\noindent \underline{Thick}. Let $\Theta$ express indiscernibility of a sequence of (tuples of) real elements. Then
\[
\exists (x_{i,r})_{i < \omega} \left( \Theta((x_{i,r})_{i < \omega}) \wedge \bigwedge_{i < \omega} \Xi(x_{i,r}, x_i) \right)
\]
expresses indiscernibility of $(x_i)_{i < \omega}$ in $\MM^\E$. Here we use that a sequence in $\MM^\E$ is indiscernible if and only if there is an indiscernible sequence of representatives, see \thref{hyperimaginary-indiscernible-transfer}.
\end{proof}
\begin{example}
\thlabel{hyperimaginary-does-not-preserve-boolean}
Being Boolean is not preserved when moving from $T$ to $T^\E$. For example, start with any theory $T$ in full first-order logic (considered as a positive theory through Morleyisation, see \thref{full-first-order-logic-as-special-case}). Then $T$ is by construction Boolean. Assume that $T$ is complete with an infinite model, and consider the type-definable equivalence relation
\[
E(x, y) = \{\phi(x) \leftrightarrow \phi(y) : \phi \text{ is a formula in } T\}.
\]
Then $E$ expresses that $x$ and $y$ have the same type. Set $\E = \{E\}$, then in $\MM^\E$ the elements of the sort $S_E$ are exactly the types (with free variables $x$) in $\MM$. If there are infinitely many types with free variables $x$ (e.g., $x$ is an infinite tuple of variables in the sort whose underlying set is infinite) then $\MM^\E$ has a bounded infinite definable set. So $T^\E$ cannot be Boolean: inequality on $S_E$ is not equivalent to a positive formula (modulo $T^\E$), because otherwise by compactness we would find arbitrarily many elements in $S_E$.

Another example is to consider the full first-order theory of the real numbers as an ordered field $T = \Th(\R; 0, 1, +, \cdot, -, \leq)$. This is also known as the theory of \emph{real closed fields}. Consider the type-definable equivalence relation
\[
E(x, y) = \{ -1/n < x - y < 1/n : n < \omega \}.
\]
Then $E$ expresses that $x$ and $y$ are infinitesimally close. Set $\E = \{E\}$, then the real unit interval $[0, 1]$ is definable in the sort $S_E$ in $\MM^\E$ by the formula $[0] \leq x \wedge x \leq [1]$. Here $[0]$ and $[1]$ are the $E$-equivalence classes of $0$ and $1$ respectively, and the relation $\leq$ is an abbreviation for $R_{x_r \leq y_r}(x, y)$.
\end{example}
Adding hyperimaginaries does not impact many model-theoretic properties of the theory, such as simplicity or stability. We will prove this for those two here, based on the characterisations in these notes (\thref{simplicity} and \thref{stability} respectively).
\begin{theorem}
\thlabel{hyperimaginaries-respect-simplicity}
The theory $T$ is simple if and only if $T^\E$ is simple.
\end{theorem}
\begin{proof}
We will prove that $T$ is \NTP if and only if $T^\E$ is \NTP. The left to right direction is trivial. We prove the converse.

Suppose that $\phi(x, y)$ has \TP in $T^\E$, as witnessed by parameters $([a_\eta])_{\eta \in \omega^{<\omega}}$ and some formula $\psi(y_1, \ldots, y_k)$. Let $\Sigma_\phi(x_r, y_r)$ and $\Sigma_\psi(y_{1,r}, \ldots, y_{k,r})$ be as in \thref{hyperimaginary-formulas-as-real-partial-types}. Then
\[
\Sigma_\phi(x_r, y_{1,r}) \cup \ldots \cup \Sigma_\phi(x_r, y_{1,k}) \cup \Sigma_\psi(y_{1,r}, \ldots, y_{1,k})
\]
is inconsistent. Hence there are $\phi'(x_r, y_r) \in \Sigma_\phi$ and $\psi'(y_{1,r}, \ldots, y_{k,r}) \in \Sigma_\psi$ such that
\[
\phi'(x_r, y_{1,r}) \wedge \ldots \wedge \phi'(x_r, y_{k,r}) \wedge \psi'(y_{1,r}, \ldots, y_{k,r})
\]
is inconsistent modulo $T^\E$. As the above is an $\L$-formula, it is also inconsistent modulo $T$. We claim that $\phi'(x_r, y_r)$ has \TP, as witnessed by $(a_\eta)_{\eta \in \omega^{<\omega}}$ and $\psi'(y_1, \ldots, y_k)$. We check \thref{tree-property}.
\begin{enumerate}[label=(\roman*)]
\item Let $\sigma \in \omega^\omega$. Then $\{\phi(x, [a_{\sigma|_n}]) : n < \omega\}$ is consistent. So there is $[b]$ such that $\MM^\E \models \phi([b], [a_{\sigma|_n}])$ for all $n < \omega$. That is, we have $\MM \models \Sigma_\phi(b, a_{\sigma|_n})$ for all $n < \omega$. In particular, $\{\phi'(x, a_{\sigma|_n}) : n < \omega\}$ is consistent.
\item Let $\eta \in \omega^{<\omega}$ and let $i_1 < \ldots < i_k < \omega$. Then $\MM^\E \models \psi([a_{\eta^\frown i_1}], \ldots, [a_{\eta^\frown i_k}])$, so $\MM \models \Sigma_\psi(a_{\eta^\frown i_1}, \ldots, a_{\eta^\frown i_k})$ and in particular $\MM \models \psi'(a_{\eta^\frown i_1}, \ldots, a_{\eta^\frown i_k})$.
\end{enumerate}
\end{proof}
\begin{theorem}
\thlabel{hyperimaginaries-respect-stability}
The theory $T$ is stable if and only if $T^\E$ is stable.
\end{theorem}
\begin{proof}
If $T^\E$ is $\lambda$-stable then $T$ is $\lambda$-stable by \thref{hyperimaginary-types-equal-upgrades-to-real-types-equal}. For the converse we let $\lambda$ be an upper bound for the lengths of the tuples of variables appearing in the equivalence relations in $\E$. Note that $\lambda$ is still small compared to the monster. We will prove that if $T$ is $\lambda$-stable then $T^\E$ is $\lambda$-stable.

Let $B$ be a set of parameters from $\MM^\E$ with $|B| \leq \lambda$ and let $x$ be a finite tuple of hyperimaginary variables. Write $\S_x(B)$ for the set of types over $B$ in free variables $x$. We need to show that $|\S_x(B)| \leq \lambda$. Enumerate $B$ as a tuple $[b]$. Then $[b]$ has length $\leq \lambda$ and since every hyperimaginary is represented by at most $\lambda$ elements, we have $|b| \leq \lambda$. Let $\S_{x_r}(b)$ be the set of $\L$-types over $b$ in free variables $x_r$. For every finite subtuple $x' \subseteq x_r$ we have by assumption that there are at most $\lambda$ many $\L$-types over $b$ in free variables $x'$. Since types are determined by their finite restrictions, we have that $|\S_{x_r}(b)| \leq |x_r|^{<\omega} \times \lambda \leq \lambda^{<\omega} \times \lambda = \lambda$. Define a map $\S_{x_r}(b) \to \S_x(B)$ by
\[
\tp_\L(a/b) \mapsto \tp_{\L_\E}([a]/[b]),
\]
which is well-defined by \thref{hyperimaginary-types-equal-upgrades-to-real-types-equal}. By construction and saturation of $\MM^\E$ this map is surjective and so $|\S_x(B)| \leq \lambda$, as required.
\end{proof}
\section{Continuous logic}
\label{sec:continuous-logic}
In this section we will see how continuous logic can be studied using positive logic. Unlike the situation with full first-order logic, positive logic is not a direct generalisation of continuous logic (see also \thref{continuous-translation-needs-saturation}). However, for many abstract model-theoretic purposes, positive logic is more general and developing the abstract theory in the generality of positive logic allows us to immediately apply it to continuous logic.

The framework for continuous logic that we will consider is that of \cite{ben-yaacov_model_2008}. We will use the notation and terminology from there, which we will assume the reader to be familiar with. We just note that what they call a \emph{$\kappa$-universal domain} is what we call a monster model (at least, when $\kappa$ is bigger than all ``small'' cardinals, see Section \ref{sec:monster-models}), and we will omit the $\kappa$ from the notation. For simplicity of notation we will assume that all bounded intervals in our continuous logic are simply $[0, 1]$.
\begin{definition}
\thlabel{continuous-to-positive-translation}
Let $\U$ be a universal domain of some continuous theory $T$ in some metric signature $\L$. We define $\L_\pos$\nomenclature[Lpos]{$\L_\pos$}{Positive version of the metric signature $\L$} to be the following relational signature (in the first-order sense). For every $\L$-formula $\phi(x)$ we introduce a relation symbol $R_\phi(x)$ of the same arity. We then turn $\U$ into an $\L_\pos$-structure $\MM_\pos$\nomenclature[MMpos]{$\MM_\pos$}{Positive monster obtained from a continuous theory} by interpreting $R_\phi$ as the set $\{a \in \U : \phi(a) = 0\}$. Let $T_\pos$ be the positive theory of $\MM_\pos$. That is, $T_\pos$\nomenclature[Tpos]{$T_\pos$}{Positive version of the continuous theory $T$} is the set of all h-inductive sentences in $\L_\pos$ that are true in $\MM_\pos$.
\end{definition}
\begin{convention}
\thlabel{fixed-continuous-universal-domain}
For the remainder of this section, $\U$ is a universal domain of some fixed continuous theory $T$ in some metric signature $\L$, and we let $\L_\pos$, $\MM_\pos$ and $T_\pos$ be as in \thref{continuous-to-positive-translation}.
\end{convention}
\begin{remark}
\thlabel{continuous-bounded-formula}
We write $\dotminus$ for the truncated subtraction operation $[0, 1]^2 \to [0,1]$. That is, for $r, s \in [0, 1]$
\[
r \dotminus s = \begin{cases}
r - s & \text{if } r \geq s, \\
0 & \text{else.}
\end{cases}
\]
Then given any $\L$-formula $\phi(x)$ and any $r \in [0, 1]$ we get an $\L$-formula $\varphi(x) \dotminus r$. Then $R_{\phi \dotminus r}$ is interpreted in $\MM_\pos$ as
\[
\{a \in \U : \phi(a) \leq r \}.
\]
We will thus use the notation $R_{\phi \leq r}$ for $R_{\phi \dotminus r}$. Similarly, we write $R_{\phi \geq r}$ for $R_{r \dotminus \phi}$.
\end{remark}
\begin{proposition}
\thlabel{continuous-logic-formula-translation}
Modulo $T_\pos$, we have the following equivalences of formulas:
\begin{enumerate}[label=(\roman*)]
\item $\bot$ is equivalent to $R_r$, where $r$ is the $\L$-formula with constant value $r$, for any $r > 0$;
\item $x = y$ is equivalent to $R_d(x, y)$, where $d(x, y)$ is the metric on $\U$;
\item $R_\phi(x) \wedge R_\psi(x)$ is equivalent to $R_{\max(\phi, \psi)}(x)$;
\item $R_\phi(x) \vee R_\psi(x)$ is equivalent to $R_{\min(\phi, \psi)}(x)$;
\item $\exists y R_\phi(x, y)$ is equivalent to $R_{\inf_y \phi}(x)$.
\end{enumerate}
In particular, every $\L_\pos$-formula is equivalent to $R_\phi$ for some $\L$-formula $\phi$.
\end{proposition}
\begin{proof}
Items (i)--(iv) follow immediately from the definitions. The final claim follows by induction on the construction of $\L_\pos$-formulas, using (i) and (ii) for the base case and using (iii)--(v) for the inductive steps. So we are left to prove (v).

Suppose that $\MM_\pos \models R_{\inf_y \phi}(a)$. Then $\U \models \inf_y \phi(a, y) = 0$ and so the set $\L$-conditions $\Sigma(y) = \{ \phi(a, y) \dotminus \frac{1}{n} = 0 : 1 \leq n < \omega \}$ is finitely satisfiable in $\U$. By saturation of $\U$ there must be $b \in \U$ such that $\U \models \Sigma(b)$, which is exactly saying that $\U \models \phi(a, b) = 0$. That is, $\MM_\pos \models R_\phi(a, b)$, and so $\MM_\pos \models \exists y R_\phi(a, y)$.

Conversely, suppose that there is $b \in \MM_\pos$ such that $\MM_\pos \models R_\phi(a, b)$. Then $\U \models \phi(a, b) = 0$, and so $\U \models \inf_y \phi(a, y) = 0$. We conclude that $\MM_\pos \models R_{\inf_y \phi}(a)$, as required.
\end{proof}
\begin{theorem}
\thlabel{continuous-universal-domain-yields-positive-monster}
We have the following properties for $\MM_\pos$ and $T_\pos$.
\begin{enumerate}[label=(\roman*)]
\item The structure $\MM_\pos$ is a monster model of $T_\pos$.
\item The structures $\MM_\pos$ and $\U$ have the same automorphisms.
\item The theory $T_\pos$ is Hausdorff.
\end{enumerate}
\end{theorem}
\begin{proof}
\underline{(i)} We first prove that $\MM_\pos$ is a p.c.\ model of $T_\pos$. Suppose that $\MM_\pos \not \models \theta(a)$ for some $\L_\pos$-formula $\theta(x)$ and some $a \in \MM_\pos$. By \thref{continuous-logic-formula-translation} there is an $\L$-formula $\phi(x)$ such that $\theta(x)$ is equivalent to $R_\phi(x)$. Set $r = \phi(a)$, then $r \neq 0$ because $\MM_\pos \not \models R_\phi(a)$. Hence $0 < r/2 < r$ and so $R_{\phi \geq r/2}(\MM_\pos) \cap R_\phi(\MM_\pos) = \emptyset$. Therefore $R_{\phi \geq r/2}(x)$ is an obstruction of $R_\phi(x)$ modulo $T_\pos$, while $\MM_\pos \models R_{\phi \geq r/2}(a)$. We thus conclude that $\MM_\pos$ is a p.c.\ model of $T_\pos$.

Next we check that $\MM_\pos$ is as saturated as $\U$. Let $\Sigma(x)$ be a small set of $\L_\pos$-formulas with parameters in $\MM_\pos$, which is finitely satisfiable in $\MM_\pos$. By \thref{continuous-logic-formula-translation} $\Sigma(x)$ is (equivalent to) a set of relation symbols of the form $R_\phi(x)$, where $\phi(x)$ is an $\L$-formula. Define the following set of $\L$-conditions
\[
\Sigma'(x) = \{\phi(x) = 0 : \phi(x) \text{ is an $\L$-formula s.t.\ } R_\phi(x) \in \Sigma(x) \}.
\]
Then $\Sigma'(x)$ is finitely satisfiable in $\U$ and so by saturation there is $a \in \U$ with $\U \models \Sigma'(a)$. By construction $\MM_\pos \models \Sigma(a)$, as required.

Finally, we check that $\MM_\pos$ is as homogeneous as $\U$. Let $a$ and $b$ be small tuples in $\MM_\pos$ be such that $\tp(a; \MM_\pos) = \tp(b; \MM_\pos)$. Then they satisfy the same set of $\L_\pos$-relation symbols, which is precisely saying that they satisfy the same set of $\L$-conditions, and so by homogeneity of $\U$ there is an automorphism $f$ of $\U$ that sends $a$ to $b$. By (ii), the proof of which stands on itself, $f$ is also an automorphism of $\MM_\pos$.

\underline{(ii)} Let $f$ be an automorphism of the underlying set of $\U$ and $\MM_\pos$. First, suppose that $f$ is an automorphism of $\U$. Then $f$ preserves null-sets of formulas set-wise. So it respects the relation symbols in $\L_\pos$, and it is thus an automorphism of $\MM_\pos$. Conversely, suppose that $f$ is an automorphism of $\MM_\pos$. Let $\phi(x)$ be an $\L$-formula, let $a \in \U$ and set $r = \phi(a)$. Then $\MM_\pos \models R_{\phi \leq r}(a)$ and $\MM_\pos \not \models R_{\phi \leq s}(a)$ for all $s < r$. Hence $\MM_\pos \models R_{\phi \leq r}(f(a))$ and $\MM_\pos \not \models R_{\phi \leq s}(f(a))$ for all $s < r$. This says exactly that $\phi(f(a)) \leq r$ and $\phi(f(a)) > s$ for all $s < r$. So $\phi(f(a)) = r$, and we conclude that $f$ is an automorphism of $\U$.

\underline{(iii)} Let $p(x)$ and $q(x)$ be two distinct types. As $\MM_\pos$ is a monster model (see item (i)), there are realisations $a \in \MM_\pos$ and $b \in \MM_\pos$ of $p$ and $q$ respectively. Let $\theta(x) \in p(x)$ such that $\theta(x) \not \in q(x)$. By \thref{continuous-logic-formula-translation}, $\theta(x)$ is (equivalent to) $R_\phi(x)$ for some $\L$-formula $\phi(x)$. Set $r = \phi(b)$, so $r > 0$. Clearly we have $T_\pos \models \forall x(R_{\phi \leq r/2}(x) \vee R_{\phi \geq r/2}(x))$, and by construction $R_{\phi \leq r/2}(x) \not \in q(x)$ while $R_{\phi \geq r/2}(x) \not \in p(x)$. We conclude that $T_\pos$ is Hausdorff.
\end{proof}
The fact that $\MM_\pos$ and $\U$ have the same automorphisms means that they agree on what a type is, and so they enjoy the same model-theoretic properties. For example, one is simple if and only if the other is (see also \thref{continuous-tp-iff-positive-tp,continuous-kim-pillay}).
\begin{remark}
\thlabel{continuous-translation-needs-saturation}
We have made essential use of the fact that $\U$ is very saturated. In particular, in proving that $\exists y R_\phi(x, y)$ is equivalent to $R_{\inf_y \phi}(x)$ modulo $T_\pos$ (i.e., \thref{continuous-logic-formula-translation}(v)). Without saturation we still get that $\exists y R_\phi(x, y)$ implies $R_{\inf_y \phi}(x)$, but for the other direction we only get values for $y$ that take $\phi(x, y)$ arbitrarily close to $0$, but we might never actually reach $0$.

Of course, to make the argument in \thref{continuous-logic-formula-translation}(v) work, being $\omega$-saturated would be enough. So we could prove that every $\omega$-saturated metric model of $T$ becomes a positively $\omega$-saturated p.c.\ model of $T_\pos$.

We have not considered the converse. Fully investigating this is beyond our scope, but we briefly discuss the obstacles and possibilities here. Firstly, any p.c.\ model $M$ of $T_\pos$ can quickly be turned it into a metric structure because the exact values of the metric, function symbols and relation symbols are captured by the positive formulas. For example, for elements $a$ and $b$ their distance will be the smallest $r$ such that $M \models R_{d \leq r}(a, b)$, where $d(x, y)$ is the metric symbol. The theory $T_\pos$ captures that all symbols behave as required (e.g., $d$ is a metric and the function symbols all have the correct modulus of uniform continuity, etc.). The only issue is that the underlying metric space of a metric structure must be complete. This will obviously generally fail when starting with an arbitrary p.c.\ model and can be fixed by requiring $M$ to be $\omega_1$-saturated. Indeed, given a Cauchy sequence $(a_n)_{n < \omega}$ in $M$ one easily writes down a set of formulas $\Sigma(x)$ with parameters $(a_n)_{n < \omega}$ that expresses that $x$ is the limit of $(a_n)_{n < \omega}$. The fact that $(a_n)_{n < \omega}$ is a Cauchy sequence means that $\Sigma(x)$ is finitely satisfiable in $M$ (in fact, by elements from $(a_n)_{n < \omega}$). By saturation $\Sigma(x)$ is satisfiable in $M$, which means that $(a_n)_{n < \omega}$ has a limit in $M$.

In summary, in the above we have essentially shown the following. For any infinite $\kappa$, a $\kappa$-saturated metric model of $T$ yields a positively $\kappa$-saturated p.c.\ model of $T_\pos$. At the same time, we have sketched a proof of: if $\kappa \geq \omega_1$ then every positively $\kappa$-saturated p.c.\ model of $T_\pos$ yields a $\kappa$-saturated metric model of $T$.
\end{remark}
\begin{example}
\thlabel{continuous-logic-not-boolean}
While $T_\pos$ is always Hausdorff, this is the best we can do. That is, $T_\pos$ will generally not be Boolean. For example, consider $[0, 1]$ with the Euclidean metric as a metric structure (so $\L$ is the empty language, i.e., we only have the symbol $d$ for the metric). Let $T$ be the continuous theory of $[0, 1]$ and let $\U$ be the corresponding universal domain. Suppose for a contradiction that $T_\pos$ is Boolean. Then we can consider the following set of $\L_\pos$-formulas:
\[
\Sigma(x) = \{x \neq 0\} \cup \{R_{d \leq 1/n}(x, 0) : 1 \leq n < \omega\}.
\]
Clearly $\Sigma(x)$ is finitely satisfiable, and so there is $a \in \MM_\pos$ that realises $\Sigma(x)$. However, for such $a$ we have both $a \neq 0$ and $d(a, 0) = 0$, a contradiction. So we conclude that $T_\pos$ cannot be Boolean.
\end{example}
In the proof that $\MM_\pos$ is a p.c.\ model we have already seen that $\MM_\pos \not \models R_\phi(a)$ is always witnessed by an obstruction of the form $R_{\phi \geq \varepsilon}(x)$ for some $\varepsilon > 0$. In fact, these are essentially the only possible obstructions, as is made precise below.
\begin{proposition}
\thlabel{continuous-negations}
Let $\theta(x)$ be an $\L_\pos$-formula and let $\phi(x)$ be an $\L$-formula, such that $\theta(x)$ is an obstruction of $R_\phi(x)$ modulo $T_\pos$. Then there is $\varepsilon > 0$ such that $T_\pos \models \forall x(\theta(x) \to R_{\phi \geq \varepsilon}(x))$.
\end{proposition}
\begin{proof}
By \thref{continuous-logic-formula-translation} there is an $\L$-formula $\psi(x)$ such that $\theta(x)$ is equivalent to $R_\psi(x)$ modulo $T_\pos$. Consider the set of $\L$-conditions
\[
\Sigma(x) = \{\psi(x) = 0\} \cup \{\phi(x) \dotminus \frac{1}{n} = 0 : 1 \leq n < \omega\}.
\]
Then $\Sigma(x)$ must be unsatisfiable in $\U$, as a realisation $a$ would satisfy both $\U \models \psi(a) = 0$ and $\U \models \phi(a) = 0$, contradicting that $R_\psi(x)$ is an obstruction of $R_\phi(x)$ modulo $T_\pos$. There is thus some $1 \leq n < \omega$ such that $\{\psi(x) = 0, \phi(x) \dotminus \frac{1}{n} = 0\}$ is unsatisfiable in $\U$. That is, for all $a \in \U$ we have that if $\psi(a) = 0$ then $\phi(a) > \frac{1}{n}$. In other words, $R_\psi(x)$ implies $R_{\phi \geq 1/n}(x)$ modulo $T_\pos$, as required.
\end{proof}
As an example of how results in positive logic can be applied to continuous logic, using the above translation, we will treat simplicity in continuous logic (i.e., the main results from Chapter \ref{ch:simple-theories}).
\begin{definition}
\thlabel{continuous-tree-property}
Let $k \geq 2$ be a natural number. An $\L$-formula $\phi(x, y)$ is said to have the \emph{$k$-tree property} ($k$-\TP) if there are $(a_\eta)_{\eta \in \omega^{<\omega}}$ in $\U$ and some $\varepsilon > 0$ such that:
\begin{enumerate}[label=(\roman*)]
\item for all $\sigma \in \omega^\omega$ the set $\{ \phi(x, a_{\sigma|_n}) : n < \omega \}$ is consistent,
\item for all $\eta \in \omega^{<\omega}$ and $i_1 < \ldots < i_k < \omega$ we have that
\[
\inf_x(\max(\phi(x, a_{\eta^\frown i_1}), \ldots, \phi(x, a_{\eta^\frown i_1}))) \geq \varepsilon.
\]
\end{enumerate}
An $\L$-formula $\phi(x, y)$ has the \emph{tree property} (\TP) if there exists some natural number $k \geq 2$ such that $\phi(x, y)$ has $k$-\TP.

The theory $T$ has the \emph{tree property} (\TP) if there is a formula that has the tree property, and otherwise we say that $T$ is \NTP or \emph{simple}.
\end{definition}
\begin{proposition}
\thlabel{continuous-tp-iff-positive-tp}
An $\L$-formula $\phi(x, y)$ has $k$-\TP (in the sense of \thref{continuous-tree-property}) if and only if the $\L_\pos$-formula $R_\phi(x, y)$ has $k$-\TP (modulo $T_\pos$, in the sense of \thref{tree-property}).

In particular $T$ is simple if and only if $T_\pos$ is simple.
\end{proposition}
\begin{proof}
Throughout this proof we freely use the translation between $\L$-formulas and $\L_\pos$-formulas from \thref{continuous-logic-formula-translation}. In particular, for $k \geq 2$ we write $\psi(y_1, \ldots, y_k)$ for the $\L$-formula
\[
\inf_x(\max(\phi(x, y_1), \ldots, \phi(x, y_k))),
\]
so $R_{\psi}(y_1, \ldots, y_k)$ is equivalent to $\exists x(R_\phi(x, y_1) \wedge \ldots \wedge R_\phi(x, y_k))$ modulo $T_\pos$.

If $\phi(x, y)$ has $k$-\TP then the same tree of parameters witnesses $k$-\TP for $R_\phi(x, y)$. We take $R_{\psi \geq \varepsilon}(y_1, \ldots, y_k)$ as the witnessing obstruction of the formula $\exists x(R_\phi(x, y_1) \wedge \ldots \wedge R_\phi(x, y_k))$, where $\varepsilon$ is as in \thref{continuous-tree-property}.

Conversely, if $R_\phi(x, y)$ has $k$-\TP then this is witnessed by some tree of parameters and an obstruction $\theta(y_1, \ldots, y_k)$ for $\exists x(R_\phi(x, y_1) \wedge \ldots \wedge R_\phi(x, y_k))$. By \thref{continuous-negations} there is $\varepsilon > 0$ such that $T_\pos \models \forall y_1, \ldots, y_k(\theta(y_1, \ldots, y_k) \to R_{\psi \geq \varepsilon}(y_1, \ldots, y_k))$, which shows that the same tree of parameters, together with $\varepsilon$, witnesses $k$-\TP for $\phi(x, y)$.

The final claim about simplicity follows because in both cases being simple is defined as being \NTP, and because every $\L_\pos$-formula is equivalent (modulo $T_\pos$) to one of the form $R_\phi$.
\end{proof}
\begin{theorem}[Kim-Pillay for continuous logic]
\thlabel{continuous-kim-pillay}
The theory $T$ is simple if and only if there is an independence relation $\ind$ on $\U$ satisfying \textsc{invariance}, \textsc{monotonicity}, \textsc{normality}, \textsc{existence}, \textsc{full existence}, \textsc{base monotonicity}, \textsc{extension}, \textsc{symmetry}, \textsc{transitivity}, \textsc{finite character}, \textsc{local character} and \textsc{independence theorem}. Furthermore, in this case $\ind = \ind^d$.
\end{theorem}
\begin{remark}
\thlabel{continuous-kim-pillay-notes}
We make the following notes about \thref{continuous-kim-pillay}.
\begin{enumerate}[label=(\roman*)]
\item We have not defined what it means for two tuples of $\U$ to have the same type, but this can be understood as being in the same automorphism orbit.
\item If one does not wish to translate the definition of Lascar strong types to continuous logic then the \textsc{independence theorem} property can be weakened by only allowing $\lambda_T$-saturated models in the base, see also \thref{kim-pillay-weaken-independence-theorem}. By \thref{continuous-translation-needs-saturation} it does not matter here if we speak about saturation in the continuous sense or in the positive sense.
\item We have not defined dividing for continuous logic. There is a definition in \cite[Definition 14.11]{ben-yaacov_model_2008}, which is what one would expect. Importantly, both $\U$ and $\MM_\pos$ agree on when a type divides (pointing out, once more, that they agree on types in the first place). So there is no difference between computing the $\ind^d$ relation in $\U$ or $\MM_\pos$. This also means that simplicity of $T$ can be characterised as local character of $\ind^d$, see \thref{simplicity-equivalences}.
\end{enumerate}
\end{remark}
\begin{proof}[Proof of \thref{continuous-kim-pillay}]
Apply \thref{kim-pillay} to $T_\pos$, noting that $T_\pos$ is Hausdorff (\thref{continuous-universal-domain-yields-positive-monster}(iii)) and thus thick and $\MM_\pos$ and $\U$ agree on types (so $\ind$ has the same properties in either structure), while \thref{continuous-tp-iff-positive-tp} tells us that $T$ is simple if and only if $T_\pos$ is simple. The final claim follows because $\MM_\pos$ and $\U$ agree on what dividing means (see \thref{continuous-kim-pillay-notes}(iii)).
\end{proof}
\section{Bibliographic remarks}
\label{sec:bibliographic-remarks-examples}
We discuss the bibliographic remarks separately for each of the two sections in this chapter.

\subsection{Hyperimaginaries, Section \ref{sec:hyperimaginaries}}
The original motivation of \cite{ben-yaacov_positive_2003} was to create a model-theoretic framework where ``hyperimaginary elements could be adjoined as parameters to the language, the same way we used to do it with real and imaginary ones since the dawn of time''. In \cite[Example 2.16]{ben-yaacov_positive_2003} a brief description is given of how to add hyperimaginaries, which is exactly the construction that was worked out in this chapter. The first place where these details appear is in \cite[Section 10C]{dobrowolski_kim-independence_2022} (\thref{set-of-formulas-implies-formula-then-small-set-implies-it,type-definable-equivalence-relation-containing-formula,hyperimaginary-definable-with-small-hyperimaginaries} come from \cite[Section 3]{ben-yaacov_thickness_2003}). The present version is slightly easier and smoother because we do not notationally distinguish between real and hyperimaginary sorts (see \thref{hyperimaginary-projection-notation}).

The proof that simplicity is respected by adding hyperimaginaries (\thref{hyperimaginaries-respect-simplicity}) is essentially the same as \cite[Theorem 10.18]{dobrowolski_kim-independence_2022}, where a similar result is proved for \NSOP[1]. As is noted there as well, this proof works for any model-theoretic dividing line that is defined like \TP, \SOP[1], etc. So \thref{hyperimaginaries-respect-stability} could be proved in a similar way using the order property (see also \thref{stable-iff-no-order-property}). However, since we have not worked out the equivalence between not having the order property and stability, we chose to work with the type counting definition that we gave.

\subsection{Continuous logic, Section \ref{sec:continuous-logic}}
In \cite{ben-yaacov_model_2008} it is already mentioned that a continuous theory can be studied as a \emph{compact abstract theory}, which is the name that \cite{ben-yaacov_positive_2003,ben-yaacov_simplicity_2003,ben-yaacov_thickness_2003} use for what is essentially positive logic. In fact, \cite[Section 2]{ben-yaacov_positive_2003} gives us a recipe that allows us to cook up a positive theory from the data of a continuous theory. In Section \ref{sec:continuous-logic} we just work this out explicitly.

There is nothing special about the tree property in \thref{continuous-tp-iff-positive-tp}. We could translate various other combinatorial model-theoretic properties (such as \OP, \SOP[1], \SOP[2], \IP, \TP[1], \TP[2], etc.) in a similar fashion. That is, one takes their definition for positive logic (see e.g., \cite{dmitrieva_dividing_2023}) and replaces the existence of an obstruction $\chi'$ of some formula $\chi$ by the existence of an $\varepsilon > 0$ such that $\chi$ has value at least $\varepsilon$ wherever $\chi'$ was supposed to hold.

We could also translate the results concerning stability (i.e., Chapter \ref{ch:stable-theories}) similarly to how we translated those for simplicity (e.g., \thref{continuous-kim-pillay}), but this already appears in print in \cite[Section 14]{ben-yaacov_model_2008}.

\begin{appendices}
\renewcommand\chaptername{Appendix}
\newcommand{\lazybegin}[2]{\begin{#1}[\thref{#2}]}
\newcommand{\lazyend}[1]{\end{#1}}

\chapter{The lazy model-theoretician's guide to positive logic}
\chaptermark{Lazy guide to positive logic} 
\label{ch:lazy-model-theoritician}
We give an as-brief-as-possible summary of the necessities for positive logic. The ``proofs'' here are just sketches or indications of the main ingredients in the actual arguments. Each definition, lemma, proposition and theorem refers to the relevant counterpart within the notes, where the details can be found.

\section{Basics}
Definitions of signature (or language), structure and first-order formula are the same as usual. We work in some fixed signature $\L$, which we often drop from the notation. We will not distinguish tuples from single elements, so the notation $a \in M$ means that $a$ is some tuple in $M$ (and similarly for variables).
\lazybegin{definition}{positive-formulas-and-theory}
A \emph{positive formula} is one that is built from atomic formulas using $\top, \bot, \wedge, \vee$ and $\exists$.

An \emph{h-inductive sentence} is one of the form $\forall x(\phi(x) \to \psi(x))$, where $\phi(x)$ and $\psi(x)$ are positive formulas. An \emph{h-universal sentence} is an h-inductive sentence of the form $\forall x(\phi(x) \to \bot)$.

A \emph{positive theory} is a set of h-inductive sentences.
\lazyend{definition}
The signature does not necessarily contain a symbol for inequality.
\begin{convention}[Stay positive!]
We will drop the ``positive'' from terms from now on. That is, we will just say ``formula'' and ``theory'' instead of ``positive formula'' and ``positive theory'' respectively.
\end{convention}
\lazybegin{definition}{homomorphism-preserves-truth,immersion}
A function $f: M \to N$ between structures is called a \emph{homomorphism} if for every formula $\phi(x)$ and every $a \in M$ we have:
\[
M \models \phi(a) \implies N \models \phi(f(a)).
\]
We call $f$ an \emph{immersion} if for every formula $\phi(x)$ and every $a \in M$ we have:
\[
M \models \phi(a) \Longleftrightarrow N \models \phi(f(a)).
\]
\lazyend{definition}
The main objects of study in positive model theory are the positively closed models. In some literature these are also called existentially closed models (abbreviated as e.c.\ models) or positively existentially closed models (abbreviated as pec models).
\lazybegin{definition}{pc-model,equivalence-pc-characterisations}
We call a model $M$ of a theory $T$ a \emph{positively closed model}, or \emph{p.c.\ model}, if the following equivalent conditions hold:
\begin{enumerate}[label=(\roman*)]
\item every homomorphism $f: M \to N$ with $N \models T$ is an immersion;
\item for every $a \in M$ and $\phi(x)$ the following holds, if there is a homomorphism $f: M \to N$ with $N \models T$ and $N \models \phi(f(a))$ then already $M \models \phi(a)$;
\item for every $a \in M$ and $\phi(x)$ such that $M \not \models \phi(a)$ there is $\psi(x)$ such that $T \models \neg \exists x(\phi(x) \wedge \psi(x))$ and $M \models \psi(a)$.
\end{enumerate}
\lazyend{definition}
\begin{proof}[Proof of the above equivalence.]
Proving (iii) $\Rightarrow$ (i) $\Rightarrow$ (ii) is straightforward. For (ii) $\Rightarrow$ (iii), use compactness and the method of diagrams.
\end{proof}
It is useful to have a name for a formula like $\psi$ in (iii) above.
\lazybegin{definition}{negation}
Let $T$ be a theory and $\phi(x)$ be a formula. A formula $\psi(x)$ such that $T \models \neg \exists x(\phi(x) \wedge \psi(x))$ is called an \emph{obstruction of $\phi(x)$}.
\lazyend{definition}
What is called ``an obstruction of a formula'' here is called ``a negation of a formula'' in some literature.
\lazybegin{theorem}{directed-unions,pc-models-closed-under-directed-unions}
The class of models and the class of p.c.\ models of a theory $T$ are both closed under unions of chains.
\lazyend{theorem}
\begin{proof}
Show that a formula holds in the union of a chain of structures if and only if it holds in some structure in the chain. The statement for the class of models then follows, after which the statement about p.c.\ models is an easy consequence.
\end{proof}
The links in the chains in the above statement are homomorphisms. For p.c.\ models these are then automatically immersions (and thus injective maps), for arbitrary models these may not be injective so ``union'' should not be taken literally (see \thref{directed-system} for a precise definition). Furthermore, ``chains'' can equivalently be replaced by ``directed systems''.
\lazybegin{theorem}{continue-to-pc-model}
Every model $M$ of a theory $T$ can be continued to a p.c.\ model of $T$. That is, there is some p.c.\ model $N$ of $T$ with a homomorphism $f: M \to N$.
\lazyend{theorem}
\begin{proof}
Enumerate all possible formulas with parameters in $M$ as $(\phi_i)_{i < \kappa}$. Inductively construct a chain of models $(M_i)_{i < \kappa}$ with $M_0 = M$ as follows: if there is a homomorphism $M_i \to N$ with $N \models T$ and $N \models \phi_i$ then set $M_{i+1} = N$, otherwise set $M_{i+1} = M_i$. Call the union of this chain $M^1$ and repeat the process to find a chain $(M^i)_{i < \omega}$ such that its union is the desired $N$.
\end{proof}
\lazybegin{theorem}{compactness}
Let $T$ be a theory and let $\Sigma(x)$ be a set of positive formulas. Suppose that for every finite $\Sigma_0(x) \subseteq \Sigma(x)$ there is $M \models T$ with $a \in M$ such that $M \models \Sigma_0(a)$. Then there is a p.c.\ model $N$ of $T$ with $a \in N$ such that $N \models \Sigma(a)$.
\lazyend{theorem}
\begin{proof}
First use the usual compactness and then continue to a p.c.\ model.
\end{proof}
The category of p.c.\ models satisfies the amalgamation property.
\lazybegin{proposition}{amalgamation-bases}
Let $M \xleftarrow{f} M_0 \xrightarrow{g} M'$ be a span of p.c.\ models of some theory $T$. Then there exists a p.c.\ model $N$ of $T$ and $M \xrightarrow{f'} N \xleftarrow{g'} M'$ such that $f'f = g'g$.
\lazyend{proposition}
\begin{proof}
Use compactness and the method of diagrams to find an amalgamation and then continue to a p.c.\ model.
\end{proof}
\lazybegin{definition}{size-of-theory}
For a theory $T$ we let $|T|$ be the cardinality of the set of formulas, up to logical equivalence.
\lazyend{definition}
\lazybegin{theorem}{lowenheim-skolem}
Let $M$ be a p.c.\ model of $T$ and let $A \subseteq M$. Then there is a p.c.\ model $M_0 \subseteq M$ with $A \subseteq M_0$ and $|M_0| \leq |A| + |T|$, such that the inclusion is an elementary embedding.
\lazyend{theorem}
\begin{proof}
The $M_0$ exists by the usual L\"owenheim-Skolem. Combine this with the fact that if $M_1 \to M_2$ is an immersion with $M_2$ a p.c.\ model then $M_1$ is a p.c.\ model.
\end{proof}
There are many equivalent characterisations of a complete theory in full first-order logic. These are no longer equivalent in positive logic, but the important analogous property the following.
\lazybegin{definition}{jcp,equivalence-jcp-characterisations}
A theory $T$ is said to have the \emph{joint continuation property}, or \emph{JCP}, if the following equivalent conditions hold.
\begin{enumerate}[label=(\roman*)]
\item For any two models $M$ and $M'$ of $T$ there is a model $N$ of $T$ with homomorphisms $M \to N \leftarrow M'$.
\item For any two p.c.\ models $M$ and $M'$ of $T$ there is a model $N$ of $T$ with homomorphisms $M \to N \leftarrow M'$.
\item For any two h-universal sentences $\phi$ and $\psi$ we have that $T \models \phi \vee \psi$ implies $T \models \phi$ or $T \models \psi$.
\end{enumerate}
\lazyend{definition}
There are further characterisations similar to ``$T = \Th(M)$ for some model $M$'' (see \thref{jcp}), but they require definitions that we skip here.
\begin{proof}[Proof of the above equivalence.]
For (i) $\Leftrightarrow$ (ii), use that every model continues to a p.c.\ model. For (iii) $\Rightarrow$ (i) use compactness and the method of diagrams. For (i) $\Rightarrow$ (iii) use that truth of h-universal sentences is reflected by homomorphisms. 
\end{proof}
\section{Monster model and indiscernible sequences}
As is common in model theory, it is convenient to work in a monster model. The construction from full first-order logic of such a model goes through in positive logic (see \thref{building-saturated-homogeneous-model}). As usual, we need to fix a notion of smallness, for which we invite the reader to pick their favourite one (e.g., smaller than some fixed inaccessible cardinal).
\lazybegin{definition}{monster-model}
Let $T$ be a theory with JCP. A \emph{monster model} of $T$ is a model $\MM$ of $T$ that is:
\begin{itemize}
\item \underline{Positively closed}: $\MM$ is a p.c.\ model of $T$.
\item \underline{Very homogeneous}: any partial immersion $f: \MM \to \MM$ with small domain extends to an automorphism on all of $\MM$, equivalently any two small tuples $a$ and $b$ in $\MM$ satisfy the same set of formulas if and only if there is an automorphism $f$ of $\MM$ such that $f(a) = b$.
\item \underline{Very saturated}: any small set of formulas with parameters in $\MM$ that is finitely satisfiable in $\MM$ is satisfiable in $\MM$.
\end{itemize}
\lazyend{definition}
\begin{mdframed}
\Large{\begin{convention}
From now on we work in a monster model $\MM$, so all p.c.\ models, tuples and sets are assumed to be small and to live in $\MM$.
\end{convention}}
\end{mdframed}
\begin{convention}
We generally omit the monster model $\MM$ from the notation. We also fix the following notation. Everything is small unless explicitly mentioned otherwise.
\begin{itemize}
\item For a tuple $a$ and a set $B$ we write
\[
\tp(a/B) = \{ \phi(x, b) : b \in B \text{ and } \models \phi(a, b)  \}
\]
for the set of formulas over $B$ that are satisfied by $a$, and we call this the \emph{type of $a$ over $B$}.
\item We write $a \equiv_B a'$ to mean $\tp(a/B) = \tp(a'/B)$.
\item We write $\Aut(\MM/B)$ for the set of autmorphisms of $\MM$ that fix $B$ pointwise. So by homogeneity we have $a \equiv_B a'$ if and only if there is $f \in \Aut(\MM/B)$ with $f(a) = a'$.
\end{itemize}
\end{convention}
We stress that a \emph{type} for us is a maximally consistent set of formulas. That is, a set of formulas $p(x)$ of the form $\tp(a/B)$. An arbitrary consistent set of formulas is called a \emph{partial type} (see also Section \ref{sec:types-and-type-spaces}).

Sometimes smaller saturated p.c.\ models will be useful.
\lazybegin{definition}{saturated-model,saturation-extend-variables}
Let $\kappa$ be an infinite cardinal. A structure $M$ is called \emph{positively $\kappa$-saturated} if, for every $A \subseteq M$ with $|A| < \kappa$, every set $\Sigma(x)$ of formulas over $A$, with $|x| \leq \kappa$, that is finitely satisfiable in $M$ is satisfiable in $M$.
\lazyend{definition}
\lazybegin{proposition}{building-saturated-model}
Let $A$ be any parameter set. Then for all $\kappa \geq |A| + |T|$ there is a positively $\kappa^+$-saturated p.c.\ model $M$ of $T$ with $|M| \leq 2^\kappa$ and $A \subseteq M$.
\lazyend{proposition}
\begin{proof}
By downward L\"owenheim-Skolem we may assume $A$ to be a p.c.\ model. Then inductively construct a chain of p.c.\ models such that the next link realises all finitely satisfiable sets of formulas over $< \kappa$ parameters in the current link. A proper choice of cardinals in this process yields $M$ as the union of this chain.
\end{proof}
\lazybegin{definition}{indiscernible-sequence-over-parameters}
Let $B$ be a set of parameters. An \emph{indiscernible sequence over $B$} is an infinite sequence $(a_i)_{i \in I}$ such that for any $i_1 < \ldots < i_n$ and $j_1 < \ldots < j_n$ in $I$ we have
\[
a_{i_1} \ldots a_{i_n} \equiv_B a_{j_1} \ldots a_{j_n}.
\]
We will also abbreviate this as a \emph{$B$-indiscernible sequence}.
\lazyend{definition}
Indiscernible sequences are often constructed by first constructing a very long sequence and then using the lemma below to find some indiscernible sequence that is \emph{based on} the very long sequence.
\lazybegin{definition}{lambda-t}
Write $\lambda_\kappa = \beth_{(2^\kappa)^+}$ for any cardinal $\kappa$ and $\lambda_T = \lambda_{|T|}$.
\lazyend{definition}
\lazybegin{lemma}{base-indiscernible-sequence-on-long-sequence}
Let $B$ be any parameter set and let $\kappa$ be any cardinal. Then for any sequence $(a_i)_{i \in I}$ of $\kappa$-tuples with $|I| \geq \lambda_{|T| + |B| + \kappa}$ there is a $B$-indiscernible sequence $(a_i')_{i < \omega}$ such that for all $n < \omega$ there are $i_1 < \ldots < i_n$ in $I$ with $a'_1 \ldots a'_n \equiv_B a_{i_1} \ldots a_{i_n}$.
\lazyend{lemma}
\begin{proof}
Erd\H{o}s-Rado.
\end{proof}

\section{Boolean, (semi-)Hausdorff and thick}
\lazybegin{definition}{morleyisation}
Given a \emph{positive fragment} $\Delta$ of our signature (i.e., a set of full first-order formulas, closed under sub-formulas, change of variables, conjunction and disjunction) we define the \emph{($\Delta$-)Morleyisation} $\Mor(\Delta)$ to be the following theory. We extend the signature by a relation symbol $R_\phi(x)$ for each $\phi(x) \in \Delta$ and let $\Mor(\Delta)$ express that $R_\phi(x)$ and $\phi(x)$ are equivalent.
\lazyend{definition}
To make sure $\Mor(\Delta)$ is a \emph{positive} theory, the equivalence of $\phi(x)$ and $R_\phi(x)$ has to axiomatised by induction on the complexity of $\phi(x)$ (see \thref{morleyisation-does-exist}).

By taking $\Delta$ to be the set of all full first-order formulas, every full first-order formula becomes a positive formula (modulo $\Mor(\Delta)$). In this way we can study full first-order logic using positive logic.
\lazybegin{definition}{boolean-theory,equivalence-boolean-characterisations}
We call a theory $T$ \emph{Boolean} if the following equivalent conditions hold.
\begin{enumerate}[label=(\roman*)]
\item Every model of $T$ is a p.c.\ model.
\item Every homomorphism between models of $T$ is an immersion.
\item For every positive formula $\phi(x)$ there is a positive formula $\psi(x)$ such that $T \models \forall x(\neg \phi(x) \leftrightarrow \psi(x))$.
\item For every full first-order formula $\phi(x)$ there is a positive formula $\psi(x)$ such that $T \models \forall x(\phi(x) \leftrightarrow \psi(x))$.
\item Every homomorphism between models of $T$ is an elementary embedding.
\end{enumerate}
\lazyend{definition}
\begin{proof}[Proof of the above equivalence.]
The implications (i) $\Rightarrow$ (ii) and (iv) $\Rightarrow$ (v) $\Rightarrow$ (i) are straightforward. For (ii) $\Rightarrow$ (iii) one proves the following using compactness and the method of diagrams: if $\phi(x)$ is a full first-order formula whose truth is preserved by all homomorphisms of models of $T$ then $\phi(x)$ is equivalent to a positive formula modulo $T$. Finally, (iii) $\Rightarrow$ (iv) is by induction on the complexity of the formula.
\end{proof}
\lazybegin{definition}{hausdorff-thick}
Let $T$ be a theory with JCP. We call $T$:
\begin{itemize}
\item \emph{Hausdorff} if for any two distinct types $p(x)$ and $q(x)$ there are $\phi(x) \not \in p(x)$ and $\psi(x) \not \in q(x)$ such that $\models \forall x(\phi(x) \vee \psi(x))$;
\item \emph{semi-Hausdorff} if equality of types is type-definable, so there is a partial type $\Omega(x, y)$ such that $\tp(a) = \tp(b)$ if and only if $\models \Omega(a, b)$;
\item \emph{thick} if being an indiscernible sequence is type-definable, so there is a partial type $\Theta((x_i)_{i < \omega})$ such that a sequence $(a_i)_{i < \omega}$ is indiscernible if and only if $\models \Theta((a_i)_{i < \omega})$.
\end{itemize}
\lazyend{definition}
The above definitions can be made sense of for any theory $T$, so without assuming JCP. In this case we need to refer to all p.c.\ models of $T$ rather than just the single monster model.
\lazybegin{proposition}{hausdorff-implies-thick,hausdorff-not-boolean,semi-hausdorff-not-hausdorff,thick-not-semi-hausdorff,not-thick}
Boolean implies Hausdorff implies semi-Hausdorff implies thick. None of these implications are reversible and there are non-thick positive theories.
\lazyend{proposition}
\begin{proof}
For Boolean $\Rightarrow$ Hausdorff pick any formula $\phi(x)$ that is in one type and not in the other and then $\neg \phi(x)$ takes the role of $\psi(x)$. To prove Hausdorff $\Rightarrow$ semi-Hausdorff $\Rightarrow$ thick we take
\[
\{\phi(x, y) :\; \text{for all } a, b \text{ with } \tp(a) = \tp(b) \text{ we have } \models \phi(a, b)\}
\]
and
\[
\bigcup \{ \Omega(x_{i_1}, \ldots, x_{i_n}; x_{j_1}, \ldots, x_{j_n}) : n < \omega, i_1 < \ldots < i_n < \omega, j_1 < \ldots < j_n < \omega \}.
\]
for $\Omega(x, y)$ and $\Theta((x_i)_{i < \omega})$ respectively. Finally, the referenced examples show that none of the implications are reversible.
\end{proof}
\lazybegin{proposition}{hausdorff-equivalences}
A theory $T$ with JCP is Hausdorff if and only if the following amalgamation property holds. For any span $M \xleftarrow{f} M_0 \xrightarrow{g} M'$ of models of the theory
\[
T' = \{ \chi \text{ an h-inductive sentence} : \MM \models \chi \}
\]
there is an amalgam $M \xrightarrow{f'} N \xleftarrow{g'} M'$, so $f'f = g'g$, with $N \models T'$.
\lazyend{proposition}
As before, there is a version of the above without assuming JCP.
\begin{proof}
First use compactness to prove the following intermediate statement (see \thref{separate-types-equivalences}). For any $M \models T'$ and $a \in M$ there is a type $p(x)$ such that for any homomorphism $f: M \to N$ where $N$ is a p.c.\ model of $T'$ we have that $N \models p(f(a))$. The equivalence of this intermediate statement with the amalgamation property for models of $T'$ is a straightforward argument involving compactness and the method of diagrams.
\end{proof}

\section{Simple theories}
\lazybegin{definition}{dividing}
Let $\Sigma(x, b)$ be a set of formulas over $Cb$. We say that $\Sigma(x, b)$ \emph{divides over $C$} if there is a $C$-indiscernible sequence $(b_i)_{i < \omega}$ with $b_i \equiv_C b$ for all $i < \omega$ such that $\bigcup_{i < \omega} \Sigma(x, b_i)$ is inconsistent.
\lazyend{definition}
\lazybegin{proposition}{dividing-in-terms-of-automorphic-indiscernible-sequences}
The following are equivalent:
\begin{enumerate}[label=(\roman*)]
\item $\tp(a/Cb)$ does not divide over $C$;
\item for every $C$-indiscernible sequence $(b_i)_{i < \omega}$ with $b_0 = b$ there is a $Ca$-indiscernible sequence $(b'_i)_{i < \omega}$ with $(b'_i)_{i < \omega} \equiv_{Cb} (b_i)_{i < \omega}$;
\item for every $C$-indiscernible sequence $(b_i)_{i < \omega}$ with $b_0 = b$ there is $a' \equiv_{Cb} a$ such that $(b_i)_{i < \omega}$ is $Ca'$-indiscernible.
\end{enumerate}
\lazyend{proposition}
\begin{proof}
This mostly comes down to moving things around with automorphisms. For (i) $\Rightarrow$ (ii) and (iii) we also use compactness to elongate sequences so that we can base new indiscernible sequences on them.
\end{proof}
\lazybegin{definition}{lascar-distance-and-strong-type,same-lstp-iff-same-types-over-sequence-of-models}
Assume thickness. Let $a$ and $a'$ be two tuples of the same length and let $B$ be any parameter set. We say that $a$ and $a'$ have the same \emph{Lascar strong type over $B$} and write $a \equivls_B a'$ if the following equivalent conditions hold.
\begin{enumerate}[label=(\roman*)]
\item There are $a = a_0, \ldots, a_n = a'$ such that $a_i$ and $a_{i+1}$ are on a $B$-indiscernible sequence for all $0 \leq i < n$.
\item There are $a = a_0, \ldots, a_n = a'$ and positively $\lambda_T$-saturated p.c.\ models $M_1, \ldots, M_n$ such that $a_i \equiv_{M_i} a_{i+1}$ for all $0 \leq i < n$.
\end{enumerate}
\lazyend{definition}
\begin{proof}[Proof of the above equivalence.]
We need the following fact: given $C \supseteq B$ and any $B$-indiscernible sequence $(a_i)_{i < \omega}$ there is $C'$ with $C' \equiv_B C$ such that $(a_i)_{i < \omega}$ is $C'$-indiscernible (see \thref{extend-base-set-of-indiscernible-sequence}). This is proved by elongating the original sequence, basing a new indiscernible sequence on it and applying an automorphism. To prove (i) $\Rightarrow$ (ii) we let $M \supseteq B$ be a positively $\lambda_T$-saturated p.c.\ model and repeatedly apply the preceding fact with $M$ in the role of $C$. The converse is a compactness argument using thickness (see \thref{same-type-over-saturated-model-lascar-distance-2}).
\end{proof}
We assumed thickness because we will only need Lascar strong types in that context. Another equivalent definition of Lascar strong types can be given in terms of bounded invariant equivalence relations. The equivalence of that condition to (i) above holds in any theory (see \thref{equivalence-lascar-strong-type-conditions}).
\lazybegin{definition}{independence-relation}
An \emph{independence relation} $\ind$ is a ternary relation on small subsets of the monster model. If $A$, $B$ and $C$ are in the relation we write
\[
A \ind_C B,
\]
which should be read as ``$A$ is independent from $B$ over $C$''. We also allow tuples in the relation, which are then interpreted as the set they enumerate.
\lazyend{definition}
\lazybegin{definition}{independence-properties}
Let $\ind$ be an independence relation. We define the following properties for $\ind$, where $a$ and $b$ are arbitrary tuples and $C$ is an arbitrary set.
\begin{description}
\item[\textsc{invariance}] For any $f \in \Aut(\MM)$ we have that $a \ind_C b$ implies $f(a) \ind_{f(C)} f(b)$.
\item[\textsc{monotonicity}] For any $a' \subseteq a$ and $b' \subseteq b$ we have that $a \ind_C b$ implies $a' \ind_C b'$.
\item[\textsc{normality}] If $a \ind_C b$ then $Ca \ind_C Cb$.
\item[\textsc{existence}] We always have $a \ind_C C$.
\item[\textsc{full existence}] There is always $b'$ with $b' \equiv_C b$ such that $a \ind_C b'$.
\item[\textsc{base monotonicity}] If $C \subseteq C' \subseteq b$ then $a \ind_C b$ implies $a \ind_{C'} b$.
\item[\textsc{extension}] If $a \ind_C b$ then for any $d$ there is $d'$ with $d' \equiv_{Cb} d$ and $a \ind_C bd'$.
\item[\textsc{symmetry}] If $a \ind_C b$ then $b \ind_C a$.
\item[\textsc{transitivity}] If $C \subseteq C'$ with $a \ind_C C'$ and $a \ind_{C'} b$ then $a \ind_C b$.
\item[\textsc{finite character}] If for all finite $a' \subseteq a$ and all finite $b' \subseteq b$ we have $a' \ind_C b'$ then $a \ind_C b$.
\item[\textsc{local character}] For every cardinal $\kappa$ there is a cardinal $\lambda$ such that for all $a$ with $|a| < \kappa$ and any $C$ there is $C' \subseteq C$ with $|C'| < \lambda$ and $a \ind_{C'} C$.
\item[\textsc{independence theorem}] If $a \ind_C b$, $a' \ind_C c$ and $b \ind_C c$ with $a \equivls_C a'$ then there is $a''$ with $a'' \equivls_{Cb} a$ and $a'' \equivls_{Cc} a'$ such that $a'' \ind_C bc$.
\item[\textsc{stationarity}] For any $C$ such that $a \equiv_C a'$ implies $a \equivls_C a'$ for all $a, a'$, we have that $a \ind_C b$, $a' \ind_C b$ and $a \equiv_C a$ implies $a \equiv_{Cb} a'$.
\end{description}
\lazyend{definition}
\lazybegin{definition}{non-dividing-independence}
Let $A, B, C$ be sets and let $a$ and $b$ enumerate $A$ and $B$ respectively. Then we write
\[
A \ind^d_C B
\]
if $\tp(a/Cb)$ does not divide over $C$. We call this relation \emph{dividing independence}.
\lazyend{definition}
\lazybegin{theorem}{dividing-basic-properties}
Dividing independence satisfies the following properties: \textsc{invariance}, \textsc{monotonicity}, \textsc{normality}, \textsc{existence}, \textsc{base monotonicity}, \textsc{finite character} and \textsc{left transitivity}. This final property is the same as \textsc{transitivity} with the sides of the independence relation swapped: if $C \subseteq C'$ then $C' \ind^d_C b$ and $a \ind^d_{C'} b$ implies $a \ind^d_C b$.
\lazyend{theorem}
\begin{proof}
All of this is standard manipulation of indiscernible sequences.
\end{proof}
\lazybegin{definition}{tree-property}
\thlabel{lazy-tree-property}
Let $k \geq 2$ be a natural number. A formula $\phi(x,y)$ is said to have the \emph{$k$-tree property} ($k$-\TP) if there are $(a_\eta)_{\eta \in \omega^{< \omega}}$ and an obstruction $\psi(y_1, \ldots, y_k)$ of the formula $\exists x (\phi(x, y_1) \wedge \ldots \wedge \phi(x, y_k))$ such that:
\begin{enumerate}[label=(\roman*)]
\item for all $\sigma \in \omega^\omega$ the set $\{ \phi(x, a_{\sigma|_n}): n < \omega\}$ is consistent,
\item for all $\eta \in \omega^{< \omega}$ and $i_1 < \ldots < i_k < \omega$ we have $\models \psi(a_{\eta^\frown i_1}, \ldots, a_{\eta^\frown i_k})$.
\end{enumerate}

A formula $\phi(x,y)$ has the \emph{tree property} (\TP) if there exists a natural number $k \geq 2$ such that $\phi(x,y)$ has $k$-\TP.

A theory has the \emph{tree property} (\TP) if there is a formula that has the tree property, and otherwise it is \NTP.
\lazyend{definition}
\lazybegin{theorem}{simplicity-equivalences}
\thlabel{lazy-simplicity-equivalences}
A theory $T$ is \NTP if and only if dividing independence $\ind^d$ satisfies \textsc{local character}.
\lazyend{theorem}
There are more detailed statements in \thref{simplicity-equivalences}, referring to the exact cardinals involved in \textsc{local character}.
\begin{proof}
Both directions are proved by contrapositive. Assuming \TP, say as witnessed by $\phi(x, y)$, one first uses compactness to make the tree as big as necessary (this is why (ii) in \thref{lazy-tree-property} is important). Then inductively construct a branch in the tree, which results in a sequence $(b_i)_{i < \lambda}$ and a realisation $a$ of $\{\phi(x, b_i) : i < \lambda\}$ (see \thref{lazy-tree-property}(i)) such that $\tp(a/(b_i)_{i < \lambda})$ divides over $(b_i)_{i < \gamma}$ for all $\gamma < \lambda$, contradicting \textsc{local character}. The sequence witnessing dividing is found every time as a subsequence of the immediate successors of each $b_i$.

Conversely, failure of \textsc{local character} implies that there are $a$ and $C$ such that $\tp(a/C)$ does not divide over $C'$ for all $C' \subseteq C$ with $|C'| < |T|^+$. We use this to inductively build a tree of height $|T|^+$ whose zero branch consists of tuples in $C$ and such that all branches have the same type. In the induction step we let $C'$ be the zero branch constructed so far, so $\tp(a/C)$ divides over $C'$. This yields an indiscernible sequence along which some $\psi(y_1, \ldots, y_k)$ holds that is an obstruction of $\exists x(\phi(x, y_1) \wedge \ldots \wedge \phi(x, y_k))$ for some $\phi(x, C) \in \tp(a/C)$. This indiscernible sequence will form the next level above the zero branch, and automorphic copies complete the levels above the other branches. As the tree has height $|T|^+$ we can use the pigeonhole principle the restrict to a subtree with a uniform choice of $\phi$ and $\psi$, which is exactly saying that $\phi(x, y)$ has \TP.
\end{proof}
\lazybegin{definition}{simplicity}
We call a theory $T$ \emph{simple} if the equivalent conditions from \thref{lazy-simplicity-equivalences} hold.
\lazyend{definition}
\lazybegin{theorem}{simple-thick-implies-full-existence}
Assume thickness. If $T$ is simple then dividing independence has \textsc{full existence}.
\lazyend{theorem}
\begin{proof}
This is a long and technical proof. The key is that the thickness assumption makes it so that a certain technical concept, namely that of a dividing sequence, type-definable. We need infinite dividing sequences, but to construct these we only need to deal with finite dividing sequences by type-definability and compactness.
\end{proof}
The importance of \textsc{full existence} is that we can build Morley sequences and prove Kim's lemma (\thref{lazy-kims-lemma}).
\lazybegin{definition}{dividing-morley-sequence}
A \emph{Morley sequence (over $C$)} is a $C$-indiscernible sequence $(a_i)_{i \in I}$ such that $a_i \ind^d_C (a_j)_{j < i}$ for all $i \in I$.
\lazyend{definition}
\lazybegin{proposition}{morley-sequences-exist}
\thlabel{lazy-morley-sequences-exist}
Assume thickness. If $T$ is simple then for any $a$ and $C$ there is a Morley sequence $(a_i)_{i < \omega}$ over $C$ with $a_0 = a$.
\lazyend{proposition}
\begin{proof}
Use \textsc{full existence} to find a long enough $\ind^d_C$-independent sequence (i.e., a Morley sequence without the indiscernibility). Then base a $C$-indiscernible sequence on it.
\end{proof}
\lazybegin{theorem}{kims-lemma}
\thlabel{lazy-kims-lemma}
Suppose that $T$ is simple and let $\Sigma(x, b)$ be a set of formulas over $Cb$. If $\bigcup_{i < \omega} \Sigma(x, b_i)$ is consistent for some Morley sequence $(b_i)_{i < \omega}$ over $C$ with $b_0 = b$ then $\Sigma(x, b)$ does not divide over $C$.

In particular, assuming thickness, we have that $\Sigma(x, b)$ divides over $C$ if and only if there is a Morley sequence $(b_i)_{i < \omega}$ with $b_0 = b$ such that $\bigcup_{i < \omega} \Sigma(x, b_i)$ is inconsistent.
\lazyend{theorem}
\begin{proof}
The first part is proved by usual argument that relies on \textsc{local character}. The second part then follows from the existence of Morley sequences (\thref{lazy-morley-sequences-exist}).
\end{proof}
\lazybegin{theorem}{dividing-extension,dividing-extension-partial-types}
Assume thickness. If $T$ is simple then given a partial type $\Sigma(x, b)$ that does not divide over $C$ there is a type $p(x, b) \supseteq \Sigma(x, b)$ that does not divide over $C$.

In particular, dividing independence satisfies \textsc{extension}.
\lazyend{theorem}
\begin{proof}
Take a Morley sequence over $C$ starting with $b$ and use compactness to elongate it to $(b_i)_{i < \lambda}$. By non-dividing of $\Sigma(x, b)$ there is a realisation $a$ of $\bigcup_{i < \lambda} \Sigma(x, b_i)$. By the pigeonhole principle there is an infinite subsequence $I \subseteq \lambda$ such that $a b_i \equiv_C a b_j$ for all $i,j \in I$, so taking $\tp(a/Cb_i)$ for $i \in I$ results in the required type.
\end{proof}
The properties \textsc{symmetry}, \textsc{transivitiy} and \textsc{independence theorem} also hold in thick simple theories (see \thref{dividing-symmetry,dividing-transitivity,independence-theorem}). We summarise everything in a Kim-Pillay style theorem. The proofs of all these theorems are analogous to the full first-order setting.
\lazybegin{theorem}{kim-pillay}
\thlabel{lazy-kim-pillay}
Assume thickness. A theory $T$ is simple if and only if there is an independence relation $\ind$ satisfying \textsc{invariance}, \textsc{monotonicity}, \textsc{normality}, \textsc{existence}, \textsc{full existence}, \textsc{base monotonicity}, \textsc{extension}, \textsc{symmetry}, \textsc{transitivity}, \textsc{finite character}, \textsc{local character} and \textsc{independence theorem}. Furthermore, in this case, $\ind = \ind^d$.
\lazyend{theorem}
We have decided to only treat simplicity for thick theories, which simplifies the treatment and allows us to stay closer to the treatment in full first-order logic. This still captures a large class of positive theories, see for example \cite[Section 2]{kamsma_positive_2024} for a list of (classes of) examples of thick theories. Even though much is still possible without the thickness assumption, \textsc{full existence} may fail \cite[Example 4.3]{ben-yaacov_simplicity_2003} and the treatment of simplicity becomes much more complicated (see also \thref{full-existence-without-thickness}).
\section{Stable theories}
\lazybegin{definition}{stability}
Let $\lambda$ be an infinite cardinal. A theory $T$ is called \emph{$\lambda$-stable} if for all parameter sets $B$ with $|B| \leq \lambda$ there are at most $\lambda$ many types in finitely many variables over $B$. We call $T$ \emph{stable} if it is $\lambda$-stable for some $\lambda$.
\lazyend{definition}
\lazybegin{definition}{definable-type}
Let $p(x)$ be a type over $B$ and let $\phi(x, y)$ be a formula without parameters. A \emph{$\phi$-definition over $C$} of $p(x)$ is a set of formulas $\d_p \phi(y)$ over $C$ with $|\d_p \phi(y)| \leq |T|$ such that for all $b \in B$ we have
\[
\phi(x, b) \in p(x) \quad \Longleftrightarrow \quad \models \d_p \phi(b).
\]
We say that $p(x)$ is \emph{definable over $C$} if it has a $\phi$-definition over $C$ for every formula $\phi(x, y)$. If $p(x)$ is definable over $B$ then we just say that $p(x)$ is \emph{definable}.
\lazyend{definition}
\lazybegin{definition}{binary-tree-rank,binary-tree-rank-lemma}
A formula $\phi(x, y)$ is said to have the \emph{binary tree property} if there is an obstruction $\psi(x, y)$ and parameters $(b_\eta)_{\eta \in 2^{< \omega}}$ such that for all $\sigma \in 2^\omega$ the set
\[
\{ \phi(x, b_{\sigma|_n}) : \sigma(n) = 0 \} \cup \{ \psi(x, b_{\sigma|_n}) : \sigma(n) = 1 \}
\]
is consistent.
\lazyend{definition}
\lazybegin{definition}{order-property}
A formula $\phi(x, y)$ has the \emph{order property} (\OP) if there are sequences $(a_i)_{i < \omega}$ and $(b_i)_{i < \omega}$ and an obstruction $\psi(x, y)$ of $\phi(x, y)$ such that for all $i,j < \omega$, we have
\begin{align*}
&\models \phi(a_i, b_j) &\text{if } i < j, \\
&\models \psi(a_i, b_j) &\text{if } i \geq j.
\end{align*}
\lazyend{definition}
\lazybegin{theorem}{stable-characterisations,stable-iff-no-order-property}
The following are equivalent for a theory $T$:
\begin{enumerate}[label=(\roman*)]
\item $T$ is stable,
\item no formula has the binary tree property,
\item every type is definable,
\item $T$ is $\lambda$-stable for every $\lambda$ with $\lambda^{|T|} = \lambda$,
\item no formula has the order property.
\end{enumerate}
\lazyend{theorem}
\begin{proof}
The proofs closely follow the standard proofs from full first-order logic, only (v) $\Rightarrow$ (i) is much more involved (and is not treated in these notes, see Section \ref{sec:bibliographic-remarks-stable} and \cite[Lemma 3.18]{dmitrieva_dividing_2023}). For the proofs of (i) $\Rightarrow$ (ii) and (i) $\Rightarrow$ (v) one proves the contrapositive by constructing many different types. For (ii) $\Rightarrow$ (iii) the key insight is that having the binary tree property (with respect to a fixed $\psi$) is type-definable. So by compactness we get a finite bound on the height of such trees. The existence of such a finite binary tree (again, with respect to fixed $\psi$) can be expressed by just a formula. Taking these formulas, while $\psi$ ranges over the obstructions of $\phi$, yields the required definition $\d_p \phi(y)$. Finally, for (iii) $\Rightarrow$ (iv) we simply count the number of possible definitions over a fixed parameter set, which bounds the number of possible types.
\end{proof}
It is also possible to establish the above equivalences on a formula-by-formula basis (e.g., a formula does not have the binary tree property if and only if it does not have the order property), which is done in \cite[Theorem 3.11]{dmitrieva_dividing_2023}.
\lazybegin{theorem}{stable-implies-simple}
Every stable theory is simple.
\lazyend{theorem}
\begin{proof}
First show that, roughly, definable types do not divide over the parameter set over which they are definable (see \thref{definable-types-do-not-divide}). Then note that there is a bound on the cardinality of the parameters that are needed to define a type, and so we get \textsc{local character} for dividing independence.
\end{proof}
\lazybegin{definition}{stationary-type}
A \emph{stationary type} is a type $p(x) = \tp(a/C)$ that admits exactly one non-dividing extension to any parameter set. That is, for any $B \supseteq C$, there is a type $p'(x) \supseteq p(x)$ over $B$ such that:
\begin{enumerate}[label=(\roman*)]
\item $p'(x)$ does not divide over $C$;
\item for any type $r(x) \supseteq p(x)$ over $B$ that does not divide over $C$ we have $r(x) = p'(x)$.
\lazyend{enumerate}
\end{definition}
\lazybegin{theorem}{stable-theory-stationary-iff-lstp}
Assume thickness. If $T$ is a stable theory then $\tp(a/C)$ is stationary if and only if we have for all $a'$ that $a \equiv_C a'$ implies $a \equivls_C a'$. In particular, $\ind^d$ satisfies \textsc{stationarity} in stable theories.
\lazyend{theorem}
\begin{proof}
First we note that $T$ is simple, and so we can use all the properties of $\ind^d$ as summarised in \thref{lazy-kim-pillay}. For the left to right we extend $\tp(a/C)$ to a global type, which can be shown to be $C$-invariant (see \thref{invariant-type}), from which the conclusion follows in a standard manner (see \thref{type-extends-to-global-invariant-implies-lstp}). For the converse we assume for a contradiction that there are two distinct non-dividing extensions of $\tp(a/C)$. Then, mainly using the \textsc{independence theorem} for $\ind^d$, we can inductively build a binary tree of non-dividing distinct types, ultimately yielding many distinct types and contradicting stability.
\end{proof}
\lazybegin{theorem}{simple-is-stable-iff-stationarity}
Assume thickness. If $T$ is simple and $\ind^d$ satisfies \textsc{stationarity} then $T$ is stable.
\lazyend{theorem}
\begin{proof}
By \textsc{local character} of $\ind^d$ there is a bound $\lambda$, such that any type (in finitely many variables) does not divide over some parameter set of cardinality $\leq \lambda$. At the same time, for any type $p(x)$ over a parameter set of cardinality $\leq \lambda$ there is a bound on the number of non-dividing extensions, by \textsc{stationarity}. Combining this yields the type counting definition of stability.
\end{proof}
\lazybegin{theorem}{stable-kim-pillay}
Assume thickness. A theory $T$ is stable if and only if there is an independence relation $\ind$ satisfying \textsc{invariance}, \textsc{monotonicity}, \textsc{normality}, \textsc{existence}, \textsc{full existence}, \textsc{base monotonicity}, \textsc{extension}, \textsc{symmetry}, \textsc{transitivity}, \textsc{finite character}, \textsc{local character} and \textsc{stationarity}. Furthermore, in this case, $\ind = \ind^d$.
\lazyend{theorem}
\begin{proof}
This is just piecing together previous results, mainly relying on \thref{lazy-kim-pillay}, with as the only new ingredient that \textsc{stationarity} implies \textsc{independence theorem} (see \thref{stationarity-implies-independence-theorem}).
\end{proof}
\end{appendices}

\clearpage
\phantomsection
\addcontentsline{toc}{chapter}{Bibliography}
\bibliographystyle{alpha}
\bibliography{bibfile}

\clearpage
\phantomsection
\addcontentsline{toc}{chapter}{\nomname}
\printnomenclature[3.25cm]

\clearpage
\phantomsection
\addcontentsline{toc}{chapter}{Index of terms}
\printindex


\end{document}